\documentclass[11pt]{amsart}
\usepackage[centering, margin=1in]{geometry}                
\usepackage{graphicx}
\usepackage{bm}

\usepackage[all,cmtip]{xy}
\usepackage{dsfont}
\usepackage{amssymb}
\usepackage{epstopdf}
\usepackage{url}
\usepackage{xcolor}
\usepackage{hyperref}
\definecolor{couleur_cite}{rgb}{0.05,.4,0.05}
\definecolor{couleur_link}{rgb}{0.05,0.05,0.4}
\hypersetup{
bookmarksopen,bookmarksnumbered,colorlinks,linkcolor=couleur_link,citecolor=couleur_cite,
pdfauthor=Farrell Brumley and Djordje Milicevic,pdftitle=Counting cusp forms by analytic conductor}
\usepackage{enumerate}
\usepackage{exscale} 
\usepackage{wasysym}
\usepackage{caption}
\usepackage{overpic}
\usepackage{stmaryrd}
\DeclareFontFamily{U}{mathx}{}
\DeclareFontShape{U}{mathx}{m}{n}{<-> mathx10}{}
\DeclareSymbolFont{mathx}{U}{mathx}{m}{n}
\DeclareMathAccent{\widecheck}{0}{mathx}{"71}

\usepackage[active]{srcltx} 
\usepackage{mathrsfs}

\usepackage{soul}
\usepackage[normalem]{ulem}

\def\aa{\mathfrak{a}}
\newcommand\dd{\mathfrak{d}}
\newcommand\h{\mathfrak{h}}
\newcommand\ee{\mathfrak{e}}
\newcommand{\m}{\mathfrak{m}}
\newcommand\nn{\mathfrak{n}}
\newcommand{\p}{\mathfrak{p}}
\newcommand{\pp}{\mathfrak{p}}
\newcommand{\q}{\mathfrak{q}}

\newcommand{\A}{\mathbb{A}}
\newcommand{\C}{\mathbb{C}}
\newcommand{\N}{\mathbb{N}}
\newcommand{\PP}{\mathbb{P}}
\newcommand{\Q}{\mathbb{Q}}
\newcommand{\R}{\mathbb{R}}
\newcommand{\Z}{\mathbb{Z}}

\newcommand{\GL}{\mathrm{GL}}
\newcommand{\PGL}{\mathrm{PGL}}

\newcommand{\SL}{\mathrm{SL}}
\newcommand{\SO}{\mathrm{SO}}

\newcommand{\U}{\mathrm{U}}

\def\norm{\mathbf{N}}
\def\bfk{\mathbf{K}}

\newcommand{\Ad}{\mathrm{Ad}}
\def\tr{\mathop{\mathrm{tr}}}

\def\mRe{\mathop{\mathrm{Re}}}
\def\mIm{\mathop{\mathrm{Im}}}

\def\rest{\mathop{\mathrm{Res}}}
\def\pl{\mathrm{pl}}
\def\vol{\mathop{\textnormal{vol}}\nolimits}

\theoremstyle{definition}
\newtheorem{remark}{Remark}
\newtheorem{remarks}[remark]{Remarks}
\newtheorem{defn}{Definition}

\theoremstyle{plain}
\newtheorem{theorem}{Theorem}[section]
\newtheorem{conj}{Conjecture}
\newtheorem{cor}[theorem]{Corollary}
\newtheorem{lemma}[theorem]{Lemma}
\newtheorem{prop}[theorem]{Proposition}
\numberwithin{equation}{section}
\begin{document}

\title[Counting cusp forms by analytic conductor]{Counting cusp forms by analytic conductor\\ [0.3cm] Comptage des formes cuspidales\\ par conducteur analytique}
\date{\today}                                         

\author[]{Farrell Brumley}
\address{Institut Galil\'ee, Universit\'e Sorbonne Paris Nord, 93430 Villetaneuse, France}
\email{brumley@math.univ-paris13.fr}

\author[]{Djordje Mili\'cevi\'c}
\address{Bryn Mawr College, Department of Mathematics, 101 North Merion Avenue, Bryn Mawr, PA 19010, U.S.A.}
\email{dmilicevic@brynmawr.edu}

\thanks{The first author was partially supported by the ANR Grant PerCoLaTor: ANR-14-CE25. The second author was partially supported by the National Security Agency (Grant H98230-14-1-0139), National Science Foundation (Grants DMS-1503629 and DMS-1903301), and through ARC grant DP130100674. He thanks the Max Planck Institute for Mathematics for their support and exceptional research infrastructure. The United States Government is authorized to reproduce and distribute reprints notwithstanding any copyright notation herein.}



\begin{abstract}
Let $F$ be a number field and $n\geqslant 1$ an integer. The {\it universal family} is the set $\mathfrak{F}$ of all unitary cuspidal automorphic representations on ${\rm GL}_n$ over $F$, ordered by their analytic conductor. We prove an asymptotic for the size of the truncated universal family $\mathfrak{F}(Q)$ as $Q\rightarrow\infty$, under a spherical assumption at the archimedean places when $n\geqslant 3$. We interpret the leading term constant geometrically and conjecturally determine the underlying Sato--Tate measure. Our methods naturally provide uniform Weyl laws with logarithmic savings in the level and strong quantitative bounds on the non-tempered discrete spectrum for ${\rm GL}_n$.

\medskip

\noindent
Soient $F$ un corps de nombres et $n\geqslant 1$ un entier. La {\it famille universelle} $\mathfrak{F}$ est l'ensemble de toutes les repr\'esentations cuspidales unitaires automorphes de ${\rm GL}_n$ sur $F$, muni de l'ordre induit par le conducteur analytique. Nous obtenons un \'equivalent asymptotique pour le cardinal de la famille universelle tronqu\'ee $\mathfrak{F}(Q)$ lorsque $Q\rightarrow \infty$, sous une hypoth\`ese de sph\'ericit\'e aux places archim\'ediennes si $n\geqslant 3$. Nous interpr\'etons g\'eom\'etriquement le terme dominant and d\'eterminons conjecturalement la mesure de Sato--Tate sous-jacente. Nos m\'ethodes fournissent une loi de Weyl uniforme avec un gain logarithmique dans le niveau et des bornes quantitatives fortes sur le spectre discret non temp\'er\'e pour ${\rm GL}_n$.
\end{abstract}

\maketitle

While automorphic forms can be notoriously difficult to study individually using analytic techniques, desired results can often be obtained by embedding them into a larger family of cusp forms of favorable size. In this article, we address the question of the asymptotic size of the universal family, which contains all cuspidal automorphic representations on $\GL_n$ over a fixed number field $F$, and is ordered by the analytic conductor of Iwaniec and Sarnak~\cite{IwaniecSarnak2000}.

The analytic conductor $Q(\pi)$ of a cusp form $\pi$ is an archimedean fattening of the classical arithmetic conductor of Casselman~\cite{Casselman1973} and Jacquet--Piatetski-Shapiro--Shalika~\cite{JacquetPiatetski-ShapiroShalika1981}: indeed, it is the product of the conductors $q(\pi_v)$ of all local components $\pi_v$, each of these arising from the local functional equation of the standard $L$-function $L(s,\pi_v)$. One way to understand the significance of the analytic conductor to the theory of $L$-functions is that its square-root determines the effective length of the partial sums appearing in the global approximate functional equation for $L(s,\pi)$. In turn, the analytic conductor controls the complexity of an array of analytic problems involving $L$-functions, such as the evaluation of moments, subconvexity, nonvanishing, extreme value problems, and numerical computations. From a broader and related perspective, $Q(\pi)$ quantifies the size of a system of equations large enough to pin down $\pi$ exactly. In this respect, it has a close connection with the requisite number of twists in the Converse Theorem.

Our interest here is in the role that $Q(\pi)$ plays as a natural height function in the automorphic context. To this end, we denote by $\mathfrak{F}$ the countable discrete set of all irreducible unitary cuspidal automorphic representations $\pi$ of $\GL_n(\A_F)$, considered up to unitary twist $|\det |^{it}$, organized into a family under the ordering induced by $Q(\pi)$. Following \cite{Sarnak08}, we shall refer to $\mathfrak{F}$ as the \textit{universal family}. In recent years, Sarnak has repeatedly emphasized the importance of understanding the statistical properties of the set
\[ \mathfrak{F}(Q)=\{\pi\in\mathfrak{F}: Q(\pi)\leqslant Q\}. \]
It may come as a surprise how little is known about $\mathfrak{F}(Q)$.

In this paper we investigate the cardinality $|\mathfrak{F}(Q)|$, for increasing $Q$.
Historically, the first result in this direction is the finiteness of $\mathfrak{F}(Q)$, which was established in \cite{Brumley2006} (see also \cite{Moreno1985}). Later, Michel and Venkatesh~\cite{MichelVenkatesh2010} showed that $|\mathfrak{F}(Q)|$ has at most polynomial growth in $Q$.

Our main theorem is the determination of the asymptotic size of $\mathfrak{F}(Q)$, subject to a spherical hypothesis on the archimedean component of $\pi$ when $n>2$. This allows us to answer in the affirmative the question posed by Michel--Venkatesh in \cite{MichelVenkatesh2010} regarding the limiting behavior of $\log |\mathfrak{F}(Q)|/\log Q$. More precisely, we find that $|\mathfrak{F}(Q)|$ has pure power growth of the order $Q^{n+1}$, with no logarithmic factors. Moreover, we interpret the leading term constant in a way which is consistent with analogous problems for counting rational points of bounded height.

\tableofcontents

\section{Introduction}\label{intro1}

Having briefly described the central problem, we now proceed to stating more precisely our main asymptotic result on the counting function of $\mathfrak{F}(Q)$, the trace formula input on which it depends, as well as an interpretation of the leading term constant.

\subsection{Weyl--Schanuel law}\label{into-1.1}

We formulate, in Conjecture \ref{weyl-schanuel-conj} below, the expected asymptotic behavior of $|\mathfrak{F}(Q)|$. Following \cite{SarnakShinTemplier2016}, we refer to this asymptotic as the \emph{Weyl--Schanuel law}. Indeed, it can si\-mul\-ta\-ne\-ous\-ly be viewed as a sort of universal Weyl law, and as an automorphic analogue to Schanuel's well-known result on the number of rational points of bounded height on projective~spaces. 

To describe the conjecture, we shall need to set up some notation. Let $\Pi(\GL_n(\A_F))$ denote the restricted direct product, over all places $v$ of $F$, of the local unitary duals of $\GL_n(F_v)$ relative to the local unramified duals. Let $\Pi(\GL_n(\A_F)^1)$ be the subset consisting of those $\pi$ whose central characters are trivial on the diagonal embedding of the positive reals. We give $\Pi(\GL_n(\A_F)^1)$ the subspace topology derived from the direct product topology. We may embed the universal family $\mathfrak{F}$ into $\Pi(\GL_n(\A_F)^1)$ by taking local components, and the notion of analytic conductor extends to all of the latter space.

Let $\GL_n(\A_F)^1$ be the subgroup of $g\in\GL_n(\A_F)$ with $|\det g|_{\A_F}=1$.  We equip $\GL_n(\A_F)^1$ with Tamagawa measure, denoted by $\omega_{\GL_n}$; in particular, $\omega_{\GL_n}$ assigns the automorphic quotient $\GL_n(F)\backslash\GL_n(\A_F)^1$ volume 1. This then induces a normalization of Plancherel measure $\widehat{\omega}^{\rm pl}$ on $\Pi(\GL_n(\A_F)^1)$. Let
 \begin{equation}\label{eq:def:tau-hat}
\mathscr{C}(\mathfrak{F})=\frac{1}{n+1}\mathop{\int\nolimits}\limits_{\Pi(\GL_n(\A_F)^1)}Q(\pi)^{-n-1}\,{\rm d}\widehat{\omega}^{\rm pl}(\pi),
\end{equation}
with the integral being regularized as in \S\ref{leading-term}. 

We may now state the following
\begin{conj}[Weyl--Schanuel law]
\label{weyl-schanuel-conj}
Fix a number field $F$ and an integer $n\geqslant 1$. Then
\[
|\mathfrak{F}(Q)|\sim \mathscr{C}(\mathfrak{F}) Q^{n+1} \quad\text{as }\,Q\to\infty.
\]
\end{conj}

Conjecture \ref{weyl-schanuel-conj} is motivated by the heuristic, borrowed from the setting of Schanuel's theorem, that the asymptotic behavior of $|\mathfrak{F}(Q)|$ should align with that of the ``Plancherel volume of the conductor ball'':
\begin{equation}\label{eq:big-adelic-integral}
V_{\mathfrak{F}}(Q)=\int\limits_{\substack{\pi\in\Pi(\GL_n(\A_F)^1)\\ Q(\pi)\leqslant Q}}{\rm d}\widehat{\omega}^{\rm pl}(\pi).
\end{equation}
In Proposition \ref{MainTermLemma} an asymptotic evaluation of the finite integral $V_{\mathfrak{F}}(Q)$ is given, in which $\mathscr{C}(\mathfrak{F})$ appears as the leading term constant. Note that $V_{\mathfrak{F}}(Q)$ contains no automorphic information.
 
In this paper, we establish the above predicted asymptotics for $|\mathfrak{F}(Q)|$ in many cases, with explicit logarithmic savings in the error term. Namely, we prove the following

\begin{theorem}\label{main-theorem}
Let $F$ be a number field and $n\geqslant 1$ be an integer. The Weyl--Schanuel law holds for $n\leqslant 2$, as well as for $n\geqslant 3$, when restricted to the archimedean spherical spectrum.
\end{theorem}

In addition, we address related counting and equidistribution problems and prove uniform Weyl laws (with explicit savings in the level aspect), estimates on the size of complementary spectrum, and uniform estimates on terms appearing in Arthur's trace formula for $\GL_n$.

\subsection{Main auxiliary results}\label{sec:results}

To prove Theorem \ref{main-theorem}, we first reduce Conjecture \ref{weyl-schanuel-conj} to certain trace formula estimates, and then prove these estimates in many cases. We elaborate on the precise form of these estimates in \S\ref{sec:ELM}, where we define the Effective Limit Multiplicity property, or {\it Property (ELM)}, which encapsulates them. The archimedean unitary dual $\Pi(\GL_n(F_{\infty})^1)$ enjoys a decomposition \eqref{eq:intro:discrete-decomp} according to the full set $\mathcal{D}$ of discrete inducing parameters, one of which gives the spherical spectrum. For every subset $\Delta\subseteq\mathcal{D}$, we may consider the corresponding subset $\Pi(\GL_n(F_{\infty})^1)_{\Delta}$ and a refined version of Property (ELM) with respect to $\Delta$. We refer the reader to Definition \ref{TFhyp} for more details (including some of the notation used here), and proceed now to a description of the two main auxiliary results which are used to prove Theorem \ref{main-theorem}.

Our first main theorem, proved in Part \ref{part1}, is the reduction of Conjecture \ref{weyl-schanuel-conj} to Property $(\mathrm{ELM})$.

\begin{theorem}\label{main-implication}
Let $F$ be a number field and $n\geqslant 1$ an integer. Property $(\mathrm{ELM})$ implies Conjecture \ref{weyl-schanuel-conj} in the following effective form
\[
|\mathfrak{F}(Q)|=\mathscr{C}(\mathfrak{F})Q^{n+1}\left(1+\textnormal{O}\left(\frac{1}{\log Q}\right)\right).
\]
Moreover, if Property $(\mathrm{ELM})$ holds with respect to some given subset $\Delta\subseteq\mathcal{D}$, then the restricted family
$\mathfrak{F}_\Delta(Q)$ consisting of those $\pi\in \mathfrak{F}(Q)$ for which $\pi_\infty\in\Pi(\GL_n(F_\infty)^1)_{\Delta}$ satisfies
\[
|\mathfrak{F}_\Delta(Q)|=\mathscr{C}_{\Delta}(\mathfrak{F}) Q^{n+1}\left(1+\textnormal{O}\left(\frac{1}{\log Q}\right)\right),
\]
where, making any choice of normalization of archimedean Plancherel measure ${\rm d}\pi_\infty$,
\[
\mathscr{C}_{\Delta}(\mathfrak{F})=\mathscr{C}(\mathfrak{F})\frac{\int_{\Pi(\GL_n(F_\infty)^1)_{\Delta}}q(\pi_\infty)^{-n-1}\, {\rm d}\pi_\infty}{\int_{\Pi(\GL_n(F_\infty)^1)}q(\pi_\infty)^{-n-1}\, {\rm d}\pi_\infty}.
\]
All implied constants depend on $F$ and $n$ (as well as $\Delta$, where applicable).
\end{theorem}

One of the crucial ingredients in the proof of Theorem \ref{main-implication} is Proposition \ref{comp-mu}, in which we provide upper bounds on the sum over the discrete spectrum of $\GL_n$, where each $\pi$ appears with an exponential weight measuring how far its archimedean component $\pi_\infty$ is from being tempered. The role of the latter proposition is to show, at various places in our arguments, that discrete $\pi$ for which $\pi_\infty$ is non-tempered contribute negligibly to the total automorphic count. In particular, the combination of Propositions~\ref{KR} and \ref{temp-err} shows that the number of cuspidal $\pi\in\mathfrak{F}(Q)$ with $\pi_\infty$ non-tempered is in fact $\mathrm{O}(Q^{n+1}/\log^3Q)$. (The Ramanujan conjecture states that such $\pi$ do not exist.) The fact that the archimedean place is singled out here is due to the expanding support conditions on the test functions we consider at infinity.

In our second main theorem, formulated in Theorem \ref{money-cor} and proved throughout Part \ref{part3}, we establish Property (ELM) in certain cases. These are described in the following result.

\begin{theorem}\label{master}
Let $F$ be a number field and $n\geqslant 1$ be an integer.
\begin{enumerate}
\item For $n\leqslant 2$, Property $(\mathrm{ELM})$ holds.
\item For $n\geqslant 3$, Property $(\mathrm{ELM})$ holds with respect to the spherical part of $\Pi(\GL_n(F_\infty)^1)$.
\end{enumerate}
\end{theorem}

The combination of the above two theorems yields our main result, Theorem \ref{main-theorem}. The restriction to the archimedean spherical spectrum for $n\geqslant 3$ in Theorem \ref{master} is a purely technical constraint, having only to do with explicit spectral inversion of archimedean test functions. We believe that this restriction can be removed, by following a different approach to bounding the weighted archimedean orbital integrals appearing in the Arthur trace formula. We plan to address this in a subsequent work.

\subsection{Effective Limit Multiplicity (ELM) property}\label{sec:ELM}

A natural framework for counting automorphic representations is provided by Arthur's non-invariant trace formula. This is an equality of distributions $J_{\rm spec}=J_{\rm geom}$, along with an expansion of both sides according to primitive spectral or geometric data. Roughly speaking, the most regular part of the spectral side of the trace formula $J_{\rm cusp}$, coming from the cuspidal contribution, is expected to be governed by the most singular part of the geometric side $J_{\rm cent}$, coming from the central elements.

In the body of the paper, we shall use measure conventions for the trace formula which align with the literature we cite. In particular, rather than $\omega_{\GL_n}$, we shall equip $\GL_n(\A_F)^1$ with the measure $\mu_{\GL_n}$ which assigns the standard maximal compact subgroup at finite places unit volume.\footnote{\label{rem:intro:measures}
One might ask what shape Conjecture \ref{weyl-schanuel-conj} would take under this alternative choice of measures. To describe this, we follow the notational convention in \cite{Shin2012,ShinTemplier2016} in which, if a Haar measure $m$ on a group $G$ is fixed, then $\widehat{m}^{\rm pl}$ denotes the induced normalization of Plancherel measure on the unitary dual $\Pi(G)$. Let $\widehat{\mu}^{\, \rm pl}$ be the normalization of Plancherel measure corresponding to $\mu_{\GL_n}$. Using $\widehat{\mu}^{\, \rm pl}$ in place of $\widehat{\omega}^{\, \rm pl}$ in the integral \eqref{eq:big-adelic-integral} would have the effect of decomposing the adelic integral in \eqref{eq:def:tau-hat} into a product of two factors, one coming from the automorphic volume $\mu_{\GL_n}(\GL_n(F)\backslash\GL_n(\A_F)^1)$, the other coming from the analogous spectral integral involving $\widehat{\mu}^{\, \rm pl}$.} Then $\mu_{\GL_n}$ assigns the automorphic space $\GL_n(F)\backslash\GL_n(\A_F)^1$ the measure $D_F^{n^2/2}\Delta_F^*(1)$, where $D_F$ is the absolute discriminant of $F$, and $\Delta_F^*(1)$ is the residue at $s=1$ of the global motivic $L$-function $\Delta_F(s)=\zeta_F(s)\zeta_F(s+1)\cdots \zeta_F(s+n-1)$ for $\GL_n$. Then, for a function $\phi\in \mathcal{H}(\GL_n(\A_F)^1)$ we let 
\[
J_1(\phi)=D_F^{n^2/2}\Delta_F^*(1)\phi(1)\quad\text{and}\quad J_{\rm cent}(\phi)=D_F^{n^2/2}\Delta_F^*(1)\sum_{\gamma\in Z(F)}\phi(\gamma)
\]
be the identity and central contributions to the trace formula. Furthermore, let
\[
J_{\rm cusp}(\phi)={\rm tr}(R_{\rm cusp}(
\phi))\qquad\text{and}\qquad J_{\rm disc}(\phi)={\rm tr}(R_{\rm disc}(
\phi))
\]
be the cuspidal and discrete contributions, where $R_\bullet$ is the restriction of the right-regular representation of $\GL_n(\A_F)^1$ on $L^2_\bullet(\GL_n(F)\backslash\GL_n(\A_F)^1)$.
Finally put
\begin{equation}
\label{DefEqh}
J_{\rm error}(\phi)=J_{\rm disc}(\phi)-J_{\rm cent}(\phi),
\end{equation}
the estimation of which will be our primary concern. See \S\S\ref{fine-geom}--\ref{JMspec} for more precise definitions of the distributions $J_{\rm cent}$ and $J_{\rm disc}$.

We shall in fact be interested in $J_{\rm error}(\phi)$ for $\phi$ of the form $\varepsilon_{K_1(\q)}\otimes f$, where $f\in C^\infty_c(\GL_n(F_\infty)^1)$ and $\varepsilon_{K_1(\q)}$ is the idempotent element in the Hecke algebra associated with the standard Hecke congruence subgroup $K_1(\q)$. The latter subgroup, by the work of Casselman \cite{Casselman1973} and Jacquet--Piatetski-Shapiro--Shalika \cite{Jacquet2012, JacquetPiatetski-ShapiroShalika1981}, is known to pick out from the cuspidal spectrum those representations of conductor dividing $\q$. As with many spectral counting problems, one expects $J_{\rm error}(\varepsilon_{K_1(\q)}\otimes f)$ to be small in terms of various quantities involving $\q$ and $f$. If this can be properly quantified, one can hope to deduce that a sharp cuspidal count modeled by $J_{\rm cusp}(\varepsilon_{K_1(\q)}\otimes f)$ is roughly equal to $J_1(\varepsilon_{K_1(\q)}\otimes f)$. Indeed, with our choices of $f$, the passage from $J_{\rm cent}$ to $J_1$ causes no difficulty, and we will further be able to show (see Section~\ref{section:outline}) that $J_{\rm cusp}$ provides the dominant contribution to $J_{\rm disc}$.

Our interest is in taking $f$ whose Fourier transform $h(\pi_\infty)={\rm tr}\, \pi_\infty(f)$ satisfies a localization property around a given tempered unitary representation in $\Pi(\GL_n(F_\infty)^1)$. Moreover, we would like to have some control over the error in the localization. In general this error is quantified by the support of the test function $f$. Indeed, if ${\rm supp} f\subset \bfk_\infty \exp (B(0,R)) \bfk_\infty$, where $B(0,R)$ is the ball of radius $R$ in the Lie algebra of the diagonal torus, then the walls of the corresponding $h$, that is, the regions where it transitions to rapid decay, have width on the scale of $1/R$. The following few paragraphs make this precise.

The unitary dual of $\GL_n(F_\infty)^1$ breaks up as a disjoint union 
\begin{equation}\label{eq:intro:discrete-decomp}
\Pi(\GL_n(F_\infty)^1)=\bigcup_{\underline{\delta}\in\mathcal{D}}\Pi(\GL_n(F_\infty)^1)_{\underline{\delta}}
\end{equation}
indexed by discrete data $\mathcal{D}$. More precisely, $\mathcal{D}$ is the set of conjugacy classes of pairs $(M,\delta)$ consisting of a Levi subgroup $M$ of $\GL_n(F_\infty)^1$ and a discrete series representation $\delta$ of $M^1$. Given a discrete spectral parameter $\underline{\delta}\in\mathcal{D}$ represented by $(M,\delta)$, a continuous spectral parameter $\mu\in i\mathfrak{h}_M^*$, and a real number $R>0$, we will be interested in test functions $f_R^{\delta,\mu}$ whose support lies in $\bfk_\infty \exp (B(0,cR)) \bfk_\infty$ (for a suitable $c>0$) and whose Fourier transform $h_R^{\delta,\mu}$ localizes about the irreducible tempered unitary representation indexed by the data $(\delta,\mu)$. We shall need bounds on $J_{\rm error}(\varepsilon_{K_1(\q)}\otimes f_R^{\delta,\mu})$ for such archimedean localizing functions $h_R^{\delta,\mu}$. The exact conditions we impose upon $h_R^{\delta,\mu}$ (where the parameter $R$ controls the rate of localization and the corresponding Paley--Wiener type conditions) are formulated in Definition \ref{defn-localizer}.

The following definition expresses uniform bounds on this quantity which are sufficiently strong for our applications. The dependence in the error term with respect to the archimedean spectral data is via a natural Plancherel majorizer $\beta_M^G(\delta,\nu)\,\text{d}\nu$, introduced in Definition \ref{BetaMajorizer}.

\begin{defn}[Effective Limit Multiplicity $(\mathrm{ELM})$]\label{TFhyp}
Let a number field $F$ and an integer $n\geqslant 1$ be fixed. We say that \emph{Property {\rm (ELM)} holds with respect to a subset $\Delta\subseteq\mathcal{D}$} if there exist constants $C\geqslant 0$ and $c,\theta>0$ such that the following holds: for all $\underline{\delta}\in\Delta$ represented by $(M,\delta)$, $\mu\in i\mathfrak{h}_M^*$,
and $R\geqslant 1$, and for every spectral localizer $h_R^{\delta,\mu}$ about $[\underline{\delta},\mu]$ in $\mathcal{PW}_{R,\underline{\delta}}$, there exists an $f_R^{\delta,\mu}\in\mathcal{H}(G_{\infty}^1)_{cR,\underline{\delta}}$ such that $h_R^{\delta,\mu}(\nu)=\tr\pi_{\delta,\nu}(f_R^{\delta,\mu})$ and such that, for every integral ideal $\q$ of $\mathcal{O}_F$, we have
\[
J_{\rm error}(\varepsilon_{K_1(\q)}\otimes f_R^{\delta,\mu})\ll e^{CR}\norm{\q}^{n-\theta}\beta_M^G(\delta,\mu).
\]
When (ELM) holds with respect to all of $\mathcal{D}$, one then says that \emph{Property $(\mathrm{ELM})$ holds}.
\end{defn}

Property (ELM) depends on the choice of number field $F$ and integer $n\geqslant 1$. Since these are fixed throughout the paper, we will not recall this dependency elsewhere. We shall make more extensive comments about Property (ELM) in \S\ref{blurb}. For the moment, we content ourselves to a few brief remarks.

\begin{remarks}
${}$
\begin{enumerate}
\item We have expressed Property (ELM) with respect to the particular subgroups $K_1(\q)$ since only these arise in our applications. More generally, one could ask for analogous bounds for arbitrary sequences of compact open subgroups in $\GL_n(\A_f)$ whose volumes tend to zero. Our proof of Theorem \ref{master}, which establishes Property (ELM) in many cases, would continue to hold for such subgroups, since we are appealing to the powerful results \cite{FinisLapid2018} of Finis--Lapid. 
\item The estimate is trivial in the archimedean spectral parameters $\underline{\delta}$ and $\mu$. Indeed, as will be seen in Section~\ref{spec-loc}, we have
\begin{equation}\label{eq:arch-MT}
\int_{i\h_M^*}h_R^{\delta,\mu}(\nu)\beta_M^G(\delta,\nu)\, {\rm d}\nu\asymp\int_{B_M(\mu,1/R)}\beta_M^G(\delta,\nu)\,\text{d}\nu\asymp R^{-\dim\mathfrak{h}_M}\beta_M^G(\delta,\mu).
\end{equation}
Since the polynomial factor $R^{\dim\mathfrak{h}_M}$ can be subsumed into the exponential factor $e^{CR}$ (with a different constant $C$), it is in fact equivalent to state Property (ELM) with the majorant $\beta_M^G(\delta,\mu)$ replaced by any of the three expressions appearing in \eqref{eq:arch-MT}. All three appear naturally in our proofs and are generically comparable to the archimedean component of the identity contribution to the trace formula $J_1(\varepsilon_{K_1(\q)}\otimes f_R^{\delta,\mu})$; see \S\ref{UniformUpperBoundsSubsection}. While it might at first seem surprising that no non-trivial savings at infinity is needed in order to deduce our main result, it is rather the power savings in the level which is of critical importance in our applications. To get a better feeling for the various ranges of parameters, and corresponding savings, see \S\ref{Approach2}.
\item The terminology ``Effective Limit Multiplicity" was chosen in reference to the power savings in the level as well as the uniformity in all parameters that the statement provides, which can be seen as a quantification of the limit multiplicity property for $\GL_n$ at the archimedean place.
\end{enumerate}
\end{remarks}

\subsection{On the leading term}\label{leading-term}

We now give a precise definition of the leading term constant in the Weyl--Schanuel conjecture.

We now recall the Tamagawa measure for $\GL_n$. We let $\omega(g)=\det (g)^{-n}({\rm d}g_{11}\wedge\cdots\wedge {\rm d}g_{nn})$ be the unique (up to scalar) invariant rational differential form of highest degree on $\GL_n$. Then $\omega$ induces a Haar measure $\omega_{\GL_n,v}$ on $\GL_n(F_v)$ at every place $v$. We put the $\omega_{\GL_n,v}$ together into a global measure in the following way. For every finite place $v$ of $F$ we let $\Delta_v(s)=\zeta_v(s)\zeta_v(s+1)\cdots\zeta_v(s+n-1)$ be the local factor of the motivic $L$-function $\Delta_F(s)$ described in \S\ref{sec:ELM}. Similarly, let $\zeta_F(s)=\prod_{v<\infty}\zeta_v(s)$ be the Dedekind zeta function for $F$, with residue $\zeta_F^*(1)$ at $s=1$. Define a global Haar measure on $\GL_n(\A_F)$ by
\begin{equation}\label{eq:defn:omegaG-ast}
\frac{1}{\zeta_F^*(1)}\frac{1}{D_F^{n^2/2}}\prod_{v<\infty}\zeta_v(1)\omega_{\GL_n,v} \prod_{v\mid \infty}\omega_{\GL_n,v} =\frac{1}{\Delta_F^*(1)}\frac{1}{D_F^{n^2/2}}\prod_v\Delta_v(1)\omega_{\GL_n,v}\prod_{v\mid\infty} \omega_{\GL_n,v}.
\end{equation}
Using the short exact sequence
\[
1\rightarrow \GL_n(\A_F)^1\rightarrow\GL_n(\A_F)\xrightarrow{\log |\det|_{\A_F}}\R\rightarrow 0,
\]
we factorize $\GL_n(\A_F)\simeq \GL_n(\A_F)^1 \times \R$. Recall from \cite[\S 2]{Ono1966} that the Tamagawa measure $\omega_{\GL_n}$ on $\GL_n(\A_F)^1$ is defined as the unique measure which decomposes the measure in \eqref{eq:defn:omegaG-ast} as $\omega_{\GL_n}\times {\rm d}t$, where ${\rm d}t$ is the standard Lebesgue measure on $\R$. Then it follows from \cite[Theorem 3.1.1]{Weil1982} that the Tamagawa number $\tau(\GL_n)=\int_{\GL_n(F)\backslash \GL_n(\A_F)^1}{\rm d}\omega_{\GL_n}$ of $\GL_n$ is $1$.

On the spectral side, for every place $v$ of $F$, we now fix a normalization of local Plancherel measures $\widehat{\omega}_v^{\rm pl}$ on $\Pi(\GL_n(F_v))$ by taking Plancherel inversion to hold relative to $\omega_{\GL_n,v}$. At finite places $v$ the measure $\widehat{\omega}_v^{\rm pl}$ assigns the unramified unitary dual volume $\Delta_v(1)$; see \cite[p. 31]{Weil1982}. Similarly, we let $\widehat{\omega}^{\rm pl}$ be the adelic Plancherel measure  on $\Pi(\GL_n(\A_F)^1)$ for which the Plancherel inversion formula holds relative to $\omega_{\GL_n}$. This is the Plancherel measure appearing in \eqref{eq:big-adelic-integral}.

We now explain the regularization of the integral in \eqref{eq:def:tau-hat}. We let $\omega_\infty$ be the unique measure on $\GL_n(F_\infty)^1$ decomposing $\omega_{\GL_n,\infty}$ as $\omega_\infty\times {\rm d}t$ and write $\widehat{\omega}_\infty^{{\rm pl}}$ for the Plancherel measure on $\Pi(\GL_n(F_\infty)^1)$ corresponding to $\omega_\infty$. We define local measures $\tau_{\mathfrak{F},v}$ on $\Pi(\GL_n(F_v))$, for $v$ finite, and $\tau_{\mathfrak{F},\infty}$ on $\Pi(\GL_n(F_\infty)^1)$, by setting (for open subsets $A$)
\begin{equation}\label{eq:defn:plv}
\tau_{\mathfrak{F},v}(A)=\int_A q(\pi_v)^{-n-1}\,{\rm d}\widehat{\omega}_v^{\rm pl}(\pi_v)\quad (v<\infty),\quad \tau_{\mathfrak{F},\infty}(A)=\int_A q(\pi_\infty)^{-n-1}\,{\rm d}\widehat{\omega}_\infty^{{\rm pl}}(\pi_\infty).
\end{equation}
Since $\widehat{\omega}_v^{\rm pl}$ is supported on the tempered spectrum, so too is $\tau_{\mathfrak{F},v}$. In Lemma \ref{local-count} we show that for finite places $v$ the volume of $\tau_{\mathfrak{F},v}$ is finite and, using $\mu_{\GL_n,v}=\Delta_v(1)\omega_{\GL_n,v}$, equals $\Delta_v(1)\zeta_v(1)/\zeta_v(n+1)^{n+1}$.  We deduce that the measure on $\Pi(\GL_n(\A_F)^1)$ given by the regularized product
\begin{equation}\label{defn-hat-nu}
\tau_{\mathfrak{F}}=D^{n^2/2}\Delta_F^*(1)\zeta_F^*(1)\bigg(\prod_{v<\infty} \Delta_v(1)^{-1}\zeta_v(1)^{-1}\tau_{\mathfrak{F},v}\bigg)\tau_{\mathfrak{F},\infty}
\end{equation}
converges. Then the regularized integral \eqref{eq:def:tau-hat} in Conjecture \ref{weyl-schanuel-conj} is, \textit{by definition}, equal to the volume
\begin{equation}\label{eq:defn:tau=pl}
{\rm vol}(\tau_\mathfrak{F})=\tau_{\mathfrak{F}}(\Pi(\GL_n(\A_F)^1))
\end{equation}
of the finite measure $\tau_{\mathfrak{F}}$.
 
From the definition \eqref{eq:defn:plv}, one can interpret ${\rm vol}(\tau_{\mathfrak{F}})$ as the regularized Euler product of local conductor zeta functions evaluated at $s=n+1$. This point of view is emphasized in Section \ref{sec:cond-zeta}; see, for example, Remark \ref{rem:cond-zeta}. The measure $\tau_{\mathfrak{F}}$ also features in the conjectured equidistribution of the universal family which we discuss in Section~\ref{equidsitribution-sec}. 

\subsection{Schanuel's theorem}\label{sec:Schanuel}
The Weyl--Schanuel law of Conjecture \ref{weyl-schanuel-conj} is reminiscent of the familiar problem of counting rational points on projective algebraic varieties. In particular, one can set up an analogy between counting $\pi\in\mathfrak{F}$ with analytic conductor $Q(\pi)\leqslant Q$ and counting $x\in\PP^n(F)$ with exponential Weil height $H(x)\leqslant B$. 

To be more precise, let us recall some classical results on counting rational points in projective space $\PP^n$, where $n\geqslant 1$. For a rational point $x\in\PP^n(F)$, given by a system of homogeneous coordinates $x=[x_0:x_1:\cdots :x_n]$, we denote by
\[
H(x)=\prod_v\max (|x_0|_v,|x_1|_v,\ldots ,|x_n|_v)^{1/d} \qquad (d=[F:\Q])
\]
the absolute exponential Weil height of $x$, the product being taken over the set of normalized valuations of $F$. An asymptotic for the above counting function was given by Schanuel~\cite{Schanuel1979}. The leading constant was later reinterpreted by Peyre \cite{Peyre1995}, as part of his refinement of the conjectures of Batyrev--Manin~\cite{BatyrevManin1990}. Following Peyre, we write $\tau_H(\PP^n)$ for the volume of the Tamagawa measure of $\PP^n$ with respect to $H$; then $\tau_H(\PP^n)$ is equal to $\zeta^*_F(1)/\zeta_F(n+1)$ times some archimedean volume factors. Then Schanuel proved that
\[
|\{x\in\PP^n(F):H(x)\leqslant B\}|\sim \frac{1}{n+1}\tau_H(\PP^n)B^{n+1}\qquad\text{ as } \quad B\rightarrow\infty.
\]
Schanuel in fact gave an explicit error term of size $\text{O}(B\log B)$ when $n=1$ and $F=\Q$ and $\text{O}(B^{n+1-1/d})$ otherwise. Later, Chambert-Loir and Tschinkel \cite{Chambert-LoirTschinkel2010} showed how the Tamagawa measure $\tau_H(\PP^n)$ appears naturally when calculating the volume of a height ball.

More generally, in the same spirit as the Batyrev--Manin--Peyre conjectures for counting rational points on Fano varieties, given a reductive algebraic group $\bm{G}$ over $F$ and a representation $\rho: {}^L\bm{G}\rightarrow\GL_n(\C)$ of the $L$-group (see \cite[\S 2.6]{Borel1979}), then assuming an appropriate version of the local Langlands conjectures, one can pull back the $\GL_n$ conductor to $\bm{G}$. Under appropriate conditions on $\rho$ assuring a finiteness property, one would like to understand the asymptotic properties of the counting function associated to global cuspidal automorphic $L$-packets of $\bm{G}(\A_F)$ of bounded analytic conductor. This problem has been investigated in some cases by Lesesvre \cite{Lesesvre2020} and Brooks--Petrow \cite{BrooksPetrow2018}, as well as in Petrow \cite{Petrow2018}, which furthermore explores eligibility requirements on $\rho$. Our methods suggest that, any time this problem can be solved, the leading term constant will be given in terms of the Plancherel volume of the $\rho$-conductor ball
\[
\int\limits_{\substack{\pi\in\Pi(\GL_n(\A)^1)\\ Q(\pi,\rho)\leqslant Q}} {\rm d}\widehat{\omega}^{\rm pl}(\pi).
\]
This analogy served as an inspiration and organizing principle throughout the elaboration of this article, where we address the setting of general linear groups and the standard embedding. Finally, we emphasize that we exploit several important features of $\GL_n$ throughout the proofs of Theorems \ref{main-theorem} and \ref{master}, including the well-understood newform theory in Section \ref{sec:cond-zeta}, and strong estimates on the continuous spectrum in Section \ref{EisensteinSection}.

\subsection{Hecke congruence subgroups}

What makes the general linear groups particularly amenable to conductor counting are the well-understood multiplicity and volume growth properties of the Hecke congruence subgroups $K_1(\q)$ and $K_0(\q)$ defining the arithmetic conductor, in the newform theory of Casselman~\cite{Casselman1973} and Jacquet--Piatetski-Shapiro--Shalika~\cite{JacquetPiatetski-ShapiroShalika1981}.

To better understand the role of $K_1(\q)$ in the leading term asymptotics of the Weyl--Schanuel law, consider the ``index zeta function" of the system of Hecke congruence subgroups $K_1(\q)$:
\[
\prod_{v<\infty}\sum_{r\geqslant 0}\frac{[\GL_n(\mathfrak{o}_v): K_1(\p_v^r)]}{\norm\p_v^{rs}}.
\]
If $p_n$ denotes the arithmetic function $\mathfrak{a}\mapsto\norm\mathfrak{a}^n$ on ideals, and $\mu$ is the M\"obius function over $F$, then it follows from \eqref{eq:Euler-phi} that the above can be rewritten as
\[
\prod_{v<\infty}\sum_{r\geqslant 0}\frac{(p_n\star\mu)(\p_v^r)}{\norm\p_v^{rs}}=\prod_{v<\infty}\frac{\zeta_v(s-n)}{\zeta_v(s)}=\frac{\zeta_F(s-n)}{\zeta_F(s)}.
\]
The abscissa of convergence of the above Euler product determines the order of magnitude of the asymptotic growth of $|\mathfrak{F}(Q)|$. Its regularized value at $s=n+1$, $\zeta_F^*(1)/\zeta_F(n+1)$, contributes to the volume ${\rm vol}(\tau_\mathfrak{F})$; see \eqref{eq:adelic-int-w-zeta}, where, it should be remarked that the arithmetic factor $\zeta_F(n+1)^{-n}$ comes from inverting, through a sieving process, the series formed from the dimensions of old forms.

\subsection{Comments on other asymptotic aspects}

Throughout this paper, both $n$ and $F$ will be considered as fixed. Nevertheless, it would be interesting to understand the behavior of $|\mathfrak{F}(Q)|$ as $n$ and $F$ vary (either simultaneously with $Q$ or for $Q$ fixed). We remark on two aspects:

\medskip

\begin{enumerate}
\item In Conjecture \ref{weyl-schanuel-conj}, one could set $Q=1+\epsilon$ (for a small $\epsilon>0$) and vary either $n$ or $F$. This would count the number of everywhere unramified cuspidal automorphic representations of $\GL_n$ over $F$ whose archimedean spectral parameters are constrained to a small ball about the origin. In this set-up, if $F$ is fixed and $n$ gets large, we recover the number field version of a question of Venkatesh, as described for function fields in \cite[\S 4]{FouvryKowalskiMichel2013}.

\medskip

\item On the other hand, we may fix $n$ (again keeping $Q=1+\epsilon$) and allow $D_F$ to get large. For example, when $n=1$ this counts the size of a ``regularization" of the group of ideal class characters, which has size about $D_F^{1/2+o(1)}$ by Siegel's theorem; this lines up with the power of $D_F$ in Conjecture  \ref{weyl-schanuel-conj}. When $n=2$, we recover the number field version of the famous result of Drinfeld \cite{Drinfeld1981}. Note that in the function field case the role of $D_F^{1/2}$ is played by the quantity $q^{g-1}$, as one can see by comparing the Tamagawa measures in \cite{Ono1965} and \cite[\S 3.8]{DeligneFlicker2013}. Thus, when $n=2$ the factor of $D_F^2$ in the leading term in Conjecture  \ref{weyl-schanuel-conj} corresponds to $q^{4(g-1)}$ for function fields, and when multiplied by $|{\rm Pic}(X_0)|=q-1$ this recovers the leading term of $q^{4g-3}$. 

\end{enumerate}

\medskip

\noindent We emphasize that we are not making any conjectures about the nature of the above asymptotic counts (1) and (2) when either $F$ or $n$ is allowed to move. The above discussion is meant purely to evoke parallels with other automorphic counting problems in the literature.

\subsection{Acknowledgements}

We would like to thank Nicolas Bergeron, Andrew Booker, Laurent Clozel, James Cogdell, Guy Henniart, Emmanuel Kowalski, Erez Lapid, Dipendra Prasad, Andre Reznikov, and Peter Sarnak for various enlightening conversations. We are particularly indebted to Simon Marshall for suggestions leading to a simplification of the proof of Proposition \ref{comp-mu} and to Peter Sarnak for originally suggesting this problem. Finally, we would like to express our most sincere thanks to the referees; their extensive and detailed reports led us to make substantial improvements to the exposition.

\section{Equidistribution and Sato--Tate measures: conjectures}\label{equidsitribution-sec}

Beyond the counting statement of Conjecture~\ref{weyl-schanuel-conj}  we in fact conjecture that the universal family $\mathfrak{F}(Q)$ {\it equidistributes}, as $Q\rightarrow\infty$, to a probability measure on $\Pi(\GL_n(\A_F)^1)$ that we now explicitly identify. This allows us to properly interpret the leading term constant in the conjectural Schanuel--Weyl law and our main theorems. We expect that our techniques can be leveraged to yield a proof of these equidistribution conjectures and plan to address this in follow up work.

The universal family $\mathfrak{F}(Q)$ gives rise, by way of the embedding into $\Pi(\GL_n(\A_F)^1)$ via local components, to two automorphic counting measures
\begin{equation}
\label{CountingMeasuresDef}
\frac{1}{Q^{n+1}}\sum_{\pi\in\mathfrak{F}(Q)}\delta_\pi\qquad\text{ and }\qquad\frac{1}{|\mathfrak{F}(Q)|}\sum_{\pi\in\mathfrak{F}(Q)}\delta_\pi
\end{equation}
on $\Pi(\GL_n(\A_F)^1)$. We would like to understand their limiting behavior as $Q\rightarrow\infty$.

Recall the measure $\tau_{\mathfrak{F}}$ on $\Pi(\GL_n(\A_F)^1)$ of \S\ref{leading-term}, whose volume enters Conjecture~\ref{weyl-schanuel-conj}. Denoting by $\Pi_0(\GL_n(\A_F)^1)$ the subset of $\Pi(\GL_n(\A_F)^1)$ consisting of $\pi$ with $\pi_\infty$ spherical, the statement of our Theorem~\ref{master} verifies that
\begin{equation}
\frac{1}{Q^{n+1}}\sum_{\pi\in\mathfrak{F}(Q)}\delta_\pi(A)\longrightarrow  \frac1{n+1}\tau_{\mathfrak{F}}(A),
\end{equation}
for the sets $A=\Pi(\GL_n(\A_F)^1)$ for $n\leqslant 2$ and for $A=\Pi_0(\GL_n(\A_F)^1)$ for every $n\in\mathbb{N}$.

We conjecture that this holds more generally:
\begin{conj}[Equidistribution]
\label{master-conj}
As $Q\to\infty$,
\[ \frac{1}{Q^{n+1}}\sum_{\pi\in\mathfrak{F}(Q)}\delta_\pi\longrightarrow  \frac{1}{n+1} \tau_{\mathfrak{F}}. \]
\end{conj}

Conjecture \ref{master-conj} implies, in particular, the Weyl--Schanuel law (Conjecture~\ref{weyl-schanuel-conj}). It moreover implies that the closure of $\mathfrak{F}$, relative to the direct product topology, contains the support of $\tau_{\mathfrak{F}}$, namely, the everywhere tempered subspace of $\Pi(\GL_n(\A_F)^1)$. Concretely, this means that given a finite set of places $S$, and an unramified tempered representation $\sigma_S\in\Pi(\GL_n(F_S))$, one can find $\pi\in\mathfrak{F}$ whose $S$-components $\pi_S$ are arbitrarily close to $\sigma_S$. This would extend to a wide degree of generality an old observation of Piatetski-Shapiro and Sarnak~\cite{Sarnak1987} for level one Maass forms for compact congruence quotients of $\SL_n$ (see Corollary 1.3 and the discussion on page 330 of \cite{Sarnak1987}). 

\begin{remark}
By contrast, if we give $\Pi(\GL_n(\A_F))$ the {\it restricted} product topology, the set $\mathfrak{F}$ is discrete in $\Pi(\GL_n(\A_F))$; this is a point of view more adapted to computational problems of isolating and numerically computing cusp forms \cite{Sarnak2015}.
\end{remark}

Note that neither object in Conjecture \ref{master-conj} is a probability measure. By contrast, the second measure in \eqref{CountingMeasuresDef} is. We therefore set
\[ \tau_{\mathfrak{F}}^\circ=\frac{\tau_{\mathfrak{F}}}{{\rm vol}(\tau_\mathfrak{F})}; \]
this is a well-defined probability measure on $\Pi(\GL_n(\A_F)^1)$. Conjecture \ref{master-conj} then implies
\begin{equation}\label{master-conj2}
\frac{1}{|\mathfrak{F}(Q)|}\sum_{\pi\in\mathfrak{F}(Q)}\delta_\pi\longrightarrow \tau_{\mathfrak{F}}^\circ.
\end{equation}

\subsection{Sato--Tate measure}\label{sec:S-T}

Granting ourselves the statement \eqref{master-conj2}, we may identify the Sato--Tate measure $\mu_{\rm ST}(\mathfrak{F})$ of the universal family $\mathfrak{F}$.

Let us consider the local components of $\tau_{\mathfrak{F}}^\circ$. For a place $v$ we define $\tau_{\mathfrak{F},v}^\circ=\tau_{{\mathfrak{F},v}}/{\rm vol}(\tau_{\mathfrak{F},v})$, a probability measure on $\Pi(\GL_n(F_v))$ supported on the tempered spectrum; then $\tau_{\mathfrak{F}}^\circ=\prod_v \tau_{\mathfrak{F},v}^\circ$. Let $T$ denote the diagonal torus inside the Langlands dual group $\GL_n(\C)$ of $\GL_n$, and let $W$ be the associated Weyl group. For finite places $v$, the Satake isomorphism identifies the unramified admissible dual of $\mathrm{GL}_n(F_v)$ with the quotient $T/W$. It then makes sense to speak of the restriction of $\tau_{\mathfrak{F},v}^\circ$ to $T/W$, which we write (abusing notation) as $\tau_{\mathfrak{F},v}^\circ|_T$. In~\cite{SarnakShinTemplier2016}, Sarnak, Shin, and Templier define $\mu_{\rm ST}$ by the formula
\[
\mu_{\rm ST}(\mathfrak{F})=\lim_{x\rightarrow\infty}\frac{1}{x}\sum_{q_v < x}(\log q_v)\cdot \tau_{\mathfrak{F},v}^\circ|_T;
\]
thus $\mu_{\rm ST}(\mathfrak{F})$ is a measure on $T/W$, provided that the limit exists (which we shall presently show to be the case).

Note that, under the above identification, the tempered unramified unitary dual corresponds with $T_c/W$ where $T_c$ is the compact torus $U(n)\cap T$. Thus the restriction $\tau_{\mathfrak{F},v}^\circ|_T$ is supported on $T_c/W$ and we may think of the Sato--Tate measure as being defined on $T_c/W$. Now Lemma \ref{local-count} implies that for finite places $v$ the volume of $\tau_{\mathfrak{F},v}$ is given by $\Delta_v(1)\frac{\zeta_v(1)}{\zeta_v(n+1)^{n+1}}$. Thus, letting $\widehat{\omega}_v^{\rm pl}|_{T_c}$ denote the restriction of $\widehat{\omega}_v^{\rm pl}$ to $T_c/W$, we have
\[
\tau_{\mathfrak{F},v}^\circ|_{T_c}=\frac{\zeta_v(n+1)^{n+1}}{\Delta_v(1)\zeta_v(1)}\widehat{\omega}_v^{\rm pl}|_{T_c}=(1+{\rm O}(q_v^{-1}))\widehat{\omega}_v^{\rm pl}|_{T_c}.
\]
Since $\widehat{\omega}_v^{\rm pl}=\Delta_v(1)\widehat{\mu}_v^{\rm pl}$, we deduce that
\[
\mu_{\rm ST}(\mathfrak{F})=\lim_{q_v\rightarrow\infty}\widehat{\omega}_v^{\rm pl}|_{T_c}=\lim_{q_v\rightarrow\infty}\widehat{\mu}_v^{\rm pl}|_{T_c},
\]
where the latter limit is known to exist and have the following description (see, for example, \cite[p. 330]{Sarnak1987}).
\begin{cor}\label{cor-conj}
Assume Conjecture \ref{master-conj}. Then the Sato--Tate measure $\mu_{\rm ST}(\mathfrak{F})$ of the universal family is the push-forward of the probability Haar measure on $U(n)$ to $T_c/W$.\end{cor}

Using the Weyl integration formula, we have
\[
\mathrm{d}\mu_{\rm ST}(\mathfrak{F})(e^{it_1},\ldots ,e^{it_n})=\frac{1}{n!}\prod_{j<k}\big|e^{it_j}-e^{it_k}\big|^2\,\frac{\mathrm{d}t_1}{2\pi}\cdots\frac{\mathrm{d}t_n}{2\pi}.
\]
In particular, it follows from Corollary \ref{cor-conj} that the indicators
\[
i_1(\mathfrak{F})=\int_T|\chi(t)|^2 \,{\rm d}\mu_{\rm ST}(\mathfrak{F})(t),\quad i_2(\mathfrak{F})=\int_T\chi(t)^2\,{\rm d}\mu_{\rm ST}(\mathfrak{F})(t),\quad i_3(\mathfrak{F})=\int_T \chi(t^2)\,{\rm d}\mu_{\rm ST}(\mathfrak{F})(t)
\]
introduced in~\cite{SarnakShinTemplier2016}, where $\chi(t)={\rm tr}(t)$, take values $i_1(\mathfrak{F})=1$, $i_2(\mathfrak{F})=0$, and $i_3(\mathfrak{F})=0$ on the universal family. This is consistent with the expectation that the universal family $\mathfrak{F}$ is of unitary symmetry type.

\section{Outline of the proof}\label{section:outline}

To set up the proofs of Theorems \ref{main-theorem}--\ref{master}, we begin by decomposing the universal family according to discrete data:
\begin{enumerate}
\item the first such datum is the arithmetic conductor ideal $\q$ in the ring of integers of $F$ (see \S\ref{sec:global-cond});
\item the second is an archimedean spectral parameter, which enters through the decomposition of the admissible dual of $\GL_n(F_\infty)^1$ into a disjoint union over a discrete set of parameters $\underline{\delta}\in\mathcal{D}$. See \S\ref{disc-aut} for more details.
\end{enumerate}
The set $\mathcal{D}$ consists of equivalence classes of square-integrable representations on Levi subgroups of $\GL_n(F_{\infty})$. More precisely, the elements $\underline{\delta}$ of $\mathcal{D}$ are conjugacy classes of pairs $(M,\delta)$ consisting of a cuspidal Levi subgroup $M$ and an essentially discrete series representation $\delta$ of $M$. We can represent any class $\underline{\delta}\in\mathcal{D}$ by a square-integrable representation $\delta$ of $M^1=M/A_M$, where $M$ is a standard Levi subgroup (with blocks of descending size) and $A_M$ is the split component of its center. We shall call $(M,\delta)$ a \textit{standard representative} of $\underline{\delta}$.

It remains to impose a condition on the continuous archimedean spectral parameter. This can be done by specifying a nice $W(A_M)_\delta$-invariant subset $\Omega$ of the $\delta$-unitary spectrum $\mathfrak{h}_{\delta,{\rm un}}^*$, defined in \eqref{delta-unitary}. For example, if $\underline{\delta}\in\mathcal{D}$ has standard representative $(M,\delta)$, the set
\begin{equation}
\label{DefinitionOmegaX}
\Omega_{\underline{\delta},X}=\{\nu\in\mathfrak{h}_{\delta,{\rm un}}^*: q(\pi_{\delta,\nu})\leqslant X\}
\end{equation}
selects unitary representations $\pi_\infty\in\Pi(\GL_n(F_\infty)^1)_{\underline{\delta}}$ of archimedean conductor $q(\pi_\infty)\leqslant X$.

Given an ideal $\q$ and archimedean spectral data $[\delta,\Omega]$ as above, let $\mathfrak{H}(\q,\underline{\delta},\Omega)$ denote the set of $\pi\in\mathfrak{F}$ such that $\q(\pi)=\q$, $\underline{\delta}_\pi=\underline{\delta}$, and $\nu_\pi\in\Omega$. Then
\begin{equation}\label{obvious-decomp}
|\mathfrak{F}(Q)|=\sum_{1\leqslant\norm\q\leqslant Q}\sum_{\underline{\delta}\in\mathcal{D}}|\mathfrak{H}(\q,\underline{\delta},\Omega_{\underline{\delta},Q/\norm\q})|.
\end{equation}
In the parlance of \cite{SarnakShinTemplier2016,ShinTemplier2016}, the set $\mathfrak{H}(\q,\underline{\delta},\Omega)$ is what is called a {\it harmonic family}. One of the hallmarks of a harmonic family is that it can be studied by means of the trace formula. For this reason, it will be more convenient to work with the weighted sum
\begin{equation}\label{weightedsum}
N(\q,\underline{\delta},\Omega)=\sum_{\pi\in\mathfrak{H}(\q,\underline{\delta},\Omega)} \dim V_{\pi_f}^{K_1(\q)},
\end{equation}
which counts each $\pi\in\mathfrak{H}(\q,\underline{\delta},\Omega)$ with a weight corresponding to the dimension of the space of old forms for $\pi_f$. The quantities $|\mathfrak{H}(\q,\underline{\delta},\Omega)|$ and $N(\q,\underline{\delta},\Omega)$ can be related via newform theory, yielding
\begin{equation}\label{univ-decomp}
|\mathfrak{F}(Q)|=\sum_{1\leqslant\norm\q\leqslant Q}\; \sum_{\dd|\q}
\sum_{\underline{\delta}\in\mathcal{D}}\lambda_n(\q/\dd)N(\dd,\underline{\delta},\Omega_{\underline{\delta},Q/\norm\q}),
\end{equation}
where $\lambda_n=\mu\star\cdots\star\mu$ is the $n$-fold Dirichlet convolution of the M\"obius function on $F$. Indeed, from the dimension formula \eqref{eq:Reeder} of Reeder we deduce
\[
N(\q,\underline{\delta},\Omega)=\sum_{\dd\mid\q}\;\sum_{\pi\in \mathfrak{H}(\dd,\underline{\delta},\Omega)}d_n(\q/\dd)=\sum_{\dd\mid\q}d_n(\q/\dd)|\mathfrak{H}(\dd,\underline{\delta},\Omega)|,
\]
where $d_n=1\star\cdots\star 1$ is the Dirichlet inverse to $\lambda_n$. This equality holds for every integral ideal $\q$ and is hence an equality of arithmetical functions. Since the inverse of $d_n(\m)$ under Dirichlet convolution is $\lambda_n(\m)$, M\"obius inversion yields
\[
|\mathfrak{H}(\q,\underline{\delta},\Omega)|=\sum_{\dd|\q}\lambda_n(\q/\dd)N(\dd,\underline{\delta},\Omega).
\]
Taking $\Omega=\Omega_{\delta,Q/\norm\q}$, the claim \eqref{univ-decomp} then follows from \eqref{obvious-decomp}.

From this point, the proof of Theorem \ref{main-implication} proceeds as follows. We approximate $N(\q,\underline{\delta},\Omega)$ by the discrete spectral distribution $J_{\rm disc}$ of the trace formula, using a test function which 
\begin{enumerate}
\item {\it exactly picks out} the weight $\dim V_{\pi_f}^{K_1(\q)}$ and the condition $\underline{\delta}_\pi=\underline{\delta}$,
\item but which {\it smoothly approximates} the condition $\nu_\pi\in \Omega$, with the auxiliary parameter $R>0$ controlling the degree of localization. 
\end{enumerate}
The quality of this approximation is estimated in Part \ref{part1}, where we execute the passage from smooth to sharp count of the tempered spectrum in harmonic families. We obtain asymptotic results on the size of the spectrum and strong upper bounds on the size of the complementary spectrum for individual large levels $\q$, which are of independent interest. Here it should be noted that we require uniformity in $\q$, $\underline{\delta}$, and the domain $\Omega$, as all of them vary in our average~\eqref{univ-decomp}.

The successful execution of these steps of course depends on the trace formula input, which enters our argument through suitable applications of Property (ELM). Summing over $\q$ and appropriate spectral data as in \eqref{univ-decomp} then proves Theorem \ref{main-implication}.

\subsection{Prototypical example: classical Maass forms}\label{Approach2}

Since much of the work required to prove Theorem \ref{main-implication} involves the treatment of the continuous parameter $\nu_\pi$, it makes sense to illustrate the difficulties by describing the simplest case, where we restrict to the spherical spectrum for $\mathrm{GL}_2$ over $\mathbb{Q}$, consisting of even Maa{\ss} cusp forms. In classical language, we seek an asymptotic for the number of Hecke--Maa{\ss} cuspidal newforms on congruence quotients $Y_1(q)=\Gamma_1(q)\,\backslash\,\mathbb{H}$ of level $q$ and Laplacian eigenvalue $\lambda=1/4+r^2$ satisfying the bound $q(1+|r|)^2\leqslant Q$.

\subsubsection{Why existing results are insufficient}

A familiar environment for automorphic counting problems is that of Weyl's law. A Weyl law for $\GL_2$ over $\Q$, which is uniform in the level $q$, can be found in \cite[Corollary 3.2.3]{Palm2012}, where it is shown that $N_{\Gamma_1(q)}(T)$, the count of the cuspidal spectrum on $Y_1(q)$ with spectral parameter up to $T$, satisfies
\begin{equation}\label{MP}
 N_{\Gamma_1(q)}(T)=\frac{\mathrm{Vol}(Y_1(q))}{4\pi}T^2-\varphi(q)d(q)\frac{2}{\pi}T\log T+\mathrm{O}(q^2T).
 \end{equation}
Since $\mathrm{Vol}(Y_1(q))\asymp q^2$, by taking $1+T=\sqrt{Q/q}$ and summing over $q$, one expects the main term in the asymptotic for $|\mathfrak{F}(Q)|$ to be of size $\sum_{q\leqslant Q}q^2\big(\sqrt{Q/q}-1\big)^2\asymp Q^3$. Unfortunately, the total error term is also of size $Q^3$.

We see from this that one cannot simply sum \eqref{MP} over $q$ to count the universal family. This is not surprising, since when the level $q$ is of size comparable to $Q$ and $1+T=\sqrt{Q/q}$ is bounded, the error term in \eqref{MP} is of the same size as the main term, yielding only an upper bound. The loss of information in this range is deadly, since limit multiplicity theorems \cite{Sarnak1987} (or  Conjecture~\ref{weyl-schanuel-conj} more generally) suggest that $N_{\Gamma_1(q)}(1)\asymp q^2$, which would then in turn show that the bounded eigenvalue range contributes to $|\mathfrak{F}(Q)|$ with positive proportion. 

The important point here is that we cannot assume even a condition of the form $T\geqslant\frac1{100}$ if we wish to recover the correct leading constant in Theorem~\ref{main-implication}, since the complementary range contributes with positive proportion to the universal count.

\subsubsection{Weak spectral localization}

For our purposes, what we require from a uniform Weyl law for $N_{\Gamma_1(q)}(T)$ is an error term that is not only uniform in $q$ but in fact gives explicit savings in $q$ for bounded $T$. The gain in such an error term is the measure by which one can localize about a given eigenvalue in the cuspidal spectrum of $Y_1(q)$. For example, \eqref{log-q-savings} below implicitly requires the ability to estimate this spectrum for $T$ in ranges of size comparable to $1/\log q$.

Nevertheless, observe that even a modest improvement in the $q$-dependence in the error term with a complete loss of savings in the eigenvalue aspect -- something of the form
\begin{equation}\label{log-q-savings}
 N_{\Gamma_1(q)}(T)=\frac{\mathrm{Vol}(Y_1(q))}{4\pi}\left(\int_{-T}^Tr\tanh\pi r\, {\rm d}r+\mathrm{O}\left(\frac{1+T^2}{\log q}\right)\right)
\end{equation}
for $T>0$ -- is sufficient and yields an asymptotic of the form $c_0Q^3+\mathrm{O}(Q^3/\log Q)$, with an absolute $c_0>0$. The gain by $\log q$ in the error term, along with the absence of savings in the $T$ aspect, in the above expression coincides precisely with the type of estimate for error terms that we have encoded into Property (ELM).

Note that the demands one places on the savings in the $T$- and $q$-aspects are on unequal footings: we lose if we fail to show savings in the $q$-aspect (which is hard to acquire), but can afford to use the trivial bound in $T$ (which is easy to improve). For example, when $q=1$ we may clearly get by with the bound of $T^2$ -- or worse! This is essentially due to the fact that the parameter $T$ corresponds to one place only, whereas $q$ encodes all finite places.

\subsubsection{Correspondence with expanding geometric support}

One approaches \eqref{log-q-savings} through an application of the Selberg trace formula, which states
\begin{equation}
\label{SelbergTraceFormula}
\sum_{j\geqslant 0}h(r_j)=\frac{\mathrm{vol}(Y_1(q))}{4\pi}\int_{\mathbb{R}}h(r) r\tanh\pi r\,\text{d}r+\sum_{[\gamma]}f(\log N\gamma)\frac{\log N\gamma_0}{N\gamma^{1/2}-N\gamma^{-1/2}}+\dots,
\end{equation}
where $h$ is an even Paley--Wiener function, $\frac12+ir_j$ runs through spectral parameters of cuspidal eigenforms on $Y_1(q)$, $f$ is the inverse Fourier transform of $h$, and $[\gamma]$ runs through hyperbolic conjugacy classes in $\Gamma_1(q)$. Here, the real number $N\gamma>1$ satisfies ${\rm tr}\,\gamma= N\gamma^{1/2}+N\gamma^{-1/2}$, and $[\gamma_0]$ is the unique primitive hyperbolic conjugacy class associated with $[\gamma]$. The remaining terms arise from non-hyperbolic conjugacy classes on the geometric side and the Eisenstein spectrum on the spectral side. 

One generally works with functions $h$ which approximate the characteristic function $\chi_I$ of a spectral interval $I$ (or ball, in higher rank). One way of constructing such $h$ is by convolving $\chi_I$ with a suitably nice $h_0\in\mathcal{PW}(\C)$, centered at the origin, or, in fact, for a parameter $R>1$, with the rescaled function $\nu\mapsto h_0( R\nu)$, which makes the walls around the interval $I$ of length $1/R$. For example, if $I=[-T,T]$, we may essentially localize $r$ to $[-T+\mathrm{O}(1/R),T+\mathrm{O}(1/R)]$ and the de-smoothing process in Weyl's Law incurs an error of size $T/R$.

Note, however, that the spectral test functions $h$ obtained in this way, while serving our purposes, have two inevitable drawbacks:
\begin{enumerate}
\item they assign an exponential weight to the non-tempered spectrum (as large as $e^{CR|r_j|}$ for $\frac12+ir_j$, $r_j\not\in\mathbb{R}$),
\item have Fourier transforms $f$ supported on a ball of radius $R$ about the origin.
\end{enumerate}
In the remaining subsections of this exposition of the Maass case, we discuss how we deal with the obstacles created by these two properties.

\subsubsection{Exponentially weighted discrete spectrum}

After estimating the contributions from the Eisenstein series and the non-identity terms in the geometric side, a Weyl law of the form \eqref{log-q-savings} follows by converting the smooth count of \eqref{SelbergTraceFormula} to a sharp-cutoff. This conversion requires local bounds on the discrete spectrum, which itself involves another application of the trace formula. See \cite[Section 2]{LapidMuller2009} for a nice overview of this, by now, standard procedure.

Note, however, that by the first drawback above, the bounds on the discrete spectrum we require are exponentially weighted by the distance to the tempered axis. Estimating this weighted count is closely related to density estimates for exceptional Maass forms. In the context of quotients by the upper half-plane this is classical, but in higher rank a delicate construction of positive-definite dominating test functions is required. This is described in more detail in \S\ref{S10-desciption}.

\subsubsection{Expanding geometric side}\label{sec-low-lying-length-spectrum}

In the case of a fixed level and large eigenvalue, it is possible to localize $r$ within ${\rm O}_{\Gamma_1(q)}(1)$, without seeing any of the hyperbolic spectrum. In light of the Prime Geodesic Theorem, which states that
\begin{equation}
\label{PrimeGeodesicTheorem}
\#\{\text{primitive }\gamma:\log N\gamma\leqslant T\}\sim_{\Gamma_1(q)}e^T/T,
\end{equation}
this approach can be pushed to the limit by entering up to $\mathrm{O}_{\Gamma_1(q)}(\log T)$ of the hyperbolic spectrum, which leads to the familiar (and currently best available) error term $\mathrm{O}_{\Gamma_1(q)}(T/\log T)$ in a refinement of \eqref{MP}, for a fixed level $q$. In fact, it is well-known that a purely analytic use of the trace formula can only give $\log T$ savings over the local Placherel bound even in the upper bounds for multiplicities of Laplacian eigenvalues.  Analogous reasoning holds in the $q$-aspect, as we now describe.

Estimates on $N_{\Gamma_1(q)}(T)$ in bounded ranges of $T$ with error terms that feature explicit savings in $q$ correspond to instances of \eqref{SelbergTraceFormula} such that the support of $f$ is expanding for large $q$; thus, controlling the number and magnitude of conjugacy classes of $\gamma\in\Gamma_1(q)$ in \eqref{SelbergTraceFormula} is an essential ingredient in any limit multiplicity-type statement. One can use effective Benjamini--Schramm type statements \cite{AbertBergeronBiringerGelanderNikolovRaimbaultSamet2017}, adapted to this non-compact setting \cite{Raimbault2017}, to show that the number of closed geodesics of length at most $R$ in $Y_1(q)$ is at most ${\rm O}(e^{CR})$, for some constant $C>0$.\footnote{For the example $Y_1(q)$, an elementary argument shows that there are in fact no closed geodesics of length $\ll\log q$. However, this fact is not robust: it already disappears for $\Gamma_0(q)$ or for the analog of $\Gamma_1(q)$ over number fields.} This control allows us to use functions $h$ arising as Fourier transforms of functions supported up to $\mathrm{O}(\log q)$. After some work\footnote{Working directly with the Selberg trace formula, and being less wasteful in the $T$ aspect, would yield a $\mathrm{O}(q^2T/\log q)$ error term in \eqref{MP}.} to estimate all other contributions to \eqref{SelbergTraceFormula}, we obtain \eqref{log-q-savings}. More precisely, Proposition \ref{MainCountingResultGLn}, when specialized to the case of $\GL_2$ over $\Q$ yields the following refinement of \eqref{log-q-savings}: there is $\theta>0$ such that
\begin{equation}\label{eq:GL2Q-uniform-Weyl}
N_{\Gamma_1(q)}(T)=\frac{\mathrm{Vol}(Y_1(q))}{4\pi}\left(\int_{-T}^Tr\tanh\pi r\, {\rm d}r+\mathrm{O}\left(\frac{\min \{T,T^2\}}{\log q}+\frac{1}{\log^3 q}+q^{-\theta}T^2\right)\right)
\end{equation}
for all $T>0$. Taking $T=1/\log q$, this in particular implies that the dimension of the $\lambda=1/4$ eigenspace is at most ${\rm O}(\mathrm{Vol}(Y_1(q))/\log^3 q)$.

\subsection{Overview of Part \ref{part1}}
 
We now return to the general setting, and describe in more detail the contents of each section.

\subsubsection*{Section~\ref{sec:mainimplicationproof}: Preliminaries}
In this section, we set up the notation for Proposition \ref{MainCountingResultGLn}, the basic estimate of Part \ref{part1}. This result roughly states that for fixed discrete data $\q$ and $\underline{\delta}$, and a nice set $P$ in the tempered subspace $i\mathfrak{h}_M^*$, Property (ELM) allows us to control the difference between the sharp count $N(\q,\underline{\delta},P)$ of \eqref{weightedsum} and the expected main term $D_F^{n^2/2}\Delta_F^*(1)\varphi_n(\q)\int_P{\rm d}\widehat{\omega}^{\rm pl}_\infty$. The error terms appearing in Proposition \ref{MainCountingResultGLn} depend on an approximation parameter $1\leqslant R\ll \log (2+\norm\q)$ and arise from several sources, namely, the passage from smooth to sharp count (in Section~\ref{sec:smooth-to-sharp-tempered}), the estimation of complementary spectrum (in Section~\ref{ExceptionalSpectrumSection}), and the applications of Property (ELM). As we explain in Section~\ref{sec:mainimplicationproof}, Proposition~\ref{MainCountingResultGLn} provides a Weyl law with explicit level savings and uniformity in the region $P$ and the discrete data.

\subsubsection*{Section \ref{spec-loc}: Spectral localizers}
In this section we define various archimedean Paley--Wiener functions $h_R^{\delta,\mu}$ which localize around given spectral parameters $(\delta,\mu)$, for $\mu\in i\h_M^*$, and provide some basic estimates for their analytic behavior. Similarly, we define $h_R^{\delta,P}$ for nice subsets $P$ of $i\h_M^*$, with the basic idea that the characteristic function $\chi_P(\nu)$ is very well approximated by $h_R^{\delta,P}(\nu)$ on points $\nu$ that are firmly inside or outside $P$. The test functions $f_R^{\delta,\mu}$ or $f_R^{\delta,P}$ associated with these spectral localizers, through an invocation of the Paley--Wiener theorem of Clozel--Delorme, will be used in the trace formula to prove Proposition \ref{MainCountingResultGLn}.

\subsubsection*{Section~\ref{ExceptionalSpectrumSection}: Exponentially weighted discrete spectrum}\label{S10-desciption} We would like to approximate $N(\q,\underline{\delta},P)$ using $J_{\rm temp}(\varepsilon_{K_1(\q)}\otimes f_R^{\delta,P})$, the contribution of the tempered discrete spectrum to the trace formula. But an application of Property (ELM) requires working with $J_{\rm disc}$ rather than $J_{\rm temp}$. We must therefore control the difference
\begin{equation}\label{sec3.2-defn-J-comp}
J_{\rm comp}(\varepsilon_{K_1(\q)}\otimes f_R^{\delta,P})=J_{\rm disc}(\varepsilon_{K_1(\q)}\otimes f_R^{\delta,P})-J_{\rm temp}(\varepsilon_{K_1(\q)}\otimes f_R^{\delta,P}).
\end{equation}
Note that the spectral sampling functions $h_R^{\delta,\mu}$ and $h_R^{\delta,P}(\nu)$, being of Paley--Wiener type, act differently on spectral parameters $\nu_{\pi}$ off the tempered spectrum $i\mathfrak{h}_M^{\ast}$: they exhibit exponential growth in $\|{\rm Re}\,\nu_{\pi}\|$. In fact, the rate of exponential growth is directly related to the size $R$ of the support of the test functions used on the geometric side; see \S\ref{PWCD} for details. For this reason, the contributions from $\pi\in \Pi_{\rm disc}(\bm{G}(\A_F)^1)_{\underline{\delta}}$ for which $\pi_{\infty}$ is not tempered must be estimated separately; specifically, for a suitable parameter $R>0$ we require an upper bound for the exponentially weighted sum of the shape
\[
\sum_{\substack{\pi\in \Pi_{\rm disc}(\GL_n(\A_F)^1)_{\underline{\delta}}\\ \|{\rm Im}\, \nu_{\pi}-\mu\|\ll 1/R}}\dim V_{\pi_f}^{K_1(\q)}e^{R\|{\rm Re}\,\nu_{\pi}\|}.
\]
This is majorized, using an application of Property (ELM) and a very delicate construction of archi\-mede\-an positive-definite dominating test functions, in Section \ref{ExceptionalSpectrumSection}, combining inputs from Sections \ref{ExceptionalSpectrumSection} and \ref{sec:Paley-Wiener}. The principal difficulty in this construction is that every spectral localizer which detects, with exponential weights, the desired exceptional eigenvalues $\nu_{\pi}$ with $\|\mathrm{Im}\,\nu_{\pi}-\mu\|\ll 1/R$ also picks up other unwanted, and difficult to control, terms. We separate such ``good'' and ``bad'' contributions and construct the dominating test function in \S\ref{existence-test-functions-subsection}, using a delicate combination of geometric and analytic ingredients. Two main ingredients of geometric nature (the inductive construction of ``good'' and ``bad'' tubular neighborhoods and a lemma guaranteeing that the real and imaginary parts of ``bad'' contributions are out of alignment) are proved in \S\ref{good-bad-subsec} and \S\ref{angles-subsec}, while the purely analytic construction of Paley--Wiener functions with desirable asymptotics is the subject of Section \ref{sec:Paley-Wiener}.

For suitable $R\ll\log(2+\norm\q)$, the upper bound of Proposition~\ref{comp-mu} is typically comparable to the expected contribution of the tempered spectrum in the above sum. In fact, we achieve an additional savings by a power of $R$ in this majorization when the spectral center $\mu$ is singular, coming from the degree of vanishing of the Plancherel measure at $\mu$. In particular, this latter condition is automatic when estimating the exceptional spectrum, as in Proposition~\ref{KR}. This observation accounts for the extra savings by $\log^3Q$ in the contribution to $|\mathfrak{F}(Q)|$ of cusp forms which are non-tempered at infinity.

\subsubsection*{Section~\ref{sec:Paley-Wiener}: A Paley--Wiener function} In this purely analytic section, we construct a Paley--Wiener function $h(\nu)$ such that, for sizable $\mathrm{Re}\,\nu\gg 1/R$, $h(R\nu)$ exhibits exponential growth in $R\|\mathrm{Re}\,\nu\|$ if $\mathrm{Im}\,\nu$ is small (the ``good'' terms in the application in Section \ref{ExceptionalSpectrumSection}), and is sufficiently smaller if $\mathrm{Im}\,\nu$ is a bit bigger and away from full alignment with $\mathrm{Re}\,\nu$ (the ``bad'' terms). At the heart of this argument are new asymptotics for the complex Fourier transform of the smooth radial bump function on the unit ball, which we prove by a rather intricate application of the method of stationary (complex) phase and which may be of independent interest.

\subsubsection*{Section~\ref{sec:smooth-to-sharp-tempered}: Smooth to sharp for tempered parameters} Here we put to use the preceding results to prove Proposition \ref{MainCountingResultGLn}. Using the analytic properties of $h_R^{\delta,P}$, we first identify $N(\q,\underline{\delta},P)$ as the sum of $J_{\rm temp}(\varepsilon_{K_1(\q)}\otimes f_R^{\delta,P})$ with a boundary error, of the form $N(\q,\underline{\delta},\partial P(1/R))$, where $\partial P(1/R)$ is the $1/R$-fattened boundary of $P$. Then, using the results from Section \ref{ExceptionalSpectrumSection}, the term $J_{\rm temp}(\varepsilon_{K_1(\q)}\otimes f_R^{\delta,P})$ is amenable to the application of Property (ELM). Finally, we show that the boundary contributions can also be estimated from above by smooth sums, which can in turn be estimated by further applications of Property (ELM). 

\subsubsection*{Section~\ref{sec:temp-cor}: Summing error terms over discrete data} With Proposition \ref{MainCountingResultGLn} established, we can sum $N(\q,\underline{\delta},P)$ over all discrete parameters $\underline{\delta}$ and levels $\q$ to obtain the full count $|\mathfrak{F}(Q)|$ in \eqref{univ-decomp}. Bounding the resulting averages of errors terms proves Theorem \ref{main-implication}.

\subsection{Overview of Part \ref{part3}}\label{blurb}

In Part \ref{part3}, we establish Theorem \ref{master}. The proof naturally divides into two parts, corresponding to bounding $J_{\rm geom}-J_{\rm cent}$ on the geometric side and $J_{\rm spec}-J_{\rm disc}$ on the spectral side. The estimations are not symmetric in the way they are proved, nor in the degree of generality in which they are stated. We would like to briefly describe these results here, and in particular explain why we are at present unable to establish Property (ELM) in all cases.

The main result on the geometric side is Theorem~\ref{geom-side} in which we show the existence of constants $C,\theta>0$ such that for any $R>0$, integral ideal $\q$, and test function $f\in\mathcal{H}(\GL_n(F_\infty)^1)_R$, we have
\begin{equation}\label{intro:geom}
J_{\rm geom}(\varepsilon_{K_1(\q)}\otimes f)-J_{\rm cent}(\varepsilon_{K_1(\q)}\otimes f)\ll e^{CR}\norm\q^{n-\theta}\|f\|_\infty.
\end{equation}
This can be thought of as a sort of geometric limit multiplicity theorem, although it is only non-trivial in the $\q$ aspect. The exponential factor $e^{CR}$ should be compared to \eqref{PrimeGeodesicTheorem}. The latter shows that $R\ll\log\norm\q$ is an allowable range in which the main term dominates. The proof of Theorem~\ref{geom-side} occupies most of Sections \ref{sec:geom1} and \ref{sec:geom2}. Indeed, in Section \ref{sec:geom1} we reduce the problem to a local one, and in Section \ref{sec:geom2} we bound the relevant local weighted orbital integrals. 

The proof of our local estimates relies crucially on several recent developments, due to Finis--Lapid, Matz, Matz--Templier, and Shin--Templier. In particular, a central ingredient in the power savings in $\norm\q$ comes from the work of Finis--Lapid \cite{FinisLapid2018} on the intersection volumes of conjugacy classes with open compact subgroups. On the other hand, the source of the factor $\|f\|_\infty$ comes from estimating archimedean weighted orbital integrals {\it trivially}, by replacing $f$ by the product of $\|f\|_\infty$ with the characteristic function of its support. As the latter is, by hypothesis, contained in $\bfk_\infty \exp(B(0,R))\bfk_\infty$, it is enough then to have polynomial control in the support of the test function on these weighted orbital integrals. This can be viewed as the archimedean analogue of polynomial control in the Hecke depth aspect, as developed in a variety of contexts by the aforementioned authors. We in fact extract this from a careful reading of the papers of Matz \cite{Matz2017} and Matz--Templier \cite{MatzTemplier2021}. 

Comparing the bound \eqref{intro:geom} to the statement of Property (ELM), it is clear that if one takes $f=f_R^{\delta,\mu}$, then one wants to understand $\|f_R^{\delta,\mu}\|_\infty$ in terms of $R^{-\dim\mathfrak{h}_M}\beta_M^G(\delta,\mu)$, which, as we discuss after Property (ELM), is a natural majorizer of the Plancherel mass of $h_R^{\delta,\mu}$. It is at this point that we impose the condition that for $n>2$ the discrete parameter $\underline{\delta}$ is the trivial character on the torus (and omit it from the notation). In this case, the Paley--Wiener functions that we use for spectral localization can be inverted by integration against the spherical function $\varphi_{\nu}$. Since $\|\varphi_{\nu}\|_\infty\leqslant 1$ for tempered parameters $\nu$, we obtain $\|f^{\mu}_R\|_\infty\leqslant \|h_R^{\mu}\|_{L^1(\widehat{\omega}_\infty^{\rm pl})}$, as desired. For $\GL_2$, we use a slightly more general inversion formula, valid for $\tau$-spherical functions, where $\tau$ is an arbitrary $\bfk_\infty$-type; see \S\ref{sec:sph}. In any case, the test functions $f_R^{\delta,\mu}$ which we use in Property (ELM) are all defined in Section \ref{PW}, and their main properties are summarized in Proposition \ref{test-fn-bd}.

On the spectral side, our main result is Theorem \ref{spec-est}, which roughly states that, for $R\ll\log \norm\q$,
\[
J_{\rm spec}(\varepsilon_{K_1(\q)}\otimes f_R^{\delta,\mu})-J_{\rm disc}(\varepsilon_{K_1(\q)}\otimes f_R^{\delta,\mu})\ll \norm\q^{n-\theta} R^{-\dim\mathfrak{h}_M}\beta_M^G(\delta,\mu)
\]
for some $\theta>0$. The argument uses induction on $n$. Indeed the above difference can be written as a sum over proper standard Levi subgroups $\bm{M}\neq \bm{G}$ of $J_{{\rm spec}, \bm{M}}(\varepsilon_{K_1(\q)}\otimes f_R^{\delta,\mu})$, and each $\bm{M}$ is a product of $\GL_m$'s for $m<n$. The induction step itself relies critically on several ingredients. Besides the bounds on the geometric side of the trace formula of Sections \ref{sec:geom1} and \ref{sec:geom2}, and the properties of the test functions of Section \ref{PW}, the proof uses in an essential way the Tempered Winding Number property of \cite{FinisLapidMuller2015} and the Bounded Degree Property of \cite{FinisLapidMuller2012}. Our presentation follows that of several recent works, such as \cite[\S 15]{Matz2017}, but differs in that we make explicit the dependence in the parameter $R$ and in the level $\q$.

Putting all estimates together, we prove Theorem \ref{master} in Theorem \ref{money-cor}.

\part{Preliminaries}

\section{Global structures and the Arthur trace formula}\label{sec:ATF}

The goal of this section is to put in place the basic notation associated with the algebraic group $\bm{G}=\GL_n$ defined over a number field $F$, and then to state the non-invariant Arthur trace formula, for use in Part \ref{part3} of the paper.

\subsection{Field notation}\label{field}

We recall some standard notation relative to the number field $F$. 

Let $d=[F:\Q]$ be the degree of $F$ over $\Q$. Let $r_1$ and $r_2$ be the number of real and inequivalent complex embeddings of $F$, so that $r_1+2r_2=d$. Write $r=r_1+r_2$. Then $r$ is the number of inequivalent embeddings of $F$ into $\C$. Let $\mathcal{O}_F$ be the ring of integers of $F$. For an ideal $\mathfrak{n}$ of $\mathcal{O}_F$ let $\norm (\mathfrak{n})=|\mathcal{O}_F/\mathfrak{n}|$ be its norm. Write $D_F$ for the absolute discriminant of $F$.

For a normalized valuation $v$ of $F$, inducing a norm $|\cdot |_v$, we write $F_v$ for the completion of $F$ relative to $v$. For $v<\infty$ let $\mathcal{O}_v$ be the ring of integers of $F_v$, $\p_v$ the maximal ideal of $\mathcal{O}_v$, $\varpi_v$ any choice of uniformizer, and $q_v$ the cardinality of the residue field.

Let $\zeta_F(s)=\prod_{v<\infty}\zeta_v(s)$ for ${\rm Re}(s)>1$ be the Dedekind zeta function of $F$. Write $\zeta_F^*(1)$ for the residue of $\zeta_F(s)$ at $s=1$. We let $\A_F$ denote the ring of adele ring of $F$ and $\A_f$ the ring of finite adeles.

\subsection{Levi and parabolic subgroups}\label{sec:subgp}

We let $\bm{G}=\GL_n$, viewed as an algebraic group defined over $F$. Let $\bm{P_0}$ denote the standard Borel subgroup  of upper triangular matrices and $\bm{T}_0$ the diagonal torus of $\bm{G}$. Let $\Phi^{\bm{G}}=\Phi(\bm{T}_0,\bm{G})$ be the set of roots of $\bm{T}_0$ on the Lie algebra $\mathfrak{g}$ of $\bm{G}$ and $\Phi^{\bm{G},+}$ the subset of positive roots with respect to $\bm{P}_0$. Let $\bm{Z}$ denote the center of $\bm{G}$.

A Levi subgroup of $\bm{G}$ is called {\it semistandard} if it contains $\bm{T}_0$; it is automatically defined over $F$. Let $\mathcal{L}$ denote the finite set of all semistandard Levi subgroups of $\bm{G}$. For $\bm{M}\in\mathcal{L}$ we let $\Phi^{\bm{M}}=\Phi(\bm{T}_0,\bm{M})\subset \Phi$ be the set of roots of $\bm{T}_0$ on the Lie algebra $\mathfrak{m}$ of $\bm{M}$. Write $\bm{Z}_{\bm{M}}$ for the center of $\bm{M}$. Let $\mathcal{L}(\bm{M})=\{\bm{L}\in\mathcal{L}: \bm{M}\subset \bm{L}\}$.

An $F$-parabolic subgroup $\bm{P}$ of $\bm{G}$ is called {\it semistandard} if it contains $\bm{T}_0$. Let $\mathcal{F}$ denote the finite set of all semistandard $F$-parabolic subgroups. For $\bm{P}\in\mathcal{F}$, let $\bm{U}_{\bm{P}}$ denote the unipotent radical of $\bm{P}$ and $\bm{M}_{\bm{P}}$ the unique semistandard Levi subgroup such that $\bm{P}=\bm{M}_{\bm{P}}\bm{U}_{\bm{P}}$. When $\bm{P}=\bm{P}_0$ we write $\bm{U}_0$ for $\bm{U}_{\bm{P}_0}$ and of course $\bm{M}_{\bm{P}_0}$ is simply $\bm{T}_0$. For $\bm{M}\in\mathcal{L}$ let $\mathcal{F}(\bm{M})=\{\bm{P}\in\mathcal{F}: \bm{M}\subset \bm{P}\}$; clearly, $\mathcal{F}(\bm{T}_0)=\mathcal{F}$ and $\mathcal{F}(\bm{G})=\{\bm{G}\}$. We have a map $\mathcal{F}(\bm{M})\rightarrow\mathcal{L}(\bm{M})$ sending $\bm{P}$ to $\bm{M}_{\bm{P}}$. Denote by $\mathcal{P}(\bm{M})$ the subset of $\mathcal{F}(\bm{M})$ consisting of those $F$-parabolic subgroups having Levi component $\bm{M}$. Thus $\mathcal{P}(\bm{M})=\mathcal{F}(\bm{M})- \bigcup_{\bm{L}\supsetneq \bm{M}}\mathcal{F}(\bm{L})$.

We call an $F$-parabolic subgroup $\bm{P}$ of $\bm{G}$ {\it standard} if it contains $\bm{P}_0$. Each $\bm{G}$-conjugacy class of parabolic subgroups contains a unique standard member. Similarly, a semistandard Levi subgroup $\bm{M}$ of $\bm{G}$ is {\it standard} if it is contained in a standard parabolic subgroup. Two standard Levi subgroups $\bm{M}_{\bm{P}}\subset\bm{P}$ and $\bm{M}_{\bm{Q}}\subset\bm{Q}$ may be conjugate. Write $\mathcal{F}_{\rm st}$ and $\mathcal{L}_{\rm st}$ for the respective subsets of standard elements. Then $\mathcal{F}_{\rm st}$ (resp. $\mathcal{L}_{\rm st}$) is in bijection with the set of ordered (resp., unordered) partitions of $n$. Namely, the ordered partition $(n_1,\ldots ,n_m)$ is sent to the block upper triangular subgroup of $\bm{G}$ whose diagonal blocks have successive dimensions given by $(n_1,\ldots ,n_m)$; forgetting the ordering then corresponds with taking the conjugacy class of its Levi subgroup.

For $\bm{M}\in\mathcal{L}$ we let $W_{\bm{M}}=N_{\bm{G}}(\bm{M})/\bm{M}$ denote the Weyl group of $\bm{M}$. When $\bm{M}=\bm{T}_0$ we simplify $W_{\bm{T}_0}$ to $W_0$. Then each $W_{\bm{M}}$ can be identified with a subgroup of $W_0$. More concretely,  $W_{\bm{M}}$ permutes the blocks of equal size.

For a finite non-empty collection of places $S$ of $F$ and $\bm{M}\in\mathcal{L}$ we write $\bm{M}_S=\bm{M}(F_S)$. When $\bm{M}=\bm{T}_0$ we shorten this to $\bm{T}_S=\bm{T}_0(F_S)$. In particular, when $S$ consists of all archimedean places, we write $\bm{M}_\infty$. 

\begin{remark}\label{rem:boldface-notation}
When $\bm{M}=\bm{G}$ or $\bm{M}=\bm{T}_0$, we shall often write $G_S$ in place of $\bm{G}_S$ and $T_S$ in place of $\bm{T}_S$ (dropping the boldface). We shall refrain from doing so more systematically, however, since in later sections we shall need to distinguish between two types of Levi subgroups of $G_S$: those that are defined rationally, as above, and those that are place-by-place products of Levi subgroups for each $G_v$. No such confusion can arise for the two distinguished Levi subgroups $\bm{M}=\bm{G}$ and $\bm{M}=\bm{T}_0$.
\end{remark}

For any subgroup $\bm{H}$ of $\bm{G}$ we may consider the $\R$-points $\bm{H}(\R)$, where $\R$ is embedded diagonally in $F_\infty\subset\A_F$. In particular, $\bm{G}(\R)$ is the diagonally embedded copy of $\GL_n(\R)$ inside $G_\infty=\prod_{v\mid\infty}\bm{G}(F_v)$.

\subsection{Characters, cocharacters, and split components}\label{sec:global-characters}

For $\bm{M}\in\mathcal{L}$ we let $X^*(\bm{M})$ be the group of $F$-rational characters of $\bm{M}$. Then $X^*(\bm{M})$ can be identified with $\Z^m$ for some $m\geqslant 1$. Namely, if $\bm{M}$ is isomorphic to $\GL_{n_1}\times\cdots\times\GL_{n_m}$, then $\lambda=(\lambda_i)\in\Z^m$ corresponds to the character $\chi^\lambda(g)=\prod\det g_i^{\lambda_i}$. Let $X_*(\bm{M})$ be the lattice of $F$-rational cocharacters. We then write $X_*^+(\bm{M})=\{\lambda\in X_*(\bm{M}): \langle \alpha,\lambda\rangle\geqslant 0 \text{ for all simple }\alpha\in\Phi^+\}$ for the cone of positive cocharacters.

Let $A_{\bm{M}}$ be the connected component of the identity in $\bm{Z}_{\bm{M}}(\R)$, the real points of the center of $\bm{M}$; then $A_{\bm{M}}$ is of dimension $m$ over $\R$. For example, $A_{\bm{G}}$ is the image of $\R_{>0} I_n$ embedded diagonally across all archimedean places. When $\bm{M}=\bm{T}_0$ we write $A_0$ for $A_{\bm{T}_0}$.

We put $\bm{M}(\A_F)^1=\bigcap_{\chi\in X^*(\bm{M})}\ker (|\chi|_{\A_F})$. Then $\bm{M}(F)$ is a discrete subgroup of finite covolume in $\bm{M}(\A_F)^1$. For example, if $\bm{M}=\bm{G}$ then $\bm{G}(\A_F)^1=\{g\in\GL_n(\A_F): |\det g|_{\A_F}=1\}$. More generally, if $\bm{M}$ is isomorphic to $\GL_{n_1}\times\cdots\times\GL_{n_m}$, then $\bm{M}(\A_F)^1$ is isomorphic to $\GL_{n_1}(\A_F)^1\times\cdots\times\GL_{n_m}(\A_F)^1$. One has a direct product decomposition $\bm{M}(\A_F)=\bm{M}(\A_F)^1\times A_{\bm{M}}$. 

When $S$ contains all archimedean places, we have similar decompositions for coefficients in $F_S$, namely $\bm{M}_S=\bm{M}_S^1\times A_{\bm{M}}$, where $\bm{M}_S^1=\bm{M}_S\cap\bm{M}(\A_F)^1$. Concretely, when $\bm{M}\in\mathcal{L}_{\rm st}$ is a standard Levi, of block diagonal form $\GL_{n_1}\times\cdots\times\GL_{n_m}$, then
\begin{equation}\label{eq:defn:G1}
\bm{M}_S^1=\bigg\{{\rm diag}(g_1,\ldots ,g_{n_m})\in \bm{M}(F_S): g_i=(g_{vi})_{v\in S}\in\GL_{n_i}(F_S),\, \prod_{v\in S} |\det g_{vi}|_v=1\bigg\}.
\end{equation}
Of particular importance is the case $S=\{v\mid\infty\}$, as it will be throughout Section \ref{sec:arch-rep-PWCD}.

We set $\mathfrak{a}_{\bm{M}}^*=X^*(\bm{M})\otimes_\Z\R$ and $\mathfrak{a}_{\bm{M}}={\rm Hom}_\R(\mathfrak{a}_{\bm{M}}^*,\R)={\rm Hom}_\Z(X^*(\bm{M}),\R)$. When $\bm{M}=\bm{T}_0$ we write $\aa_0$ for $\aa_{\bm{T}_0}$. Then $\mathfrak{a}_{\bm{M}}$ is spanned by the cocharacter lattice $X_*(\bm{M})$ and the map $\log: A_{\bm{M}}\rightarrow \mathfrak{a}_{\bm{M}}$, defined by $e^{\langle \chi,\log a\rangle}=|\chi(a)|$, for $a\in A_{\bm{M}}$ and $\chi\in X^*(\bm{M})$, is an isomorphism. We may furthermore identify $\aa_{\bm{M}}$ with the Lie algebra ${\rm Lie}(A_{\bm{M}})$ inside $\mathfrak{g}$, by composing the preceding isomorphism with the exponential map on Lie groups. For example, when $\bm{M}$ is standard, with (ordered) block decomposition of the form $\GL_{n_1}\times\cdots\times\GL_{n_m}$, then $\aa_{\bm{M}}={\rm Lie}(A_{\bm{M}})$ is the subspace in $\mathfrak{g}$ consisting of diagonal matrices of the form ${\rm diag}(\R I_{n_1},\cdots , \R I_{n_m})$.

For $\bm{L}\in\mathcal{L}(\bm{M})$, there is a natural inclusion $\mathfrak{a}_{\bm{L}}^*\rightarrow\mathfrak{a}_{\bm{M}}^*$ induced by restriction of cocharacters $X^*(\bm{L})$ to $X^*(\bm{M})$. By dualizing we obtain a surjective map $\aa_{\bm{M}}\rightarrow\aa_{\bm{L}}$, whose kernel we denote by $\aa_{\bm{M}}^{\bm{L}}$. Equivalently, $\aa_{\bm{M}}^{\bm{L}}$ is the annihilator of $\aa_{\bm{L}}^*$ in $\aa_{\bm{M}}$. For example, when $\bm{M}$ is standard as in the last paragraph, $\aa_{\bm{M}}^{\bm{G}}$ can be identified with the diagonal matrices ${\rm diag}(\R I_{n_1},\cdots , \R I_{n_m})$ of trace zero. Note, furthermore that there is a canonical linear isomorphism $\mathfrak{a}_M^*=X^*(M)\otimes\R\rightarrow X^*(A_M)\otimes \R$. The surjective restriction homomorphism $X^*(A_M)\rightarrow X^*(A_L)$ then induces a dual  injection $\mathfrak{a}_L\rightarrow\mathfrak{a}_M$, which splits the exact sequence $0\rightarrow\mathfrak{a}_{\bm{M}}^{\bm{L}}\rightarrow\mathfrak{a}_{\bm{M}}\rightarrow\mathfrak{a}_{\bm{L}}\rightarrow 0$. We therefore have a direct sum decomposition $\mathfrak{a}_{\bm{M}}=\mathfrak{a}_{\bm{L}}\oplus\mathfrak{a}_{\bm{M}}^{\bm{L}}$.

\subsection{Maximal compacts and the Iwasawa decomposition}\label{sec:Iwasawa-decomp}

For a finite place $v$ we let $\bfk_v=\GL_n(\mathcal{O}_v)$ be the standard maximal compact open subgroup of $G_v=\bm{G}(F_v)$ and put $\bfk_f=\prod_{v<\infty} \bfk_v$. For archimedean $v$, we let $\bfk_v$ be ${\rm O}(n)$ or ${\rm U}(n)$, according to whether $v$ is real or complex. For a finite set of places $S$, denote $\bfk_S=\prod_{v\in S}\bfk_v$ and $\bfk^S=\prod_{v\notin S}\bfk_v$. When $S=\{v\mid\infty\}$, we shall prefer to write $\bfk_\infty=\prod_{v\mid\infty}\bfk_v$ and $\bfk_f=\prod_{v<\infty}\bfk_v$, respectively. Then $\bfk=\prod_v\bfk_v=\bfk_f\bfk_\infty$ is a maximal compact subgroup of $\bm{G}(\A_F)$. 

When $\bm{M}\in\mathcal{L}_{\rm st}$ and $\bm{P}\in\mathcal{P}(\bm{M})$ has unipotent radical $\bm{U}_{\bm{P}}$, there is a global Iwasawa decomposition $\bm{G}(\A_F)=\bm{P}(\A_F)\bfk=\bm{U}_{\bm{P}}(\A_F)\bm{M}(\A_F)\bfk$. Similarly, at every place $v$ of $F$ we have an Iwasawa decomposition $G_v=\bm{P}(F_v)\bfk_v=\bm{U}_{\bm{P}}(F_v)\bm{M}(F_v)\bfk_v$.

We may write the global Iwasawa decomposition alternatively as $\bm{G}(\A_F)=\bm{U}_{\bm{P}}(\A_F)\bm{M}(\A_F)^1 A_{\bm{M}}\bfk$ and define the associated projection 
\[
H_{\bm{P}}: \bm{G}(\A_F)=\bm{U}_{\bm{P}}(\A_F)\bm{M}(\A_F)^1 A_{\bm{M}}\bfk\longrightarrow \aa_{\bm{M}}, \qquad H_{\bm{P}}(ume^Xk)=X.
\]
When $\bm{M}=\bm{T}_0$, so that $\bm{P}=\bm{P}_0$, we write this as $H_0$.

\subsection{Weyl discriminant}\label{subsec:Weyl-disc}

Let $v$ be a place of $F$ and $\sigma$ a semisimple element in $G_v$. Let $G_{\sigma,v}$ be the centralizer of $\sigma$ and $\mathfrak{g}_{\sigma,v}$ its Lie algebra inside $\mathfrak{g}_v$, the Lie algebra of $G_v$. Then the Weyl discriminant of $\sigma$ in $G_v$ is defined to be
\begin{equation}\label{def:Weyl-disc}
D^G_v(\sigma)=|\det (1-{\rm Ad}(\sigma)|_{\mathfrak{g}_v/\mathfrak{g}_{\sigma,v}})|_v=\prod_{\substack{\alpha\in\Phi_v\\\alpha(\sigma)\neq 1}}|1-\alpha(\sigma)|_v.
\end{equation}
If an arbitrary $\gamma\in G_v$ has Jordan decomposition $\gamma=\sigma\nu$, where $\sigma$ is semisimple and $\nu\in G_{\sigma,v}$ is unipotent, then we extend the above notation to $\gamma$ by setting $D^G_v(\gamma)=D^G_v(\sigma)$. Whenever there is no risk of confusion we shall simplify $D_v^G(\gamma)$ to $D_v(\gamma)$. For a finite set of places $S$, let $D_S(\gamma)=\prod_{v\in S}D_v(\gamma)$.

\begin{remark}\label{rem:Weyl-disc}
The function $\gamma\mapsto D_S(\gamma)$ on $G_S$ is locally bounded and discontinuous at irregular elements. We have, for example, $D_S(\sigma)=1$ for all central $\sigma$. In our estimates on orbital integrals in the latter sections, it is this function which will measure their dependency on $\gamma$.
\end{remark}

\subsection{Twisted Levi subgroups}\label{sec:twisted-Levi}

Although our interest in this paper is solely in $\bm{G}=\GL_n$, in applications of the trace formula one encounters more general connected reductive groups, through the centralizers of semisimple elements in $\bm{G}(F)$.

If $\gamma\in \bm{G}(F)$, we write $\bm{G}_\gamma$ for the centralizer of $\gamma$. If $\gamma=\sigma$ is semisimple, then $\bm{G}_\sigma$ is connected and reductive. Moreover, one knows that in this case $\bm{G}_\sigma$ is a twisted Levi subgroup, meaning that there are field extensions $E_1,\ldots ,E_m$ of $F$ and non-negative integers $n_1,\ldots ,n_m$ such that $\bm{G}_\sigma\simeq\rest_{E_1/F}\GL_{n_1}\times\cdots\times\rest_{E_m/F}\GL_{n_m}$. For example, if $\sigma\in\bm{G}(F)$ is regular elliptic, then $\bm{G}_\sigma$ is the torus $\rest_{E/F}\mathbb{G}_m$ given by the restriction of scalars of a degree $n$ field extension $E$ of $F$. See \cite[\S 10A]{MatzTemplier2021} for more details.

If $\bm{H}$ is a twisted Levi subgroup of $\bm{G}$ containing some $\bm{M}\in\mathcal{L}$, then an $F$-Levi subgroup of $\bm{H}$ will be called semistandard (resp., standard) if it is the restriction of scalars of a semistandard (resp., standard) Levi subgroup. We similarly extend the notions of semistandard and standard to $F$-parabolic subgroups of $\bm{H}$. We let $\mathcal{L}^{\bm{H}}$ (resp., $\mathcal{F}^{\bm{H}}$) denote the set of semistandard $F$-Levi subgroups (resp., $F$-parabolic subgroups) of $\bm{H}$. If $\bm{M}\in\mathcal{L}^{\bm{H}}$ we write $\mathcal{L}^{\bm{H}}(\bm{M})=\{\bm{L}\in\mathcal{L}^{\bm{H}}: \bm{M}\subset \bm{L}\}$ and $\mathcal{F}^{\bm{H}}(\bm{M})=\{\bm{P}\in\mathcal{F}^{\bm{H}}: \bm{M}\subset \bm{P}\}$. Finally, $\mathcal{P}(\bm{M})$ will denote the subset of $\mathcal{F}^{\bm{H}}$ consisting of parabolics having Levi component $\bm{M}$.

If $\bm{H}$ is a twisted Levi subgroup of $\bm{G}$, and $v$ is a place of $F$, let $H_v=\bm{H}(F_v)$. When $v$ is non-archimedean, we write $\bfk_v^{\bm{H}}=\bfk_v\cap H_v$. Recall that a semisimple element $\sigma\in\bm{G}(F)$ has \textit{good reduction} at $v$ if $1-\alpha(\sigma)\in\mathcal{O}_{\bar{F}}$ is either zero or a $v$-adic unit for every $\alpha\in\Phi$; equivalently, $D^G_v(\sigma)=1$. When $\bm{H}=\bm{G}_\sigma$ for $\sigma$ semisimple with good reduction at $v$ then $\bfk_v^{\bm{H}}=\bfk_v\cap H_v$ is a maximal compact subgroup of $H_v$; see \cite[Prop 7.1]{Kottwitz1986}. Finally, for archimedean $v$ it follows from \cite[Lemma 5.3]{MatzTemplier2021} that $\bfk_v^{\bm{H}}=\bfk_v\cap H_v$ is again a maximal compact subgroup of $H_v$; let $\bfk_\infty^{\bm{H}}=\prod_{v\mid\infty}\bfk_v^{\bm{H}}$.

\subsection{Measure normalization on $\bm{G}$}\label{sec:measures} 

Although the various theorems and conjectures from Sections \ref{intro1} and \ref{equidsitribution-sec} were expressed using the Plancherel measure $\widehat{\omega}^{\rm pl}$ on $\Pi(\bm{G}(\A_F)^1)$ induced by Tamagawa measure $\omega_{\bm{G}}$, it will be useful, when it comes to estimates involving the trace formula (and citation of the literature) to work with a different measure normalization throughout the rest of this paper.

For finite places $v$, let $\mu_{\bm{G},v}$ be the unique Haar measure assigning the maximal compact $\bfk_v=\bm{G}(\mathcal{O}_v)$ volume 1. At archimedean places $v$, we will continue to take the local measure $\omega_{\bm{G},v}$ defined in \S\ref{leading-term}, although it will sometimes be convenient, for notational uniformity, to denote it by $\mu_{\bm{G},v}$. We then obtain a measure $\mu_{\bm{G}}$ on $\bm{G}(\A_F)^1$ by decomposing $\prod_v\mu_{\bm{G},v}$ as $\mu_{\bm{G}}\times {\rm d}t$ via $\bm{G}(\A_F)\simeq\bm{G}(\A_F)^1\times\R$. We give the automorphic space $\bm{G}(F)\backslash \bm{G}(\A_F)^1$ the quotient measure of $\mu_{\bm{G}}$ by the counting measure on $\bm{G}(F)$. Abusing notation, we again denote this measure by $\mu_{\bm{G}}$.

We now explicitly compute the volume that $\mu_{\bm{G}}$ assigns to the automorphic space, using  the notation from \S\ref{leading-term}. Since $\omega_{{\bm G},v}$ assigns $\bfk_v$ volume $\Delta_v(1)^{-1}$ (see \cite[p. 31]{Weil1982}), it follows that $\Delta_v(1)\omega_{{\bm G},v}=\mu_{\bm{G},v}$ for all finite places. Comparing with \eqref{eq:defn:omegaG-ast}, we find $\mu_{\bm{G}}=D_F^{n^2/2}\Delta_F^*(1)\omega_{\bm{G}}$. It follows that $\mu_{\bm{G}}(\GL_n(F)\backslash\GL_n(\A_F)^1)=D_F^{n^2/2}\Delta_F^*(1)$, since the Tamagawa number of $\bm{G}=\GL_n$ is~1.

\subsection{Measure normalizations on subgroups}\label{sec:measures-on-subgps}

Now let $\bm{H}$ be a twisted Levi subgroup of $\GL_n$, defined over $F$, as in \S\ref{sec:twisted-Levi}, and write
\[
\bm{H}(\A_F)^1=\bigcap_{\chi\in X^*(\bm{H})}\ker (|\chi|_{\A_F}).
\]
In order to define the global arithmetic coefficients $a^{\bm{M}}(\gamma,S)$, as well as the weighted orbital integrals $J_{\bm{M}}(\gamma,{\bf 1}_{\bfk^S}\otimes\phi_S)$ on the geometric side of the trace formula in \S\S\ref{fine-geom} and \ref{sec-review}, we must define a normalization of Haar measure on $\bm{H}(\A_F)^1$. As was the case for $\bm{H}=\bm{G}$, we shall adopt the canonical normalization of Gross \cite{Gross1997}.

For a finite place $v$ let $\mu_{\bm{H},v}$ denote the Gross canonical measure. When $\bm{H}$ is the centralizer of a semisimple element with good reduction at a finite place $v$, then $\mu_{{\bm H},v}$ assigns the maximal compact subgroup $\bfk_v^{\bm{H}}$ volume 1. At archimedean places $v$, whenever $\bm{H}\neq\bm{G}$, we let $\mu_{\bm{H},v}$ denote any choice of Haar measure. The product $\mu_{\bm{H}}^{\ast}=\prod_v\mu_{\bm{H},v}$ is a well-defined measure on $\bm{H}(\A_F)$. Let $\{\chi_i: 1\leqslant i\leqslant m\}$ be a $\Z$-basis for $X^*(\bm{H})$, where $m={\rm rank}_\Z X^*(\bm{H})$, and consider the homomorphism
\[
\log: \bm{H}(\A_F)\ni g\mapsto (\log |\chi_1(g)|_{\A_F},\ldots ,\log|\chi_m(g)|_{\A_F})\in \R^m.
\]
Then we have a split exact sequence
\[
1\rightarrow \bm{H}(\A_F)^1\rightarrow \bm{H}(\A_F)\xrightarrow{\log}\R^m\rightarrow 0.
\]
This yields an isomorphism $\bm{H}(\A_F)\simeq \bm{H}(\A_F)^1 \times\R^m$. Let ${\rm d}t$ denote Lebesgue measure on $\R^m$. Then we let $\mu_{\bm{H}}$ denote the unique measure on $\bm{H}(\A_F)^1$ factorizing $\mu_{\bm{H}}^{\ast}$ as $\mu_{\bm{H}}\times {\rm d}t$. We again write $\mu_{\bm{H}}$ for the measure on the automorphic space $\bm{H}(F)\backslash \bm{H}(\A_F)^1$ similarly to that for $\bm{G}$.

Finally, for every $\bm{M}\in\mathcal{L}_{\rm st}$, $\bm{P}\in\mathcal{P}(\bm{M})$, and any place $v$, there is a uniquely defined Haar measure on $U_v=\bm{U}_{\bm{P}}(F_v)$ defined via the integration formula
\[
\int_{G_v} f(g){\rm d}\mu_{\bm{G},v}(g)=\int_{\bfk_v}\int_{M_v}\int_{U_v} f(muk) {\rm d}u \,{\rm d} \mu_{\bm{M},v}(m)\, {\rm d}k\qquad (f\in C_c^\infty(G_v)),
\]
where ${\rm d}k$ is the Haar probability measure on $\bfk_v$. For example, at finite places $v$, by testing this equality against the characteristic function of $\bfk_v$ it follows from the definition of $\mu_{\bm{G},v}$ and $\mu_{\bm{M},v}$ that $U_v\cap \bfk_v$ receives volume 1. The product measures on $\bm{U}_{\bm{P}}(\A_F)$, $\bm{M}(\A_F)$, and $\bfk$ factorize $\mu_{\bm{G}}=\prod_v\mu_{\bm{G},v}$ by a similar formula.

\subsection{Hecke algebras}\label{Hecke-section}

We now introduce the local and global Hecke algebras for use in the trace formula.

\subsubsection{Local case}\label{local-Hecke}

At any place $v$ we define $C^\infty_c(G_v)$ to be the space of functions on $G_v$ which are locally constant and of compact support, for $v$ finite, and smooth and of compact support for $v$ infinite. We then let $\mathcal{H}(G_v)$ denote $C^\infty_c(G_v)$, when considered as a convolution algebra with respect to the measure $\mu_{\bm{G},v}$. Similarly, for a finite set of places $S$ containing all archimedean places, we denote by $\mathcal{H}(G_S)$ or $\mathcal{H}(G_S^1)$ the convolution algebra of smooth, compactly supported, functions on $G_S$ or $G_S^1$, respectively. For non-archimedean $v$, and an open compact subgroup $K_v$ of $G_v$, let
\begin{equation}\label{idempotent}
\varepsilon_{K_v}=\frac{1}{\mu_{\bm{G},v}(K_v)}{\bf 1}_{K_v}
\end{equation}
denote the corresponding idempotent in $\mathcal{H}(G_v)$. 

Given an admissible representation $\pi_v$ of $G_v$ any $\phi_v\in \mathcal{H}(G_v)$ define an operator on the space of $\pi_v$ via the averaging
\[
\pi_v(\phi_v)=\int_{G_v}\phi_v(g)\pi_v(g)\, {\rm d}\mu_{\bm{G},v}(g).
\]
This is a trace class operator; we write ${\rm tr}\, \pi_v(\phi_v)$ for its trace. If, for a finite place $v$, $K_v$ is an open compact subgroup of $G_v$, it is straightforward to see that ${\rm tr}\,\pi_v (\varepsilon_{K_v})=\dim\pi_v^{K_v}$. 

We write $\Pi(G_v)$ for the unitary dual of $G_v$, endowed with the Fell topology. Let $\widehat{\mu}_v^{\rm pl}$ denote the Plancherel measure for $\Pi(G_v)$, normalized to satisfy
\begin{equation}\label{eq:plancherel-inv}
\phi_v(e)=\int_{\Pi(G_v)}{\rm tr}\, \pi_v(\phi_v^\vee)\, {\rm d}\widehat{\mu}_v^{\rm pl}(\pi_v)
\end{equation}
for any $\phi_v\in C_c^\infty(G_v)$, where $\phi_v^\vee(g)=\phi_v(g^{-1})$.

\subsubsection{Global case}\label{Hecke-global}

Let $\mathcal{H}(\bm{G}(\A_f))$ denote the space of finite linear combinations of factorizable functions $\otimes_{v<\infty}\, \phi_v$, where each $\phi_v$ lies in $\mathcal{H}(G_v)$ and $\phi_v={\bf 1}_{\bfk_v}$ for almost all $v$. We then take as the global Hecke algebra $\mathcal{H}(\bm{G}(\A_F)^1)$ the tensor product $\mathcal{H}(\bm{G}(\A_f))\otimes \mathcal{H}(\bm{G}_\infty^1)$. Convolution is taken with respect to the global measure $\mu_{\bm{G}}$. For admissible $\pi=\otimes_v\pi_v$ and $\phi\in\mathcal{H}(\bm{G}(\A_F)^1)$ we define the trace-class operator $\pi(\phi)$ with respect to $\mu_{\bm{G}}$. Moreover, for admissible $\pi=\otimes_v\pi_v$ and factorizable $\phi=\otimes_v\,\phi_v\in\mathcal{H}(\bm{G}(\A_F)^1)$ the global trace ${\rm tr}\,\pi(\phi)$ factorizes as $\prod_v {\rm tr}\,\pi_v(\phi_v)$.

Similarly, if $S$ is any finite set of places of $F$ containing all archimedean places, and $\bm{G}(\A_F^S)$ is the restricted tensor product of $G_v$ over places outside of $S$, we let $\mathcal{H}(\bm{G}(\A_F^S))$ denote the analogous space, with convolution taken with respect to the measure $\mu^S_{\bm{G}}=\prod_{v\notin S}\mu_{\bm{G},v}$.

\subsection{The geometric side of the Arthur trace formula}\label{fine-geom}

The next three subsections describe the non-invariant trace formula of Arthur, alluded to in \S\ref{sec:ELM} and written there as $J_{\rm geom}=J_{\rm spec}$. We note that the Arthur trace formula can be seen as a vast generalization of Poisson summation, to which it reduces in the case of $n=1$. To guide the reader, we note that \S\S\ref{fine-geom} and \ref{sec-review} describe the geometric side, while \S\ref{JMspec} describes the spectral side.

Let $\mathfrak{O}$ denote the set of semisimple conjugacy classes of $\bm{G}(F)$. As we are taking $\bm{G}=\GL_n$, a semisimple conjugacy class consists of all those $\gamma\in \bm{G}(F)$ sharing the same characteristic polynomial. Associated with each $\mathfrak{o}\in\mathfrak{O}$ Arthur \cite{Arthur1978} defines a global distribution $J_\mathfrak{o}$ on $C_c^\infty (\bm{G}(\A)^1)$, as the value at $T=0$ of a geometric distribution $J_{\mathfrak{o}}^T$ coming from his truncated kernel \cite[\S 8]{Arthur1978}. These distributions fit into the \textit{coarse geometric expansion}, given by
\begin{equation}\label{eq:J-geom-def}
J_{\rm geom}(\phi)=\sum_{\mathfrak{o}\in\mathfrak{O}}J_\mathfrak{o}(\phi).
\end{equation}
The \textit{fine geometric expansion} expresses each $J_\mathfrak{o}(\phi)$ as a linear combination of weighted orbital integrals $J_{\bm{M}}(\gamma,\phi)$, where $M\in\mathcal{L}$ and $\gamma\in \bm{M}(F)$, which are local in nature.

More precisely, in \cite[Theorem 8.1]{Arthur1986}, Arthur shows that for every equivalence class $\mathfrak{o}\in\mathfrak{O}$, there is a finite set of places $S_{\rm adm}(\mathfrak{o})$ (containing all archimedean places) which is {\it admissible} in the following sense. For any finite set of places $S$ containing $S_{\rm adm}(\mathfrak{o})$, there are real numbers $a^{\bm{M}}(\gamma,S)$, indexed by $\bm{M}\in\mathcal{L}$ and $\bm{M}(F)$-conjugacy classes of elements $\gamma\in \bm{M}(F)$ (and, in general, depending on $S$ as well as the measure normalizations on Levi subgroups from \S\ref{sec:measures}), such that 
\begin{equation}\label{eq:J-geom-fine}
J_\mathfrak{o}({\bf 1}_{\bfk^S}\otimes\phi_S)=\sum_{\bm{M}\in\mathcal{L}}\frac{|W_{\bm{M}}|}{|W|}\sum_\gamma a^{\bm{M}}(\gamma,S)J_{\bm{M}}(\gamma,{\bf 1}_{\bfk^S}\otimes\phi_S)
\end{equation}
for any function $\phi_S\in C^\infty_c(\bm{G}(F_S)^1)$. In the inner sum, $\gamma$ runs over those $\bm{M}(F)$-conjugacy classes of elements in $\bm{M}(F)$ meeting $\mathfrak{o}$. Furthermore, $J_{\bm{M}}(\gamma,{\bf 1}_{\bfk^S}\otimes\phi_S)=0$ for any $\gamma\notin \bfk^S\cap\mathfrak{o}$ and $J_{\bm{M}}(\gamma,{\bf 1}_{\bfk^S}\otimes\phi_S)=J_{\bm{M}_S}(\gamma,\phi_S)$ otherwise, the latter being an $S$-adic integral to be described in \S\ref{sec-review}. Moreover, when $\gamma$ is semisimple, Arthur \cite[Theorem 8.2]{Arthur1986} shows that for $S$ large enough the global coefficient $a^{\bm G}(\gamma,S)$ is independent of $S$ and equal to the volume of ${\bm G}_\gamma(F)\backslash\bm{G}(\A_F)^1$.

Let $\mathcal{U}_{G_\sigma}$ denote (the $F_S$-points of the) algebraic variety of unipotent elements in the centralizer of $\sigma$. Then the Jordan decomposition $\gamma=\sigma\nu$ requires $\nu$ to lie in $\mathcal{U}_{G_\sigma}$. Following \cite[\S 6]{Matz2017} and \cite[\S11D]{MatzTemplier2021}, for $\mathfrak{o}\in\mathfrak{O}$ and $\gamma=\sigma\nu\in \sigma\mathcal{U}_{\bm{G}_\sigma}(F)\cap\mathfrak{o}$ we let
\begin{equation}\label{def-S-o}
S_\mathfrak{o}=S_{\rm wild}\cup\{ v<\infty: D^G_v(\sigma)\neq 1\},
\end{equation}
where $S_{\rm wild}$ is a certain finite set of finite places depending only on $n$. Then \cite[Lemma 6.2]{Matz2017} or \cite[Remark 11.8]{MatzTemplier2021} shows that one can take $S_{\rm adm}(\mathfrak{o})=S_\mathfrak{o}\cup \{v\mid\infty\}$ in the fine geometric expansion. Note that, for any $S$ containing $S_{\rm adm}(\mathfrak{o})$, every $\sigma\in\mathfrak{o}$ has good reduction outside of $S$.

When $\mathfrak{o}$ is the semisimple conjugacy class of a central element $\gamma\in Z(F)$, we shall be particularly interested in the contribution that $\gamma$ itself makes to the $\bm{M}=\bm{G}$ part of the fine geometric expansion \eqref{eq:J-geom-fine}, namely $a^{\bm{G}}(\gamma,S)J_{\bm{G}}(\gamma,{\bf 1}_{\bfk^S}\otimes\phi_S)$. From the formula for $a^{\bm{G}}(\gamma,S)$ stated above, applied to the central element $\gamma$ and $S\supset S_{\rm wild}\cup \{v\mid\infty\}$, we obtain $a^{\bm{G}}(\gamma,S)=\mu_{\bm{G}}(\bm{G}(F)\backslash\bm{G}(\A_F)^1)$. It is natural then to define
\begin{equation}\label{eq:defn:central-cont}
J_{\rm cent}({\bf 1}_{\bfk^S}\otimes\phi_S)=D_F^{n^2/2}\Delta_F^*(1)\sum_{\gamma\in Z(F)}J_{\bm G}(\gamma,{\bf 1}_{\bfk^S}\otimes\phi_S).
\end{equation}
This is the central contribution of the trace formula that was introduced in \S\ref{sec:ELM}.

\subsection{The weighted orbital integrals}\label{sec-review}

In this section we review the definition of the weighted orbital integrals $J_{\bm{M}_S}(\gamma,\phi_{S})$. Following \cite{Arthur1988}, we define $J_{\bm{M}_S}(\gamma,\phi_S)$ in two steps, first for unipotent elements then for general $\gamma$. We emphasize that the structures in this section are purely local.

We fix a finite non-empty set of places $S$ of $F$. Where possible, we will drop the subscript $S$ (and the boldface font, violating momentarily the notational convention in Remark \ref{rem:boldface-notation}). So, for example, $G=\bm{G}_S$, $G_\sigma=\bm{G}_{\sigma}(F_S)$, $M=\bm{M}_S$, and $\bfk=\bfk_S$. By a parabolic $P\in\mathcal{F}^H(M)$ we mean the $F_S$ points of some $\bm{P}\in\mathcal{F}^{\bm{H}}(\bm{M})$. Throughout, we let $\gamma=\sigma\nu\in G$ be the Jordan decomposition of $\gamma$.

\subsubsection{The unipotent case}

Let $\nu$ be a unipotent element lying in a standard Levi subgroup $M$. The weighted orbital integral $J_M(\nu,\phi)$ in this case was introduced in \cite{Arthur1985}.

To define $J_M(\nu,\phi)$ we shall need to specify a certain unipotent radical containing $\nu$ as well as a weight function. To begin, we let $\mathcal{V}_0$ be the $M$-conjugacy class of $\nu$ and denote by $\mathcal{V}_1$ the $G$-conjugacy class given by inducing $\mathcal{V}_0$ from $M$ to $G$ along any parabolic $P$ of $G$ containing $M$ as a Levi subgroup. (The induced class $\mathcal{V}_1$, defined as the unique class which intersects $\mathcal{V}_0U_P$ densely, is independent of the choice of $P$.) Then, since unipotent conjugacy classes in $G=\GL_n$ are of Richardson type (i.e., induced from the trivial class), there is a unique standard parabolic subgroup $P_1\in\mathcal{F}_S$, say with Levi decomposition $P_1=L_1U_1$, such that $\mathcal{V}_1$ has dense intersection with $U_1$. (Specifically, the Levi factor of $P_1$ is given by the dual partition of the Jordan form of $\mathcal{V}_1$. See, for example, \cite[\S 5.5]{Humphreys1995}.) Then
\begin{equation}\label{unip-woi}
J_M(\nu,\phi)=\int_\bfk\int_{U_1} \phi(k^{-1}uk) w_{M,U_1} (u)\, \text{d}u\, \text{d}k.
\end{equation}
Here, the weight function $w_{M,U_1}$ is a complex-valued function defined on $U_1$; it is invariant under conjugacy by $\bfk^{L_1}$ (so that the above integral is well-defined) and constantly equal to $1$ when $M=G$. See \cite[p. 143]{LapidMuller2009} for more details.

\subsubsection{The general case}

The general formula, for elements $\gamma=\sigma\nu\in \sigma\mathcal{U}_{G_\sigma}\cap M$, is considerably more complicated. It will be expressed in terms of weighted unipotent orbital integrals for the Levi components of parabolic subgroups in $\mathcal{F}^{G_\sigma}(M_\sigma)$. (The latter set reduces to $\mathcal{F}(M)$ when $\sigma=1$.) We first state the original definition of Arthur then proceed to give a simplification in the case of $\GL_n$.

More precisely, if $\gamma=\sigma\nu\in \sigma\mathcal{U}_{G_\sigma}\cap M$, then \cite[Corollary 8.7]{Arthur1988} states that 
\begin{equation}\label{woi}
J_M(\gamma,\phi)=D(\sigma)^{1/2}\int_{G_\sigma\backslash G}\left(\sum_{R\in\mathcal{F}^{G_\sigma}(M_\sigma)}J_{M_\sigma}^{M_R}(\nu,\Phi_{R,y})\right)\, \text{d}y,
\end{equation}
where, for $m\in M_R$ and $y\in G$, we have put 
\[
\Phi_{R,y}(m)=\delta_R(m)^{1/2}\int_{\bfk^{G_\sigma}}\int_{N_R}\phi(y^{-1}\sigma k^{-1}mnky)v_R'(ky)\, \text{d}n\, \text{d}k.
\]
The complex-valued weight function $v'_R$ on $G$ is set to be
\begin{equation}\label{def:weight}
v_R'(z)=\sum_{\substack{Q\in\mathcal{F}(M): \; Q_\sigma=R\\ \mathfrak{a}_Q=\mathfrak{a}_R}} v_Q'(z),
\end{equation}
where $v'_Q$, defined in \cite[\S2]{Arthur1981}.

Using the expression \eqref{unip-woi} for the unipotent weighted orbital integral, valid for $\GL_n$, we may write $J_{M_\sigma}^{M_R}(\nu,\Phi_{R,y})$ more conveniently. To see this, similarly to before, we first let $\mathcal{V}_0$ denote the $M_\sigma$-conjugacy class of the unipotent element $\nu\in M_\sigma$. Next we write $\mathcal{V}_1$ for the induced unipotent class of $\mathcal{V}_0$ to $M_R$ along any parabolic in $M_R$ containing $M_\sigma$ as a Levi subgroup. Let $P_1=L_1U_1\subset M_R$ be a Richardson parabolic for $\mathcal{V}_1$. Finally, let $\mathcal{V}$ be the induced unipotent class of $\mathcal{V}_1$ to $G_\sigma$ along $R$. Then the Richardson parabolic $P=LV\subset G_\sigma$ of $\mathcal{V}$ satisfies $U=U_1N_R$. We deduce that
\begin{equation}\label{J-sigma}
J_{M_\sigma}^{M_R}(\nu,\Phi_{R,y})=\int_{\bfk^{G_\sigma}}\int_U\phi(y^{-1}\sigma k^{-1}uky)w_{M_\sigma,U}^{M_R}(u)v'_R(ky)\, \text{d}u\, \text{d}k,
\end{equation}
where $w_{M_\sigma,U}^{M_R}$ is the trivial extension of $w_{M_\sigma,U_1}^{M_R}$ to all of $U$. For more details, see \cite[\S 10.4]{Matz2017}.

\subsubsection{Remarks on special cases}

In the next few remarks, in an effort to render more comprehensible the general definition $J_M(\gamma,\phi)$, we examine several special cases.

\begin{remark}
When $\sigma=1$ so that $\gamma=\nu$ is unipotent, the outer integral in \eqref{woi} is trivial. Moreover, the function $v_R'$ vanishes on $\bfk$ unless $R=G$ when it is constantly equal to $1$. From \eqref{J-sigma} we deduce that $J_{M}^{M_R}(\nu,\Phi_{R,e})=0$ unless $R=G$, in which case (since $U=U_1$) we obtain
\[
\int_\bfk\int_{U_1}\phi(k^{-1}uk)w_{M,U_1}^G(u)\, \text{d}u\, \text{d}k,
\]
recovering the previous expression \eqref{unip-woi} for $J_{M}(\nu,\phi)$.
\end{remark}

\begin{remark}\label{rem:WOI-simplify}
Under the assumption that $G_\gamma\subset M$, the general definition simplifies greatly, yielding
\begin{equation}\label{basic-eq}
J_M(\gamma,\phi)=D(\gamma)^{1/2}\int_{G_\gamma\backslash G}\phi_v(y^{-1}\gamma y) v_M'(y)\, \text{d}y.
\end{equation}
Here, the weight function $v_M'$ is the volume of a certain complex hull; as a function on $G$, it is left invariant under $M$ (so the above integral is well-defined) and constantly equal to $1$ when $M=G$. In fact, the weight $v_Q'$ appearing in \eqref{def:weight} generalizes  $v'_M$ to arbitrary parabolics $Q\in\mathcal{F}(M)$. In particular, when $\gamma\in Z(F)$ is a central element one obtains $J_G(\gamma,\phi)=\phi(\gamma)$. From this and the definition \eqref{eq:defn:central-cont} it follows that
\begin{equation}\label{eq:def:explicit-Jcentral}
J_{\rm cent}({\bf 1}_{\bfk^S}\otimes\phi_S)=D_F^{n^2/2}\Delta_F^*(1)\sum_{\gamma\in Z(F)}\phi_S(\gamma),
\end{equation}
which makes more precise the expression given in \S\ref{sec:ELM}.

To see \eqref{basic-eq}, observe that the condition $G_\gamma\subset M$ is equivalent to $G_\gamma=M_\gamma$ and the uniqueness of the Jordan decomposition then implies $M_\sigma=G_\sigma$. In this case $\mathcal{F}^{G_\sigma}(M_\sigma)=\{G_\sigma\}$ and the sum over $R$ in \eqref{woi} reduces to the single term $R=G_\sigma$; clearly, $M_R=G_\sigma$ and $N_R=\{e\}$. Thus $U=U_1$ and the weight function $w_{M_\sigma,U}^{M_\sigma}$ is constantly equal to $1$ on all of $U$. Furthermore, $v_R'=v_M'$ in this case. From the left $M$-invariance of $v_M'$, we have $v_M'(ky)=v_M'(y)$ for $k\in \bfk^{G_\sigma}$ and $y\in G$. The corresponding integral in \eqref{J-sigma} then reduces to 
\[
v_M'(y)\int_{\bfk^{G_\sigma}}\int_{U_1}\phi(y^{-1}\sigma k^{-1}uky)\, \text{d}u\, \text{d}k=v_M'(y)J_{G_\sigma}^{G_\sigma}(\sigma\nu,\phi^y),
\]
where $\phi^y(x)=\phi(y^{-1}xy)$. Note that the latter integral is
\[
\int_{G_\gamma\backslash G_\sigma}\phi(y^{-1}x^{-1}\sigma \nu xy)\, \text{d}x.
\]
Inserting this into the integral over $y\in G_\sigma\backslash G$ we obtain the expression \eqref{basic-eq}.
\end{remark}

\begin{remark}
Note that for $\nu$ unipotent $G_\nu\not\subset M$, unless $M=G$. When $\gamma=\nu$ is unipotent and $M=G$ then the two formulae coincide, giving the invariant unipotent orbital integral. For example, the Richardson parabolic of the trivial class in $G$ is of course $G$ itself, so that both formulae collapse in this case to $J_G(1,\phi)=\phi(1)$. 
\end{remark}

\subsection{The spectral side of the Arthur trace formula}\label{JMspec}

We now turn to the spectral side of the trace formula. The work of Arthur for general groups \cite[Theorem 8.2]{Arthur82}, coupled with the absolute convergence of the spectral side of the trace for $\GL_n$, established by M\"uller--Speh \cite{MullerSpeh2004}, yields the following form of the fine spectral expansion
\[
J_{\rm spec}=\sum_{\bm{M}\in\mathcal{L}}J_{{\rm spec},\bm{M}},
\]
for distributions $J_{{\rm spec},\bm{M}}(\phi)$ to be described below. The term corresponding to $\bm{M}=\bm{G}$ is relatively straightforward, whereas the terms $J_{{\rm spec},\bm{M}}$ for $\bm{M}\in\mathcal{L}$, $\bm{M}\neq G$, will necessitate a great deal of notation; we borrow essentially from \cite[\S 4]{FinisLapidMuller2015}.  Taking for granted their existence, we let
\begin{equation}\label{def:J-Eis}
J_{\rm disc}(\phi)=J_{{\rm spec},\bm{G}}(\phi) \quad\textrm{and}\quad J_{\rm Eis}(\phi)=\sum_{\bm{M}\neq \bm{G}} J_{{\rm spec},\bm{M}}(\phi)
\end{equation}
denote the discrete and continuous contributions to the trace formula.

For $\bm{M}\in\mathcal{L}$ we write $\Pi_{\rm disc}(\bm{M}(\mathbb{A})^1)$ for the set of isomorphism classes of irreducible discrete unitary automorphic representations of $\bm{M}(\mathbb{A})^1$. We may view $\Pi_{\rm disc}(\bm{M}(\A_F)^1)$ alternatively as the set of irreducible subrepresentations of the right-regular representation on $L^2_{\rm disc}(\bm{M}(F)\backslash \bm{M}(\A_F)^1)$. Indeed, the multiplicity one theorem \cite{JacquetLanglands1970, Shalika1974, Piatetski-Shapiro1979} for the cuspidal spectrum of $\bm{G}=\GL_n$, and the description of the residual spectrum of $\GL_n$ by \cite{MoeglinWaldspurger1989}, together imply that the multiplicity with which $\pi\in\Pi_{\rm disc}(\bm{M}(\A_F)^1)$ appears in the right-regular representation $L^2(\bm{M}(F)\backslash \bm{M}(\A_F)^1)$ is one.

When $\bm{M}=\bm{G}$ then according to M\"uller--Speh \cite{MullerSpeh2004} we have
\begin{equation}\label{defn:J-spec-G}
J_{{\rm spec},\bm{G}}(\phi) =\sum_{\pi\in\Pi_{\rm disc}(\bm{G}(\mathbb{A})^1)} {\rm tr}(\pi(\phi)),
\end{equation}
justifying the notation in \eqref{def:J-Eis}. As in \S\ref{sec:ELM}, we let $J_{\rm cusp}$ denote the restriction of $J_{\rm disc}$ to those $\pi$ in the cuspidal subspace $\Pi_{\rm cusp}(\bm{G}(\mathbb{A})^1)$.

Now let $\bm{M}\in\mathcal{L}$ be a \textit{proper} standard Levi subgroup. We choose a parabolic $\bm{P}\in \mathcal{P}(\bm{M})$ containing $\bm{M}$ as a Levi subgroup. We denote by $\mathcal{A}^2(\bm{P})$ the space of all complex-valued functions $\varphi$ on the boundary degeneration $\bm{U}_{\bm{P}}(\A_F)\bm{M}(F)\backslash\bm{G}(\A_F)^1$ such that for every $x\in \bm{G}(\A_F)$ the function $\varphi_x(g)=\delta_P(g)^{-1/2}\varphi(gx)$, where $g\in \bm{M}(\A_F)$, lies in $L^2(\bm{M}(F)\backslash \bm{M}(\A_F)^1)$. We require that $\varphi$ be $\mathfrak{z}$-finite and $\bfk_f$-finite, where $\mathfrak{z}$ is the center of the universal enveloping algebra of $\mathfrak{g}_\C=\mathfrak{g}\otimes_\R\C$. Let $\bar{\mathcal{A}}^2(\bm{P})$ be the Hilbert space completion of $\mathcal{A}^2(\bm{P})$. For every $\lambda\in\aa_{\bm{M},\C}^*$ the space $\bar{\mathcal{A}}^2(\bm{P})$ receives an action $\rho(\bm{P},\lambda)$ by $\bm{G}(\A_F)^1$ given by
\[
(\rho(\bm{P},\lambda,y)(\varphi))(x)=\varphi(xy)e^{\langle\lambda,H_{\bm{P}}(xy)- H_{\bm{P}}(x)\rangle},
\]
which makes it isomorphic to the induced representation
\begin{equation}\label{L2-induced}
{\rm Ind}(L^2_{\rm disc}(\bm{M}(F)\backslash \bm{M}(\A_F)^1)\otimes e^{\langle\lambda,H_{\bm{M}}(\cdot)\rangle}).
\end{equation}
Applying this procedure for two parabolics $\bm{P}, \bm{Q}\in\mathcal{P}(\bm{M})$ as above, we obtain the (analytic continuation of the) standard intertwining operator $\mathcal{M}_{\bm{Q}|\bm{P}}(\lambda):\mathcal{A}^2(\bm{P})\rightarrow\mathcal{A}^2(\bm{Q})$, where $\lambda\in\aa_{\bm{M},\C}^*$; then $\mathcal{M}_{\bm{Q}|\bm{P}}(\lambda)$ is unitary for all $\lambda\in i\aa_{\bm{M}}^*$. For these facts see \cite[\S1]{Arthur82}.

The spectral distributions $J_{{\rm spec},\bm{M}}$ for proper Levi subgroups $\bm{M}$ will involve, as an essential input, the logarithmic derivatives of the intertwining operators $\mathcal{M}_{\bm{Q}|\bm{P}}(\lambda)$, for varying choices of $\bm{P}$ and $\bm{Q}$. These can be packaged together rather compactly, in the case of $\GL_n$. The Weyl group $W_{\bm{M}}$ acts on $\mathcal{P}(\bm{M})$ by sending $\bm{P}$ to $w_s\bm{P} w_s^{-1}$, where $w_s\in N_G(\bm{T}_0)$ is a representative. This gives rise to a map on $\bm{P}$-induced automorphic forms $s:\mathcal{A}^2(\bm{P})\rightarrow\mathcal{A}^2(s\bm{P})$ given by left-translation by $w_s^{-1}$. We write $\mathcal{M}(\bm{P},s)$ for the composition $\mathcal{M}_{\bm{P}|s\bm{P}}(0)\circ s:\mathcal{A}^2(\bm{P})\rightarrow\mathcal{A}^2(\bm{P})$. Then $\mathcal{M}(\bm{P},s)$ is a unitary operator which for $\lambda\in i\aa_{\bm{L}_s}^*$ intertwines $\rho(\bm{P},\lambda)$ with itself, where $\bm{L}_s$ denotes the smallest Levi subgroup of $\bm{G}$ containing $w_s$; note that when $s\in W_{\bm{M}}$ we have $\bm{L}_s\supset\bm{M}$ (i.e., $\bm{L}_s\in\mathcal{L}(\bm{M})$). Then the work of Finis--Lapid--M\"uller \cite[\S 2]{FinisLapidMuller2011} (see also \cite[\S 4]{FinisLapidMuller2015}) associates with certain $\dim\mathfrak{a}_{\bm{L}_s}^{\bm{G}}$-tuples $\mathcal{X}_{\bm{L}_s}(\underline{\beta})$, consisting of pairs of adjacent parabolic subgroups associated with root tuples $\underline{\beta}\in\mathfrak{B}_{\bm{P},\bm{L}_s}$, a corresponding iterated logarithmic derivative of intertwining operators, denoted $\Delta_{\mathcal{X}_{\bm{L}_s(\underline{\beta})}}(\bm{P},\lambda)$.

With the above notation in place, we may now give the definition of $J_{{\rm spec},\bm{M}}$ for $\bm{M}\neq\bm{G}$. Recall from \S\ref{sec:global-characters} that $\aa_{\bm{L}_s}=\aa_{\bm{G}}\oplus\aa_{\bm{L}_s}^{\bm{G}}$. For any $s\in W_{\bm{M}}$ and $\underline{\beta}\in\mathfrak{B}_{\bm{P},\bm{L}_s}$, we put
\begin{equation}\label{Jintegral}
J_{{\rm spec},\bm{M}}(\phi; s,\underline{\beta})=\int_{i(\mathfrak{a}_{\bm{L}_s}^{\bm{G}})^*}{\rm tr}\left(\Delta_{\mathcal{X}_{\bm{L}_s}(\underline{\beta})}(\bm{P},\lambda)\mathcal{M}(\bm{P},s)\rho(\bm{P},\lambda,\phi)\right){\rm d}\lambda,
\end{equation}
where the operators are of trace class and the integrals are absolutely convergent \cite{MullerSpeh2004}. Finally, let
\begin{equation}\label{JspecM}
J_{{\rm spec},\bm{M}}(\phi)=\frac{1}{|W_{\bm{M}}|}\sum_{s\in W_{\bm{M}}}|\det(s-1)_{\mathfrak{a}^{\bm{L}_s}_{\bm{M}}}|^{-1}\sum_{\underline{\beta}\in\mathfrak{B}_{\bm{P},\bm{L}_s}}J_{{\rm spec},\bm{M}}(\phi; s,\underline{\beta});
\end{equation}
see \cite[Corollary 1]{FinisLapidMuller2011} or \cite[Theorem 4.1]{FinisLapidMuller2015}.

\section{Archimedean structures}\label{sec:arch-rep-PWCD}

The goal of this section is to introduce the archimedean structures that will be required in our construction of spectral localizer functions in Sections \ref{spec-loc} and \ref{ExceptionalSpectrumSection}. In particular, we shall describe various Lie algebra decompositions, and recall the parametrizations of the admissible, unitary, and Hermitian duals of $\GL_n(F_\infty)^1$. Since we work in the setting of a general number field $F$, some care must be taken in the presence of more than one archimedean place.

\subsection{Local and $S$-adic Lie algebras}\label{sec:v-adic}

Let $\mathcal{L}_v$ and $\mathcal{F}_v$ denote the set of semistandard Levi and parabolic subgroups of $G_v$. We will write $M_v$ for an arbitrary element in $\mathcal{L}_v$.

For $M_v\in\mathcal{L}_v$ we denote by $M_v^1$ the largest closed subgroup of $M_v$ on which all $|\chi_v|$ are trivial, as $\chi_v$ runs over the lattice $X^*(M_v)$ of $F_v$-rational characters of $M_v$. When $v\mid\infty$, we let $A_{M_v}$ be the identity component of the $\R$-split part of the center of $M_v$; we have $M_v=M_v^1\times A_{M_v}$. When $v<\infty$, let
\[
A_{M_v}=M_v\cap \{{\rm diag}(\varpi_v^{i_1},\ldots ,\varpi_v^{i_n}): i_1,\ldots ,i_n\in\Z\},
\]
for a fixed choice of uniformizer $\varpi_v$ of $F_v$. When $M_v=T_{0,v}$, we write $A_{0,v}=A_{T_{0,v}}$.

We set $\mathfrak{a}_{M_v}^*=X^*(M_v)\otimes\R$ and $\mathfrak{a}_{M_v}={\rm Hom}_\R(\mathfrak{a}^*_{M_v},\R)$. Their complexifications are denoted $\mathfrak{a}_{M_v,\C}$ and $\mathfrak{a}_{M_v,\C}^*$. If $M_v$ is isomorphic to $\GL_{n_1}(F_v)\times\cdots\times\GL_{n_{m_v}}(F_v)$ then $\mathfrak{a}_{M_v}$ is isomorphic to $\R^{m_v}$, where $m_v=\dim \mathfrak{a}_{M_v}$. The map $\log_v: A_{M_v}\rightarrow\mathfrak{a}_{M_v}$, defined as in \S\ref{sec:global-characters}, is injective and in fact an isomorphism when $v\mid\infty$. If $M_v\subset L_v$, we let $\aa_{M_v}^{L_v}$ denote the kernel of the natural map from $\aa_{M_v}$ to $\aa_{L_v}$. For example, $\aa_{G_v}=\R I_n\subset\aa_{M_v}$, making $\aa_{M_v}^{G_v}$ consist of trace zero matrices in $\aa_{M_v}$. Similarly to \S\ref{sec:global-characters}, we have a direct sum decomposition $\aa_{M_v}=\aa_{L_v}\oplus\aa_{M_v}^{L_v}$.

The above definitions can be extended in the obvious way to the $S$-adic setting. For a finite non-empty collection of places $S$ of $F$, let $\mathcal{L}_S=\prod_{v\in S}\mathcal{L}_v$. We will write $M_S$, or simply $M$ if the context is clear, for an arbitrary element in $\mathcal{L}_S$. An $M\in \mathcal{L}_S$ is then of the form $M=\prod_{v\in S}M_v$, where $ M_v\in\mathcal{L}_v$. When $S=\{v\mid\infty\}$ and $M=\prod_{v\mid\infty}M_v\in\mathcal{L}_\infty$ we let $A_M=\prod_v A_{M_v}$ and
\[
\aa_{M,\infty}=\bigoplus_{v\mid\infty}\aa_{M_v}.
\]
When $M$ is the archimedean diagonal torus $T_{0,\infty}=\prod_{v\mid\infty}T_{0,v}$ we write $\aa_{0,\infty}$ for $\aa_{M,\infty}$. Thus, 
\begin{equation}\label{eq:a0infty-in-ginfty}
\mathfrak{a}_{0,\infty}=\left\{\prod_{v\mid\infty}{\rm diag}(a_{v1},\ldots ,a_{vn}): a_{vi}\in\R\right\}\subset \mathfrak{g}_\infty=\mathfrak{gl}_n(F\otimes\R)
\end{equation}
is of dimension $rn$ and $\aa_{G,\infty}$, consisting of central elements at every place, is of dimension $r$. For every $M\in\mathcal{L}_\infty$, the Weyl group $W(A_M)=\prod_{v\mid\infty} W(A_{M_v})$ acts by orthogonal transformations on $\aa_{M,\infty}$ and $\aa_{M,\infty}^*$. When $M\subset L$ the subspace $\aa_{M,\infty}^L=\bigoplus_v \aa_{M_v}^{L_v}$ is the kernel of the natural map $\aa_{M,\infty}\rightarrow\aa_{L,\infty}$. For any $M\in\mathcal{L}_\infty$, let $\Phi^M=\Phi(A_0,M)$ denote the set of roots of $A_0=\prod_{v\mid\infty} A_{0,v}$ on the Lie algebra of $M$. Let $\Phi^{G,+}\subset\Phi^G$ be the subset of positive roots, with respect to the ordering induced by the standard Borel subgroup, and put $\Phi^{M,+}=\Phi^M\cap \Phi^{G,+}$.

\subsection{Norms and Weyl groups}\label{norms-weyl-groups-subsec}

In this subsection we let $v\mid\infty$. We let $\mathfrak{g}_v$ be the Lie algebra of $G_v$. We denote by $\theta_v$ the usual Cartan involution of minus conjugate transpose, and let $\mathfrak{g}_v=\mathfrak{p}_v\oplus\mathfrak{k}_v$ be the corresponding Cartan decomposition. Fix an ${\rm Ad}$-invariant (i.e., conjugation invariant) non-degenerate bilinear form $B_v$ on $\mathfrak{g}_v$, which is positive-definite on $\mathfrak{p}_v$ and negative definite on $\mathfrak{k}_v$. Then $B_v$ defines an inner product $\langle ,\rangle_v$ on $\mathfrak{g}_v$ by the rule $\langle X,Y\rangle_v =-B_v(X,\theta_v Y)$. We may, for example, take $B_v(X,Y)={\rm tr}(XY)$, so that $\langle X,Y\rangle_v={\rm tr}(X {}^t\overline{Y})$. We may extend $\langle \,,\,\rangle_v$ to a Hermitian inner product on $\mathfrak{g}_{v,\C}$ in the natural way. We may restrict $\langle \, ,\, \rangle_v$ to $\mathfrak{a}_{0,\C}$ and its subspaces. For $L_v\in\mathcal{L}(M_v)$, it then follows from the definition of $\aa_{M_v}^{L_v}$ and the description of $B_v$ that the direct sum decomposition $\mathfrak{a}_{M_v}=\mathfrak{a}_{L_v}\oplus\mathfrak{a}_{M_v}^{L_v}$ is orthogonal with respect to $\langle \cdot,\cdot\rangle_v$. We shall write $\|\cdot\|$ for the norm induced by $\langle\cdot,\cdot\rangle_v$. Using $B$ to identify $\mathfrak{a}_{0,\C}$ and $\mathfrak{a}_{0,\C}^*$ we obtain a norm on the latter space, again denoted by $\| \cdot\|$. 

Recall  \S\ref{sec:Iwasawa-decomp} that $\bfk_v={\rm O}(n)$ or ${\rm U}(n)$ according to whether $v=\R$ or $v=\C$; thus $\bfk_v$ is the group of fixed points of standard Cartan involution $\Theta_v$ of inverse conjugate transpose on $G_v$, and ${\rm Lie}(\bfk_v)=\mathfrak{k}_v$. It follows from the $\Ad$-invariance of $B$ that $N_{\bfk_v}(\mathfrak{a}_{M_v})=\{k\in \bfk_v: \Ad_k \mathfrak{a}\subset\mathfrak{a}\}$ acts by norm preserving transformations on $\aa_{M_v,\C}$ and $\aa_{M_v,\C}^*$. Letting $Z_{\bfk_v}(\mathfrak{a}_{M_v})$ denote the kernel of this action we put $W(A_{M_v})=N_{\bfk_v}(\mathfrak{a}_{M_v})/Z_{\bfk_v}(\mathfrak{a}_{M_v})$. Note that since $M_v$ is split, the algebraic and analytic Weyl groups -- $W(A_{M_v})$ and $W_{M_v}=N_{G_v}(M_v)/M_v$ -- coincide \cite[\S 5.1]{BorelTits}.

Once again, these definitions extend in the obvious way to the $S$-adic setting. For example, when $S=\{v\mid\infty\}$, we denote by $\mathfrak{g}_\infty=\bigoplus_{v\mid\infty}\mathfrak{g}_v$ the Lie algebra of $G_\infty$. We extend in the obvious way the inner product on $\mathfrak{g}_v$ of the previous subsection to $\mathfrak{g}_\infty$.

\subsection{The trace zero subspace $\mathfrak{h}_M$}\label{sec:lie}

We need to adapt the purely archimedean structures of the previous paragraph to be compatible with $\bm{G}(\A_F)^1$, or more generally with $\bm{M}(\A_F)^1$, where $\bm{M}\in\mathcal{L}$ is a rational Levi subgroup. 

The diagonal embedding of $\bm{M}(\R)$ in $\bm{M}(F_\infty)$ described in \S\ref{sec:subgp} induces a diagonal embedding of $\mathfrak{m}={\rm Lie}(\bm{M})$ in $\mathfrak{m}_\infty$. For a rational Levi $\bm{M}\in\mathcal{L}$, the Lie algebra $\aa_{\bm{M}}$ should not be confused with $\aa_{M,\infty}$, where $M=\bm{M}_\infty$; the former is of dimension $m$, the latter of dimension $rm$, where $r$ is the number of places of $F$. For a rational Levi $\bm{M}\in\mathcal{L}$, we shall, as a general rule, identify $\aa_{\bm{M}}$ with its diagonal image in $\aa_{0,\infty}$. For example, when $\bm{M}=\bm{T}_0$, the following global Lie algebras, 
\begin{equation}\label{eq:a0G-in-a0}
\aa_0^{\bm{G}}=\left\{\prod_{v\mid\infty}{\rm diag}(a_1,\ldots ,a_n) : \sum_i a_i=0\right\}\subset \mathfrak{a}_0=\left\{\prod_{v\mid\infty}{\rm diag}(a_1,\ldots ,a_n): a_i\in\R\right\},
\end{equation}
defined in \S\ref{sec:global-characters}, are diagonally embedded in $\mathfrak{a}_{0,\infty}$ as in \eqref{eq:a0infty-in-ginfty}.

For $M\in\mathcal{L}_\infty$ we let $\bm{M}\in\mathcal{L}$ be a rational Levi such that $M\subset\bm{M}_\infty$. We let $\h_M^{\bm{M}}$ denote the orthogonal complement of $\aa_{\bm{M}}$ in $\aa_{M,\infty}$. As usual, when $M=T_{0,\infty}$ we shall write $\h_0^{\bm{M}}$ in place of $\h_{T_{0,\infty}}^{\bm{M}}$ and when $\bm{M}=\bm{G}$ we shall drop the superscript and write $\h_M$ for $\h_M^{\bm{G}}$. The spaces $\mathfrak{h}_M$ will be ubiquitous in everything that follows: indeed, much of the work to execute the strategy of Section \ref{section:outline} amounts to exhibiting appropriate functions of Paley--Wiener type on the dual $\mathfrak{h}_{M,\C}^*$. 

For $H\in\mathfrak{h}_M^{\bm{M}}$, let $B_M^{\bm{M}}(H,R)$ denote the ball of radius $R>0$ about $H$ in $\mathfrak{h}_M^{\bm{M}}$. Similarly $B_M^{\bm{M}}(\mu,r)$ shall denote the ball of radius $r$ in $(\mathfrak{h}_M^{\bm{M}})^*$ about $\mu\in (\mathfrak{h}_M^{\bm{M}})^*$. When $M=T_{0,\infty}$ we shall drop the $M$ subscript in the ball notation, and when $\bm{M}=\bm{G}$ we drop the superscript. In view of the Cartan decomposition $\bm{M}_\infty^1=\bfk_\infty^{\bm{M}}\exp(\mathfrak{h}_0^{\bm{M}})\bfk_\infty^{\bm{M}}$ the sets $\mathbf{K}_\infty^{\bm{M}}\exp B^{\bm{M}}(0,R)\mathbf{K}_\infty^{\bm{M}}$ form an exhaustive system of neighborhoods of the identity in $\bm{M}_\infty^1$. When $\bm{M}=\bm{G}$, we recall the boldface convention in Remark \ref{rem:boldface-notation} and agree to write $G_{\infty,\leqslant R}^1=\mathbf{K}_\infty\exp B(0,R)\mathbf{K}_\infty$. Denote by $\mathcal{H}(G_\infty^1)_R$ the space of all smooth functions on $G_\infty^1$ supported in $G_{\infty,\leqslant R}^1$.

\subsection{Subspace decompositions of $\mathfrak{h}_M$}\label{sec:hM-decomp}

We now provide several subspace decompositions of $\mathfrak{h}_M$ which will be of use later in the paper.

Since $\aa_{\bm{M}}=\aa_{\bm{G}}\oplus\aa_{\bm{M}}^{\bm{G}}$, we have $\aa_{M,\infty}=\aa_{\bm{G}}\oplus\aa_{\bm{M}}^{\bm{G}}\oplus\h^{\bm{M}}_M$ so that
\begin{equation}\label{eq:hM-aMG}
\h_M=\aa_{\bm{M}}^{\bm{G}}\oplus\h^{\bm{M}}_M.
\end{equation}
For $H\in\h_M$ we shall write $H=H_{\bm{M}}+H^{\bm{M}}$ according to this decomposition, and similarly for $\nu\in i\h_M^*$. For example, when $M=T_{0,\infty}$ and $\bm{M}=\bm{T}_0$, the inclusion $\h_0^{\bm{T}_0}\subset\h_0$ is given by
\[
\left\{\prod_{v\mid\infty}{\rm diag}(a_{v1},\ldots ,a_{vn}): \sum_v d_v a_{vi}=0 \; \forall\; i\right\}\subset\left\{\prod_{v\mid\infty}{\rm diag}(a_{v1},\ldots ,a_{vn}): \sum_vd_v\sum_i a_{vi}=0\right\},
\]
where $d_v=[F_v:\R]$. Recalling \eqref{eq:a0G-in-a0}, we deduce that $\h_0=\aa_0^{\bm{G}}\oplus\h_0^{\bm{T}_0}$, verifying \eqref{eq:hM-aMG} in this case. The trace-zero conditions defining $\h_0^{\bm{T}_0}$ and $\h_0$ are over all places $v$, which is a reflection of the fact that they are orthogonal complements of globally defined Lie algebras. Using $\langle \, ,\, \rangle$, we may identify $\mathfrak{h}_M$ with $\mathfrak{h}_M^*$ to obtain an analogous decomposition of the latter space. The decomposition \eqref{eq:hM-aMG} according to \textit{rational} Levi subgroups will be used in the trace formula estimates (the spectral side) of Section \ref{EisensteinSection}, in the special case when $M=T_{0,\infty}$.

Similarly, if $L\in\mathcal{L}_\infty$ is any archimedean Levi such that $M\subset L$, we let $\h_M^L$ denote the orthogonal complement of $\aa_{L,\infty}$ inside $\aa_{M,\infty}$. When $M=T_{0,\infty}$ we shall write $\h_0^L$ in place of $\h_{T_{0,\infty}}^L$. When $L=G$, the space $\h^G_M$ is the codimension $r$ subspace of $\aa_{M,\infty}$ consisting of elements whose $v$-adic trace, for each $v\mid\infty$, vanishes; in other words,
\begin{equation}
\label{semisimple-decomp}
\h_M^G=\bigoplus_{v\mid\infty}\aa_{M_v}^{G_v}.
\end{equation}
With $\h_G$ as above, we have the decomposition
\begin{equation}\label{eq:hM-hG}
\h_{M}=\h_G\oplus \h_M^G
\end{equation}
into \textit{central and semisimple components}, respectively. For example, when $M=T_{0,\infty}$, we have
\[
\h_G=\left\{\prod_{v\mid\infty}{\rm diag}(a_v,\ldots ,a_v): \sum_v d_va_v=0\right\}\;\textrm{ and }\;\h_0^G=\left\{\prod_{v\mid\infty}{\rm diag}(a_{v1},\ldots ,a_{vn}): \sum_i a_{vi}=0\;\forall\; v\right\}.
\]
We obtain the decomposition $\h_0=\h_G\oplus \h_0^G$ in \eqref{eq:hM-hG}, which furthermore recovers that of \cite[\S 4.2]{Matz2016}. The decompositions \eqref{eq:hM-hG} and \eqref{semisimple-decomp} will be used in the construction of spectral localizing functions in Sections \ref{spec-loc}--\ref{sec:Paley-Wiener} and \ref{PW}. The terminology of \textit{central} and \textit{semisimple} components is meant to reflect the fact that the continuous parameter $\nu\in\mathfrak{h}_{M,\C}^*$ (see \S\ref{rep-notation}) of an irreducible representation $\pi_\infty$ of $G_\infty^1$ will decompose as $\nu=\nu_Z+\nu^0$ in the dual decomposition $\h_{M,\C}^*=\h_{G,\C}^*\oplus (\h_M^G)_\C^*$, with $\nu_Z$ arising from the central character of $\pi_\infty$. We also often further write $\nu^0=(\nu^0_v)_{v\mid\infty}$ relative to the decomposition \eqref{semisimple-decomp}.

Then $W(A_M)$ acts by orthogonal transformations on $\h_{M}$, leaving $\h_G$ fixed elementwise. Moreover, $\h_M^G=\h_L^G\oplus \h_M^L$, yielding
\begin{equation}\label{eq:decomp:arch-Levis}
\h_M=(\h_G\oplus \h_L^G)\oplus \h_M^L=\h_L\oplus \h_M^L,
\end{equation}
and similarly for $\h_M^*$. For $H\in\h_M$ (resp. $\nu\in\h_M^*$) we shall write $H=H_L+H^L$ (resp. $\nu=\nu_L+\nu^L$) according to this decomposition. Note that when $L=G$ in \eqref{eq:decomp:arch-Levis} we recover \eqref{eq:hM-hG}. The decomposition \eqref{eq:decomp:arch-Levis} according to \emph{archimedean} Levi subgroups will be used in Section \ref{ExceptionalSpectrumSection}, where we estimate the contribution of the complementary spectrum.

\subsection{Unitary and admissible archimedean duals}\label{rep-notation}

Let ${\rm Irr}(G_v)$ and ${\rm Irr}(G_\infty^1)$ denote the admissible duals of  $G_v$ and $G_\infty^1$, respectively. The main goal of this subsection is to review the description of ${\rm Irr}(G_v)$ and ${\rm Irr}(G_\infty^1)$ (due to Langlands and Knapp--Zuckerman, and nicely described in \cite[\S 2]{Knapp1994}), and, in particular, to associate with every infinitesimal class of irreducible admissible representations $\pi$ an equivalence class of spectral data $[\delta,\nu]$, consisting of a discrete parameter $\underline{\delta}$ and a continuous parameter $\nu$, the latter being taken up to Weyl group symmetry. Since the unitary duals $\Pi(G_v)$ and $\Pi(G_\infty^1)$ can be identified with subsets of the respective admissible duals, this parametrization can also be applied to isomorphism classes of irreducible unitary representations.

As in Vogan \cite{Vogan1979,Vogan1986}, we may put a norm on $\Pi(\bfk_\infty)$ in the following way. Let $\bfk_\infty^0$ denote the connected component of the identity of $\bfk_\infty$;  then $\bfk_\infty^0$ is the analytic subgroup corresponding to the Lie subalgebra $\mathfrak{k}$, as in \S\ref{sec:v-adic}. Let $\bfk^T_\infty=T_{0,\infty}\cap\bfk_\infty^0$ have corresponding Cartan subalgebra $\mathfrak{k}_T$. For any $\tau\in\Pi(\bfk_\infty)$ let $\lambda_\tau\in \mathfrak{k}_T^*$ be the highest weight associated with any irreducible component of $\tau|_{\bfk_\infty^0}$. Let $\|\cdot\|^2=-B(\cdot,\theta\cdot)$ be the norm on $\mathfrak{k}_T$ given by restricting the norm on $\mathfrak{g}_\infty$ from \S\ref{norms-weyl-groups-subsec}; under the identification of $\mathfrak{k}_T$ with its dual $\mathfrak{k}_T^*$ induced by the non-degenerate bilinear form $B$, this defines a norm on $\mathfrak{k}_T^*$, which we again denote by $\|\cdot\|^2$. Then, we put $\|\tau\|=\|\lambda_\tau+2\rho_{\bfk_\infty}\|^2$, where $\rho_{\bfk_\infty}$ denotes half the sum of the positive roots of $\bfk_\infty^T$ inside $\bfk_\infty$. Given $\pi\in\Pi(G_\infty)$, Vogan defines a minimal $\bfk_\infty$-type of $\pi$ as any $\tau\in\Pi(\bfk_\infty)$ of minimal norm appearing in $\pi|_{\bfk_\infty}$. In general, a minimal ${\bf K}_\infty$-type is not necessarily unique, but for $G_{\infty}=\GL_n(F_\infty)$ it is \cite[Theorem 4.9]{Vogan1986}.

For $M\in\mathcal{L}_v$ or $M\in\mathcal{L}_\infty$ we let $\mathscr{E}^2(M^1)$ denote the set of isomorphism classes of square-integrable representations of $M^1$. We say that $M$ is {\it cuspidal} if $\mathscr{E}^2(M^1)$ is non-empty. If $M\in\mathcal{L}_v$ and $M$ is isomorphic to $\GL_{n_1}(F_v)\times\cdots\times\GL_{n_m}(F_v)$, where $n_1+\cdots +n_m=n$, then $M$ is cuspidal precisely when $1\leqslant n_j\leqslant 2$ for $v$ real and $n_j=1$ for $v$ complex. 

Following \cite[\S 6]{FinisLapidMuller2015} we let $\mathcal{D}$ denote the $G_\infty$-conjugacy classes of pairs $(M,\delta)$ consisting of a cuspidal Levi subgroup $M$ of $G_\infty$ and $\delta\in\mathscr{E}^2(M^1)$. We note that $W(A_M)=N_G(M)/M$ acts on $\mathscr{E}^2(M^1)$ via $(w.\delta)(m)=\delta(w^{-1}mw)$. We let $W(A_M)_\delta$ denote the stabilizer of $\delta$ in $W(A_M)$. Furthermore, we put
\begin{equation}\label{relative-Weyl}
W(A^{M'}_M)_\delta=\{w\in W(A_M)_\delta: w\nu=\nu\;\;\forall\,\nu\in\h_{M'}^*\}.
\end{equation}

Fixing $\underline{\delta}\in\mathcal{D}$, we may and shall (in this paragraph) take $\underline{\delta}$ to be represented by a pair $(M,\delta)$, where $M\in\mathcal{L}_{{\rm st},\infty}$ has blocks of descending size. We call such a pair $(M,\delta)$ a \textit{standard representative} of $\underline{\delta}$. We caution, however, that the representation $\delta$ in the pair $(M,\delta)$ is not uniquely determined; if $(M,\delta)$ is a standard representative for $\underline{\delta}$ then so is $(M,w.\delta)$ for all $w\in W(A_M)$. For $\nu\in\mathfrak{h}_{M,\C}^*$, we may form an essentially square-integrable representation $\delta\otimes e^\nu$ of $M$. Let $P$ be the unique standard parabolic containing $M$ as a Levi subgroup. We may then consider the (unitarily) induced representation $\pi(\delta,\nu)={\rm Ind}_P^G(\delta\otimes e^\nu)$. This is not, in general, irreducible. Nevertheless, if ${\rm Re}\,\nu$ is taken to lie in $\{\lambda\in \h_M^*: \langle \lambda,\alpha\rangle\geqslant 0 \,\textrm{ for all simple }\, \alpha\in \Phi^G\setminus \Phi^M\}$ (which one can achieve by replacing $(\delta,\nu)$ by $(w.\delta,w.\nu)$ for some $w\in W(A_M)$), then $\pi(\delta,\nu)$ admits a unique irreducible quotient; see \cite[Theorems 1 \& 4]{Knapp1994}. The infinitesimal equivalence class of this quotient is denoted $\pi_{\delta,\nu}$. Then $\pi_{\delta,\nu}\in {\rm Irr}(G_\infty^1)$. Conversely, any $\pi\in {\rm Irr}(G_\infty^1)$ arises from the preceding construction, as follows: there is a standard cuspidal Levi subgroup $M$ with blocks in descending size, a $\delta\in\mathscr{E}^2(M^1)$, and $\nu\in\mathfrak{h}_{M,\C}^*$ such that $\pi\simeq\pi_{\delta,\nu}$. Given $M$, a pair $(\delta,\nu)$ for which the last isomorphism holds is unique up to the diagonal action by $W(A_M)$, in which $w\in W(A_M)$ sends $(\delta,\nu)$ to $(w.\delta,w.\nu)$. We write $[\delta_\pi,\nu_\pi]$ for the $W(A_M)$-orbit of any such pair, and $\underline{\delta}_\pi\in\mathcal{D}$ for the conjugacy class generated by $(M,\delta)$.

Let $G'=\prod_{v\mid\R}\SL_n^\pm(\R)\prod_{v\mid\C}\SL_n(\C)$, where $\SL_n^\pm(\R)=\{g\in\GL_n(\R): \det (g)=\pm 1\}$. We furthermore put $\bfk_\infty'=\bfk_\infty\cap G'$. For $\underline{\delta}\in\mathcal{D}$, with standard representative $(M,\delta)$, we let $\tau(\pi_\delta)\in\Pi(\bfk_\infty')$ be the minimal $\bfk_\infty'$-type of $\pi_{\delta,\nu}$, for any choice of $\nu\in i\mathfrak{h}_M^*$. This is well-defined, since the $\bfk_\infty$-type decomposition of $\pi(\delta,\nu)$ is independent of $\nu$, and $\pi(\delta,\nu)$ is irreducible for $\nu\in i\mathfrak{h}_M^*$ (see \cite[Proposition 3]{Moeglin1997}). We set $\|\delta\|=\|\tau(\pi_\delta)\|$. The set $\mathcal{D}$ is equipped with a partial ordering (see \cite[\S 2.3]{ClozelDelorme1984}), in which $\underline{\delta}'\prec\underline{\delta}$ if $\|\delta'\|<\|\delta\|$ or $\underline{\delta}'=\underline{\delta}$. This partial ordering will be used in \S\ref{CD-appl} and again for the specific case of $\GL_2$ in Sections \ref{PW}--\ref{EisensteinSection}.

\subsection{Decomposition according to $\delta$-type}\label{disc-aut}

We shall in fact need to parametrize the admissible dual of $\bm{M}_\infty^1$ for any rational Levi $\bm{M}\in\mathcal{L}$. Similarly to the above, we let $\mathcal{D}^{\bm{M}}$ denote the $\bm{M}_\infty$-conjugacy classes of pairs $(M,\delta)$ consisting of a cuspidal Levi subgroup $M$ of $\bm{M}_\infty$ and $\delta\in\mathscr{E}^2(M^1)$. (When $\bm{M}=\bm{G}$ we simply write $\mathcal{D}=\mathcal{D}^{\bm{G}}$.) Extending the earlier notion, we shall say that the pair $(M,\delta)$ is a standard representative of $\underline{\delta}\in\mathcal{D}^{\bm{M}}$ if, for every block of $\bm{M}$, the corresponding subblocks of $M$ are of descending size. Given $\underline{\delta}\in\mathcal{D}^{\bm{M}}$, we write ${\rm Irr}(\bm{M}_\infty^1)_{\underline{\delta}}$ for the subset of $\pi\in {\rm Irr}(\bm{M}_\infty^1)$ with $\underline{\delta}_\pi=\underline{\delta}$. We have
\[
{\rm Irr}(\bm{M}_\infty^1)=\bigsqcup_{\underline{\delta}\in\mathcal{D}^{\bm{M}}}{\rm Irr}(\bm{M}_\infty^1)_{\underline{\delta}}.
\]
From the inclusion $\Pi(\bm{M}_\infty^1)\subset {\rm Irr}(\bm{M}_\infty^1)$, we deduce
\begin{equation}\label{Pi-delta}
\Pi(\bm{M}_\infty^1)=\bigsqcup_{\underline{\delta}\in\mathcal{D}^{\bm{M}}}\Pi(\bm{M}_\infty^1)_{\underline{\delta}},
\end{equation}
extending to $\bm{M}$ the decomposition \eqref{eq:intro:discrete-decomp} from the introduction.

We may furthermore decompose the discrete automorphic spectrum $\Pi_{\rm disc}(\bm{M}(\mathbb{A}_F)^1)$, recalled earlier in \S\ref{JMspec}, according to the archimedean decomposition \eqref{Pi-delta}. For $\pi\in\Pi_{\rm disc}(\bm{M}(\mathbb{A}_F)^1)$, we let $\underline{\delta}_\pi$ denote the discrete parameter $\underline{\delta}_{\pi_\infty}\in\mathcal{D}^{\bm{M}}$ associated with $\pi_\infty$. We have
\[
\Pi_{\rm disc}(\bm{M}(\A_F)^1)=\bigsqcup_{\underline{\delta}\in\mathcal{D}^{\bm{M}}}\Pi_{\rm disc}(\bm{M}(\A_F)^1)_{\underline{\delta}},
\]
where $\Pi_{\rm disc}(\bm{M}(\mathbb{A}_F)^1)_{\underline{\delta}}$ consists of those $\pi$ for which $\underline{\delta}_\pi=\underline{\delta}$. Finally, when $\pi\in\Pi_{\rm disc}(\bm{M}(\mathbb{A}_F)^1)_{\underline{\delta}}$, with $\underline{\delta}$ having standard representative $(M,\delta)$, we shall write $[\delta_\pi,\nu_\pi]$ for the $W(A_M^{\bm{M}})$-orbit of a pair $(\delta_\pi,\nu_\pi)$ consisting of a discrete parameter $\delta_\pi\in\mathscr{E}^2(M^1)$ and a continuous parameter $\nu_\pi\in (\mathfrak{h}_M^{\bm{M}})_\C^*$ associated with $\pi_\infty$.

\subsection{Hermitian archimedean dual}\label{sec:herm} 

We now describe the Hermitian dual of $G_\infty$. We are inspired by the treatment in Lapid--M\"uller \cite[\S 3.3]{LapidMuller2009}, who deal with the spherical case. 

For $M\in\mathcal{L}_\infty$ and $w\in W(A_M)$ we let
\begin{equation}\label{hMw*}
\mathfrak{h}_{M,w}^*=\{\nu\in\mathfrak{h}_{M,\C}^*: w\nu=-\overline{\nu}\}.
\end{equation}
For $\delta\in\mathscr{E}^2(M^1)$, we introduce the $\delta$-Hermitian spectrum
\begin{equation}\label{delta-hm}
\mathfrak{h}_{\delta,{\rm hm}}^*=\bigcup_{w\in W(A_M)_\delta}\mathfrak{h}_{M,w}^*.
\end{equation}
The key property of $\mathfrak{h}_{\delta,{\rm hm}}^*$ is that whenever $\nu\in \mathfrak{h}_{\delta,{\rm hm}}^*$ the representation $\pi_{\delta,\nu}$ admits a non-degenerate Hermitian structure \cite[Theorem 16.6]{Knapp2001}. If we put
\begin{equation}\label{delta-unitary}
\mathfrak{h}^*_{\delta,{\rm un}}=\{\nu\in\mathfrak{h}_{M,\C}^*:\pi_{\delta,\nu}\text{ unitarizable}\},
\end{equation}
then $\mathfrak{h}^*_{\delta,{\rm un}}\subset \mathfrak{h}^*_{\delta, {\rm hm}}$. For $\nu\in\mathfrak{h}_{\delta,{\rm un}}^{\ast}$ we have $\|{\rm Re}\, \nu\|\leqslant \|\rho\|$, where $\rho$ is the half-sum of positive roots on $G_\infty$ \cite[Chapter IV, Theorem 8.1]{Helgason2000}.

A key feature of the $\delta$-Hermitian spectrum is that if $\nu\in\mathfrak{h}^*_{\delta, {\rm hm}}$ is such that ${\rm Re}\, \nu\neq 0$ then ${\rm Im}\,\nu$ is forced to belong to a positive codimension subspace in $\h_M^*$. We would like to be more precise about the collection of such singular subspaces in $\h_M^*$. We begin by noting that if $\mathfrak{h}_{M,w,\pm 1}^*$ denotes the $\pm 1$ eigenspaces for $w$ acting on $\mathfrak{h}_M^*$, then $\mathfrak{h}_{M,w}^*=\mathfrak{h}_{M,w,-1}^*+i\mathfrak{h}_{M,w,+1}^*$. From \cite[Theorem 6.27]{OrlikTerao1992} it follows that for every $w\in W(A_M)$ we have $\h_{M,w,+1}^*=\h_{M_w}^*$, where $M_w$ is the smallest Levi subgroup in $\mathcal{L}_\infty$ containing $M$ and (a representative of) $w$.

Let $\mathcal{L}_\infty(\delta)$ be the collection of the $M_w$ as $w$ varies over $W(A_M)_\delta$. In view of \eqref{delta-hm}, as well as the observations of the preceding paragraph, we obtain the following inclusion
\[
\{{\rm Im}\,\nu: \nu\in\mathfrak{h}_{\delta,{\rm hm}}^*,\; {\rm Re}\, \nu\neq 0\}\subset \mathfrak{h}_{\delta,{\rm sing}}^*
\]
where
\begin{equation}\label{delta-sing}
\mathfrak{h}_{\delta,{\rm sing}}^*=\bigcup_{\substack{M'\in\mathcal{L}_\infty(\delta)\\ M'\supsetneq M}}\h_{M'}^*.
\end{equation}
We shall refer to $\mathfrak{h}_{\delta,{\rm sing}}^*$ as the \textit{$\delta$-singular subset of $\h_M^*$}.

The following lattice-type property of $\mathcal{L}_\infty(\delta)$, extending \cite[\S 3.2]{LapidMuller2009}, will be used repeatedly in Section \ref{ExceptionalSpectrumSection}.

\begin{lemma}
If $L_1$ and $L_2$ lie in $\mathcal{L}_\infty(\delta)$ then the Levi subgroup they generate $\langle L_1,L_2\rangle$ also lies in $\mathcal{L}_\infty(\delta)$. 
\end{lemma}

\begin{proof}
It suffices to prove the proposition for each archimedean place $v$, taking $L_1,L_2$ to lie in $\mathcal{L}_v(\delta_v)$. We therefore fix $v\mid\infty$ and drop the $v$ subscripts throughout the proof. 

We begin by giving an explicit description of $W(A_M)_\delta$. Conjugating $M$ if necessary, we may assume that $M$ is a standard Levi subgroup, whose block decomposition is ordered, with respect to $\delta$, in such a way that
\[
M=M_1^{m_1}\times \cdots \times M_t^{m_t},\qquad\delta=\delta_1^{\otimes m_1}\otimes\cdots\otimes\delta_t^{\otimes m_t},
\]
where the $\delta_i\in\mathscr{E}^2(M_i^1)$ are pairwise distinct. Then $W(A_M)_\delta$ is isomorphic to $\mathfrak{S}_{m_1}\times\cdots\times\mathfrak{S}_{m_t}$, and we may write $w\in W(A_M)_{\delta}$ as $(w_1,\ldots ,w_t)$, where $w_i\in\mathfrak{S}_{m_i}$, so that $M_w=M_{w_1}\times\cdots \times M_{w_t}$. It therefore suffices to prove the claim for pairs of the form $(M^m,\delta^{\otimes m})$.

We can describe any $L\in\mathcal{L}(\delta)$ explicitly as follows: if $w\in\mathfrak{S}_m$ is such that $L=M_w$, then $L$ is a product of Levi subgroups indexed by the disjoint factors in the cyclic decomposition of $w$, with a cycle of support $\{i_1,\ldots ,i_r\}$ being associated with the Levi subgroup generated by $M_{i_1},\ldots, M_{i_r}$. In this way, $M_w$ depends only on the partition of $\{1,\ldots ,m\}$ determined by $w$. With this description, it can now be checked that, given $w_1,w_2\in\mathfrak{S}_m$, we have $\langle M_{w_1},M_{w_2}\rangle=M_w$, where the partition associated with $w\in\mathfrak{S}_m$ is the finest common coarsening of those associated with $w_1,w_2$.
 \end{proof}

Finally, for every $L\in\mathcal{L}_\infty(\delta)$, we denote by $W(M,L)_\delta$ the subset of $W(A_M)_{\delta}$ consisting of transpositions fixing $\mathfrak{h}_L$. (This is not a group, as for example, it does not contain the identity.) Note that $W(M,M)_\delta=\emptyset$ but that $|W(M,L)_\delta|\geqslant 1$ for $L\supsetneq M$. See Remark \ref{rem:use-of-W(M,L)} for the role of this subset in our main estimates.

\subsection{Paley--Wiener classes}\label{PWCD}

Let $M\in\mathcal{L}_\infty$. For a function $g\in C^\infty_c(\mathfrak{h}_M)$ and $\nu\in\mathfrak{h}_{M,\C}^*$ let $\widehat{g}(\nu)=\int_{\mathfrak{h}_M} g(X)e^{\langle \nu,X\rangle}dX$ denote the Fourier transform of $g$ at $\nu$. The image of $C^\infty_c(\mathfrak{h}_M)$ under this map is the Paley--Wiener space $\mathcal{PW}(\mathfrak{h}_{M,\C}^*)$. Recall that
\begin{equation}
\label{PWDecomposition}
\mathcal{PW}(\mathfrak{h}_{M,\C}^*)=\bigcup_{R>0}\mathcal{PW}(\mathfrak{h}_{M,\C}^*)_R,
\end{equation}
where $\mathcal{PW}(\mathfrak{h}_{M,\C}^*)_R$ consists of those entire functions $h$ on $\mathfrak{h}_{M,\C}^*$ such that for all $k\geqslant 0$ we have
\[
\sup_{\nu\in\mathfrak{h}_{M,\C}^*}\left\{ |h(\nu)|e^{-R\|{\rm Re}\, \nu\|}(1+\|\nu\|)^k\right\}<\infty.
\]
Then the Fourier transform $C^\infty_c(\mathfrak{h}_M)\rightarrow\mathcal{PW}(\mathfrak{h}_{M,\C}^*)$ is an isomorphism of topological algebras, each of these spaces being taken with their natural Fr\'echet topologies. Moreover, for $R>0$, if $C^\infty_c(\mathfrak{h}_M)_R$ denotes the subspace of $g\in C^\infty_c(\mathfrak{h}_M)$ having support in the ball $B_M(0,R)$, then the Fourier transform maps $C^\infty_c(\mathfrak{h}_M)_R$ onto $\mathcal{PW}(\mathfrak{h}_{M,\C}^*)_R$; see \cite[Theorem~3.5]{Gangolli1971}.

We will also be working with decompositions $\h_{M,\C}^{\ast}=\h_{G,\C}^{\ast}\oplus(\h_M^G)^{\ast}_{\C}$ and $(\h_M^G)^{\ast}_C=\bigoplus_{v\mid\infty}(\aa_{M_v}^{G_v})^{\ast}_{\C}$ dual to \eqref{eq:hM-hG} and \eqref{semisimple-decomp}, and with the analogously defined  Paley--Wiener spaces $\mathcal{PW}((\h_{G,\C})^{\ast})_R$, $\mathcal{PW}((\h_M^G)^{\ast}_{\C})_R$, and $\mathcal{PW}((\aa_{M_v}^{G_v})^{\ast}_{\C})_R$.

We write $\mathcal{PW}$ for the space of collections of complex valued functions $H_M(\delta,\nu)$, indexed by standard Levi subgroups $M$ (with blocks in descending order), and defined on pairs $(\delta,\nu)\in \mathscr{E}^2(M^1)\times\h_{M,\C}^*$, such that, for each $M$,
\begin{enumerate}
\item $H_M(\cdot,\nu)$ is supported on a finite set of $\delta$, in the sense that for all but a finite number of $\delta\in \mathscr{E}^2(M^1)$ the function $\nu\mapsto H_M(\delta,\nu)$ is identically zero on $\h_{M,\C}^*$;
\item\label{PW-fixed-delta} for all $\delta\in \mathscr{E}^2(M^1)$ the function $H_M(\delta,\cdot)$ lies in $\mathcal{PW}(\mathfrak{h}_{M,\C}^*)$;
\item\label{PW3} $H_M(w.\delta,w.\nu)=H_M(\delta,\nu)$ for all $w\in W(A_M)$.
\end{enumerate}
Similarly, one defines $\mathcal{PW}_R$ using $\mathcal{PW}(\mathfrak{h}_{M,\C}^*)_R$ in \eqref{PW-fixed-delta}.

Lastly, for a fixed $\underline{\delta}\in\mathcal{D}$, we let $\mathcal{PW}_{\underline{\delta}}$ (resp. $\mathcal{PW}_{R,\underline{\delta}}$) be the subspace of $\mathcal{PW}$ (resp. $\mathcal{PW}_R$) consisting of functions $H_L$ on $\mathscr{E}^2(L^1)\times\h_{L,\C}^*$ for which $\nu\mapsto H_L(\sigma ,\nu)$ is identically zero on $\mathfrak{h}_{L,\C}^*$, unless $(L,\sigma)\in \underline{\delta}$. Similarly, we let $\mathcal{H}(G_\infty^1)_{\underline{\delta}}$ (resp. $\mathcal{H}(G_\infty^1)_{R,\underline{\delta}}$) denote the subspace of $\mathcal{H}(G_\infty^1)$ (resp. $\mathcal{H}(G_\infty^1)_R$) consisting of those $f$ that are mapped into $\mathcal{PW}_{\underline{\delta}}$ (resp. $\mathcal{PW}_{R,\underline{\delta}}$) under $f\mapsto {\rm tr}\,\pi_{\sigma,\nu}(f)$. (Since $\pi_{\sigma,\nu}$ depends only on the diagonal $W(A_L)$-orbit of $(\sigma,\nu)$, property \eqref{PW3} is always met.) Non-trivial examples of functions in $\mathcal{H}(G_\infty^1)_{\underline{\delta}}$ will be given in Section \ref{PW}.

\begin{remark}
Given a standard representative $(M,\delta)$ of $\underline{\delta}$, a function in $\mathcal{PW}_{\underline{\delta}}$ is uniquely determined by the choice of a single $W(A_M)_\delta$-invariant function $h\in\mathcal{PW}(\mathfrak{h}_{M,\C}^*)$. Indeed, from the latter, one can construct the associated collection of $H_L$ in $\mathcal{PW}_{\underline{\delta}}$ by setting $H_L(\sigma,\nu)$ equal to zero unless there exists $w\in W(A_M)$ satisfying $\sigma=w.\delta$, in which case we put $H_M(\sigma,\nu)=h(w^{-1}.\nu)$. This is well-defined, in light of the $W(A_M)_\delta$-invariance of $h$. Throughout this paper, we will identify $\mathcal{PW}_{\underline{\delta}}$ with the space of $W(A_M)_\delta$-invariants of $\mathcal{PW}(\mathfrak{h}_{M,\C}^*)$ via the above construction. The same discussion applies \textit{mutatis mutandis} to $\mathcal{PW}_{R,\underline{\delta}}$ and $\mathcal{PW}(\mathfrak{h}_{M,\C}^*)_R$.
\end{remark}

\section{Asymptotics of global Plancherel volume}\label{sec:cond-zeta}

In this section we asymptotically evaluate the global Plancherel volume $V_{\mathfrak{F}}(Q)$ from \eqref{eq:big-adelic-integral}, in the limit as $Q\rightarrow\infty$. Recall the constant $\mathscr{C}(\mathfrak{F})$ from \eqref{eq:def:tau-hat}. As usual we put $d=[F:\Q]$. The following is our main result.

\begin{prop}\label{MainTermLemma}
Let $0<\theta\leqslant 2/(d+1)$ when $n\geqslant 2$ and $0<\theta<\min \{1,2/(d+1)\}$ when $n=1$. Then for all $Q\geqslant 1$ we have
\[
V_{\mathfrak{F}}(Q)=\mathscr{C}(\mathfrak{F})Q^{n+1}+\textnormal{O}_\theta\big(Q^{n+1-\theta}\big).
\]
\end{prop}

It is the above proposition which essentially gives the shape of the leading term constant, as well as the precise power growth, in the Weyl--Schanuel law. The proof of Proposition \ref{MainTermLemma} culminates in \S\ref{sec:computing-zeta-residue}, where we shall also indicate how to deduce from it the following corollary.

\begin{cor}\label{MT-is-vol}
We have
\[
D_F^{n^2/2}\Delta_F^*(1)\sum_{1\leqslant \norm\q\leqslant Q}\sum_{\dd\mid\q}\lambda_n(\q/\dd)\varphi_n(\dd)\int\limits_{\substack{\pi_\infty\in\Pi(G_\infty^1)\\ q(\pi_\infty)\leqslant Q/\norm\q}}\mathrm{d}\widehat{\omega}_{\infty}^{\pl}(\pi_\infty)=\mathscr{C}(\mathfrak{F})Q^{n+1}+\textnormal{O}_\theta\big(Q^{n+1-\theta}\big),
\]
where $\varphi_n(\q)=\norm (\q)^n\prod_{\p\mid\q}\left(1-\frac{1}{\norm (\p)^n}\right)$.
\end{cor}
In view of the decomposition \eqref{univ-decomp}, the left-hand side of the above expression will be what naturally arises from our methods.

\subsection{Conversion to the canonical measure}\label{sec:conversion-to-Gross}

In order to line up with the measure conventions for Plancherel inversion in \eqref{eq:plancherel-inv}, we renormalize the spectral measures appearing in Proposition \ref{MainTermLemma}, to be taken with respect to the measure $\mu_{\bm{G},v}$.

The relation $\mu_{\bm{G}}=D_F^{n^2/2}\Delta_F^*(1)\omega_{\bm{G}}$ implies $\widehat{\omega}^{\rm pl}=D_F^{n^2/2}\Delta_F^*(1)\widehat{\mu}^{\rm pl}$. Similarly, for finite places $v$ the relation $\mu_{\bm{G},v}=\Delta_v(1)\omega_{{\bm G},v}$ implies $\widehat{\mu}_v^{\rm pl}=\Delta_v(1)^{-1}\widehat{\omega}_v^{\rm pl}$. When these local formulae are inserted into \eqref{defn-hat-nu},  we obtain
\[
\tau_{\mathfrak{F}}=D_F^{n^2/2}\Delta_F^*(1)\zeta_F^*(1)\big(\prod_{v<\infty} \zeta_v(1)^{-1}\tau_{\mathfrak{F},v}'\big)\tau_{\mathfrak{F},\infty},
\]
where the $\tau_{\mathfrak{F},v}'$ are defined as in \eqref{eq:defn:plv} but with respect to $\widehat{\mu}_v^{\rm pl}$. This shows that the quantity ${\rm vol}(\tau_{\mathfrak{F}})$ defined in \eqref{eq:defn:tau=pl} can be written as
\[
{\rm vol}(\tau_{\mathfrak{F}})=D_F^{n^2/2}\Delta_F^*(1)\zeta_F^*(1)\left(\prod_{v<\infty}\zeta_v(1)^{-1}Z_v(n+1)\right)Z_\infty(n+1),
\]
where
\begin{equation}\label{def:local-cond-zeta}
Z_v(s)=\int_{\Pi(\bm{G}(F_v))} q(\pi_v)^{-s}\, {\rm d}\widehat{\mu}_v^{\rm pl}(\pi_v),\qquad Z_\infty(s)=\int_{\Pi(\bm{G}(F_\infty)^1)}q(\pi_\infty)^{-s}\, {\rm d}\widehat{\omega}^{{\rm pl}}_\infty(\pi_\infty)
\end{equation}
are the local conductor zeta functions at finite and infinite places, respectively.

Proposition \ref{MainTermLemma} is then equivalent to the asymptotic
\begin{equation}\label{eq:Prop-asymp-mu-version}
\int\limits_{\substack{\pi\in\Pi(\bm{G}(\A_F)^1)\\ Q(\pi)\leqslant Q}}{\rm d}\widehat{\mu}^{\rm pl}(\pi)=\frac{1}{n+1}\zeta_F^*(1)\left(\prod_{v<\infty}\zeta_v(1)^{-1}Z_v(n+1)\right) Z_\infty(n+1)Q^{n+1}+\textnormal{O}_\theta\big(Q^{n+1-\theta}\big),
\end{equation}
the implied constants depending, as always, on $F$ and $n$. It is rather this version which we shall prove over the course of this section.

\subsection{Local and global conductors}\label{conductors}

We now define the local conductor of an irreducible generic representation of $\GL_n(F_v)$ for all places $v$ of $F$.

Let $\pi_v$ be an irreducible admissible representation of $\GL_n(F_v)$. The local conductor can be read off from the local functional equation of the standard $L$-function associated with $\pi_v$. Let $L(s,\pi_v)$ denote the standard $L$-function of $\pi_v$, as defined by Tate \cite{Tate1967} (for $n=1$) and Godement--Jacquet \cite{GodementJacquet1972} (for $n\geqslant 1$). When $v$ is finite, the local conductor appears in the epsilon factor, taken with respect to an unramified additive character of $F_v$. It was shown in \cite{Casselman1973, Jacquet2012, JacquetPiatetski-ShapiroShalika1981} that when, furthermore,  $\pi_v$ is generic, the local conductor encodes the existence of non-zero invariant vectors under a certain Hecke congruence subgroup. When $v$ is archimedean, the local $L$-factor is a product of Gamma factors, and the shifts that appear in them define the archimedean conductor. We review these definitions now. A more uniform approach, using the local $\gamma$-factors, can be found in \cite[\S 3.1.12]{MichelVenkatesh2010}.

\subsubsection{Non-archimedean case}\label{sec:def-cond}

Let $v$ be a non-archimedean place. For an additive character of level zero $\psi_v$, let $\epsilon(s,\pi_v,\psi_v)$ be the local espilon factor of $\pi_v$. Then there is an integer $f(\pi_v)$, independent of $\psi_v$, and a complex number $\epsilon(0,\pi_v,\psi_v)$ of absolute value 1 such that $\epsilon(s,\pi_v,\psi_v)=\epsilon(0,\pi_v,\psi_v)q_v^{-f(\pi_v)s}$. Moreover, $f(\pi_v)=0$ whenever $\pi_v$ is unramified.

Under the additional assumption that $\pi_v$ is generic, Jacquet, Piatetski-Shapiro, and Shalika~\cite{JacquetPiatetski-ShapiroShalika1981} show that the integer $f(\pi_v)$ is in fact always non-negative. One then calls $f(\pi_v)$ the conductor exponent of $\pi_v$. The conductor $q(\pi_v)$ of a generic irreducible $\pi_v$ is then defined to be $q(\pi_v)=q_v^{f(\pi_v)}$. In particular, $q(\pi_v)=1$ whenever $\pi_v$ is unramified.

For an integer $r\geqslant 0$ write $K_{1,v}(\p_v^r)$ for the subgroup of $\bfk_v$ consisting of matrices whose last row is congruent to $(0,0,\ldots ,1)$ mod $\p_v^r$. In particular, when $r=0$ we obtain the maximal compact $\bfk_v$. For ideals $\q_v=\p_v^r$ of $\mathcal{O}_v$, we have
\begin{equation}\label{eq:Euler-phi}
|\bfk_v/K_{1,v}(\q_v)|=\norm (\q_v)^n\left(1-\frac{1}{\norm (\p_v)^n}\right)=(\mu\star p_n)(\q_v)=\varphi_n(\q_v),
\end{equation}
where $p_n(\q_v)=\norm (\q_v)^n$ is the $n$-power function on $\mathcal{O}_v$-ideals. Jacquet, Piatetski-Shapiro, and Shalika~\cite{Jacquet2012, JacquetPiatetski-ShapiroShalika1981}, building on work of Casselman \cite{Casselman1973}, show that for any irreducible generic representation $\pi_v$ of $\GL_n(F_v)$ the conductor exponent $f(\pi_v)$ is equal to the smallest non-negative integer $r$ such that $\pi_v$ admits a non-zero fixed vector under $K_{1,v}(\p_v^r)$. Moreover, the space of all such fixed vectors is of dimension 1. Letting $\q_v(\pi_v)=\p_v^{f(\pi_v)}$, by the subsequent work of Reeder \cite{Reeder1991} it follows that for an $\mathcal{O}_v$-ideal $\q_v$ and an irreducible generic $\pi_v$ with $\q_v(\pi_v)|\q_v$, one has 
\begin{equation}\label{eq:Reeder}
\dim\pi_v^{K_{1,v}(\q_v)}=d_n(\q_v/\q_v(\pi_v)),
\end{equation}
where $d_n=1\star \cdots \star 1$ is the $n$-fold convolution of $1$ with itself. In particular, if $\pi_v$ is an irreducible generic representation of $\GL_n(F_v)$, one has
\begin{equation}\label{idem-trace}
{\rm tr} \left(\pi_v (\varepsilon_{K_{1,v}(\q_v(\pi_v)\p_v^r)})\right)=d_n(\p_v^r).
\end{equation}

\subsubsection{Archimedean conductor}\label{arch-cond}

For $v$ archimedean the local $L$-factor of $\pi_v$ is a product of shifted Gamma factors. We shall first describe these shifts relative to the inducing data for $\pi_v$, then use this expression to define the local conductor of $\pi_v$.

As in \S\ref{rep-notation}, we shall realize $\pi_v$ as $\pi_{\delta,\nu}$, for some $\delta\in\mathscr{E}^2(M^1)$ and $\nu\in\mathfrak{h}_{M,\C}^*$, where $M$ is a standard cuspidal Levi subgroup. Let $\sigma$ denote the essentially square-integrable representation $\delta\otimes e^\nu$ of $M$. Since $M$ is cuspidal it is isomorphic to $\GL_{n_1}(F_v)\times\cdots\times\GL_{n_m}(F_v)$, where $n_1+\cdots +n_m=n$, $1\leqslant n_j\leqslant 2$ for $v$ real and $n_j=1$ for $v$ complex. We may then decompose $\delta=\delta_1\otimes\cdots\otimes\delta_{n_m}$ and $\nu=\nu_1 \varepsilon_1+\cdots +\nu_m \varepsilon_m$, where $\nu_j=\langle \nu,\varepsilon_j\rangle\in\C$ and $(\varepsilon_j)$ is the standard basis for $\mathfrak{a}_{M,\C}^*=X^*(M)\otimes\C$. Thus $\sigma=\sigma_1\otimes\cdots\otimes\sigma_m$, where $\sigma_j=\delta_j\otimes e^{\nu_j\varepsilon_j}$. Then
\begin{equation}\label{eq:defn:L-factor-prod}
L_v(s,\pi_{\delta,\nu})=\prod_{j=1}^mL_v(s,\sigma_j)=\prod_{j=1}^mL_v(s+\nu_j,\delta_j).
\end{equation}
It therefore suffices to describe $L_v(s,\delta)$ for $\delta$ in $\mathscr{E}^2(\GL_1(\C)^1)$, $\mathscr{E}^2(\GL_1(\R)^1)$, or $\mathscr{E}^2(\GL_2(\R)^1)$. (Of course $\GL_1(\C)^1$ is just the circle group, and $\GL_1(\R)^1=\{\pm 1\}$.)

We have $\mathscr{E}^2(\GL_1(\C)^1)=\{\chi_k: k\in\Z\}$, where $\chi_k$ is the unitary character $z\mapsto (z/|z|)^k$; in this case, $L_v(s,\chi_k)=\Gamma_\C(s+|k|/2)$. Moveover, $\mathscr{E}^2(\GL_1(\R)^1)=\{{\rm sgn}^\epsilon: \epsilon=0,1\}$ and $L_v(s,{\rm sgn}^\epsilon)=\Gamma_\R(s+\epsilon)$. Finally, $\mathscr{E}^2(\GL_2(\R)^1)=\{D_k: k\geqslant 2\}$, where $D_k$ denotes the weight $k$ discrete series representation, and we have $L_v(s,D_k)=\Gamma_\C(s+(k-1)/2)$. 

We now insert these expressions into \eqref{eq:defn:L-factor-prod}. To do so in a uniform way, we include $v$ subscripts on various parameters. In particular, $m_v$ denotes the number of blocks in the Levi $M_v$, necessarily equal to $n$ when $v$ is complex. Furthermore, for $v$ real, we denote by $a_v$ the number of $\GL_1$ blocks and by $b_v$ the number of $\GL_2$ blocks of $M_v$, so that $m_v=a_v+b_v$ and $n=a_v+2b_v$. Now set $\delta_{vj}=|k_{vj}|/2$, for indices $j=1,\ldots ,n$ and $v$ complex, and
\[
\delta_{vj}=\begin{cases}
\epsilon_{vj},&  j=1,\ldots ,a_v;\\
(k_{vj}-1)/2,& j=a_v+1,\ldots ,m_v;\\
1+(k_{vj}-1)/2,& j=m_v+1,\ldots ,n,
\end{cases}
\quad\textrm{and}\quad \nu_{v(j+b_v)}=\nu_{vj}, \; j=a_v+1, \ldots, m_v,
\]
for $v$ real. Finally, for $1\leqslant j\leqslant n$, we put $\mu_{vj}=\nu_{vj}+\delta_{vj}\in\C$ for any $v\mid\infty$.

Then it follows from the above computations, as well as the duplication formula for the Gamma function, in the form $\Gamma_\C(s)=\Gamma_\R(s)\Gamma_\R(s+1)$, that
\[
L_v(s,\pi_{\delta,\nu})=\prod_{j=1}^n\Gamma_v (s+\mu_{vj}) 
\]
for any $v\mid\infty$. Iwaniec and Sarnak \cite{IwaniecSarnak2000} then define the conductor $q(\pi_v)$ of $\pi_v$ as 
\begin{equation}\label{IS-cond-first}
q(\pi_v)=\prod_{j=1}^n(1+|\mu_{vj}|)^{d_v},
\end{equation}
where $d_v=[F_v:\R]$.

\subsection{Non-archimedean local integrals}\label{sec:local-measure}

In this section we examine the local conductor zeta functions \eqref{def:local-cond-zeta}. In the definition of $Z_v(s)$, the complex parameter $s$ has large enough real part to ensure absolute convergence. For a finite place $v$ and an ideal $\q=\p_v^d$ let
\[
\mathfrak{M}_v(\q)=\int_{\q(\pi_v)=\q}{\rm d}\widehat{\mu}_v^{\rm pl}(\pi_v)
\]
be the Plancherel measure of those tempered $\pi_v$ with $\q(\pi_v)=\q$. For ${\rm Re}(s)$ large enough, we have
\[
Z_v(s)=\sum_{r\geqslant 0}\mathfrak{M}(\p_v^r)q_v^{-rs}.
\]
 
\begin{lemma}\label{local-count}
We have
\[
\mathfrak{M}_v(\q)=\sum_{\dd\mid\q}\lambda_{n+1}(\dd)\norm (\q/\dd)^n\qquad\text{ and }\qquad
Z_v(s)=\frac{\zeta_v(s-n)}{\zeta_v(s)^{n+1}}.
\]
In particular, $Z_v(n+1)=\zeta_v(1)/\zeta_v(n+1)^{n+1}$.
\end{lemma}

\begin{proof}
Applying Plancherel inversion to the idempotent $\varepsilon_{K_{1,v}(\q)}$ we obtain
\[
\frac{1}{\mu_{G,v}(K_{1,v}(\q))}=\int_{\Pi(G_v)}{\rm tr} \left(\pi_v (\varepsilon_{K_{1,v}(\q(\pi_v)\p_v^r)})\right)\,{\rm d}\widehat{\mu}_v^{\rm pl}(\pi_v)=\int_{\Pi(G_v)}\dim V_{\pi_v}^{K_{1,v}(\q)}\,{\rm d}\widehat{\mu}_v^{\rm pl}(\pi_v).
\]
From \eqref{eq:Euler-phi}, the left-hand side is $[K_{1,v}(\q):\bfk_v]/\mu_{G,v}(\bfk_v)=[K_{1,v}(\q):\bfk_v]=\varphi_n(\q)$. Thus, from \eqref{idem-trace}, we get
\[
\varphi_n(\q)=\int_{\q(\pi_v)\mid\q}d_n(\q/\q(\pi_v))\,{\rm d}\widehat{\mu}_v^{\rm pl}(\pi_v)=\sum_{\dd\mid\q}d_n(\q/\dd)\mathfrak{M}_v(\dd)=(d_n\star \mathfrak{M}_v)(\q).
\]
By M\"obius inversion (and associativity of Dirichlet convolution) this gives $\mathfrak{M}_v(\q)=(\lambda_n\star \varphi_n)(\q)=(\lambda_{n+1}\star p_n)(\q)=\sum_{\dd\mid\q}\lambda_{n+1}(\dd)\norm(\q/\dd)^n$. From
\[
\sum_{r\geqslant 0}\lambda_{n+1}(\p_v^r)q_v^{-rs}=\frac{1}{\zeta_v(s)^{n+1}},\qquad \sum_{r\geqslant 0}p_n(\p_v^r)q_v^{-rs}=\zeta_v(s-n),
\]
we obtain the value of $Z_v(s)$.
\end{proof}

\subsection{Archimedean local integral}\label{sec-arch-loc-int}

In this section, we shall work with the group $G^1=G_\infty^1$, viewed as a reductive group over $\R$. Wherever possible, we shall drop the subscript $\infty$ from the notation. So, for example, $\widehat{\omega}^{\rm pl}=\widehat{\omega}_\infty^{{\rm pl}}$ and $\pi=\pi_\infty$.

The work of Harish-Chandra \cite[Theorem 27.3]{Harish-Chandra1976} allows us to explicitly describe the Plancherel measure $\widehat{\omega}^{\rm pl}$ on $\Pi(G^1)$. Namely, for every $\underline{\delta}\in\mathcal{D}$ with standard representative $(M,\delta)$, Harish-Chandra defines constants $C_M>0$, depending only on the class of $M$, and a function $\mu_M^G(\delta,\nu)$, such that for every $h\in L^1(\widehat{\omega}^{\rm pl})$ we have
\begin{equation}\label{Pl-decomp}
\int_{\Pi(G^1)}h(\pi)\, {\rm d}\widehat{\omega}^{\rm pl}(\pi)= \sum_{\underline{\delta}\in\mathcal{D}}C_M\deg(\delta)\int_{i\mathfrak{h}_M^*} h(\pi_{\delta,\nu})\mu_M^G(\delta,\nu)\, {\rm d}\nu,
\end{equation}
where ${\rm deg}(\delta)$ is the formal degree of $\delta$. A standard reference in the setting of general groups is \cite[Theorem 13.4.1]{Wallach1992}. The density function $\mu_M^G(\delta,\nu)$ is a normalizing factor for an intertwining map \cite[Theorem 10.5.7]{Wallach1992}. In the specific context of $\GL_n$, more approachable texts are available, such as that of Knapp and Stein \cite[(6.3)]{KnappStein2015} over $\R$ and the classical work of Gelfand and Naimark \cite[p.\ 159]{GelfandNaimark1957} over $\C$.

Let $\sigma=\delta \otimes e^\nu$, where $\delta\in\mathscr{E}^2(M^1)$ and $\nu\in i\mathfrak{h}_M^*$ as above. We write $M=\prod_{v\in S_\infty}M_v$ and $\sigma=\prod_{v\in S_\infty}\sigma_v$, where $\sigma_v=\delta_v\otimes e^{\nu_v}$. We may then factorize $\sigma_v=\sigma_{v1}\otimes\cdots\otimes\sigma_{vm_v}$ according to the block decomposition of $M_v$. We have a factorization of the form
\[
\mu_M^G(\delta,\nu)=\prod_{v\in S_\infty}\prod_{1\leqslant i<j\leqslant m_v}\mu_{\GL_{n_{vi}}\times\GL_{n_{vj}}}^{\GL_{n_{vi}+n_{vj}}}(\sigma_{vi}\otimes\sigma_{vj}).
\]
Applying standard identities for the Gamma function to the presentation in \cite{KnappStein2015,GelfandNaimark1957}, we may describe the latter factors can be described in terms of archimedean Rankin--Selberg local factors. The latter are described, in all cases, in \cite[Appendix A.3]{RudnickSarnak1996}. We obtain
\begin{equation}\label{eq:gen-defn-muMG}
\mu_{\GL_{n_{vi}}\times\GL_{n_{vj}}}^{\GL_{n_{vi}+n_{vj}}}(\sigma_{vi}\otimes\sigma_{vj})=c\bigg|\frac{L_v(1,\sigma_{vi}\times\widetilde{\sigma}_{vj})}{L_v(0,\sigma_{vi}\times\widetilde{\sigma}_{vj})}\bigg|^2,
\end{equation}
where $c\neq 0$ is a constant depending only on measure normalizations. The above formula is analogous to that of the $p$-adic setting \cite[Theorem 6.1]{Shahidi1984}, and moreover $\mu_M^G(\delta,\nu)$ reduces to $|c(\nu)c(\rho)^{-1}|^{-2}$ in the case of $M=T_{0,\infty}$ and $\delta$ trivial, where
\begin{equation}\label{eq:H-C-c-function}
c(\nu)=\prod_{\alpha\in\Phi^{G,+}}\frac{\Gamma_\R(\langle\alpha,\nu\rangle)}{\Gamma_\R(\langle\alpha,\nu\rangle+1)}
\end{equation}
is the Harish-Chandra $c$-function \cite[(8.3)]{MatzTemplier2021}.

\begin{remark}\label{rem:use-of-W(M,L)}
Recall the subset $W(M,L)_\delta\subset W(A_M)_\delta$ from \S\ref{sec:herm}. The significance of this subset to us is that the Plancherel density function $\mu_M^G(\delta,\nu)$ defined above vanishes to order $2|W(M,L)_\delta|$ at generic $\nu\in i\mathfrak{h}_L^{\ast}$. In particular, the quantity $R^{-\dim\mathfrak{h}_M-2|W(M,L)_\delta|}$ featuring in several of our estimates is commensurable with the $\mu_M^G(\delta,\nu)$-measure of the ball of radius $1/R$, for $R\gg 1$.
\end{remark}

We would now like to define a majorizer of the (normalized) density function $\deg(\delta)\mu_M^G(\delta,\nu)$. Recall the notation from \S\ref{arch-cond}. 

\begin{defn}\label{BetaMajorizer}
Let $\nu_v\in i\mathfrak{h}_{M_v}^*$. When $v$ is complex (in which case $M_v=T_{0,v}$) we let
\[
\beta_0^{G_v}(\delta_v,\nu_v)=\prod_{1\leqslant i<j\leqslant n}\big(1+\big|k_{vi}-k_{vj}+\nu_{vi}-\nu_{vj}\big|\big)^2,
\]
and when $v$ is real we let
\begin{align*}
\beta_{M_v}^{G_v}(\delta_v,\nu_v)=\prod_{a_v< i\leqslant m_v}k_{vi}&\prod_{1\leqslant i<j\leqslant a_v} \big(1+\big|\nu_{vi}-\nu_{vj}\big|\big)\prod_{\substack{1\leqslant i\leqslant a_v\\ a_v< j\leqslant m_v}} \big(1+\big|k_{vj}+\nu_{vi}-\nu_{vj}\big|\big)^2\\
&\prod_{a_v< i<j\leqslant m_v} \big(1+\big|k_{vi}-k_{vj}+\nu_{vi}-\nu_{vj}\big|\big)^2\big(1+\big|k_{vi}+k_{vj}+\nu_{vi}-\nu_{vj}\big|\big)^2.
\end{align*}
Next, set $\beta_M^G(\delta,\nu)=\prod_{v\mid\infty}\beta_{M_v}^{G_v}(\delta_v,\nu_v)$.
\end{defn}

\begin{remark}
We have restricted the definition of $\beta_M^G(\delta,\nu)$ to $\nu\in i\mathfrak{h}_M^*$, since that is the support of the Plancherel measure. Using this assumption, it also follows that $\beta_M^G(\delta,\nu)$ is unchanged when $\nu$ is replaced by $w.\nu$, where $w\in W(A_M)$.
\end{remark}

\begin{lemma}\label{prop:density}${\,}$
\begin{enumerate}
\item\label{beta1}
We have $\deg(\delta)\mu_M^G(\delta,\nu)\ll \beta_M^G(\delta,\nu)$.
\item\label{beta2} For every $\lambda,\nu\in i\mathfrak{h}_M^{\ast}$,
\[ \beta_M^G(\delta,\lambda)\ll(1+\|\lambda-\nu\|)^{d_M}\beta_M^G(\delta,\nu),
\]
where
 \begin{equation}\label{def-dM}
d_M=\sum_{v\mid\infty}\sum_{1\leqslant i<j\leqslant m_v}d_vn_{vi}n_{vj}.
\end{equation}
\item \label{beta3} Let $L\in\mathcal{L}_{\infty}(\delta)$. Then, for every $\nu\in i\mathfrak{h}_L^{\ast}$ and every $\lambda\in i\mathfrak{h}_M^{\ast}$,
\[ \deg(\delta)\mu_M^G(\delta,\lambda)\ll\|\lambda-\nu\|^{2|W(M,L)_\delta|}(1+\|\lambda-\nu\|)^{d_M-|W(M,L)_\delta|}\beta_M^G(\delta,\nu). \]
\end{enumerate}
\end{lemma}

\begin{proof}
We must first explicate the Rankin--Selberg $L$-factors involved in the definition \eqref{eq:gen-defn-muMG} of the density function $\mu_M^G(\delta,\nu)$. It suffices to do so in the following cases (see \cite[p.319]{Tate1967} and \cite[Appendix A.3]{RudnickSarnak1996}):
\begin{align*}
L_\C(s, \chi_{k_1} e^{\nu_1}\times \chi_{k_2} e^{\nu_1})&=L_\C(s+\nu_1+\nu_2, \chi_{k_1+k_2})=\Gamma_\C(s+\nu_1+\nu_2+|k_1+k_2|/2),\\
L_\R(s,{\rm sgn}^{\epsilon_1}e^{\nu_1}\times {\rm sgn}^{\epsilon_2}e^{\nu_2})&=L_\R(s+\nu_1+\nu_2,{\rm sgn}^{\epsilon})=\Gamma_\R(s+\nu_1+\nu_2+\epsilon),\quad (\epsilon\equiv\epsilon_1+\epsilon_2 \;{\rm mod}\; 2)\\
L_\R(s,D_k\otimes e^{\nu_1}\times {\rm sgn}^{\epsilon}e^{\nu_2})&=L_\R(s+\nu_1+\nu_2,D_k\times {\rm sgn}^{\epsilon})=\Gamma_\C(s+\nu_1+\nu_2+(k-1)/2),\\
L_\R(s,D_{k_1}\otimes e^{\nu_1}\times D_{k_2}\otimes e^{\nu_2})&=\Gamma_\C(s+\nu_1+\nu_2+|k_1-k_2|/2)\Gamma_\C(s+\nu_1+\nu_2+(k_1+k_2)/2).
\end{align*}
We shall apply these formulae when the second member is a contragredient representation. In light of this, we note that the contragredient representations of $\chi_ke^{\nu}$, ${\rm sgn}^\epsilon e^{\nu}$, and $D_k\otimes e^\nu$ are $\chi_{-k}e^{-\nu}$, ${\rm sgn}^\epsilon e^{-\nu}$, and $D_k\otimes e^{-\nu}$, respectively. From this description of the local Rankin--Selberg $L$-factors, as well as Stirling's formula, in the form
\begin{equation}\label{eq:stirlings-formula}
\Gamma_v(1+s)/\Gamma_v(s)\ll \min(|s|,(1+|s|)^{d_v/2}),
\end{equation}
we deduce that $\mu_M^G(\delta,\nu)$ is majorized by the product over $i,j$ in $\beta_M^G(\delta,\nu)$. Taking $\deg(\delta)$ into account yields the first statement.

The second statement follows directly from Definition \ref{BetaMajorizer}, applied to $\lambda=(\lambda-\nu)+\nu$.

For the third statement, we write again $\lambda=(\lambda-\nu)+\nu$ and use the first bound in \eqref{eq:stirlings-formula} to estimate those factors of $\mu_M^G(\delta,\lambda)$ that vanish along $i\mathfrak{h}_L^{\ast}$ and the second bound for the remaining factors. This implies \eqref{beta3}, in view of Definition~\ref{BetaMajorizer}.
\end{proof}

We define a measure $\beta(\pi)\,{\rm d}\pi$ on $\Pi(G^1)$ by putting, for any $h\in L^1(\widehat{\mu}^{\rm pl})$, 
\begin{equation}\label{Pl-maj}
\int_{\Pi(G^1)}h(\pi) \beta(\pi)\, {\rm d}\pi=\sum_{\underline{\delta}\in\mathcal{D}}\; \int_{i\mathfrak{h}_M^*} h(\pi_{\delta,\nu})\beta_M^G(\delta,\nu) \,{\rm d}\nu.
\end{equation}

\begin{lemma}
\label{Plemma}
For $Q\geqslant 1$ we have $\displaystyle\int_{Q\leqslant q(\pi)\leqslant 2Q}\beta(\pi)\, {\rm d}\pi\ll Q^{n-1/d}$.
\end{lemma}

\begin{proof}
We begin by estimating the $\underline{\delta}$ sum in \eqref{Pl-maj} by an integral, as follows. Let $\underline{\delta}$ have standard representative $(M,\delta)$, with $\delta=\otimes_{v\mid\infty}\delta_v$ on $M=\prod_{v\mid\infty}M_v$. We define 
\begin{equation}\label{eq:def-HM}
\mathcal{H}_M=\left\{\mu\in \prod_{v\mid \infty}\C^n: 
\; \begin{aligned}
&\sum_{v\mid\infty} d_v\sum_{j=1}^n{\rm Im}\,\mu_{vj}=0\\
  & {\rm Re}(\mu_{vj})=0 \textrm{ for } v=\R,\; j=1,\ldots ,a_v\\
    & \mu_{vj}=\mu_{v(j+b_v)} \textrm{ for } v=\R,\; j=a_v+1,\ldots ,m_v
  \end{aligned}
 \right\}.
\end{equation}
The space $\mathcal{H}_M$ captures the values of the parameters $\mu_{vj}$ involved in the analytic conductor \eqref{IS-cond-first}, up to real shifts of bounded size. Finally write $\mathcal{H}_M(Q)=\big\{z\in \mathcal{H}_M: q(\mu)\sim Q\big\}$, where $q(\mu)=\prod_{v\mid\infty}\prod_{j=1}^n(1+|\mu_{vj}|)^{d_v}$. Here, and throughout the proof, we shall use the symbol $q\sim Q$ to mean $c_1 Q\leqslant q\leqslant c_2 Q$, where $0< c_1<c_2$ are constants that can change from line to line.

We now define a function $\beta_M^G$ on $\mathcal{H}_M$ that will capture the Plancherel majorizer of Definition \ref{BetaMajorizer}. Namely, we put $\beta_M^G(\mu)=\prod_{v\mid\infty}\beta_{M_v}^{G_v}(\mu_v)$, with
\[
\beta_{M_v}^{G_v}(\mu_v)=\prod_{\substack{a_v< i\leqslant m_v\\ v \textrm{ real}}}|{\rm Re}(\mu_{vi})|\prod_{\substack{1\leqslant i\leqslant a_v\\ i< j\leqslant n}} \big(1+\big|\mu_{vi}-\mu_{vj}\big|\big)^{d_v}\prod_{a_v< i<j\leqslant m_v} \big(1+\big|\mu_{vi}-\mu_{vj}\big|\big)^2\big(1+\big|\mu_{vi}+\overline{\mu}_{vj}\big|\big)^2.
\]
Observe that this reduces to $\prod_{1\leqslant i<j\leqslant n} \big(1+\big|\mu_{vi}-\mu_{vj}\big|\big)^{d_v}$ when $M_v=T_{0,v}$. Then, by construction, we have the upper bound
\[
\int_{q(\pi)\sim Q}\beta(\pi)\,{\rm d}\pi\ll \max_M\int_{\mathcal{H}_M(Q)}\beta_M^G(\mu)\, {\rm d}\mu,
\]
the max being taken over cuspidal Levi subgroups of $G_\infty^1$.

Fixing $M$, we now dyadically decompose the integral over ${\mathcal{H}_M(Q)}$. Let $\mathcal{R}$ denote the collection of all tuples $R=\{R_{vj}\}$ of dyadic integers, indexed by archimedean places $v$ and $j=1,\ldots ,n$. For $R\in\mathcal{R}$ we let $\mathcal{H}_{M,R}$ denote the intersection of $\mathcal{H}_M$ with $\{\mu\in\prod_{v\mid\infty}\C^n:  (1+|\mu_{vj}|)\sim R_{vj}\;\forall\; v\mid\infty\}$. If $\mathcal{R}(Q)=\{R\in\mathcal{R}:\prod_{v,j}R_{vj}^{d_v}\sim Q\}$, then $\mathcal{H}_M(Q)$ is contained in the union of the $\mathcal{H}_{M,R}$, as $R$ runs over $\mathcal{R}(Q)$. We deduce that
\begin{equation}\label{eq:Pl-dyadic-decomp}
\int_{\mathcal{H}_M(Q)}\beta_M^G(\mu)\, {\rm d}\mu\ll \sum_{R\in \mathcal{R}(Q)}\max_{\mathcal{H}_{M,R}}\beta_M^G(\mu)\vol \mathcal{H}_{M,R}. 
\end{equation}

We now bound the Plancherel majorizer $\beta_M^G(\mu)$ by a factorizable expression in the coordinates $\mu_{vj}$. Fix an $R\in\mathcal{R}$. For each $v\mid\infty$ we re-index the $R_{vj}$, if necessary, so that $R_{vn}\leqslant \cdots \leqslant R_{v1}$. We claim that, for $\mu\in\mathcal{H}_{M,R}$, 
\begin{equation}\label{eq:from-b-to-mu}
\beta_M^G(\mu)\ll \prod_{v\mid\infty}\prod_{j=1}^n \mathscr{R}_{vj}^{n-j} \qquad (\mathscr{R}_{vj}=R_{vj}^{d_v}).
\end{equation}
To see this, for indices $1\leqslant i,j\leqslant n$, let $M_{vij}=\max_{\ell\in\{i,j\}} (1+|\mu_{v\ell}|)$. Then, for $v$ complex, the claim follows by inserting $1+\big|\mu_{vi}-\mu_{vj}\big|\ll M_{vij}$ and taking into account the ordering of the $R_{vj}$. For real places $v$, a similar reasoning shows that the product over $1\leqslant i\leqslant a_v, i<j\leqslant n$ in the definition of $\beta_{M_v}^{G_v}(\mu_v)$ is at most $\prod_{1\leqslant i\leqslant a_v, i<j\leqslant n}M_{vij}$. For the remaining factors, we first observe that
\[
\prod_{a_v< i<j\leqslant m_v} \big(1+\big|\mu_{vi}-\mu_{vj}\big|\big)^2\big(1+\big|\mu_{vi}+\overline{\mu}_{vj}\big|\big)^2\ll \prod_{a_v< i<j\leqslant m_v} M_{vij}^4,
\]
and the latter, in view of the symmetry $M_{vij}=M_{vji}$ and the third condition in \eqref{eq:def-HM}, is
\[
\prod_{a_v< i<j\leqslant m_v} M_{vij}M_{vj(i+b_v)} M_{vi(j+b_v)}M_{v(i+b_v)(j+b_v)}=\prod_{\substack{a_v<i<j\leqslant n\\ j\neq i+b_v}}M_{vij}.
\]
We furthermore bound $\prod_{a_v<i\leqslant m_v}|{\rm Re}(\mu_{vi})|$ by $\prod_{a_v<i\leqslant m_v} (1+|\mu_{vi}|)=\prod_{a_v<i\leqslant m_v}M_{vi(i+b_v)}$. Putting all everything together we find $\beta_{M_v}^{G_v}(\mu_v)\ll \prod_{1\leqslant i< j\leqslant n} M_{vij}$ in the real case, and we again deduce \eqref{eq:from-b-to-mu} using the ordering of the $R_{vj}$.

To estimate the volume factor in \eqref{eq:Pl-dyadic-decomp}, let $v_0\mid\infty$ satisfy $R_{v_01}=\max_vR_{v1}$ and write $R_0=R_{v_01}$. From the first condition in \eqref{eq:def-HM}, we have $\vol(\mathcal{H}_{M,R})\ll R_0^{-1}\prod_{v,j}\mathscr{R}_{vj}$. Together, the above estimates yield
\begin{equation}\label{eq:R-upper-bd}
\max_{\mathcal{H}_{M,R}}\beta_M^G(\mu)\vol \mathcal{H}_{M,R}\ll R_0^{-1}\prod_{v\mid\infty}\prod_{j=1}^n \mathscr{R}_{vj}^{n-j+1}.
\end{equation}

We first treat the case when $d=1$, so that there is a unique archimedean place, which is real. Then \eqref{eq:R-upper-bd} yields an upper bound of $R_1^{n-1}\prod_{j\geqslant 2} R_j^{n-j+1}$, upon dropping the $v$ from the notation. Since $\prod_j R_j\sim Q$, the latter product is at most $Q^{n-1}\prod_{j\geqslant 2} R_j^{2-j}$. The sum over the dyadic integers $R_j$, for $j\geqslant 3$, is a convergent uniformly bounded geometric series. There are only ${\rm O}(1)$ remaining choices of $R_1, R_2$; indeed, from $\mathcal{H}_{M,R}\neq\emptyset$ we have $R_1\sim R_2$, and there is exactly one dyadic integer lying in any given dyadic interval. This produces the desired bound of $Q^{n-1}$ in this case.

Henceforth, we assume that $d\geqslant 2$. Using $\prod_{v,j}\mathscr{R}_{vj}\sim Q$, the right-hand side of \eqref{eq:R-upper-bd} is of size
\[
Q^{n-1}R_0^{-1}\prod_{v\mid\infty}\prod_{j\geqslant 1} \mathscr{R}_{vj}^{2-j}=Q^{n-1}\prod_v\prod_{j\geqslant 2}\mathscr{R}_{vj}^{2-j}R_0^{-1}\prod_{v\mid\infty}\mathscr{R}_{v1}.
\]
It remains to execute the sum over $R\in\mathcal{R}(Q)$. We begin by noting that for $R\in \mathcal{R}(Q)$,
\[
R_0^{-1}\prod_{v\mid\infty}\mathscr{R}_{v1}=\prod_{v\mid\infty} \mathscr{R}_{v1}^{\frac{d-1}{d}}\prod_{v\mid\infty, v\neq v_0}(R_{v1}/R_0)^{d_v/d}\ll \big(Q/\prod_{v\mid\infty}\prod_{j\geqslant 2}\mathscr{R}_{vj}\big)^{\frac{d-1}{d}}\prod_{v\mid\infty, v\neq v_0}(R_{v1}/R_0)^{d_v/d}.
\]
Thus, grouping the sum on $R_{v1}$ according to $(\min_{v\neq v_0}R_{v1}/R_0)\sim 1/N$, we obtain 
\[
\sum_{R_{v1}\in 2^\N}R_0^{-1}\prod_{v\mid\infty}\mathscr{R}_{v1}\ll \big(Q/\prod_{v\mid\infty}\prod_{j\geqslant 2}\mathscr{R}_{vj}\big)^{\frac{d-1}{d}}\sum_{N\in 2^\N}N^{-1/d}\log^{d-1}(1+N)\ll \big(Q/\prod_{v\mid\infty}\prod_{j\geqslant 2}\mathscr{R}_{vj}\big)^{\frac{d-1}{d}}.
\]
Inserting this into the remaining sum we get
\[
\sum_{R\in \mathcal{R}(Q)}\max_{\mathcal{H}_{M,R}}\beta_M^G(\mu)\vol \mathcal{H}_{M,R}\ll Q^{n-1/d}\prod_{v\mid\infty}\prod_{j\geqslant 2}\sum_{R_{vj}\in 2^\N}\mathscr{R}_{vj}^{1-j+1/d}.
\]
Since $d\geqslant 2$, the exponents in each factor are all strictly negative, in which case the geometric series are absolutely bounded. We finally obtain ${\rm O}(Q^{n-1/d})$ in all cases.
\end{proof}

\begin{cor}\label{Plemma2}
The archimedean conductor zeta function
\[
Z_{\infty}(s)=\int_{\Pi(G^1)}q(\pi)^{-s}\,\textnormal{d}\widehat{\omega}^{\pl}_{\infty}(\pi)
\]
converges absolutely for $s\in\mathbb{C}$ with ${\rm Re}\, s>n-1/d$.
\end{cor}

\begin{proof}
This follows from Lemma \ref{prop:density}\,\eqref{beta1} and Lemma \ref{Plemma}.
\end{proof}

\begin{remark}
Despite the ``spikes'' introduced by the product condition $q(\pi_{\delta,\nu})\leqslant X$, the asymptotics of the Plancherel measure of the sets $\{ \nu\in i\mathfrak{h}_M^*: q(\pi_{\delta,\nu})\leqslant X\}$ as $X\to\infty$ feature pure power growth without logarithmic factors. This is due to the following two facts: first, the Plancherel density increases into the spikes; and, second, these spikes are somewhat moderated by the trace-zero condition. To visualize this latter feature, we include the following graphics.
\begin{center}
\begin{figure}[!htbp]
\minipage{0.2\textwidth}
\captionsetup{width=.9\linewidth}
  \includegraphics[width=\linewidth]{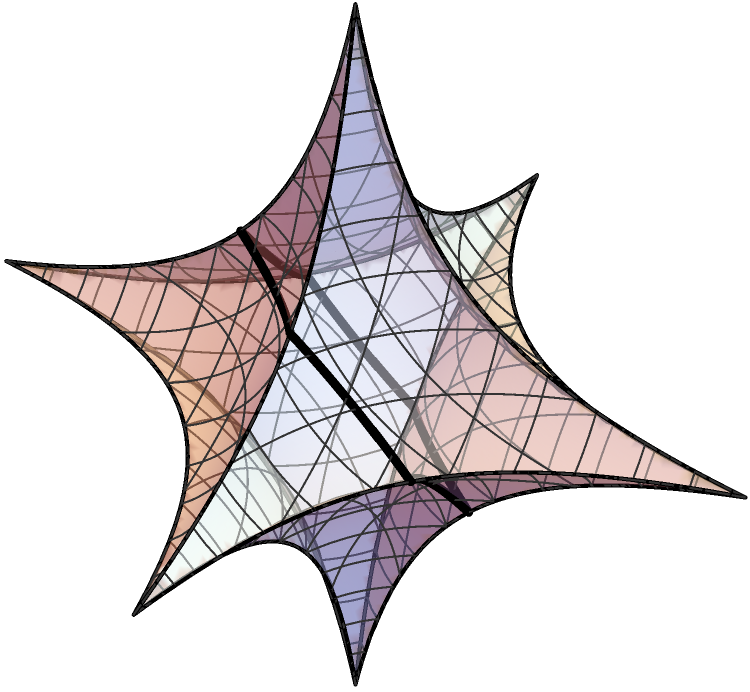}
  \caption{}
\endminipage\hfill
\minipage{0.2\textwidth}
\captionsetup{width=.9\linewidth}
  \includegraphics[width=\linewidth]{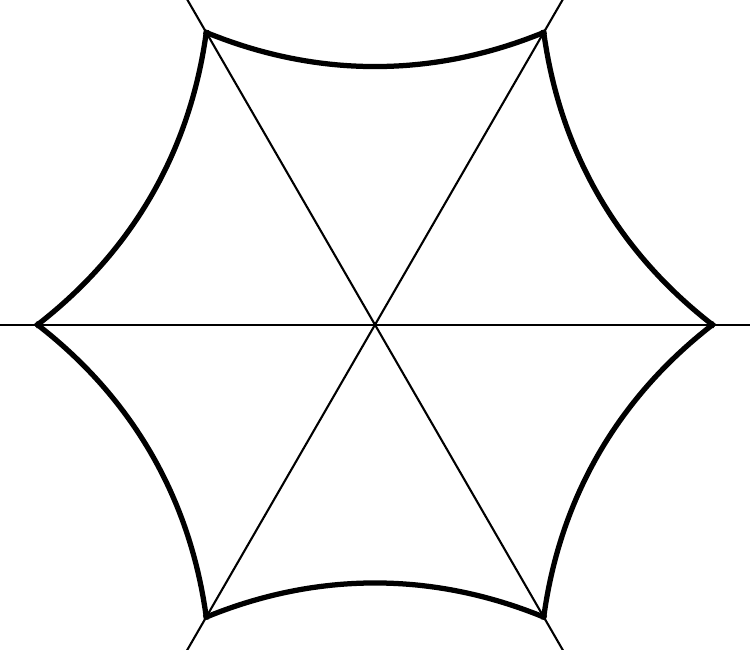}
  \caption{}
\endminipage\hfill
\minipage{0.2\textwidth}
\captionsetup{width=.9\linewidth}
  \includegraphics[width=\linewidth]{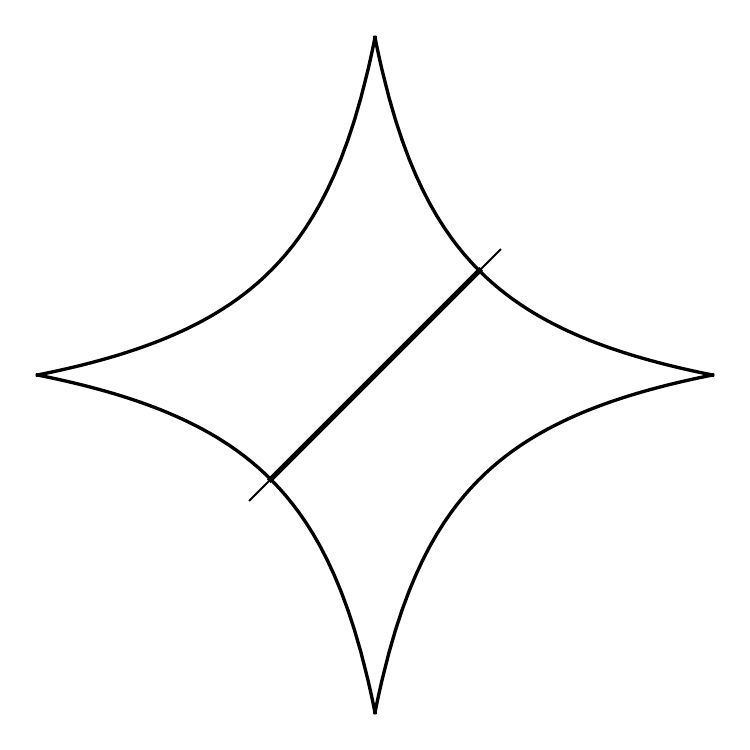}
  \caption{}
\endminipage
\end{figure}
\end{center}

\noindent In Figure 1, the region $\{(1+|x|)(1+|y|)(1+|z|)\leqslant X\}$ is drawn in $\R^3$. The spikes extend as far as $\asymp X$. The intersection with $x+y+z=0$ is indicated in bold and reproduced in the plane in Figure 2. The spikes extend as far as $\asymp X^{1/2}$. 

In Figure 3, the set $\{(1+|x|)(1+|y|)\leqslant X\}$ is drawn in $\R^2$ with spikes as far as $\asymp X$ and volume $\asymp X\log X$. The intersection with $x+y=0$ is in bold. This produces a segment of length $\asymp X^{1/2}$.
\end{remark}

\subsection{Global analytic conductor}\label{sec:global-cond}
In this section, we define the analytic condutor $Q(\pi)$ of everywhere locally generic $\pi$ in the adelic unitary dual $\Pi(\GL_n(\A_F)^1)$. This includes, on one hand, cuspidal automorphic representations $\pi$ of $\GL_n(\A_F)^1$, since their local components are generic \cite{JacquetPiatetskii-ShapiroShalika1983}, and on the other, any $\pi\in\Pi(\GL_n(\A_F)^1)$ appearing in the support of $\widehat{\omega}^{\rm pl}$, since such $\pi$ are everywhere tempered, and therefore generic.

Let $\pi=\otimes_v\pi_v \in\Pi(\GL_n(\A_F)^1)$ be everywhere locally generic. Let $\q(\pi)=\prod_{v<\infty} (\q_v(\pi_v)\cap \mathcal{O}_F)$ with each $\q_v(\pi_v)$ being defined in \S\ref{sec:def-cond}. Then $\q(\pi)$ is an  $\mathcal{O}_F$-ideal called the \textit{arithmetic conductor ideal} of $\pi$. Its absolute norm $q(\pi)=\norm(\q(\pi))\in\N$ factors as $\prod_{v<\infty} q(\pi_v)$. As almost all local components of $\pi$ are unramified, these products make sense as finite products. 

The \textit{analytic conductor} of $\pi$ is defined as $Q(\pi)=\prod_v q(\pi_v)$. Writing $q(\pi_\infty)=\prod_{v\mid\infty}q(\pi_v)$ for the archimedean conductor, it follows that $Q(\pi)=q(\pi)q(\pi_\infty)$.

\begin{remark}
For an integral ideal $\q$ in $\mathcal{O}_F$, whose completion in $\A_f$ factorizes as $\prod\p_v^{r_v}$, we agree to write $K_1(\q)$ for the open compact subgroup of $\GL_n(\A_f)$ given by $\prod_{v<\infty} K_{1,v}(\p_v^{r_v})$, with each local factor being defined in \S\ref{sec:def-cond}. Its index in $\bfk_f$ is the product over all finite places of the local indices described in \eqref{eq:Euler-phi}. Thus
$|\bfk_f/K_1(\q)|=\varphi_n(\q)$, with $\varphi_n$ as in the statement of Corollary \ref{MT-is-vol}.
\end{remark}

\subsection{Proof of Proposition \ref{MainTermLemma} and Corollary \ref{MT-is-vol}}\label{sec:computing-zeta-residue}
We deduce from Lemma \ref{local-count} that
\[
\int\limits_{\substack{\pi\in\Pi(\bm{G}(\A_F)^1)\\ Q(\pi)\leqslant Q}}{\rm d}\widehat{\mu}^{\rm pl}(\pi)=\sum_{\norm\q\leqslant Q}\prod_{\p_v^r\mid\mid\q}\mathfrak{M}(\p_v^r)\int\limits_{\substack{\pi_\infty\in\Pi(G_\infty^1)\\ q(\pi_\infty)\leqslant Q/\norm\q}}\mathrm{d}\widehat{\omega}_{\infty}^{\pl}(\pi_\infty)=\sum_{\norm\q\leqslant Q}w_n(\q)\int\limits_{\substack{\pi_\infty\in\Pi(G_\infty^1)\\ q(\pi_\infty)\leqslant Q/\norm\q}}\mathrm{d}\widehat{\omega}_{\infty}^{\pl}(\pi_\infty),
\]
where $w_n=p_n\star\lambda_{n+1}$. Let $W_n(X)=\sum_{\norm\q\leqslant X}w_n(\q)$. Exchanging the order of summation and integration,
\begin{equation}\label{adelic-vol}
\int\limits_{\substack{\pi\in\Pi(\bm{G}(\A_F)^1)\\ Q(\pi)\leqslant Q}}{\rm d}\widehat{\mu}^{\rm pl}(\pi)=\int_{\Pi(G_\infty^1)}W_n(Q/q(\pi_{\infty}))\,\mathrm{d}\widehat{\omega}_{\infty}^{\pl}(\pi_{\infty}).
\end{equation}
The statement of the proposition will follow from an asymptotic evaluation of $W_n(X)$.

Recall the classical estimate $\sum_{\norm\dd\leqslant X}1=\zeta_F^*(1)X+\text{O}(X^{1-2/(d+1)})$ on the ideal-counting function~\cite[Satz 210, p.~131]{Landau1918}. From this we deduce that given any $\sigma>-1$, $0<\theta\leqslant 2/(d+1)$, and $X>0$, we have
\begin{equation}
\label{SigmaIdealCountingAsymptotics}
\sum_{\norm\q\leqslant X}\norm\q^{\sigma}=\frac{\zeta_F^{\ast}(1)}{\sigma+1}X^{\sigma+1}+\text{O}_{\sigma,\theta}\big(X^{\sigma+1-\theta}\big),
\end{equation}
where for $X\geqslant 1$ we simply estimate $X^{\sigma+1-2/(d+1)}=\text{O}(X^{\sigma+1-\theta})$, and for $X<1$ the estimate \eqref{SigmaIdealCountingAsymptotics} holds vacuously. Using \eqref{SigmaIdealCountingAsymptotics}, we find that, for every $X>0$,
\[
\begin{aligned}
W_n(X)
=\mathop{\sum\sum}_{\norm(\dd\ee)\leqslant X}\lambda_{n+1}(\ee)\norm\dd^n&=\sum_{\norm\ee\leqslant X}\lambda_{n+1}(\ee)\bigg(\frac{\zeta_F^{\ast}(1)}{n+1}\bigg(\frac{X}{\norm\ee}\bigg)^{n+1}+\text{O}\bigg(\frac{X}{\norm\ee}\bigg)^{n+1-\theta}\bigg)\\
&=\frac{\zeta_F^{\ast}(1)}{n+1}X^{n+1}\sum_{\norm\ee\leqslant X}\frac{\lambda_{n+1}(\ee)}{\norm\ee^{n+1}}+\text{O}_F\bigg(X^{n+1-\theta}\sum_{\norm\ee\leqslant X}\frac{|\lambda_{n+1}(\ee)|}{\norm\ee^{n+1-\theta}}\bigg).
\end{aligned}
\]
From $|\lambda_{n+1}(\nn)|\leqslant d_{n+1}(\nn)\ll_\epsilon(\norm\nn)^{\epsilon}$ and the identity $\sum\lambda_{n+1}(\nn)\norm\nn^{-s}=\zeta_F(s)^{-n-1}$ we obtain
\begin{equation}\label{AsymptoticForW}
W_n(X)=\frac1{n+1}\frac{\zeta_F^{\ast}(1)}{\zeta_F(n+1)^{n+1}}X^{n+1}+\text{O}\big(X^{n+1-\theta}\big).
\end{equation}

Using \eqref{adelic-vol} and \eqref{AsymptoticForW} we see that
\begin{equation}\label{eq:adelic-int-w-zeta}
\int\limits_{\substack{\pi\in\Pi(\bm{G}(\A_F)^1)\\ Q(\pi)\leqslant Q}}{\rm d}\widehat{\mu}^{\rm pl}(\pi)=\frac1{n+1}\frac{\zeta_F^{\ast}(1)}{\zeta_F(n+1)^{n+1}}Z_\infty(n+1)Q^{\,n+1}
+\text{O}\big(Z_\infty(n+1-\theta)Q^{n+1-\theta}\big).
\end{equation}
In light of Corollary \ref{Plemma2}, both integrals converge. Inserting Lemma \ref{local-count}, we recover \eqref{eq:Prop-asymp-mu-version}.

Finally, to establish Corollary \ref{MT-is-vol}, we observe that $w_n=p_n\star\lambda_{n+1}=p_n\star\mu^{\star(n+1)}=(p_n\star\mu)\star\mu^{\star n}=\varphi_n\star\lambda_n$. Thus, arguing similarly to the beginning of this paragraph,
\[
D_F^{n^2/2}\Delta_F^*(1)\sum_{1\leqslant \norm\q\leqslant Q}\sum_{\dd\mid\q}\lambda_n(\q/\dd)\varphi_n(\dd)\int\limits_{\substack{\pi_\infty\in\Pi(G_\infty^1)\\ q(\pi_\infty)\leqslant Q/\norm\q}}\mathrm{d}\widehat{\omega}_{\infty}^{\pl}(\pi_\infty)=\int\limits_{\substack{\pi\in\Pi(\bm{G}(\A_F)^1)\\ Q(\pi)\leqslant Q}}{\rm d}\widehat{\omega}^{\rm pl}(\pi),
\]
to which we may then apply Proposition \ref{MainTermLemma}.\qed

\begin{remark}\label{rem:cond-zeta}
A natural way of determining the asymptotics of $V_{\mathfrak{F}}(Q)$ is through a Tauberian theorem, such as the quite general one to be found in \cite[Theorem A.1]{Chambert-LoirTschinkel2010}. In this way we would consider the associated zeta function
\[
\int_{\Pi(\bm{G}(\A_F)^1)}Q(\pi)^{-s}\,{\rm d}\widehat{\omega}^{\rm pl}(\pi),
\]
for $s\in\C$ of sufficiently large real part. Although we do not explicitly take this point of view in this section, our calculations suggest that the abscissa of convergence of the above integral is $s=n+1$, where it has a simple pole with residue ${\rm vol}(\tau_\mathfrak{F})$ (thus $a=n+1,b=1$, and $\Theta=\frac{1}{n+1}{\rm vol}(\tau_\mathfrak{F})$ in the notation of \cite[Theorem A.1]{Chambert-LoirTschinkel2010}).
\end{remark}

\part{Proof of Theorem \ref{main-implication}}\label{part1}

\section{Preparations}\label{sec:mainimplicationproof}

Our principal aim in Part \ref{part1} is to establish Theorem~\ref{main-implication}. For this we need to understand the behavior of $N(\q,\underline{\delta},\Omega)$ from \eqref{weightedsum} in all parameters. The bulk of the work will be to approximate $N(\q,\underline{\delta},P)$ for nice enough sets $P$ which lie in the tempered subspace $i\h_M^*$.

To formulate this precisely, we will first need to define what class of subsets $P$ we consider and associate with them appropriate boundary volumes. Once these concepts are in place we state, at the end of this short section, the desired asymptotic expression for $N(\q,\underline{\delta},P)$ in Proposition \ref{MainCountingResultGLn}.

\subsection{Nice sets and their boundaries}\label{sec:sets}

For $M\in\mathcal{L}_{\infty}$, we let $\mathfrak{B}_M$ be the $\sigma$-algebra of all Borel-measurable subsets of $i\mathfrak{h}_M^*$. For every $P\in\mathfrak{B}_M$ and $\rho>0$, let
\[
P^{\circ}(\rho)=\{\mu\in i\mathfrak{h}_M^*:B(\mu,\rho)\subseteq P\},\quad P^{\bullet}(\rho)=\{\mu\in i\mathfrak{h}_M^*:B(\mu,\rho)\cap P\neq\emptyset\},\quad \partial P(\rho)=P^{\bullet}(\rho)\setminus P^{\circ}(\rho),
\]
where $B(\mu,\rho)$ denotes the open ball of radius $\rho$ centered at $\mu$. Then, for every point $\mu\in\partial P(\rho)$, there are points $\nu_1,\nu_2\in B(\mu,\rho)$ with $\nu_1\in P$, $\nu_2\not\in P$, and hence by a continuity argument there is a point $\nu$ on the boundary $\partial P$ such that $|\mu-\nu|<\rho$; in other words,
\begin{equation}
\label{ContainmentPartialP}
\partial P(\rho)\subset\bigcup_{\nu\in\partial P}B(\nu,\rho).
\end{equation}

We record a few simple facts. For any bounded Borel set $P\in\mathfrak{B}_M$ and $\rho_2>\rho_1>0$, let
\[ P^{\bullet}(\rho_1,\rho_2)=P^{\bullet}(\rho_2)\setminus P^{\bullet}(\rho_1),\quad P^{\circ}(\rho_1,\rho_2)=P^{\circ}(\rho_1)\setminus P^{\circ}(\rho_2). \]

\begin{defn}
\label{ContainedInDefinition}
Let $X,Y\subseteq i\mathfrak{h}_M^*$ and $r>0$. We say that $X$ is \emph{$r$-contained in $Y$} if, for every $\mu\in X$, $B(\mu,r)\subseteq Y$.
\end{defn}

With these notions, we are ready for the following simple lemma.

\begin{lemma}
\label{SetContainments}
Let $P\in\mathfrak{B}_M$ be a bounded Borel set. Then:
\begin{enumerate}
\item \label{ClaimPartialP} For every $\rho,r>0$, the set $\partial P(\rho)$ is $r$-contained in $\partial P(\rho+r)$.
\item \label{ClaimBulletCirc} For every $\rho_2>\rho_1>r>0$, the set $P^{\bullet}(\rho_1,\rho_2)$ is $r$-contained in $P^{\bullet}(\rho_1-r,\rho_2+r)$, and the set $P^{\circ}(\rho_1,\rho_2)$ is $r$-contained in $P^{\circ}(\rho_1-r,\rho_2+r)$.
\end{enumerate}
\end{lemma}

\begin{proof}
These statements follow essentially by the triangle inequality. For example, for the first claim of \eqref{ClaimBulletCirc}, we need to prove that, if $\nu\in P^{\bullet}(\rho_1,\rho_2)$, then $B(\nu,r)\subset P^{\bullet}(\rho_1-r,\rho_2+r)$. Indeed, there is a $\nu_2\in B(\nu,\rho_2)\cap P$ while $B(\nu,\rho_1)\cap P=\emptyset$. Therefore, if $\nu_1\in B(\nu,r)$, then $\nu_2\in B(\nu_1,\rho_2+r)\cap P$ and so $\nu_1\in P^{\bullet}(\rho_2+r)$. On the other hand, we must have $B(\nu_1,\rho_1-r)\cap P=\emptyset$, for if $\nu_3\in B(\nu_1,\rho_1-r)\cap P$, then $\nu_3\in B(\nu,\rho_1)\cap P$, a contradiction; and so $\nu_1\not\in P^{\bullet}(\rho_1-r)$, as was to be shown. The other two claims are proved analogously.
\end{proof}

It will be convenient to consider the following family
\[
\mathscr{B}_M=\big\{P\in\mathfrak{B}_M:P\text{ is bounded and }\forall\rho>0,\,\,P^{\circ}(\rho),P^{\bullet}(\rho),\partial P(\rho)\in\mathfrak{B}_M\big\}.
\]
For example, every compact region with a piecewise smooth boundary clearly belongs to $\mathscr{B}_M$. Finally, if $\underline{\delta}\in\mathcal{D}$ has standard representative $(M,\delta)$, we define $\mathscr{B}(\underline{\delta})$ to be the family of all $W(A_M)_\delta$-invariant sets in $\mathscr{B}_M$.

\subsection{Archimedean Plancherel volumes}\label{sec:temp-bd}

Fix $M\in\mathcal{L}_\infty$ and $\delta\in\mathscr{E}^2(M^1)$. Let $P\in\mathscr{B}_M$. For $R>0$ and large enough $N\in\mathbb{N}$, we define
\begin{equation}\label{BoundaryTermsDefinitions1}
\partial\vol_{R,N}(\delta,P)=\sum_{\ell=1}^{\infty}\ell^{-N}\int_{\partial P(\ell/R)}\beta_M^G(\delta,\nu)\,\mathrm{d}\nu.
\end{equation}
We shall often drop the dependence on $N$ in the notation, writing simply $\partial\vol_R(\delta,P)$.

Intuitively, the sum in \eqref{BoundaryTermsDefinitions1} should be dominated by terms with small $\ell$, and indeed by estimating the Plancherel majorizer, the following lemma shows that $\partial\vol_R(\delta,P)$ is essentially the Plancherel majorizer volume of the $1/R$-thickened boundary of $P$.

\begin{lemma}
\label{boundary-term-reduction}
For all sufficiently large $N$,
\[ \partial\vol_{R,N}(\delta,P)\asymp\int_{\partial P(1/R)}\beta_M^G(\delta,\nu)\,\mathrm{d}\nu. \]
\end{lemma}

\begin{proof}
The lower bound follows trivially by positivity. For the upper bound, consider the double integral
\[ I=\int_{\partial P(1/R)} \int_{i\h_M^{\ast}} (1+R\|\lambda-\nu\|)^{-N} \beta_M^G(\delta,\lambda)\,\mathrm{d}\lambda\,\mathrm{d}\nu. \]
The lemma will follow from the follow two inequalities:
\begin{equation}\label{middle-I}
R^{-\dim\mathfrak{h}_M}\partial\vol_R(\delta,P)\ll I\ll R^{-\dim\mathfrak{h}_M}\int_{\partial P(1/R)}\beta_M^G(\delta,\nu)\,\mathrm{d}\nu.
\end{equation} 

For the lower bound in \eqref{middle-I}, we first restrict the inner integral over $\lambda$ to $\partial P(\ell/R)$ for any $\ell\in\N$. We switch the order of integration and let $\lambda\in\partial P(\ell/R)$ be arbitrary. By definition, $B(\lambda,\ell/R)$ contains a point in $P$ and a point in $P^c$ and thus by convexity and continuity two such points at a distance no more than $1/R$; taking $\nu_0$ to be the midpoint between these two, we have that $B(\nu_0,1/2R)\subseteq\partial P(1/R)$ and $\|\lambda-\nu\|<(\ell+\frac12)/R$ for every $\nu\in B(\nu_0,1/2R)$. We may therefore restrict the inner integral over $\nu$ to the ball $B(\nu_0,1/2R)$, minorize the integrand by $\ell^{-N}$, while picking up the volume of the ball $R^{-\dim\h_M}$, so that
\[
I\geqslant\int_{\partial P(\ell/R)} \beta_M^G(\delta,\lambda) \int_{\partial P(1/R)} (1+R\|\lambda-\nu\|)^{-N} \,\mathrm{d}\nu \,\mathrm{d}\lambda\gg R^{-\dim\mathfrak{h}_M}\ell^{-N}\int_{\partial P(\ell/R)} \beta_M^G(\delta,\lambda)  \,\mathrm{d}\lambda.
\]
Multiplying both sides by $\ell^{-2}$ and summing on $\ell$ establishes the first inequality in \eqref{middle-I}.

To prove the second inequality in \eqref{middle-I}, we use Part \eqref{beta2} of Lemma \ref{prop:density}. The inner integral in $I$ is
\begin{align*}
&\ll \int_{i\mathfrak{h}_M^{\ast}}\big(1+R\|\lambda-\nu\|)^{-N}\big(1+\|\lambda-\nu\|\big)^{d_M}\beta_M^G(\delta,\nu)\,\mathrm{d}\lambda\\
&\ll R^{-\dim\mathfrak{h}_M}\beta_M^G(\delta,\nu)\int_{i\mathfrak{h}_M^{\ast}}(1+\|\lambda\|)^{-N+d_M}\,\mathrm{d}\lambda
\ll R^{-\dim\mathfrak{h}_M}\beta_M^G(\delta,\nu),
\end{align*}
uniformly for every $N>d_M+\dim\mathfrak{h}_M$.
\end{proof}

Next, recalling the description of Hermitian archimedean dual and the notation $|W(M,L)_\delta|$ from \S\ref{sec:herm}, we define for every $L\in\mathcal{L}_{\infty}(\delta)$
\begin{equation}\label{BoundaryTermsDefinitions2}
\overline{\vol}_{R,L}(\delta,P)=R^{-{\rm codim}_{\h_M}(\h_L)-2|W(M,L)_\delta|}\int_{i\h_L^*}(1+d(\nu,P)\cdot R)^{-N}\beta_M^G(\delta,\nu)\,\mathrm{d}\nu.
\end{equation}
When $L=M$ we drop the dependence on $L$ from the notation and simply write $\overline{\vol}_R(\delta,P)=\overline{\vol}_{R,M}(\delta,P)$. Note that
\[
\overline{\vol}_R(\delta,P)=\int_P\beta_M^G(\delta,\nu)\,\mathrm{d}\nu+\mathrm{O}(\partial\vol_R(\delta,P));
\]
moreover, $\overline{\vol}_{R_1}(\delta,P)\asymp\overline{\vol}_{R_2}(\delta,P)$ if $R_1\asymp R_2$.

Finally we shall write
\begin{equation}\label{BoundaryTermsDefinitions3}
\vol_R^\star(\delta,P)=\sum_{\substack{L\in \mathcal{L}_{\infty}(\delta)\\ L\neq M}}\overline{\vol}_{R,L}(\delta,P).
\end{equation}

\begin{remarks}\label{N-dependence}
${}$
\begin{enumerate}
\item It will be plain from our arguments that a sufficiently large $N\in\mathbb{N}$ (in terms of $n$ and $F$) can be chosen once and for all to ensure convergence of sums we later encounter, and we normally suppress the dependence on $N$ in the notation; if we want to emphasize this dependence (for example, in \S\ref{SmoothToSharpProofSubsection}), we shall write $\partial\vol_{R,N}(\delta,P)$ or (when $L=M$) $\overline{\vol}_{R,N}(\delta,P)$.
\item Under certain geometric conditions on $P$, namely if $d(\nu,P)\asymp d(\nu,P\cap i\mathfrak{h}_L^{\ast})$ for all $\nu\in i\mathfrak{h}_L^{\ast}$, it can be shown that $\overline{\vol}_{R,L}(\delta,P)\asymp\int_{P^{\bullet}(1/R)\cap i\mathfrak{h}_L^{\ast}}\beta_M^G(\delta,\nu)\,\mathrm{d}\nu$, similarly to Lemma~\ref{boundary-term-reduction}. This would be the case, for example, for a ball or a box in $i\mathfrak{h}_M^{\ast}$ centered at the origin. We do not pursue this observation in detail since it is not required for our main application to Theorem~\ref{main-implication}.
\end{enumerate}
\end{remarks}

\subsection{Tempered count for fixed discrete data}\label{sec:temp-bd-2}

The central ingredient to the proof of Theorem \ref{main-implication} is the following result.

\begin{prop}
\label{MainCountingResultGLn}
Assume that Property {\rm (ELM)} holds with respect to $\underline{\delta}\in\mathcal{D}$. There are constants $c,C,\theta>0$ such that for $P\in\mathscr{B}(\underline{\delta})$, integral ideals $\q$ with $\norm\q\geqslant C$, and $1\leqslant R\leqslant c\log(2+\norm\q)$, 
\begin{equation}\label{NqdP}
\begin{aligned}
N(\q,\underline{\delta},P)&=D_F^{n^2/2}\Delta_F^*(1)\varphi_n(\q)\int_P{\rm d}\widehat{\omega}^{\rm pl}_\infty\\
&\qquad+\textnormal{O}\bigg(\varphi_n(\q)\big(\partial\vol_R(\delta,P)+\vol^{\star}_R(\delta,P)\big)+\norm\q^{n-\theta}\overline{\vol}_R(\delta,P)\bigg).
\end{aligned}
\end{equation}
If $\norm\q\leqslant C$, then \eqref{NqdP} holds with the first term replaced by $\textnormal{O}\big(\varphi_n(\q)\int_P {\rm d}\widehat{\omega}_\infty^{\rm pl}\big)$.
\end{prop}

Proposition \ref{MainCountingResultGLn} will be proved in Section \ref{sec:smooth-to-sharp-tempered}, after having introduced an appropriate class of test functions in Section \ref{spec-loc} and estimated the (exponentially weighted) discrete spectrum in Section \ref{ExceptionalSpectrumSection}. Then, in Section \ref{sec:temp-cor}, we make the deduction from Proposition \ref{MainCountingResultGLn} to Theorem \ref{main-implication}.

Our methods in fact yield a main term in Proposition~\ref{MainCountingResultGLn} for every $\q$; in the case $\norm\q\leqslant C$, its shape is mildly affected by the roots of unity in $F$ (see \eqref{DensityChiqLDefinition}). Since the bound $\textnormal{O}\big(\varphi_n(\q)\int_P {\rm d}\widehat{\omega}_\infty^{\rm pl}\big)$ in that range is sufficient for the proof of Theorem \ref{main-implication}, we have not included the finer result for small $\q$.

Proposition~\ref{MainCountingResultGLn} provides a Weyl law with explicit level savings, with full flexibility and uniformity with respect to the discrete parameter $\underline{\delta}\in\mathcal{D}$ and the region $P\in\mathscr{B}(\underline{\delta})$. For example, specializing to the ball $B(0,r)\subseteq i\mathfrak{h}_M^{\ast}$ of radius $r\gg 1/\log(2+\norm\q)$, we obtain
\[ N(\q,\underline{\delta},B(0,r))=\left(1+\mathrm{O}_{\delta}\left(\frac1{r\cdot \log(2+\norm\q)}+\frac1{\norm\q^{\theta+o(1)}}\right)\right)D_F^{n^2/2}\Delta_F^*(1)\varphi_n(\q)\widehat{\omega}_{\infty}^{\pl}(\delta, B(0,r)). \]
Note that, for every fixed $\delta$,
\[
\widehat{\omega}_{\infty}^{\pl}(\delta,B(0,r))\sim A_{\delta} r^{d_M+\dim\mathfrak{h}_M} \quad\text{as}\quad r\to\infty,
\]
where $d_M$ is defined in \eqref{def-dM}, and
\[
\widehat{\omega}_{\infty}^{\pl}(\delta, B(0,r))\sim a_{\delta}r^{2|W(M,G)_\delta|+\dim\mathfrak{h}_M}\quad\text{as}\quad r\to 0^+,
\] 
for some explicit constants $A_{\delta},a_{\delta}>0$.

\section{Spectral localizing functions}\label{spec-loc}

In this section we introduce a class of Paley--Wiener functions which localize about given arch\-i\-me\-dean spectral parameters, and then estimate the central contributions to the geometric side of the trace formula for the test functions giving rise to them.

\subsection{Definition of spectral localizer}

Let $\underline{\delta}\in\mathcal{D}$ and a real number $R>0$ be given.
Recall the space $\mathcal{PW}_{R,\underline{\delta}}$ from \S\ref{PWCD}. Let $(M,\delta)$ be a standard representative for $\underline{\delta}$, and fix a spectral parameter $\mu\in i\mathfrak{h}_M^*$. As in \S\ref{rep-notation}, we let $[\delta,\mu]$ denote the $W(A_M)$-orbit of $(\delta,\mu)$.
In this section, we will be interested in functions $h_R^{\delta,\mu}\in\mathcal{PW}_{R,\underline{\delta}}$ which localize about the $W(A_M)_\delta$-orbit of $\mu$ in $i\h^*_M$ (equivalently, under the identification from \S\ref{PWCD}, about $[\delta,\mu]$).

\begin{defn}\label{defn-localizer}
A function $h_R^{\delta,\mu}\in\mathcal{PW}_{R,\underline{\delta}}$ will be called a \emph{spectral localizer} (about $[\delta,\mu]$) if
\begin{enumerate}
\item \label{nonneg-item} $h_R^{\delta,\mu}(\nu)\geqslant 0$ for all $\nu\in i\h_M^*$;
\item the normalization
\[
\int_{i\h_M^*}h_R^{\delta,\mu}(\nu)\, {\rm d}\nu=R^{-\dim\h_M}
\]
holds;
\item \label{rapid-decay-item} the estimate
\[
h_R^{\delta,\mu}(\nu)\ll_N e^{R\|{\rm Re}\,\nu\|}\max_{w\in W(A_M)_{\delta}}(1+R\|\nu-w\mu\|)^{-N}
\]
holds for all $\nu\in\h_{M,\mathbb{C}}^*$ and all $N>0$;
\item\label{h-factorization} there are functions
\[
h_R^{\mu_Z}\in\mathcal{PW}(\mathfrak{h}_{G,\C}^*)_R\qquad\text{and}\qquad h_R^{\delta,\mu^0}=\prod_{v\mid\infty}h_R^{\delta_v,\mu^0_v}\in\mathcal{PW}((\h_M^G)_\C^*)
\]
with $h_R^{\delta_v,\mu^0_v}\in\mathcal{PW}((\aa_{M_v}^{G_v})_{\C}^{\ast})_R$ such that $h^{\delta,\mu}_R(\nu)=h_R^{\mu_Z}(\nu_Z)h_R^{\delta,\mu^0}(\nu^0)$. Here, we write $M=\prod_{v\mid\infty}M_v$ and $\delta=\prod_{v\mid\infty}\delta_v$ and use the decomposition of $\mathfrak{h}_M^*$ from \S\ref{sec:hM-decomp} to write $\mu=\mu_Z+\mu^0$, with $\mu_Z\in\mathfrak{h}_{G,\C}^*$ and $\mu^0=(\mu^0_v)_{v\mid\infty}\in (\h_M^G)_\C^*$, $\mu^0_v\in(\aa_{M_v}^{G_v})_{\C}^{\ast}$.
\end{enumerate}

Clearly, items \eqref{nonneg-item}--\eqref{rapid-decay-item} are ensured by the factorized conditions (with a suitably small $c>0$)
\begin{equation}
\label{factorized-spec-loc}
\begin{aligned}
&h_R^{\mu_Z}\geqslant 0\text{ on }i\h_G^{\ast},\quad \int_{i\h_G^{\ast}}h_R^{\mu_Z}=R^{-\dim\h_G},\quad h_R^{\mu_Z}(\nu_Z)\ll_Ne^{cR\|\mRe\nu_Z\|}(1+R\|\nu_Z-\mu_Z\|)^{-N},\\
& h_R^{\delta_v,\mu_v^0}\geqslant 0\text{ on }i(\aa_{M_v}^{G_v})^{\ast}, \quad \int_{i(\aa_{M_v}^{G_v})^{\ast}}h_R^{\delta_v,\mu_v^0}=R^{-\dim\aa_{M_v}^{G_v}},\\ &h_R^{\delta_v,\mu_v^0}(\nu_v^0)\ll_Ne^{cR\|\mRe\nu_v^0\|}\max_{w\in W(A_{M_v})_{\delta_v}}(1+R\|\nu_v^0-w\mu_v^0\|)^{-N}.
\end{aligned}\end{equation}

Finally, a \emph{system} of spectral localizers $h_R^{\delta,\mu}$ over all $[\delta,\mu]$ is said to be \emph{normalized} if, additionally, for every $\nu\in i\h_M^{\ast}$,
\[
\int_{i\mathfrak{h}_M^{\ast}}h_R^{\delta,\mu}(\nu)\,\mathrm{d}\mu=R^{-\dim\h_M}.
\]
\end{defn}

\begin{remarks} $\,$
\begin{enumerate}
\item Recall that the notation $\mathcal{PW}_{R,\underline{\delta}}$ includes the assumption of $W(A_M)_\delta$-invariance. Note that $W(A_M)_{\delta}$ acts trivially on $\mathfrak{h}_{G,\C}^*$, so that $h_R^{\delta,\mu^0}$ is necessarily $W(A_M)_{\delta}$-invariant. 
\item We also clarify that the spectral localizers $h_R^{\delta,\mu}$ will be used with varying $R$ and $[\delta,\mu]$, and that the implied constant in \eqref{rapid-decay-item} is understood to be independent of $R$, $[\delta,\mu]$, or any other parameters (apart from $F$ and $n$), unless explicitly indicated otherwise.
\item We note that Property \eqref{rapid-decay-item} is in agreement with the definition of the Paley--Wiener space $\mathcal{PW}(\mathfrak{h}_{M,\C}^*)_R$ in \S\ref{PWCD}, although the former features a factor of $R$ in the polynomial decay factor while the latter does not. Indeed, 
\[
\min(1,R)^k (1+\|\nu\|)^k \ll_k (1 + R\|\nu\|)^k \ll_k \max(1,R)^k (1+\|\nu\|)^k.
\]
Since the definition of $\mathcal{PW}(\mathfrak{h}_{M,\C}^*)_R$ makes no statement about the dependence of the stated supremum on $R$ or $k$, including such a factor would yield the same space.
\item These factorization conditions are imposed (and hence ensured in the constructions in \S\ref{const-spec-loc} and \S\ref{existence-test-functions-subsection}) solely for their utility in Section \ref{PW}, especially the second factorization of $h_R^{\delta,\mu^0}(\nu^0)$ over the individual archimedean places; they play no role in the actual application of spectral localizers in this and the later sections. For a similar decomposition to \eqref{h-factorization} in the literature, see \cite[\S 4.2]{Matz2016} in the context of $M=T_{0,\infty}$.
\end{enumerate}
\end{remarks}

\subsection{Construction of spectral localizers}\label{const-spec-loc}

In this subsection, we show how to construct an explicit normalized system of spectral localizers satisfying Definition~\ref{defn-localizer}, which will be used in the proof of Proposition~\ref{MainCountingResultGLn} in Section~\ref{sec:smooth-to-sharp-tempered}. We remark, however, that everything in \S\ref{UniformUpperBoundsSubsection} (resp., \S\ref{sec:approx-estimates}) applies to any spectral localizer (resp., any normalized system of spectral localizers) verifying Definition~\ref{defn-localizer} and condition \eqref{defn-hR}, independently of any explicit construction.

It suffices to construct the central and semisimple factors of $h_R^{\delta,\mu}$ which appear in the definition.

\begin{itemize}
\item[--] {\it Abelian localizer:} We let $\,\;\widehat{}:C_c^\infty(\mathfrak{h}_G)\rightarrow\mathcal{PW}(\mathfrak{h}_{G,\C}^*)$ be the Fourier transform. Let $g_0\in C^\infty_c(\mathfrak{h}_G)$ be supported in the ball of radius 1, satisfy $g_0(0)=1$, and have non-negative Fourier transform $h_0=\widehat{g_0}$ on $i\mathfrak{h}_G^*$. For a real parameter $R>0$ we write $g_{0,R}(X)=R^{-\dim\mathfrak{h}_G} g_0(R^{-1}X)$ and $h_{0,R}=\widehat{g_{0,R}}$; in particular, $h_{0,R}$ lies in $\mathcal{PW}(\mathfrak{h}_{G,\C}^*)_R$. We let $g_R^{\mu_Z}(X)=g_{0,R}(X)e^{-\langle\mu_Z,X\rangle}$ and $h_R^{\mu_Z}=\widehat{g_R^{\mu_Z}}$. Then we have $h_R^{\mu_Z}(\nu_Z)=h_{0,R}(\nu_Z-\mu_Z)$.

\item[--] {\it Semisimple localizer:} Next we let $g_{1,v}\in C^\infty_c(\aa_{M_v}^{G_v})$ be supported in the ball of radius 1 about the origin and satisfy $g_{1,v}(0)=1$. Assume $h_{1,v}=\widehat{g_{1,v}}$ is  non-negative on $i(\aa_{M_v}^{G_v})^\ast$. Let $g_{1,R,v}(X)=R^{-\dim\aa_{M_v}^{G_v}}g_{1,v}(R^{-1}X)$ and put $g_{R,v}^{\mu^0_{v}}(X)=g_{1,R,v}(X)e^{-\langle \mu^0_{v}, X\rangle}$. Then if $h_{1,R,v}, h_{R,v}^{\mu^0_{v}}\in\mathcal{PW}((\aa_{M_v}^{G_v})_{\C}^{\ast})_R$ are the Fourier transforms of $g_{1,R,v}$ and $g_{R,v}^{\mu^0_{v}}$, we have $h_R^{\mu^0_{v}}(\nu_{v})=h_{1,R,v}(\nu_{v}-\mu^0_{v})$. Finally, for $\nu_0=(\nu_{0,v})_{v\mid\infty}\in (\h_{M_v}^{G_v})_{\C}^{\ast}$, we set $h_R^{\mu_0}(\nu^0)=\prod_{v\mid\infty}h_{R,v}^{\mu^0_v}(\nu^0_v)$.
\end{itemize}

With the above constructions, we now set $h_R^\mu(\nu)=h_R^{\mu^Z}(\nu^Z)h_R^{\mu^0}(\nu^0)$ and let $h^{\delta,\mu}_R$ be the unique function in $\mathcal{PW}_{R,\underline{\delta}}$ such that, for $\nu\in \mathfrak{h}_{M,\C}^*$,
\[
h^{\delta,\mu}_R(\nu)=
\frac{1}{|W(A_M)_\delta|}\sum_{w\in W(A_M)_\delta} h^\mu_R(w\nu).
\]
In other words, the function $h^{\delta,\mu}_R(\sigma,\nu)$ is zero if $\sigma\notin\underline{\delta}$ and is otherwise equal to the displayed equation. (This is in accordance with our notational conventions for functions in $\mathcal{PW}_{R,\underline{\delta}}$, as described before Definition \ref{defn-localizer}.) It is clear that this choice satisfies \eqref{nonneg-item}--\eqref{rapid-decay-item} as well as the normalization property. The factorization property~\eqref{h-factorization} also holds as it is preserved in the $W(A_M)_{\delta}$-average.

\subsection{Central contributions}\label{UniformUpperBoundsSubsection}

Recall the definition of the space $\mathcal{H}(G^1_{\infty})_{R,\underline{\delta}}$ from \S\ref{PWCD}. In this subsection, let $\mu\in i\h_M^{\ast}$ and assume that we have at our disposal a test function $f^{\delta,\mu}_R\in\mathcal{H}(G^1_\infty)_{cR,\underline{\delta}}$ (with a constant $c>0$ depending on $F$, $n$ only) such that 
\begin{equation}\label{defn-hR}
h^{\delta,\mu}_R: \nu\mapsto {\rm tr}\, \pi_{\delta,\nu}(f^{\delta,\mu}_R)
\end{equation}
is a spectral localizer in $\mathcal{PW}_{R,\underline{\delta}}$ about $[\underline{\delta},\mu]$,
as described by Definition~\ref{defn-localizer}. In this subsection, we evaluate and bound the central contributions $J_{\rm cent}$ to the trace formula, used with $f_R^{\delta,\mu}$ as the archimedean component of the test function. From the expression \eqref{eq:def:explicit-Jcentral}, we have
\begin{equation}\label{central-gammas}
J_{\rm cent}(\varepsilon_{K_1(\q)}\otimes f_R^{\delta,\mu})
=D_F^{n^2/2}\Delta_F^*(1)\varphi_n(\q)\sum_{\substack{\gamma\in Z(F)\cap K_1(\q)\\\gamma_{\infty}\in G_{\infty,\leqslant cR}^1}} f_R^{\delta,\mu}(\gamma).
\end{equation}
Note that, by compactness, the sum over $\gamma$ in \eqref{central-gammas} is always finite. The next lemma estimates the number of terms in \eqref{central-gammas} and, in fact, shows that, in the range $1\leqslant R\ll \log(2+\norm\q)$ of interest to us, the sum typically contains only the identity element.

\begin{lemma}
\label{GL1LengthSpectrumGapNewLemma}
There exist constants $c_2,C_2>0$ such that, with the field notation from \S\ref{field}:
\begin{enumerate}
\item\label{general-gamma} the sum over $\gamma$ in \eqref{central-gammas} has at most $\mathrm{O}\big(1+(R/\log(2+\norm\q))^{r_1+r_2-1}\big)$ nonzero terms.
\item\label{gamma-sum} if $1\leqslant R\leqslant c_2\log(2+\norm\q)$, the sum over $\gamma$ in \eqref{central-gammas} consists only of $\gamma=1$ and possibly a subset of non-identity roots of unity in $\mathcal{O}_F^{\times}$.
\item if, additionally $\norm\q>C_2$, then the sum over $\gamma$ in \eqref{central-gammas} reduces to the identity $\gamma=1$.
\end{enumerate}
\end{lemma}

\begin{proof}
Let $V=\prod_{v\mid\infty}\R$ be Minkowski space and denote by $H=\{(x_v)_{v\mid\infty}\in V: \sum_v d_vx_v=0\}$ the trace-zero hyperplane. The logarithm map $\log_F: \mathcal{O}_F^\times\rightarrow V$, $\gamma\mapsto (\log |\gamma^{(v)}|)_{v\mid\infty}$, where $\gamma^{(v)}$ denotes the image of $\gamma$ under a real or complex embedding associated with the place $v$, takes values in $H$. Its kernel consists precisely of the roots of unity in $\mathcal{O}_F^{\times}$.

Note that elements $\gamma\in Z(F)\cap K_1(\q)$ are given by diagonal elements corresponding to a unit $u$ in $\mathcal{O}_F^\times$ congruent to $1$ mod $\q$. If $c_2$ is taken small enough, the image under $\log_F$ of such $u$ is a lattice in $H$ having trivial intersection with $B(0,cc_2\log(2+\norm\q))$. From this, \eqref{general-gamma} follows by a simple lattice point counting argument. Moreover, for $cR\leqslant cc_2\log(\norm\q)$, the only $\gamma\in Z(F)\cap K_1(\q)$ contributing to \eqref{central-gammas} correspond to roots of unity congruent to $1$ mod $\q$. This is a finite set which for $\q$ large enough is just $1$.
\end{proof}

Recall the Plancherel inversion formula given in \eqref{eq:plancherel-inv} and the explicit decomposition \eqref{Pl-decomp}. If $\omega_{\delta,\nu}$ denotes the central character of $\pi_{\delta,\nu}$, we obtain, for $\gamma\in Z(F)$,
\begin{equation}
\label{plancherel-gammas}
f_R^{\delta,\mu}(\gamma)=C_M\deg(\delta)\int_{i\mathfrak{h}_M^*} h_R^{\delta,\mu}(\nu)\omega_{\delta,\nu}^{-1}(\gamma)\mu_M^G(\delta,\nu)\,\mathrm{d}\nu.
\end{equation}
Thus $J_{\rm cent}(\varepsilon_{K_1(\q)}\otimes f_R^{\delta,\mu})$ is given by
\begin{equation}
\label{DensityChiqLDefinition}
D_F^{n^2/2}\Delta_F^*(1)\varphi_n(\q)\sum_{\substack{\gamma\in Z(F)\cap K_1(\q)\\\gamma_{\infty}\in G_{\infty,\leqslant cR}^1}}C_M\deg(\delta)\int_{i\mathfrak{h}_M^*} h_R^{\delta,\mu}(\nu)\omega_{\delta,\nu}^{-1}(\gamma)\mu_M^G(\delta,\nu)\,\mathrm{d}\nu.
\end{equation}
In particular,
\begin{equation}\label{Jcent-gen}
J_{\rm cent}(\varepsilon_{K_1(\q)}\otimes f_R^{\delta,\mu})\ll \varphi_n(\q)\cdot |\{\gamma\in Z(F)\cap K_1(\q):\gamma_{\infty}\in G_{\infty,\leqslant cR}^1\}|\cdot \|h_R^{\delta,\mu}\|_{L^1(\hat\omega_{\infty}^{\pl})}.
\end{equation}

It will be useful to have the following estimate, which is a variation of~\cite[Proposition 6.9]{DuistermaatKolkVaradarajan1979}. Recall the Plancherel majorizer $\beta_M^G$ of Definition \ref{BetaMajorizer} and the subset $W(M,L)_\delta$ of \S\ref{sec:herm}.

\begin{lemma}
\label{L1NormEstimateLemma}
For $L\in\mathcal{L}_{\infty}(\delta)$. Then, for every $\mu\in i\mathfrak{h}_L^{\ast}$ and every $R\gg 1$ we have
\[
\|h_R^{\delta,\mu}\|_{L^1(\hat\omega_{\infty}^{\pl})}\ll R^{-\dim\h_M-2|W(M,L)_\delta|}\beta_M^G(\delta,\mu).
\]
\end{lemma}

\begin{proof}
From the definition, we have that
\[ \|h_R^{\delta,\mu}\|_{L^1(\widehat{\omega}_{\infty}^{\pl})}=C_M\deg(\delta)\int_{i\mathfrak{h}_M^{\ast}}h_R^{\delta,\mu}(\nu)\mu_M^G(\delta,\nu)\,\mathrm{d}\nu.\]
Combining Lemma~\ref{prop:density}~\eqref{beta3} with the rapid decay of $h_R^{\delta,\mu}$, the right-hand side is at most
\[
\max_{w\in W(A_M)_{\delta}}\int_{i\mathfrak{h}_M^{\ast}}\big(1+R\|\nu-w\mu\|)^{-N}\|\nu-w\mu\|^{2|W(M,L)_\delta|}\big(1+\|\nu-w\mu\|\big)^{d_M-|W(M,L)_\delta|}\beta_M^G(\delta,\mu)\,\mathrm{d}\nu.
\]
This is in turn bounded by
\[
R^{-\dim\mathfrak{h}_M-2|W(M,L)_\delta|}\beta_M^G(\delta,\mu)\int_{i\mathfrak{h}_M^{\ast}}(1+\|\nu\|)^{-N+d_M+|W(M,L)_\delta|}\,\mathrm{d}\nu
\ll R^{-\dim\mathfrak{h}_M-2|W(M,L)_\delta|}\beta_M^G(\delta,\mu),
\]
uniformly for $N>\dim\mathfrak{h}_M+\max_{M,L}(d_M+|W(M,L)_{\delta}|)$. 
\end{proof}

We obtain the following estimate for the central contributions.

\begin{lemma}\label{cor-central}
For $\underline{\delta}\in\mathcal{D}$ with standard representative $(M,\delta)$, integral ideal $\q$, and $R\gg 1$, the following holds:
 if $\mu\in i\mathfrak{h}_L^{\ast}$ for some $L\in\mathcal{L}_{\infty}(\delta)$, then
\[
J_{\rm cent}(\varepsilon_{K_1(\q)}\otimes f_R^{\delta,\mu})\ll R^{-\dim\h_M-2|W(M,L)_\delta|}\big(1+(R/\log(2+\norm\q))^{r_1+r_2-1}\big)\varphi_n(\q)\beta_M^G(\delta,\mu).
\]
\end{lemma}
\begin{proof}
We use Lemma \ref{GL1LengthSpectrumGapNewLemma} \eqref{general-gamma} to estimate the number of $\gamma$ contributing to the second factor in \eqref{Jcent-gen}. We then apply Lemma \ref{L1NormEstimateLemma} to the last factor in \eqref{Jcent-gen}.
\end{proof}

\subsection{Approximation estimates}\label{sec:approx-estimates}

Let us assume that we have at our disposal a system of test functions $f_R^{\delta,\mu}\in\mathcal{H}(G_{\infty}^1)_{cR,\underline{\delta}}$ as in \S\ref{UniformUpperBoundsSubsection} over all $[\underline{\delta},\mu]$, locally integrable when seen as a function of $(g,\mu)\in G^1_{\infty}\times i\h_M^{\ast}$, and whose spherical transforms $h_R^{\delta,\mu}\in\mathcal{PW}_{R,\underline{\delta}}$ given by \eqref{defn-hR} form a normalized system of spectral localizers in the sense of Definition~\ref{defn-localizer}. In this subsection we will show how this data can be used, by an averaging procedure, to construct additional functions in $\mathcal{PW}_{R,\underline{\delta}}$ which smoothly approximate characteristic functions of certain sets. We will moreover construct  test functions that give rise to the latter via the spherical transform.

Specifically, for a set $P$ lying in the class $\mathscr{B}(\underline{\delta})$ defined in \S\ref{sec:sets}, and a parameter $R>0$, we put
\begin{equation}\label{defhP}
h_R^{\delta,P}(\nu)=R^{{}\dim\h_M}\int_P h^{\delta,\mu}_R(\nu)\,\mathrm{d}\mu, \qquad
f_R^{\delta,P}(g)=R^{{} \dim\h_M}\int_P f^{\delta,\mu}_R(g)\,\mathrm{d}\mu.
\end{equation}
Note that if Property (ELM) holds for $\underline{\delta}$ (see Definition \ref{TFhyp}), then it holds using $h_R^{\delta,P}$, for any $P\in\mathscr{B}(\underline{\delta})$, with the upper bound
\begin{equation}\label{averageELM}
J_{\rm error}(\varepsilon_{K_1(\q)}\otimes f_R^{\delta,P}) \ll e^{CR}\norm\q^{n-\theta}\int_P\beta_M^G(\delta,\nu)\,{\rm d}\nu.
\end{equation} 
Here, as in \eqref{DefEqh}, we have denoted $J_{\rm error}=J_{\rm disc}-J_{\rm cent}$.

We now quantify the extent to which $h_R^{\delta,P}$ approximates the characteristic function of $P$. We shall use the notation $P^\circ(\rho)$ and $P^\bullet(\rho)$ from \S\ref{sec:sets}.

\begin{lemma}
\label{GLnTestFunctionEstimates}
Let notations be as above. Then:
\begin{enumerate}
\item\label{first-guy0} For every $\nu\in\mathfrak{h}_{M,\C}^*$, $h_R^{\delta,P}(\nu)\ll e^{R\|{\rm Re}\,\nu\|}$;
\item\label{first-guy1} For every $\nu\in \mathfrak{h}_{M,\C}^*$ such that $i\mIm\nu\notin P^{\bullet}(\rho)$,
\[
h_R^{\delta,P}(\nu)\ll_N e^{R\|{\rm Re}\,\nu\|}(R\rho)^{-N};
\]
\item\label{second-guy} For every $\nu\in i\mathfrak{h}_M^*$, $0\leqslant h_R^{\delta,P}(\nu)\leqslant 1$, and, for every $\rho>0$, $N\in\mathbb{N}$,
\[
h_R^{\delta,P}(\nu)=\begin{cases}
1+\textnormal{O}_N\left((R\rho)^{-N}\right), &\nu\in P^{\circ}(\rho);\\
\textnormal{O}_N\left((R\rho)^{-N}\right), &\nu\notin P^{\bullet}(\rho).
\end{cases}
\]
\end{enumerate}
\end{lemma}

\begin{proof}
From \eqref{rapid-decay-item} of Definition~\ref{defn-localizer}, $i\mIm\nu\notin P^{\bullet}(\rho)$ and the $W(A_M)_{\delta}$-invariance of $P$, it follows that
\begin{equation}
\label{UpperBoundhRdP}
\begin{aligned}
h_R^{\delta,P}(\nu)
&\ll_N R^{\dim\mathfrak{h}_M}e^{R\|\mathrm{Re}\,\nu\|}\int_{\| i\mIm\nu-\mu\|\geqslant\rho}(1+R\| i\mIm\nu-\mu\|)^{-N-\dim\mathfrak{h}_M}\,\mathrm{d}\mu\\
&\ll e^{R\|\mathrm{Re}\,\nu\|}\int_{R\rho}^{\infty}(1+t)^{-N-\dim\mathfrak{h}_M}t^{\dim\mathfrak{h}_M-1}\,\mathrm{d}t\ll e^{R\|\mathrm{Re}\,\nu\|}(R\rho)^{-N},
\end{aligned}
\end{equation}
proving \eqref{first-guy1}. The estimate \eqref{first-guy0} follows similarly.

Next, let $\nu\in i\mathfrak{h}_M^*$. The inequality $0\leqslant h_R^{\delta,P}(\nu)\leqslant 1$ follows immediately from the definition of $h_R^{\delta,P}$ and the normalization of the system $h_R^{\delta,\mu}$. Further, if $\nu\in P^{\circ}(\rho)$, then
\[ h_R^{\delta,P}(\nu)=R^{\dim\h_M}\int_{i\h_M^{\ast}}h_R^{\delta,\mu}(\nu)\,\mathrm{d}\mu+\mathrm{O}\bigg(R^{\dim\h_M}\int_{\|\nu-\mu\|\geqslant\rho}h_R^{\delta,\mu}(\nu)\,\mathrm{d}\mu\bigg)
=1+\mathrm{O}_N((R\rho)^{-N}), \]
by the normalization and estimating the remainder as in \eqref{UpperBoundhRdP}. This establishes the first estimate in \eqref{second-guy}; the second one is analogous.
\end{proof}

\begin{lemma}\label{local-sharp}
For $P\in\mathscr{B}(\underline{\delta})$, $R>0$, and $\gamma\in Z(F)$, we have
\[
\deg(\delta)\int_{i\mathfrak{h}_M^*}h^{\delta,P}_R(\nu)\omega_{\delta,\nu}^{-1}(\gamma)\mu_M^G(\delta,\nu){\rm d}\nu
=\deg(\delta)\int_P\omega_{\delta,\nu}^{-1}(\gamma)\mu_M^G(\delta,\nu){\rm d}\nu+\mathrm{O}\big( \partial\vol_R(\delta,P)\big).
\]
\end{lemma}

\begin{proof}
We begin by decomposing the integral according to
\begin{equation}\label{P-decomp}
i\mathfrak{h}_M^*=\partial P(1/R)\cup\big(P^c\setminus \partial P(1/R)\big)\cup\big(P\setminus \partial P(1/R)\big).
\end{equation}
The integral over $i\mathfrak{h}_M^*$ in the lemma may therefore be rewritten as $\int_P\omega_{\delta,\nu}^{-1}(\gamma)\mu_M^G(\delta,\nu)\,\mathrm{d}\nu$ plus
\begin{align*}
\text{O}\bigg(\int_{\partial P(1/R)}\mu_M^G(\delta,\nu)\,\mathrm{d}\nu\bigg) &+\sum_{\ell=1}^{\infty}\int_{P^{\bullet}(\ell/R,(\ell+1)/R)}h^{\delta,P}_R(\nu)\omega_{\delta,\nu}^{-1}(\gamma)\mu_M^G(\delta,\nu)\,\mathrm{d}\nu\\
& -\sum_{\ell=1}^{\infty}\int_{P^{\circ}(\ell/R,(\ell+1)/R)}\big(1-h^{\delta,P}_R(\nu)\big)\omega_{\delta,\nu}^{-1}(\gamma)\mu_M^G(\delta,\nu)\,\mathrm{d}\nu.
\end{align*}
Using Lemma~\ref{GLnTestFunctionEstimates} and accounting for the factor $\deg(\delta)$, the three terms above are majorized by
\[
\int_{\partial P(1/R)}\beta_M^G(\delta,\nu)\,{\rm d}\nu+\sum_{\ell=1}^{\infty}\ell^{-N}\bigg(\int_{P^{\bullet}((\ell+1)/R)\setminus P}\beta_M^G(\delta,\nu)\,{\rm d}\nu +\int_{P\setminus P^{\circ}((\ell+1)/R)}\beta_M^G(\delta,\nu)\,{\rm d}\nu \bigg).
\]
The last error term is indeed $\text{O}\big(\partial\vol_R(\delta,P)\big)$, as desired.
\end{proof}

We return to the estimation and evaluation of the central contributions for the functions $f_R^{\delta,P}$.

\begin{lemma}\label{cor-central2}
Let $\underline{\delta}\in\mathcal{D}$ have standard representative $(M,\delta)$ and take $P\in\mathscr{B}(\underline{\delta})$. Let $\q$ be an integral ideal and $R\gg 1$. Then
\begin{enumerate}
\item\label{Jcent1} the central term $J_{\rm cent}(\varepsilon_{K_1(\q)}\otimes f_R^{\delta,P})$ satisfies
\begin{align*}
J_{\rm cent}(\varepsilon_{K_1(\q)}\otimes f_R^{\delta,P})
&\ll \big(1+(R/\log(2+\norm\q))^{r_1+r_2-1}\big)\varphi_n(\q)\\
&\qquad\times\min\Big(\deg(\delta)\int_P\mu_M^G(\delta,\nu)\,{\rm d}\nu+\partial\vol_R(\delta,P),\int_P\beta_M^G(\delta,\nu)\, {\rm d}\nu\Big);
\end{align*}

\item\label{Jcent3} there exist constants $c_2,C_2>0$ such that if $R\leqslant c_2\log(2+\norm\q)$ and $\norm\q>C_2$, then
\[
J_{\rm cent}(\varepsilon_{K_1(\q)}\otimes f_R^{\delta,P})=D_F^{n^2/2}\Delta_F^*(1)\varphi_n(\q)C_M\deg(\delta)\int_P\mu_M^G(\delta,\nu)\,\mathrm{d}\nu+\mathrm{O}\big( \varphi_n(\q)\partial\vol_R(\delta,P)\big).
\]
\end{enumerate}
\end{lemma}

\begin{proof}
The second bound in \eqref{Jcent1} follows from an application of Lemma \ref{cor-central}, recalling the definitions \eqref{defhP}. To prove \eqref{Jcent3} we apply the third part of Lemma \ref{GL1LengthSpectrumGapNewLemma} to reduce \eqref{central-gammas} to the identity contribution, and then use the integral representation \eqref{plancherel-gammas} and Lemma \ref{local-sharp}. The first bound in \eqref{Jcent1} is analogous, accounting for contributions from the units in \eqref{central-gammas}, the number of which is bounded in the first part of Lemma~\ref{GL1LengthSpectrumGapNewLemma}.
\end{proof}

\subsection{Application of the Paley--Wiener theorem of Clozel--Delorme}\label{CD-appl}

Let $h_R^{\delta,\mu}\in\mathcal{PW}_{R,\underline{\delta}}$ be as in Definition \ref{defn-localizer}. In this section, we describe how one can 
use the results of Harish-Chandra and Zuckerman \cite[Prop 6.6.7]{Vogan1981} and the Paley--Wiener theorem of Clozel--Delorme \cite{ClozelDelorme1984}, to find test functions $f_R^{\delta,\mu}\in\mathcal{H}(G_\infty^1)_{R,\underline{\delta}}$ which recover $h_R^{\delta,\mu}$ through the relation \eqref{defn-hR}. We present the construction here primarily for the possible interest in its $\underline{\delta}$-aspect. We emphasize, however, that in the cases in which we verify Property (ELM) (that is, in Theorem~\ref{master}), we do not need this construction or rely on the theorem of Clozel--Delorme; instead the test functions $f_R^{\delta,\mu}$ are constructed via explicit spherical inversion (including for non-trivial $K$-types when $n\leqslant 2$) in Section \ref{PW}.

From \cite[Prop 6.6.7]{Vogan1981} there are upper triangular (with respect to the partial ordering $\prec$) arrays $(n(\underline{\sigma},\underline{\sigma}'))_{\underline{\sigma},\underline{\sigma}'\in\mathcal{D}}$ and $(m(\underline{\sigma},\underline{\sigma}'))_{\underline{\sigma},\underline{\sigma}'\in\mathcal{D}}$, with integer coefficients and (multiplied as matrices) inverse to one another, such that for any $f\in\mathcal{H}(G_\infty^1)$ we have
\begin{equation}\label{Vogan-formulae}
{\rm tr}\, \pi(\sigma,\nu)(f)=\sum_{\underline{\sigma}'}n(\underline{\sigma},\underline{\sigma}'){\rm tr}\, \pi_{\sigma',\nu}(f)\quad\text{ and }\quad
\textrm{tr}\, \pi_{\sigma,\nu}(f)=\sum_{\underline{\sigma}'}m(\underline{\sigma},\underline{\sigma}'){\rm tr}\, \pi(\sigma',\nu)(f).
\end{equation}
Note that both sums in \eqref{Vogan-formulae} consist of finitely many terms for every $f\in\mathcal{H}(G_{\infty}^1)$. With the first formula in \eqref{Vogan-formulae} in mind, we set
\[
H_R^{\delta,\mu}(\sigma,\nu)=\sum_{\underline{\sigma}'}n(\underline{\sigma},\underline{\sigma}')h_R^{\delta,\mu}(\sigma',\nu).
\]
Note that, since $h_R^{\delta,\mu}$ is supported on $\delta$ in the first variable, we have $H_R^{\delta,\mu}(\sigma,\nu)=n(\underline{\sigma},\underline{\delta}) h_R^{\delta,\mu}(\delta,\nu)$. In particular, $H_R^{\delta,\mu}$ is supported in the first variable on $\underline{\sigma}\prec\underline{\delta}$. We deduce that $H_R^{\delta,\mu}\in\mathcal{PW}_R$ and hence there exists, by the theorem of Clozel--Delorme, a function $f_R^{\delta,\mu}\in\mathcal{H}(G_\infty^1)_R$ such that $\textrm{tr}\,\pi(\sigma,\nu)(f_R^{\delta,\mu})=H_R^{\delta,\mu}(\sigma,\nu)$. The second formula in \eqref{Vogan-formulae}, when applied to $f=f_R^{\delta,\mu}$, then yields
\[
\textrm{tr}\, \pi_{\sigma,\nu}(f_R^{\delta,\mu})=\sum_{\underline{\sigma}'}m(\underline{\sigma},\underline{\sigma}')H_R^{\delta,\mu}(\sigma',\nu)=h_R^{\delta,\mu}(\delta,\nu)\sum_{\underline{\sigma}'}m(\underline{\sigma},\underline{\sigma}')n(\underline{\sigma}',\underline{\delta})=h_R^{\delta,\mu}(\delta,\nu)\mathbf{1}_{\underline{\sigma}=\underline{\delta}},
\]
showing that $f_R^{\delta,\mu}\in\mathcal{H}(G_\infty^1)_{R,\underline{\delta}}$ verifies the relation \eqref{defn-hR}.

\section{Bounding the discrete and exceptional spectrum}\label{ExceptionalSpectrumSection}

The goal of this section is to provide bounds on two (similarly defined) exponentially weighted sums over the discrete spectrum. Throughout we shall assume Property {\rm (ELM)}, introduced in Definition \ref{TFhyp}. 

Let $\underline{\delta}\in\mathcal{D}$ have standard representative $(M,\delta)$. Let $\mu\in i\h_L^*$ for some $L\in\mathcal{L}_{\infty}(\delta)$. Let $\q$ be an integral ideal. For a real parameter $R>0$ let
\begin{equation}\label{def-DR}
D_R(\q,\underline{\delta},\mu)=\sum_{\substack{\pi\in \Pi_{\rm disc}(\bm{G}(\A_F)^1)_{\underline{\delta}}\\ i\mIm\nu_\pi\in B_M(\mu,1/R)}} \dim V_{\pi_f}^{K_1(\q)}e^{R\|{\rm Re}\,\nu_{\pi}\|}.
\end{equation}

The following result will be used in the proof of Proposition \ref{thm:general-n-stronger}. Recall the subset $W(M,L)_\delta$ from \S\ref{sec:herm}. 

\begin{prop}\label{comp-mu}
Let notations be as above. Assume that Property {\rm (ELM)} holds with respect to $\underline{\delta}$. Then there are constants $C\geqslant 0$, $\theta>0$ such that, for $R\geqslant 1$,
\[
D_R(\q,\underline{\delta}, \mu)\ll \big(R^{-2|W(M,L)_\delta|}+e^{CR}\norm\q^{-\theta}\big)R^{-\dim \h_M}\varphi_n(\q)\beta_M^G(\delta,\mu).
\]
In particular, there is $c>0$ such that for $1\leqslant  R\leqslant c\log(2+\norm\q)$,
\[ 
D_R(\q,\underline{\delta}, \mu)\ll R^{-\dim \h_M-2|W(M,L)_\delta|}\varphi_n(\q)\beta_M^G(\delta,\mu).
\]
\end{prop}

\begin{remark}
For $n=1$, the quantity $D_R(\q,\underline{\delta}, \mu)$ simply counts the number of Hecke characters $\chi$ of conductor $\q$, of discrete archimedean parameter $\delta$, and continuous parameter $\|\nu_\chi-\mu\|\leqslant 1/R$. The proof of Proposition~\ref{comp-mu} bounds this cardinality by $R^{-(r_1+r_2-1)}(1+(R/\log(2+\norm\q))^{r_1+r_2-1})\varphi_1(\q)$ in fact \textit{for all} $R\geqslant 1$.
\end{remark}

Let $\underline{\delta}\in\mathcal{D}$ and $P\in\mathscr{B}(\underline{\delta})$. We shall also need the following sum
\begin{equation}\label{KR-def}
K_R(\q,\underline{\delta},P)=\sum_{\substack{\pi\in \Pi_{\rm disc}(\bm{G}(\A_F)^1)_{\underline{\delta}}\\\nu_{\pi}\not\in i\mathfrak{h}_M^{\ast}}}\dim V_{\pi_f}^{K_1(\q)}e^{R\|{\rm Re}\, \nu_{\pi}\|}\big(1+R\cdot d(i\mIm\nu_{\pi},P)\big)^{-N}.
\end{equation}
Note that, when compared with $D_R(\q,\underline{\delta},\mu)$, the above sum is now over $\pi$ with $\pi_\infty$ non-tempered, and the membership of $i\mIm\nu_{\pi}$ in $B_M(\mu,1/R)$ is replaced by power decay outside of $P$.  As in Remark \ref{N-dependence}, we have suppressed the dependence on $N$ in this sum, but will occasionally revive it for clarity, writing $K_{R,N}(\q,\underline{\delta},P)$.

The following result will be used in Section \ref{sec:smooth-to-sharp-tempered}, in the proof of Proposition~\ref{MainCountingResultGLn}.

\begin{prop}\label{KR}
Let notations be as above. Assume that Property {\rm (ELM)} holds with respect to $\underline{\delta}$. Then there are constants $C\geqslant 0$, $\theta>0$ such that, for $R\geqslant 1$,
\[
K_R(\q,\underline{\delta},P)\ll \big(1+R^{2|W(A_M)_{\delta}|}e^{CR}\norm\q^{-\theta}\big)\varphi_n(\q) {\rm vol}_R^\star(\delta,P).
\]
In particular, there is $c>0$ such that for $1\leqslant  R\leqslant c\log(2+\norm\q)$,
\[ 
K_R(\q,\underline{\delta},P)\ll \varphi_n(\q) {\rm vol}_R^\star(\delta,P). 
\]
\end{prop}

As will become clear shortly, the bulk of the work necessary to prove Propositions \ref{comp-mu} and \ref{KR} is the construction of appropriate test functions. This turns out to be a highly non-trivial analytic problem.

\subsection{Reduction to test functions}\label{reduction-to-test-functions-subsection}

Let $\underline{\delta}\in\mathcal{D}$ have standard representative $(M,\delta)$. Recall the $\delta$-Hermitian dual from \S\ref{sec:herm}, and let $\mu\in i\h_M^*$ and $R>0$. 

We grant ourselves momentarily the functions $H_R^{\delta,\mu}\in \mathcal{PW}_{R,\underline{\delta}}$ in Lemma \ref{Lem1} and use them to prove Propositions \ref{comp-mu} and \ref{KR}. These are linear combinations with uniformly bounded coefficients of certain spectral localizers $h_R^{\delta,\mu_{M'}}$ ($M'\supseteq M$). Invoking Property (ELM) on each of them and forming the same linear combination of the resulting test functions $f_R^{\delta,\mu_{M'}}$, we obtain a test function we denote by $F_R^{\delta,\mu}\in\mathcal{H}(G_\infty^1)_{c'R,\underline{\delta}}$ (with $c'>0$ depending on $n$, $F$ only), such that $\tr\pi_{\delta,\nu}(F_R^{\delta,\mu})=H_R^{\delta,\mu}(\nu)$ and, for every integral ideal $\q\subseteq\mathcal{O}_F$,  $J_{\mathrm{error}}(\varepsilon_{K_1(\q)}\otimes f_R^{\delta,\mu_{M'}})$ and the same quantity for $F_R^{\delta,\mu}$ verify the upper bound specified in Property (ELM).

\begin{proof}[Proof of Proposition \ref{comp-mu}.] Recall the definition of $J_{\rm disc}$ from \eqref{def:J-Eis}--\eqref{defn:J-spec-G}. Properties \eqref{2} and \eqref{3} of Lemma \ref{Lem1} show that there are constants $C,c>0$ such that
\begin{equation}\label{DR-bound}
D_{cR}(\q,\underline{\delta},\mu)\leqslant C^{-1}J_{\rm disc}(\varepsilon_{K_1(q)}\otimes F_R^{\delta,\mu}).
\end{equation}
For each of the spectral localizers $h_R^{\delta,\mu_{M'}}$ in Property \eqref{new3} of Lemma \ref{Lem1}, we apply the decomposition \eqref{DefEqh} to $J_{\rm disc}(\varepsilon_{K_1(\q)}\otimes f_R^{\delta,\mu_{M'}})$. For $n=1$, $J_{\rm error}=J_{\rm disc}-J_{\rm cent}=0$ trivially; this is the Poisson summation formula. For $n>1$,
we estimate $J_{\rm error}$ using Property {\rm (ELM)}. We bound $J_{\rm cent}$ using Lemma \ref{cor-central}. Note that $\beta_M^G(\delta,\mu_{M'})\asymp\beta_M^G(\delta,\mu)$ in light of $\|\mu^{M'}\|\ll 1$. This proves Proposition \ref{comp-mu}, noting that the bound of Lemma~\ref{cor-central} is dominated by the first term in the upper bound of Proposition~\ref{comp-mu} for $R\ll\log(2+\norm\q)$ and by the second term otherwise.
\end{proof}

\begin{proof}[Proof of Proposition \ref{KR}.] We now let
\[
J_{\rm temp}(\phi)=\sum_{\substack{\pi\subset \Pi_{\rm disc}(\bm{G}(\A)^1)\\ \pi_\infty \textrm{ temp} }}{\rm tr}\,\pi(\phi)
\]
be the contribution to the discrete automorphic spectrum arising from $\pi\simeq\pi_f\otimes \pi_\infty$ for which $\pi_\infty$ is tempered, and write $J_{\rm comp}(\phi)=J_{\rm disc}(\phi)-J_{\rm temp}(\phi)$ as in \eqref{sec3.2-defn-J-comp}. Note that the restriction of $J_{\rm comp}$ to test functions $\varepsilon_{K_1(q)}\otimes \mathcal{H}(G_\infty^1)_{\underline{\delta}}$ is supported on discrete automorphic representations $\pi\simeq\pi_f\otimes\pi_\infty$ for which $V_{\pi_f}^{K_1(q)}\neq 0$ and $\pi_\infty$ is of the form $\pi_{\delta,\nu_\pi}$ with ${\rm Re}\, \nu_\pi\neq 0$. More precisely, in view of the Hermitian structure of $\pi_\infty$, the continuous parameter $\nu_\pi$ must lie in the $\delta$-singular subset $\mathfrak{h}_{\delta,{\rm sing}}^*$ of $\h_{M,\C}^*$, defined in \eqref{delta-sing}. We consider
\begin{equation}\label{sum-int}
\sum_{\substack{L\in\mathcal{L}_{\infty}(\delta)\\L\neq M}}R^{{}\dim \h_L}\int_{i\h_L^*}J_{\rm comp}(\varepsilon_{K_1(q)}\otimes F_R^{\delta,\mu})(1+d(\mu,P)\cdot R)^{-N}\,{\rm d}\mu.
\end{equation}
On one hand, from Lemma \ref{Lem1}, the sum-integral in \eqref{sum-int} is bounded below by
\begin{align*}
C&\sum_{\substack{L\in\mathcal{L}_{\infty}(\delta)\\L\neq M}}R^{{}\dim \h_L}\int_{i\h_L^*}(1+d(\mu,P)\cdot R)^{-N}\sum_{\substack{\| i\mIm\nu_\pi-\mu\|\leqslant R^{-1}\\ \nu_{\pi}\not\in i\h_M^{\ast},\,i\mIm\nu_\pi\in i\h_L^*}}\dim V_{\pi_f}^{K_1(\q)}e^{cR\|{\rm Re}\,\nu_\pi\|}\, {\rm d}\mu\\
&=C\sum_{\substack{L\in\mathcal{L}_{\infty}(\delta)\\L\neq M}}R^{{}\dim \h_L}\sum_{\substack{\nu_{\pi}\not\in i\h_M^{\ast}\\ i\mIm\nu_\pi\in i\h_L^*}}\dim V_{\pi_f}^{K_1(\q)}e^{cR\|{\rm Re}\,\nu_\pi\|}\int\limits_{\substack{\mu\in i\h_L^*\\ \|\mu-i\mIm\nu_\pi\|\leqslant R^{-1}}}(1+d(\mu,P)\cdot R)^{-N}\, {\rm d}\mu\\
&\asymp \sum_{\substack{L\in\mathcal{L}_{\infty}(\delta)\\L\neq M}}\sum_{\substack{\nu_{\pi}\not\in i\h_M^{\ast}\\i\mIm\nu_\pi\in i\h_L^*}}\dim V_{\pi_f}^{K_1(\q)}e^{cR\|{\rm Re}\,\nu_\pi\|}(1+d(i\mIm\nu_\pi,P)\cdot R)^{-N},
\end{align*}
where the automorphic sums are over $\pi\in \Pi_{\rm disc}(\bm{G}(\A_F)^1)_{\underline{\delta}}$. The last quantity is simply $K_{cR}(\q,\underline{\delta},P)$. On the other hand, Property \eqref{2} of Lemma \ref{Lem1} ensures that 
\[
J_{\rm comp}(\varepsilon_{K_1(q)}\otimes F_R^{\delta,\mu})\leqslant J_{\rm comp}(\varepsilon_{K_1(q)}\otimes F_R^{\delta,\mu})+J_{\rm temp}(\varepsilon_{K_1(q)}\otimes F_R^{\delta,\mu})=J_{\rm disc}(\varepsilon_{K_1(q)}\otimes F_R^{\delta,\mu}).
\]
As in the proof of Proposition \ref{comp-mu}, we now use Property (ELM) and Lemma \ref{cor-central}, along with the fact that $\beta_M^G(\delta,\mu_{M'})\asymp\beta_M^G(\delta,\mu)$ for $\|\mu^{M'}\|\ll 1$ to bound \eqref{sum-int} by
\[
\varphi_n(\q)\sum_{\substack{L\in\mathcal{L}_{\infty}(\delta)\\L\neq M}}
R^{-{\rm codim}_{\h_M}(\h_L)}\big(R^{-2|W(M,L)_\delta|}+e^{CR}\norm\q^{-\theta}\big)\int_{i\h_L^*}\beta^G_M(\delta,\mu)(1+d(\mu,P)\cdot R)^{-N}\,{\rm d}\mu.
\]
Recalling the definition of ${\rm vol}_R^\star(\delta,P)$ in \eqref{BoundaryTermsDefinitions2} and \eqref{BoundaryTermsDefinitions3}, the terms involving $R^{-2|W(M,L)_{\delta}|}$ in the above quantity add up exactly to $\varphi_n(\q) {\rm vol}_R^\star(\delta,P)$. Using $W(M,L)_{\delta}\subseteq W(A_M)_{\delta}$ to bound the remaining terms and putting all estimates together completes the proof of Proposition \ref{KR}.
\end{proof}

\subsection{Existence of test functions}\label{existence-test-functions-subsection}

The main result of this subsection is the following lemma, which proves the existence of test functions used in \S\ref{reduction-to-test-functions-subsection} to prove Propositions~\ref{comp-mu} and \ref{KR}. This delicate construction, in turn, also crucially relies on three results of geometric and analytic nature, which we prove as Lemmata~\ref{nightmare}, \ref{NSA}, and \ref{asymp}.

We recall from \S\ref{sec:herm} the $\delta$-unitary and -Hermitian duals $\mathfrak{h}_{\delta,{\rm un}}^*$ and $\h_{\delta,{\rm hm}}^*$.

\begin{lemma}\label{Lem1}
Let $\underline{\delta}\in\mathcal{D}$ have standard representative $(M,\delta)$, and let $\mu\in i\h_M^*$ and $R\geqslant 1$. There is $H_R^{\delta,\mu}\in\mathcal{PW}_{R,\underline{\delta}}$ verifying the following properties:
\begin{enumerate}
\item\label{2} $H_R^{\delta,\mu}\geqslant 0$ on $\mathfrak{h}_{\delta,{\rm hm}}^*$;
\item\label{new3} $H_R^{\delta,\mu}$ is a linear combination, with coefficients bounded independently of $\delta$ or $\mu$, of spectral localizers $h_R^{\delta,\mu_{M'}}$ over $M'\supseteq M$ such that $\|\mu^{M'}\|\ll 1$;
\item\label{3} There are constants $C>0$ and $0<c<1$ such that for all $\nu\in\h_{\delta, {\rm un}}^{\ast}$ satisfying
\begin{equation}\label{range1}
\|i\mIm\nu-\mu\|\leqslant R^{-1}
\end{equation}
we have $H_R^{\delta,\mu}(\nu)\geqslant Ce^{c R\|{\rm Re}\,\nu\|}$.
\end{enumerate}
\end{lemma}

\begin{remark}
It will be seen in the proof of Lemma~\ref{asymp} that the constant $0<c<1$ in \eqref{asymp2-eq} may in fact be taken as close to 1 as desired by taking $\eta>0$ (as well as the variable $u>0$ appearing in the proof) very small. In turn, the value of the constant $c$ in the local version of Lemma~\ref{Lem1}~\eqref{3} (see (\ref{three-prime}$'$) below) may also be taken as close to 1 as desired; this conclusion should be contrasted with the obvious upper bound of $\ll e^{R\|{\rm Re}\,\nu_v^0\|}$.
\end{remark}

\begin{proof}
For $A>0$ as in Lemma~\ref{nightmare} and $v\mid\infty$ let
\[
E_{0,v}=\{\nu_v^0\in\h_{\delta_v, {\rm un}}^{\ast}:\|\mathrm{Re}\,\nu_v^0\|\leqslant AR^{-1}\},\quad
E_{1,v}=\{\nu_v^0\in\h_{\delta_v, {\rm un}}^{\ast}:\|\mathrm{Re}\,\nu_v^0\|>AR^{-1}\},
\]
where we recall the notation $\h_{\delta_v,\mathrm{un}}^{\ast}\subset(\aa_{M_v}^{G_v})_{\C}^{\ast}$ from \S\ref{sec:herm}.
Then $\h^{\ast}_{\delta,\mathrm{un}}\subseteq\h^{\ast}_G\oplus\bigoplus_{v\mid\infty}(E_{0,v}\cup E_{1,v})$. Recalling the decomposition $\nu=\nu_Z+\nu^0$ from \S\ref{sec:lie}, we will construct the desired function $H_R^{\delta,\mu}\in\mathcal{PW}_{R,\underline{\delta}}$ as
\begin{equation}
\label{HRdeltamu-factor}
H_R^{\delta,\mu}(\nu)=h_{0,R}^{\mu_Z}(\nu_Z)\prod_{v\mid\infty}\big(H_{0,v,R}^{\delta_v,\mu_v^0}(\nu^0_v)+H_{1,v,R}^{\delta_v,\mu_v^0}(\nu^0_v)\big),
\end{equation}
where $h_{0,R}^{\mu_Z}\in\mathcal{PW}(\h_{G,\C}^{\ast})$ and $W(A_{M_v})_{\delta_v}$-invariant functions $H_{\epsilon,v,R}^{\delta_v,\mu_v^0}\in\mathcal{PW}((\aa_{M_v}^{G_v})^{\ast}_{\C})$ ($\epsilon\in\{0,1\}$) are such that:
\begin{enumerate}[\quad (1$'$)]
\item\label{one-prime} $h_{0,R}^{\mu_Z}\geqslant 0$ on $i\h_G^{\ast}$ and $H_{\epsilon,v,R}^{\delta_v,\mu_v^0}\geqslant 0$ on $\h^{\ast}_{\delta_v,\mathrm{un}}$;
\item\label{two-prime} the functions $h_{0,R}$ and $H_{0,v,R}^{\delta_v,\mu_v^0}$ satisfy the factorized spectral localizer conditions \eqref{factorized-spec-loc}, and $H_{1,v,R}^{\delta_v,\mu_v^0}$ is a linear combination of $H_{1,v,R}^{\delta_v,(\mu_v^0)_{M'_v}}$ as in \eqref{new3} satisfying these conditions.
\item\label{three-prime} there are constants $C>0$ and $0<c<1$ such that:
\begin{itemize}
\item $h_{0,R}^{\mu_Z}(\nu_Z)\geqslant\frac12$ ($\nu_Z\in i\h_G^{\ast}$, $\|\nu_Z-\mu_Z\|\leqslant R^{-1}$);
\item $H_{0,v,R}^{\delta_v,\mu^0_v}(\nu_v^0)\geqslant\frac1{16}$ ($\nu_v^0\in E_{0,v}$, $\|i\mIm\nu_v^0-\mu_v^0\|\leqslant R^{-1}$);
\item $H_{1,v,R}^{\delta_v,\mu^0_v}(\nu_v^0)\geqslant Ce^{cR\|\mathrm{Re}\,\nu_v^0\|}$ ($\nu_v^0\in E_{1,v}$, $\|i\mIm\nu_v^0-\mu_v^0\|\leqslant R^{-1}$).
\end{itemize}
\end{enumerate}
This clearly yields the statement of the theorem. (For $\nu_v^0\in E_{0,v}$, $Ce^{cR\|\mRe\nu_v^0\|}$ is bounded above by an absolute constant. The above conditions thus imply that $H_R^{\delta,\mu}(\nu)\gg e^{cR\sum_{v\mid\infty}\|\mRe\nu_v^0\|}$ for every $\nu\in\h_{\delta,\mathrm{un}}^{\ast}$ satisfying \eqref{range1}, whence \eqref{3}.)

The construction of $h_{0,R}^{\mu_Z}\in\mathcal{PW}(\h_{G,\C}^{\ast})$ is straightforward; we take a sufficiently small $0<\delta'\leqslant 1$, fix a $g_Z\in C_c^{\infty}(\h^{\ast}_G)$ be a real even function supported in the $\delta'$-ball around $0$ and satisfying $\int g_Z=1$, and consider the Fourier transform $h_Z=\widehat{g_Z\star g_Z}$. Then, $h_Z\geqslant 0$ on $i\h_G^{\ast}$, and we may take $\delta'>0$ sufficiently small to assure that $h_Z(\nu_Z)\geqslant\frac12$ for all $\nu_Z\in i\h_G^{\ast}$ with $\|\nu_Z\|\leqslant 1$. Then $h_{0,R}^{\mu_Z}(\nu_Z):=h_Z(R(\nu_Z-\mu_Z))$ has all the desired properties (\ref{one-prime}$'$)--(\ref{three-prime}$'$).

It remains to construct $H_{\epsilon,v,R}^{\delta_v,\mu_v^0}\in\mathcal{PW}((\aa_{M_v}^{G_v})^{\ast}_{\C})$ with the desired properties. This is now a purely local task, and, to lighten the notation, \textit{for the rest of the proof we consistently drop the $_v$ and $^0$ decorations. Thus, $H_{\epsilon,R}^{\delta,\mu},E_{\epsilon},\nu,\mu,\aa_M^G,W(A_M)_{\delta},\dots$ all refer to their local counterparts, as do the archimedean structures referenced from Section \ref{sec:arch-rep-PWCD}.}

We first construct $H_{0,R}^{\delta,\mu}\in\mathcal{PW}((\aa_M^G)^{\ast}_{\C})$. This is essentially accomplished by \cite[Lemma 4.3]{BrumleyMarshall2020}, which we only need to adapt slightly. Let $\delta''>0$ be sufficiently small, let $g_0\in C_c^{\infty}(\aa_M^G)$ be a real even function supported in the $\delta''$-ball around $0$ and satisfying $\int g_0=1$; then, the Fourier transform $h_0=\widehat{g_0\star g_0}\in\mathcal{PW}((\aa_M^G)^{\ast}_{\C})$ satisfies $h_0\geqslant 0$ on $i(\aa_M^G)^{\ast}$.
As in the proof of \cite[Lemma 4.3]{BrumleyMarshall2020}, we may take $\delta''>0$ sufficiently small to assure that $\mathrm{Re}\,h_0(\nu)\geqslant\frac12$ for all $\nu\in (\aa_M^G)^{\ast}_{\C}$ with $\|\mathrm{Im}\,\nu\|\leqslant 1$ and $\|\mathrm{Re}\,\nu\|\leqslant A$, and that $\mathrm{Re}\,h_0(\nu)\geqslant -C''\delta''$ (with a $C''>0$ depending on $G$ only) for all $\nu\in (\aa_M^G)^{\ast}_{\C}$ with $\|\mathrm{Re}\,\nu\|\leqslant A$. We set
\begin{equation}\label{defn-h0R-upperbds}
H_{0,R}^{\delta,\mu}(\nu)=\bigg(\sum_{w\in W(A_M)_{\delta}}h_0\big(R(w\nu-\mu)\big)\bigg)^2.
\end{equation}
Then $H_{0,R}^{\delta,\mu}\in\mathcal{PW}((\aa_M^G)^{\ast}_{\C})_R$ is $W(A_M)_{\delta}$-invariant by construction.
Now, by the definition of $\h_{\delta,{\rm hm}}^*$ in \eqref{hMw*}--\eqref{delta-hm}, we have $W(A_M)_{\delta}.\nu=W(A_M)_{\delta}.(-\overline{\nu})$ for all $\nu\in\h_{\delta,{\rm un}}^*$. It follows that the sum in \eqref{defn-h0R-upperbds} is unchanged under the substitution $\nu\mapsto -\overline{\nu}$ on the $\delta$-Hermitian spectrum. Furthermore, since $g_0$ was taken real and even, then $h_0=\widehat{g_0\star g_0}$ enjoys the relations $h_0(-\nu)=h_0(\nu)$ and $h_0(\overline{\nu})=\overline{h_0(\nu)}$, for all $\nu\in(\aa_M^G)^{\ast}_{\C}$. Finally, using $-\mu=\overline{\mu}$, we conclude that the sum in \eqref{defn-h0R-upperbds} is real-valued. Thus $H_{0,R}^{\delta,\mu}\geqslant 0$ on $\h_{\delta,{\rm hm}}^*$. This establishes (\ref{one-prime}$'$), while (\ref{two-prime}$'$) is immediate.
 Finally, for $\delta''>0$ sufficiently small, we have that for $\nu\in E_0$ satisfying $\|i\mIm\nu-\mu\|\leqslant R^{-1}$,
\[ \sum_{w\in W(A_M)_{\delta}}h_0\big(R(w\nu-\mu)\big)\geqslant \mathrm{Re}\,h_0\big(R(\nu-\mu)\big)-|W(A_M)_{\delta}|C''\delta''\geqslant\tfrac14, \]
and hence $H_{0,R}^{\delta,\mu}(\nu)\geqslant\frac1{16}$ as desired for (\ref{three-prime}$'$).

We now proceed to the principal task of constructing $H_{1,R}^{\delta,\mu}\in\mathcal{PW}((\aa_M^G)^{\ast}_{\C})$, which in particular needs to satisfy $H_{1,R}^{\delta,\mu}(\nu)\geqslant Ce^{cR\|\mathrm{Re}\,\nu\|}$ for $\nu\in E_1$ such that
\begin{equation}\label{range}
\| i\mIm\nu-\mu\|\leqslant R^{-1}\qquad\text{ and }\qquad  \|{\rm Re}\, \nu\|\geqslant AR^{-1}.
\end{equation}
Our approach in this complementary range is inspired by that of \cite[Lemma 4.3]{BrumleyMarshall2020}, although the argument is necessarily much more elaborate. Note that if $\mu$ is such that no $\delta$-Hermitian $\nu$ satisfies \eqref{range}, then condition (\ref{three-prime}$'$) is vacuous and the function $H_{1,R}^{\delta,\mu}$ identically equal to zero satisfies the remaining conditions. Otherwise, $\mu$ should be of distance at most $R^{-1}$ from the $\delta$-singular subset $i\h_{\delta,{\rm sing}}^*$ of \eqref{delta-sing}. Let $M_\mu\in\mathcal{L}(\delta)$ be maximal for the property that $\|\mu^{M_\mu}\|\leqslant R^{-1}$; in view of \eqref{delta-sing}, such a maximal $M_{\mu}\supseteq M$ is distinct from $M$. Now if the lemma is true for $\mu_{M_\mu}$ then it is true for $\mu$ (by taking for $H_{1,R}^{\delta,\mu}$ the function $H_{1,R}^{\delta,\mu_{M_\mu}}$). We may therefore assume that $\mu\in i\h_{M_\mu}^*$.

With $h\in\mathcal{PW}((\aa_M^G)^{\ast}_{\C})_1$ as in Lemma \ref{asymp} below and $B>0$ as in Lemma~\ref{nightmare}, we put, for $\nu\in(\aa_M^G)^{\ast}_{\C}$,
\begin{equation}\label{defn-hR-upperbds}
H_{1,R}^{\delta,\mu}(\nu)=\sum_{\substack{M'\supseteq M_{\mu}\\ \|\mu^{M'}\|\leqslant B}}\bigg(\sum_{w\in W(A_M)_\delta}h\big(\tfrac12R(w\nu-\mu_{M'})\big)\bigg)^2.
\end{equation}
As in the first case we conclude that $H_{1,R}^{\delta,\mu}$ is $W(A_M)_{\delta}$-invariant
and that the inner sum in \eqref{defn-hR-upperbds} is real valued on the $\delta$-Hermitian spectrum, whence $H_{1,R}^{\delta,\mu}\geqslant 0$ on $\h_{\delta,{\rm hm}}^*$. This establishes (\ref{one-prime}$'$), while (\ref{two-prime}$'$) is immediate by expanding and grouping \eqref{defn-hR-upperbds}.

For the proof of (\ref{three-prime}$'$), we shall show that for all $\nu\in\mathfrak{h}_{\delta,{\rm un}}^\ast$ verifying \eqref{range} there is $M'\supset M_\mu$ (depending on $\nu$) such that $\|\mu^{M'}\|\leqslant B$ and
\begin{equation}\label{guy}
\bigg(\sum_{w\in W(A_M)_\delta}h\big(\tfrac12 R(w\nu-\mu_{M'})\big)\bigg)^2\geqslant Ce^{c R\|{\rm Re}\,\nu\|}.
\end{equation}
Dropping the other terms by positivity yields the lemma.

To prove \eqref{guy} we let $\nu\in\mathfrak{h}^*_{\delta,{\rm un}}$ satisfy \eqref{range} and apply Lemma \ref{nightmare} below; this gives rise to an $M'$ satisfying the indicated properties. Recall from \eqref{relative-Weyl} the definition of $W(A^{M'}_M)_\delta$, and write $W_{\rm bad}$ for the complementary set $W(A_M)_\delta\setminus W(A^{M'}_M)_\delta$. Then
\begin{equation}
\begin{aligned}\label{hello}
\bigg|\sum_{w\in W(A_M)_\delta}&h\big(\tfrac12R(w\nu-\mu_{M'})\big)\bigg|\\
&\geqslant \bigg|\sum_{w\in W(A^{M'}_M)_\delta}h\big(\tfrac12R(w\nu-\mu_{M'}))\bigg|-|W_{\rm bad}|\max_{w\in W_{\rm bad}}|h\big(\tfrac12R(w\nu-\mu_{M'})\big)|.
\end{aligned}
\end{equation}
Note that $W(A_M)_\delta\subset {\rm O}(\aa_M^G,\langle,\rangle)$. By property \eqref{asymp1} of Lemma \ref{asymp}, the fact that $M'\supseteq M_\mu$, and the definition of $W(A^{M'}_M)_\delta$, we have 
\begin{equation}\label{WMsum}
\bigg|\sum_{w\in W(A^{M'}_M)_\delta}h\big(\tfrac12R(w\nu-\mu_{M'})\big)\bigg|=|W(A^{M'}_M)_\delta|\big|h\big(\tfrac12R(\nu-\mu_{M'})\big)\big|\geqslant \big|h\big(\tfrac12R(\nu-\mu_{M'})\big)\big|.
\end{equation}

It now suffices to establish an upper bound for the second term in \eqref{hello}. For this, we shall avail ourselves of Lemmata~\ref{nightmare} and \ref{NSA}, along with the remaining properties \eqref{asymp2} and \eqref{asymp3} of Lemma \ref{asymp}. Note that \eqref{bad-guy-range} of Lemma~\ref{nightmare} implies that $w\nu-\mu_{M'}$, for $w\in W_{\rm bad}$, satisfies the first inequality on the left-hand side of \eqref{with-modulus}, with $\kappa=\kappa_{r_{M'}}$ (and $r_{M'}=\dim\h_{M'}$).
Similarly, \eqref{good-guy-range} implies that $\nu-\mu_{M'}$ satisfies the inequality on the right-hand side of \eqref{with-modulus}, with $\eta=\eta_{r_{M'}}$. Moreover, \eqref{good-guy-range}, \eqref{bad-guy-range}, and property~\eqref{thing2} of Lemma~\ref{nightmare} show that the second statement of Lemma \ref{NSA} applies
and that $w\nu-\mu_{M'}$ for $w\in W_{\rm bad}$ also satisfies the second (angle) inequality on the left-hand side of \eqref{with-modulus}. These observations, together with \eqref{range}, allow us to apply property \eqref{asymp3} of Lemma \ref{asymp}. Recalling the value of $\epsilon$ from Lemma~\ref{nightmare} \eqref{thing2}, we deduce that
\begin{equation}\label{2parts2}
|W_{\rm bad}|\max_{w\in W_{\rm bad}}\big|h\big(\tfrac12R(w\nu-\mu_{M'})\big)\big|\leqslant 
\frac12  \big|h\big(\tfrac12R(\nu-\mu_{M'})\big)\big|.
\end{equation}
Inserting \eqref{WMsum} and \eqref{2parts2} into \eqref{hello} yields
\[
\bigg|\sum_{w\in W(A_M)_\delta}h\big(\tfrac12R(w\nu-\mu_{M'})\big)\bigg|\geqslant \frac12 \big|h\big(\tfrac12R(\nu-\mu_{M'})\big)\big|.
\]
From this and property \eqref{asymp2} of Lemma \ref{asymp}, with $\eta=\eta_{r_{M'}}$, the lower bound \eqref{guy} follows.
\end{proof}

\subsection{Good and bad Weyl group elements}\label{good-bad-subsec}

The proof of Lemma~\ref{Lem1} rests on three key ingredients, the geometric Lemmata~\ref{nightmare} and \ref{NSA} and the analytic Lemma~\ref{asymp}, which we proceed to prove in order. Moreover, Lemma~\ref{Lem1} produces spectral localizing functions in their factorable form \eqref{HRdeltamu-factor} in connection with the factorability condition in Definition~\ref{defn-localizer} of spectral localizers, which in turn is introduced purely for purposes of verifying the (ELM) property in Section \ref{PW}.

With this in mind, continuing the convention from the proof of Lemma~\ref{Lem1}, \textit{for the rest of sections \S\ref{good-bad-subsec} and \S\ref{angles-subsec} as well as throughout Section \ref{sec:Paley-Wiener}, we work with local groups and spaces $G=G_v,M=M_v,\h_M=\h_{M_v}=\aa_{M_v}^{G_v},\dots$ and consistently drop the $_v$ and $^0$ decorations.} However we note that all statements and proofs in these sections are true verbatim also for $G=G_{\infty},M=M_{\infty},\h_M=\h_{M_{\infty}},\dots$, and that those statements could just as well be used for a proof of Lemma~\ref{Lem1} except for the factorability condition in Definition~\ref{defn-localizer}.

In this subsection, we prove a geometric lemma which was a crucial ingredient in the proof of Lemma \ref{Lem1}, and which in turn references statements of the forthcoming geometric Lemma~\ref{NSA} and analytic Lemma~\ref{asymp}.
For $M\in\mathcal{L}$ we let $r_M=\dim\h_M$; we also write $r_0=r_{T_0}$. (Thus, $r_G=0$ with our local convention; though we keep the notation $r_G$ for compatibility with the $G_{\infty}$ case.)

\begin{lemma}\label{nightmare}
Let $\underline{\delta}\in\mathcal{D}$ have standard representative $(M,\delta)$. Let $\mu\in i\h_M^*$ be contained in a $\delta$-singular subspace $i\h_{M_\mu}^*$ for some $M_\mu\in\mathcal{L}(\delta)$ strictly containing $M$. Let the constants $\theta_0,A_0>0$ be as in the second part of Lemma~\ref{NSA}.

There are constants $A,B>0$ and two finite systems $(\eta_i)_{i=r_G}^{r_{M_{\mu}}}$, $(\kappa_i)_{i=r_G+1}^{r_{M_{\mu}}}$ of positive constants, depending only on $G$, satisfying the following properties:
\begin{enumerate}
\item\label{thing1} for every $i=r_G,\dots,r_{M_\mu}$ the constants $(\eta_i,A)$ verify Property~\eqref{asymp2} of Lemma \ref{asymp};
\item\label{thing2} for every $i=r_G+1,\ldots ,r_{M_\mu}$ the constants $(\kappa_i,\eta_i,A)$ satisfy $\kappa_i/\eta_i\geqslant A_0$ and verify Property~\eqref{asymp3} of Lemma \ref{asymp} with $\epsilon=\frac12 |W(A_M)_\delta|^{-1}$ and $\theta_0$ as fixed above;
\item\label{thing3} for all $\nu\in\mathfrak{h}^\ast_{\delta,{\rm un}}$ verifying \eqref{range} there is $M'\in\mathcal{L}(\delta)$ containing $M_\mu$ (depending on $\mu$ and $\nu$) such that $\|\mu^{M'}\|\leqslant B$ and
\begin{align}
\| i\mIm\nu-\mu_{M'}\|&\leqslant\eta_{r_{M'}}\|\mRe\nu\|\label{good-guy-range}\\
\| w \cdot i\mIm\nu - \mu_{M'} \| &> \kappa_{r_{M'}} \| \mRe \nu \|\quad (M'\neq G,\,\,w\in W(A_M)_\delta\setminus W(A^{M'}_M)_\delta),\label{bad-guy-range}
\end{align}
where $W(A^{M'}_M)_\delta$ is as in \eqref{relative-Weyl}.
\end{enumerate}
\end{lemma}

In item \eqref{thing1}, we mean that Property~\eqref{asymp2} of Lemma~\ref{asymp} holds with constants $A,B,\eta_i,c>0$, for a suitable choice of $B,c>0$. Similarly, in item~\eqref{thing2}, with the specified $\epsilon,\theta_0>0$, it is meant that $\eta_i,A$ are suitable constants to be returned by Property~\eqref{asymp3} of Lemma~\ref{asymp} for the choice of inputs $\epsilon,\kappa_i,\theta_0>0$.

To deconstruct what Lemma~\ref{nightmare} is saying, consider that Lemma~\ref{Lem1} provides for a combination $H_R^{\delta,\mu}$ of spectral localizers that suitably detect $\nu\in\mathfrak{h}_{\delta,\mathrm{un}}^*$ with $i\mIm\nu$ close to $\mu\in i\mathfrak{h}_M^*$, and which are in turn built out of averages over Weyl group translates of the form $h(R(w\nu-\mu))$ as in \eqref{defn-hR-upperbds}.

For $\mu$ (and thus $\mathrm{Im}\,\nu$) close to $i\mathfrak{h}_G^*$, all occurring $h(R(w\nu-\mu))$ can be made large; in fact, $\nu$ can then be at least as well detected by re-centering the spectral localizer from $\mu$ to $\mu_G$. For $\mu$ and $\nu\in\mathfrak{h}_{\delta,\mathrm{un}}^*$ in a more generic position, one faces terms with very different sizes of $\mathrm{Im}(w\nu-\mu)$ and asymptotics of $h(R(w\nu-\mu))$ (cf. Remark~\ref{SPL-remark}), and a way is needed to clearly distinguish between the ``good'' and ``bad'' Weyl elements $w$ according to the size of $\|w\cdot i\mIm\nu-\mu\|$. This difficulty is how to make such a discrete distinction for any given $\mu$;  certainly it cannot be made in a continuous fashion as $\mu$ varies.

To resolve this problem, the basic idea is that, in the intermediate range, when $\|w\cdot i\mIm\nu-\mu\|$ is sufficiently small to prevent $w$ from being ``bad'' yet not sufficiently small to make $w$ ``good'', then both $i\mIm\nu$ and $\mu\in i\mathfrak{h}_M^*$ must be reasonably close to $i\mathfrak{h}_{M,w,+1}^*$, and $\nu$ might in fact be better detected by re-centering to a (nearby) more singular point such as $\mu_{M_w}$. Such re-centering, though, affects whether other $w$ should be labeled good or bad.

Lemma~\ref{nightmare} makes this descent precise and provides, for each singular $\mu\in i\mathfrak{h}_M^*$ and a nearby $\nu\in\mathfrak{h}_{\delta,\mathrm{un}}^*$, a specific choice of $M'\in\mathcal{L}(\delta)$ such that, re-centering the localizer to the nearby $\mu_{M'}$, all Weyl translates are either ``clearly good'' or ``clearly bad'' (all are ``clearly good'' if $M'=G$). We construct tubular neighborhoods around each $i\mathfrak{h}_{M'}^*$ so that each $\nu\in\mathfrak{h}_{\delta,\mathrm{un}}^*$ can be detected by re-centering $\mu$ to $\mu_{M'}$ for the largest $M'$ to whose corresponding tube it belongs.

An inductive scheme of this broad nature, though in the absence of exponentially increasing weights and on a fixed scale, was used first in the foundational work of Duistermaat--Kolk--Varadarajan~\cite[Proposition 7.1]{DuistermaatKolkVaradarajan1979} and then by Lapid--M\"uller~\cite[Proposition 4.5]{LapidMuller2009}. The relevance of the specific scale in \eqref{good-guy-range} and \eqref{bad-guy-range} is transparent on comparing with \eqref{asymp2-eq} and \eqref{with-modulus} in Lemma~\ref{asymp}.

\begin{proof}
The values of $A>0$ and the system $(\kappa_i,\eta_i)$ depend on fixed choices of constants depending only on $G$ which we now specify:
\begin{enumerate}
\item[$\bullet$] For $L\in\mathcal{L}(\delta)$ and $\eta>0$ let
$\mathcal{T}_L(\eta)=\{\nu\in i\h_M^{\ast}:\|\nu^L\|\leqslant\eta\}$
denote the tube of radius $\eta$ about $i\h_L^*$ inside $i\h_M^*$. Let $0<\tilde{\eta}_{r_0}<\tilde{\eta}_{r_0+1}<\cdots <\tilde{\eta}_{r_G}$ be a fixed system of radii such that for any $L_1,L_2\in\mathcal{L}(\delta)$ one has
\begin{equation}
\label{TubeIntersectionProperty}
\mathcal{T}_{L_1}(\tilde\eta_{r_{L_1}})\cap\mathcal{T}_{L_2}(\tilde\eta_{r_{L_2}})\subset \mathcal{T}_L(\tilde\eta_{r_L}),
\end{equation}
where $L=\langle L_1,L_2\rangle\in\mathcal{L}(\delta)$ is generated by $L_1$ and $L_2$.
It is not hard to see that the property \eqref{TubeIntersectionProperty} then also holds for any other system of radii $(\eta_i)_{i=1}^{r_G}$ satisfying $\eta_i/\tilde\eta_i\leqslant\eta_j/\tilde\eta_j$ for every $i>j$. See Figure~\ref{TubularFigure}.

\item[$\bullet$] There exists a constant $C\geqslant 1$ such that for all $L,L'\in\mathcal{L}(\delta)$ and all $\mu\in i\mathfrak{h}_{L'}^{\ast}$ one has
\begin{equation}\label{LM-geometric}
\|\mu^{\langle L,L'\rangle}\|\leqslant C\|\mu^L\|.
\end{equation}
\end{enumerate}
To see the second point (which can be extracted from \cite[p. 136]{LapidMuller2009}), note that it suffices to verify \eqref{LM-geometric} for fixed $L,L'\in\mathcal{L}(\delta)$ and for $\mu\in i(\h_{L'}^{\langle L,L'\rangle})^{\ast}$, and indeed only for nonzero (and hence only for norm one) such $\mu$.
The statement then follows from the continuity of $\mu\mapsto\|\mu\|/\|\mu^L\|$ on the compact 1-ball in $i(\h_{L'}^{\langle L,L'\rangle})^{\ast}$.

\begin{figure}
\centering
\begin{overpic}[width=0.5\linewidth]{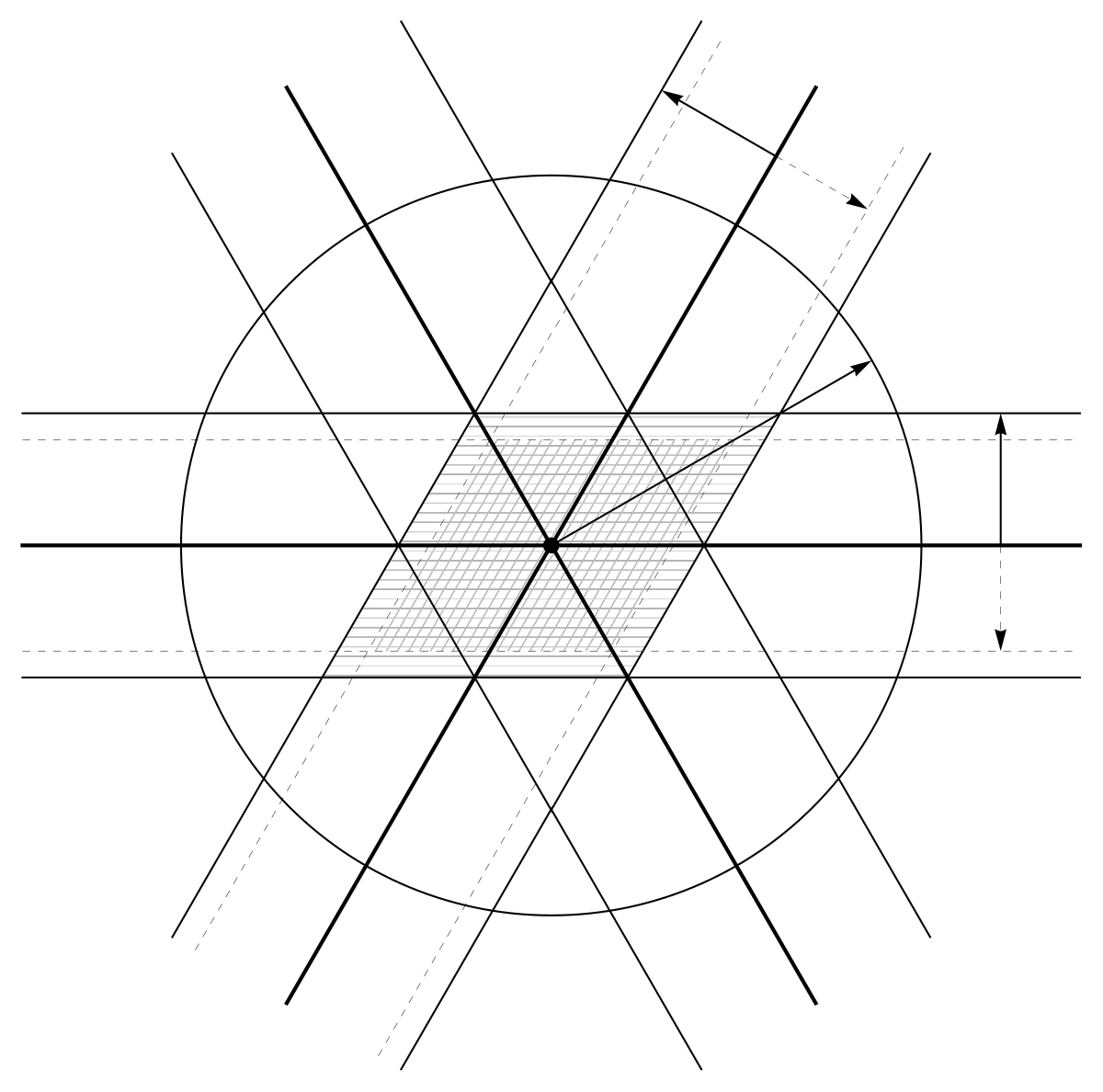}
\put(53,46){$\scriptstyle i\h_G^{\ast}$}
\put(4,46){$\scriptstyle i\h_{L_1}^{\ast}$}
\put(30,10){$\scriptstyle i\h_{L_2}^{\ast}$}
\put(30,88){$\scriptstyle i\h_{L_3}^{\ast}$}
\put(90,24){$\color{gray}{\scriptstyle i\h_{T_0}^{\ast}}$}
\put(92,54){$\scriptstyle\tilde{\eta}_1$}
\put(92,44){$\scriptstyle\eta_1$}
\put(66,89){$\scriptstyle\tilde{\eta}_1$}
\put(75,84){$\scriptstyle\eta_1$}
\put(72,58){$\scriptstyle\tilde{\eta}_2$}
\end{overpic}
\caption{\footnotesize The system of subspaces $i\h_L^{\ast}$ (three 1-dimensional, and one 2- and 0-dimensional each) corresponding to the five semistandard Levi subgroups $L\in\mathcal{L}_{\infty}$ of $\SL_3(\mathbb{R})$, and their corresponding tubular neighborhoods $\mathcal{T}_L(\tilde{\eta}_{r_L})$. Also depicted are the intersection property \eqref{TubeIntersectionProperty} (with $L=\langle L_1,L_2\rangle=G$, and with the intersection on the left-hand side shaded) and its preservation when passing to radii $(\eta_1,\tilde{\eta}_2)$ with $\eta_1<\tilde{\eta}_1$.}
\label{TubularFigure}
\end{figure}

We first set $\eta_{r_G}$ and $A_{r_G}$ to be values of $\eta$ and $A$ for which property~\eqref{asymp2} of Lemma \ref{asymp} hold. (Note that this property is invoked just once and for all. Also note that in the base case $M_{\mu}=G$, $M'=G$ in \eqref{thing3}, $\mu\in i\mathfrak{h}_G^{\ast}$, and $W(A_M^G)_{\delta}=W(A_M)_{\delta}$.)

The remaining constants indexed by $i=r_G+1,\dots,r_{M_\mu}$ will be defined by induction on $i$. We set $c_i=\min_{j<i}\eta_j/\tilde\eta_j$ and then take any $\kappa_i$ satisfying $0<\kappa_i<\frac{c_i\tilde\eta_{r_0}}{8|W(A_M)_\delta|C}$.

Applying Property~\eqref{asymp3} of Lemma \ref{asymp} with $\kappa=\kappa_i$, $\epsilon=\frac12|W(A_M)_\delta|^{-1}$, and $\theta_0$ yields constants $\eta_i'$ and $A_i$. We then set $\eta_i=\min\big\{\eta_i',\frac{\kappa_i}{A_0+1/2}\big\}$. Since $\eta_i/\tilde\eta_i\leqslant\eta_j/\tilde\eta_j$ for $i>j$,
\eqref{TubeIntersectionProperty} holds with the system of $\eta_i$'s in place of $\tilde\eta_i$.
We record that (particularly in view of $\eta_i\leqslant\min(\eta'_i,\eta_{r_G})$) the constants $(\eta_i,A)$ and $(\eta_i,\kappa_i,A)$ verify Properties~\eqref{asymp2} and \eqref{asymp3} of Lemma~\ref{asymp} (with $\epsilon=\frac12|W(A_M)_\delta|^{-1}$ and $\theta_0$) for every $A\geqslant\max(A_i,A_{r_G})$.

Finally we put $A=\max_i\{A_i,2\eta_i^{-1}\}$ and $B=\frac12\eta_{r_G}\|\rho\|$. Recall from \S\ref{sec:herm} that the half-sum of positive roots $\rho$ bounds the real parts in the unitary dual. Our construction of constants $A$, $B$, $(\eta_i)$ and $(\kappa_i)$ ensures that properties \eqref{thing1} and \eqref{thing2} of Lemma~\ref{nightmare} hold; it remains to verify \eqref{thing3}.

For $\nu\in\h_{\delta,{\rm hm}}^*$ satisfying \eqref{range}, let $M'\in\mathcal{L}(\delta)$ containing $M_\mu$ be maximal for the property
\begin{equation}\label{defM'}
\|\mu^{M'}\|\leqslant \frac12\eta_{r_{M'}}\|{\rm Re}\,\nu\|.
\end{equation}
This is well-defined, since if $M_1$ and $M_2$ satisfy this bound, then so does $\langle M_1,M_2\rangle \in\mathcal{L}(\delta)$. Moreover $\|\mu^{M'}\|\leqslant B$ is immediate.

Using $\eqref{range}$ we have that $\|i\mIm\nu-\mu\| \leqslant R^{-1}\leqslant A^{-1}\|\mRe\nu\|$, and so
\begin{equation}\label{funpart}
\| i\mIm\nu-\mu_{M'}\|\leqslant
\|i\mIm\nu-\mu\|+\|\mu^{M'}\|
\leqslant (\frac12\eta_{r_{M'}}+A^{-1})\|{\rm Re}\,\nu\|\leqslant \eta_{r_{M'}}\|{\rm Re}\,\nu\|. \end{equation}
This proves the upper bound \eqref{good-guy-range}. 
 
We proceed to some preliminary estimates toward \eqref{bad-guy-range}. We claim that for all $L\in\mathcal{L}(\delta)$, not contained in $M'$, we have
\begin{equation}\label{warm-up}
\|\mu^L\|>\frac{c_{r_{M'}} \tilde\eta_{r_0}}{2C}\|{\rm Re}\, \nu\|.
\end{equation}
Assuming otherwise, we apply the inequality \eqref{LM-geometric} with $L'=M_\mu$ to obtain the upper bound $\|\mu^{\langle L,M_\mu\rangle}\|\leqslant  \frac12 c_{r_{M'}} \tilde\eta_{r_{\langle L,M_\mu\rangle}}\|{\rm Re}\, \nu\|$. Now from \eqref{defM'} and using $\eta_{r_{M'}}\leqslant c_{r_{M'}}\tilde\eta_{r_{M'}}$ we also have $\|\mu^{M'}\|\leqslant \frac12c_{r_{M'}}\tilde\eta_{r_{M'}}\|{\rm Re}\,\nu\|$. Setting $M''=\langle M',\langle L,M_\mu\rangle\rangle=\langle L,M'\rangle$, the compatibility of the constants $\tilde{\eta_i}$ then shows that $\|\mu^{M''}\|\leqslant \frac12c_{r_{M'}}\tilde\eta_{r_{M''}}\|{\rm Re}\,\nu\|$. Now since $L\not\subset M'$, the subgroup $M''$ is strictly larger than $M'$. Thus $c_{r_{M'}}\leqslant \eta_{r_{M''}}/\tilde\eta_{r_{M''}}$ and $\|\mu^{M''}\|\leqslant \frac12\eta_{r_{M''}}\|{\rm Re}\,\nu\|$. But this contradicts the maximality of $M'$, establishing \eqref{warm-up}. Moreover, we may bootstrap \eqref{warm-up} to show that
\begin{equation}\label{bootstrap}
\| (\mu_{M'})^L\|\geqslant  \frac{3c_{r_{M'}} \tilde\eta_{r_0}}{8C}\|\mRe\nu\|
\end{equation}
for all $L\in\mathcal{L}(\delta)$ not contained in $M'$. This can be seen from
\[
\|\mu^L-(\mu_{M'})^L\|=\|(\mu-\mu_{M'})^L\|=\|(\mu^{M'})^L\|\leqslant \|\mu^{M'}\|\leqslant \frac12\eta_{r_{M'}}\|{\rm Re}\,\nu\|\leqslant \frac{c_{r_{M'}}\tilde\eta_{r_0}}{8C} \|{\rm Re}\,\nu\|
\]
along with the triangle inequality.

We can now prove \eqref{bad-guy-range}. Arguing by contradiction, we let $w\in W(A_M)_\delta\setminus W(A^{M'}_M)_\delta$ and suppose that $\| w\cdot i\mIm\nu - \mu_{M'} \|\leqslant \kappa_{r_{M'}}\| \mRe \nu \|$. Then, using this and \eqref{funpart}, we get
\[ \| \mu_{M'} - w\mu_{M'} \|
\leqslant \| \mu_{M'} - w\cdot i\mIm\nu \| + \| w\cdot(i\mIm\nu - \mu_{M'} )\|
\leqslant (\kappa_{r_{M'}}+ \eta_{r_{M'}})\| \mRe\nu \|. \]
From this it follows by induction on $k$ that $\| \mu_{M'} - w^k\mu_{M'} \| \leqslant k(\kappa_{r_{M'}}+ \eta_{r_{M'}}) \| \mRe\nu \|$. From this and the expression $(\mu_{M'})_{M_w}=|W(A_M)_\delta|^{-1}\sum_{k=1}^{|W(A_M)_\delta|}w^k\mu_{M'}$, we conclude that 
\[
\|(\mu_{M'})^{M_w} \| = \| \mu_{M'}- (\mu_{M'})_{M_w} \| \leqslant |W(A_M)_\delta|(\kappa_{r_{M'}}+ \eta_{r_{M'}}) \| \mRe\nu \|<(\frac18+\frac14)\frac{c_{r_{M'}} \tilde\eta_{r_0}}{C}\| \mRe\nu \|.
\]
Since $w\not\in W(A^{M'}_M)_\delta$, $M_w\not\subset M'$, we may now apply \eqref{bootstrap} with $L=M_w$ to get a contradiction. 
\end{proof}

\subsection{Angles}\label{angles-subsec}

\emph{We keep the local notational conventions from \S\ref{good-bad-subsec}.}
The following result provides geometric information on vectors $\nu\in\mathfrak{h}_{M,\mathbb{C}}^{\ast}$ which will be used as inputs in our application of Lemma~\ref{asymp}. Put briefly, it says that the real and imaginary parts of the argument of $h$ on the left-hand side of \eqref{2parts2} are bounded away from being parallel. To quantify this, here and in Section~\ref{sec:Paley-Wiener} we will write $\angle(\nu_1,\nu_2)$ for the unoriented positive angle between the subspaces spanned by $\nu_1,\nu_2\in\h^{\ast}_M$ (that is, the shortest positive distance from the angle between $\nu_1$ and $\nu_2$ to $\pi\mathbb{Z}$).

\begin{lemma}\label{NSA}
There exists $\theta_0>0$ (depending only on $M$) such that, for every $\nu=\mathrm{Re}\,\nu+i\mIm\nu\in\mathfrak{h}^{\ast}_{\delta,\mathrm{hm}}\setminus i\mathfrak{h}^{\ast}_M$ and every $w\in W(A_M)_{\delta}$ not fixing $\nu$, we have $\angle(\mathrm{Re}\,\nu,w\mathrm{Im}\,\nu-\mathrm{Im}\,\nu)\geqslant\theta_0$.

In particular, there are $\theta_0,A_0>0$ such that, for every $\nu\in\mathfrak{h}^{\ast}_{\delta,\mathrm{hm}}\setminus i\mathfrak{h}^{\ast}_M$, $\mu\in i\mathfrak{h}^{\ast}_M$, and $w\in W(A_M)_{\delta}$ satisfying $\|w\cdot i\mIm\nu-\mu\|\geqslant A_0\| i\mIm\nu-\mu\|$, $\nu_1=w\nu-\mu$ satisfies $\angle(\mathrm{Re}\,\nu_1,\mathrm{Im}\,\nu_1)\geqslant\theta_0$.
\end{lemma}

\begin{proof}
As in Section~\ref{sec:arch-rep-PWCD} we may use the standard basis to coordinatize $\mathfrak{a}^{\ast}_M$ so that $\mathfrak{a}_M^{\ast}$ and $\mathfrak{h}_M^{\ast}$ are identified with $\mathbb{R}^{\dim\mathfrak{a}_M}$ and its trace zero subspace, respectively, and so that
$W(A_M)_\delta$ acts on $\mathfrak{h}_{M,\C}^*$ by permuting its coordinates: $w\nu=w(\nu_i)=(\nu_{w(i)})$.

Recalling the notations of \S\ref{sec:herm}, we claim that, for arbitrary $w_0,w\in W(A_M)_{\delta}$,
\begin{equation}
\label{ToProve-Transverse}
\mathfrak{h}^{\ast}_{M,w_0,-1}\cap (w-\mathrm{id})\mathfrak{h}^{\ast}_{M,w_0,+1}=\{0\}.
\end{equation}
Indeed, suppose, contrary to \eqref{ToProve-Transverse}, that $\mathrm{Im}\,\nu\in\mathfrak{h}^{\ast}_{M,w_0,+1}$ is such that $0\neq w\mathrm{Im}\,\nu-\mathrm{Im}\,\nu\in\mathfrak{h}^{\ast}_{M,w_0,-1}$. Let $S=S(\mathrm{Im}\,\nu)$ be the subset of indices $i\leqslant n$ of $\mathrm{Im}\,\nu$ such that $\{\mathrm{Im}\,\nu_{w^k(i)}:k\in\mathbb{Z}\}$  is not reduced to $\{\mathrm{Im}\,\nu_i\}$. From $w\mathrm{Im}\,\nu\neq\mathrm{Im}\,\nu$, it is clear that $S$ is non-empty. We may choose $i_0\in S$ such that $\mathrm{Im}\,\nu_{i_0}\geqslant\mathrm{Im}\,\nu_i$ for all $i\in S$ and $\mathrm{Im}\,\nu_{i_0}>\mathrm{Im}\,\nu_{w(i_0)}$. Let $j_0=w_0(i_0)$. Then $\mathrm{Im}\,\nu_{j_0}=\mathrm{Im}\,\nu_{i_0}$ since $\mathrm{Im}\,\nu\in\mathfrak{h}_{M,w_0,+1}^*$, and $\mathrm{Im}\,\nu_{w(j_0)}-\mathrm{Im}\,\nu_{j_0}=-(\mathrm{Im}\,\nu_{w(i_0)}-\mathrm{Im}\,\nu_{i_0})$ since $w\mathrm{Im}\,\nu-\mathrm{Im}\,\nu\in\mathfrak{h}_{M,w_0,-1}^*$. We find that $\mathrm{Im}\,\nu_{w(j_0)}=2\mathrm{Im}\,\nu_{i_0}-\mathrm{Im}\,\nu_{w(i_0)}>\mathrm{Im}\,\nu_{i_0}=\mathrm{Im}\,\nu_{j_0}$. This, in particular, shows that $j_0\in S$, and by the $w$-invariance of $S$ we obtain $w(j_0)\in S$. But then $\mathrm{Im}\,\nu_{w(j_0)}>\mathrm{Im}\,\nu_{i_0}$ contradicts the maximality of $\mathrm{Im}\,\nu_{i_0}$.

It follows from \eqref{ToProve-Transverse} that there exists an angle $\theta_0(w,w_0)>0$ such that $\angle(\mathrm{Re}\,\nu,(w-\mathrm{id})\mathrm{Im}\,\nu)\geqslant\theta_0(w,w_0)$ for every $0\neq\mathrm{Re}\,\nu\in\mathfrak{h}^{\ast}_{M,w_0,-1}$ and every $\mathrm{Im}\,\nu\in\mathfrak{h}^{\ast}_{M,w_0,+1}$ with $w\mathrm{Im}\,\nu\neq\mathrm{Im}\,\nu$; indeed, angles between unit vectors in the two subspaces $\mathfrak{h}^{\ast}_{M,w_0,-1}$ and $(w-\mathrm{id})\mathfrak{h}^{\ast}_{M,w_0,+1}$ are strictly positive and we may by compactness take the smallest one, with the statement vacuously true if one of these subspaces is $\{0\}$.

Recall now the description of $\mathfrak{h}^{\ast}_{\delta,\mathrm{hm}}$ in \eqref{delta-hm}, along with the decomposition $\mathfrak{h}_{M,w}^*=\mathfrak{h}_{M,w,-1}^*+i\mathfrak{h}_{M,w,+1}^*$. The first claim of the lemma now follows by taking $\theta_0>0$ to be the minimum value of $\theta_0(w,w_0)$, as $M$ and $w,w_0\in W(A_M)_{\delta}$ vary over the finitely many choices.

For the second claim, let $\nu=\mathrm{Re}\,\nu+i\mathrm{Im}\,\nu$, and note that 
\[
\angle( {\rm Re}\, \nu_1, {\rm Im}\, \nu_1-({\rm Im}\,\nu-\mu/i))=\angle(w\mathrm{Re}\,\nu, w\mathrm{Im}\,\nu -\mathrm{Im}\,\nu)=\angle(\mathrm{Re}\,\nu, w^{-1}\mathrm{Im}\,\nu-\mathrm{Im}\,\nu).
\]
Now $\nu\in\mathfrak{h}^{\ast}_{\delta,\mathrm{hm}}\setminus i\mathfrak{h}^{\ast}_M$ and $w^{-1}$ cannot fix $\mathrm{Im}\,\nu$, so that $\angle(\mathrm{Re}\,\nu, w^{-1}\mathrm{Im}\,\nu-\mathrm{Im}\,\nu)\geqslant\theta_0$ by the first claim. Since $\|{\rm Im}\,\nu_1\|\geqslant A_0\|({\rm Im}\,\nu-\mu/i)\|$, the second claim follows by taking $A_0=1/\sin\frac12\theta_0$ and noting that then $\angle({\rm Re}\,\nu_1,{\rm Im}\,\nu_1)\geqslant\theta_0-\arcsin(A_0^{-1})\geqslant\frac12\theta_0$.
\end{proof}

\section{A Paley--Wiener function}\label{sec:Paley-Wiener}

\emph{Throughout this section, we keep the local notational conventions from \S\ref{good-bad-subsec}.}
The principal result of this section is Lemma~\ref{asymp}. This technical lemma provides for Paley--Wiener functions $h(\nu)$ with desirable asymptotics that are crucial in the proof of Lemma~\ref{Lem1}. 

The proof of Lemma~\ref{asymp} is based on the principle of stationary phase, and in particular derives inspiration from standard treatments of the Fourier transform of the uniform measure on the round sphere. Our setting is non-standard (relative to the existing literature) due the presence of a complex phase, which leads to an integral transform \emph{a priori} interpolating between Laplace- and Fourier-type integrals with competing exponential and oscillatory behavior in two independent (but not necessarily orthogonal) directions.

\subsection{Construction of Paley-Wiener function}

In the proof of the crucial Lemma~\ref{Lem1}, essential use is made of a Paley--Wiener function $h$ on $\mathfrak{h}_{M,\C}^{\ast}$ with properties that ensure that the average in \eqref{defn-hR-upperbds} detects the complementary spectrum with suitable exponential weights. Lemma~\ref{asymp}, which is of purely analytic nature, spells out these properties and constructs such $h(\nu)$. 

\newcommand\minstar{\mathop{\min\nolimits^*}}
\newcommand\maxstar{\mathop{\max\nolimits^*}}

\begin{lemma}\label{asymp}
Let $M\in\mathcal{L}$. There is a real-valued $g\in C^\infty_c(\h_M)$ whose Fourier transform $h\in\mathcal{PW}(\h_{M,\C}^*)_1$ satisfies:
\begin{enumerate}
\item\label{asymp1} $h(k\nu)=h(\nu)$ for all $k\in {\rm O}(\h_M,\langle,\rangle)$ and all $\nu\in\h_{M,\C}^*$;
\item\label{asymp2} there are constants $A,B,\eta,c>0$ such that for $R\geqslant 1$, $\sigma\geqslant AR^{-1}$, 
\begin{equation} \label{asymp2-eq}
\mathop{\min}_{\substack{\|{\rm Im}\,\nu\|\leqslant \eta\sigma\\ \|{\rm Re}\,\nu\|=\sigma}} |h(R\nu)|\geqslant Be^{cR\sigma};
\end{equation}
\item\label{asymp3} for every $\epsilon,\kappa,\theta_0>0$ there is $0<\eta\leqslant 1$ (depending only on $\kappa$, $\theta_0$) and $A>1$ (depending on $\epsilon$, $\kappa$, $\theta_0$) such that for $R\geqslant 1$, $\sigma\geqslant AR^{-1}$,
we have
\begin{equation}\label{with-modulus}
\mathop{\max}_{\substack{\|{\rm Im}\,\nu\|> \kappa\sigma,\,\|{\rm Re}\,\nu\|=\sigma,\\ \angle(\mathrm{Re}\,\nu,\mathrm{Im}\,\nu)\geqslant\theta_0}} |h(R\nu)|
\leqslant \epsilon \mathop{\min}_{\substack{\|{\rm Im}\,\nu\|\leqslant \eta\sigma\\\|{\rm Re}\,\nu\|=\sigma}} |h(R\nu)|.
\end{equation}
\end{enumerate}
\end{lemma}

\begin{remark}
\label{SPL-remark}
Weyl group invariance~\eqref{asymp1} is required to guarantee non-negativity in \eqref{defn-hR-upperbds}.

Properties \eqref{asymp2} and \eqref{asymp3} require perhaps
greater explanation. 
As a Paley-Wiener function of exponential type $R$, $h(R\nu)$ exhibits exponential growth as high as $e^{(R+o(1))\|\mathrm{Re}\,\nu\|}$ in the non-tempered directions and rapid decay along the tempered subspace as soon as $\|{\rm Im}\,\nu\|\gg 1/R$. Properties \eqref{asymp2} and \eqref{asymp3} express the delicate interplay between these two asymptotic behaviors. The angular condition in \eqref{asymp3} is used to show clear distinction in exponential behavior in \eqref{with-modulus}; cf.~\eqref{claim-1}--\eqref{claim-2}.

It might also assist the reader to recall how these two properties are used in the proof of Lemma \ref{Lem1}. Property \eqref{asymp2} is used to show that the ``good Weyl elements'' make a large contribution to the average in \eqref{defn-hR-upperbds}. These good elements, by Lemma \ref{nightmare}, have orbits whose tempered components remain in a small ball about the origin. Property \eqref{asymp3} is used to show that the ``bad Weyl elements'' are well-controlled; these (again by Lemma \ref{nightmare}) have orbits which lie far from the origin.
\end{remark}

\begin{proof}
Let $b\in C^\infty_c(\R)$ be the bump function equal to $e^{-1/(1-x^2)}$ in $[-1,1]$ and vanishing outside of this interval. For $H\in\h_M$, let $g(H)=b(\|H\|)$. In the purely archimedean-analytic Lemma~\ref{lem:SPL} below, which is also of independent interest, we shall obtain precise asymptotics of $h=\widehat{g}$ (critically as a Paley--Wiener function on all of $\h_{M,\C}^{\ast}$). Using these asymptotics, we shall now show that $g$ satisfies all properties of Lemma~\ref{asymp}.

It is clear that $h\in\mathcal{PW}(\h_{M,\C}^*)_1$ since $g$ is supported in the ball of radius $1$. To see property \eqref{asymp1}, identify $\h_M$ equipped with the non-degenerate bilinear form $\langle\,,\rangle$ of \S\ref{norms-weyl-groups-subsec} with a Euclidean $\mathbb{R}^d$, and, as above, let $\omega$ denote the surface measure of the unit sphere $S^{d-1}$ in $\mathbb{R}^d$. The Fourier transform at $\nu\in\C^d$ of the function $X\mapsto b(\|X\|)$ can be written as
\begin{equation}\label{radial}
\int_0^1 e^{-1/(1-r^2)}\widehat{{\rm d}\omega}(r\nu)r^{d-1}{\rm d}r.
\end{equation}
The ${\rm O}(d,\R)$-invariance of this expression then follows from that of $\omega$. This proves \eqref{asymp1}.

Statements \eqref{asymp2} and \eqref{asymp3} are consequences of Lemma \ref{lem:SPL} below, where the notation $\mathcal{N}(\nu)$ derives from \eqref{xi-xi}. Specifically, denoting $\|\mathrm{Re}\,\nu\|=\sigma$, $\|\mathrm{Im}\,\nu\|=t$, and $\angle(\mathrm{Re}\,\nu,\mathrm{Im}\,\nu)=\theta$, we will prove that:
\begin{itemize}
\item[--] for every $\epsilon>0$, there exist $\eta,A>0$ such that, whenever $\sigma\geqslant A$ and $t\leqslant\eta\sigma$,
\begin{equation}
\label{claim-1}
\mathrm{Re}\big(\mathcal{N}(\nu)-\sqrt{2\mathcal{N}(\nu)}\big)\geqslant (1-\epsilon)\sigma\quad\text{and}\quad |\widehat{g}(\nu)|\gg e^{(1-2\epsilon)\sigma};
\end{equation}
\item[--] for every $\kappa,\theta_0>0$, there exist $\delta,A>0$ such that, whenever $\sigma\geqslant A$, $t\geqslant\kappa\sigma$, and $\theta\geqslant\theta_0$,
\begin{equation}
\label{claim-2}
\mathrm{Re}\big(\mathcal{N}(\nu)-\sqrt{2\mathcal{N}(\nu)}\big)\leqslant (1-\delta)\sigma\quad\text{and}\quad |\widehat{g}(\nu)|\ll e^{(1-\delta)\sigma}.
\end{equation}
\end{itemize}
It is clear that \eqref{asymp2} and \eqref{asymp3} follow from the above statements. Indeed, let $B,B'>0$ be implied constants in the second estimates in \eqref{claim-1} and \eqref{claim-2}. Fixing any $0<\epsilon<\frac12$, \eqref{claim-1} yields $\eta,A>0$ such that for $R\sigma\geqslant A$,
\[ \min_{\substack{\|\mathrm{Im}(R\nu^0)\|\leqslant\eta(R\sigma)\\\|\mathrm{Re}(R\nu^0)\|=R\sigma}}|\widehat{g}(R\nu^0)|\geqslant Be^{cR\sigma} \]
with $c=1-2\epsilon>0$. Similarly, given $\epsilon,\kappa,\theta_0>0$, let $\delta, A'>0$ be values yielded by \eqref{claim-2}, let $\eta,A>0$ be values yielded by \eqref{claim-1} invoked with $\epsilon=\min(\frac14\delta,\frac14)$, and let $A''=\log(B''/B'\epsilon)/(\delta/2)$; then, for $R\sigma\geqslant\max(A,A',A'')$,
\[ \max_{\substack{\|\mathrm{Im}(R\nu)\|>\kappa(R\sigma),\,\|\mathrm{Re}(R\nu)\|=R\sigma\\\angle(\mathrm{Re}(R\nu),\mathrm{Im}(R\nu))\geqslant\theta_0}}|\widehat{g}(R\nu)|\leqslant \frac{B''}{B'} e^{-(\delta/2) R\sigma}\cdot B'e^{(1-2\epsilon)R\sigma}\leqslant\epsilon\min_{\substack{\|\mathrm{Im}(R\nu)\|\leqslant\eta(R\sigma)\\\|\mathrm{Re}(R\nu)\|=R\sigma}}|\widehat{g}(R\nu)|. \]

Writing $\langle\nu,\nu\rangle=u+iv$, we have that
\[ u=\sigma^2-t^2,\quad v=\pm 2\sigma t\cos\theta,\quad u^2+v^2=(\sigma^2+t^2)^2-4\sin^2\theta \sigma^2t^2, \]
as well as $\mathrm{Re}\,\mathcal{N}(\nu)=\mathrm{Re}\sqrt{u+iv}=\big(\frac12(u+\sqrt{u^2+v^2})\big)^{1/2}$. For $t\leqslant\eta\sigma$ with $\eta<1$, this gives
\[ \mathrm{Re}\big(\mathcal{N}(\nu)-\sqrt{2\mathcal{N}(\nu)}\big)\geqslant u^{1/2}-\sqrt2(u^2+v^2)^{1/8}\geqslant\sigma\sqrt{1-\eta^2}-2\sqrt{\sigma}. \]
This clearly implies the first claim in \eqref{claim-1}; the second claim in \eqref{claim-1} follows by invoking Lemma~\ref{lem:SPL} and noting that, for $\eta>0$ sufficiently small, $|\mathrm{arg}\,\mathcal{N}(\nu)|\leqslant\frac{\pi}3$ and $|\mathcal{N}(\nu)|^{-1}\geqslant e^{-\epsilon\sigma/d}$.

On the other hand, for $t\geqslant\kappa\sigma$ and $\theta\geqslant\theta_0$, we find, using the elementary inequality $\sqrt{1-x}\leqslant 1-\frac12x$ for $x\leqslant 1$, that
\[ \sqrt{u^2+v^2}=(\sigma^2+t^2)\sqrt{1-4\sin^2\theta\sigma^2t^2/(\sigma^2+t^2)^2}\leqslant\sigma^2+t^2-2\sigma^2\sin^2\theta_0\kappa^2/(1+\kappa^2), \]
and therefore (since the square root is the principal branch)
\[ \mathrm{Re}\big(\mathcal{N}(\nu)-\sqrt{2\mathcal{N}(\nu)}\big)\leqslant\mathrm{Re}\,\mathcal{N}(\nu)\leqslant\sigma\big(1-\sin^2\theta_0\kappa^2/(1+\kappa^2)\big)^{1/2}. \]
This estimate implies the first claim in \eqref{claim-2}; the second claim in \eqref{claim-2} then follows by Lemma~\ref{lem:SPL} by using simply $|\mathcal{N}(\nu)|\gg 1$.
\end{proof}

\subsection{(Complex) Fourier transform of the sphere}\label{sec:FT-sphere}

Consider $\R^d$ equipped with its standard inner product $\langle\,,\rangle$, and let $\omega$ denote the surface measure of the corresponding unit sphere $S^{d-1}$ in $\mathbb{R}^d$. We now collect some facts about the Fourier transform $\widehat{\mathrm{d}\omega}(\nu)$, for $\nu\in(\mathbb{C}^d)^{\ast}$.

When $\nu$ is either real or imaginary, then we may use the rotational invariance of $\omega$ to assume that $\nu=\nu_1$ points in the first standard coordinate direction. Then, using spherical coordinates, we obtain
\begin{equation}\label{omega-hat1}
\widehat{\mathrm{d}\omega}(\nu)=|\omega_{d-2}|\int_0^\pi e^{\nu_1 \cos\varphi} (\sin\varphi)^{d-2}\,\text{d}\varphi
=(2\pi)^{d/2}(i\nu_1)^{-d/2+1}J_{d/2-1}(i\nu_1),
\end{equation}
where $|\omega_{d-2}|=2\pi^{(d-1)/2}/\Gamma(\frac12(d-1))$ is the volume of the $(d-2)$-sphere, and $J_{d/2-1}(\nu)$ is the usual $J$-Bessel function~\cite[(10.9.4)]{NIST2010}. The above oscillatory integral has non-degenerate stationary points at $\varphi=0,\pi$. The classical method of stationary phase when $\nu_1$ is imaginary, or Laplace's method when $\nu_1$ is real, can then be used to obtain asymptotics of $\widehat{\mathrm{d}\omega}(\nu)$, which of course agree with the usual asymptotics for the $J$-Bessel function with large argument.

Returning to general $\nu\in\mathbb{C}^d$, extend $\langle \,,\rangle$ complex linearly to all of $(\mathbb{C}^d)^{\ast}$; we denote this extension by $\langle \,,\rangle_{\mathbb{C}}$. Thus, if $\nu\in(\mathbb{C}^d)^{\ast}$ is given by $\nu=(\nu_1,\ldots ,\nu_d)$ in the canonical orthonormal basis for $\langle\,,\rangle$, then $\langle \nu,\nu\rangle_{\mathbb{C}}=\sum_i \nu_i^2$. Moreover, if both ${\rm Re}\,\nu$ and ${\rm Im}\,\nu$ are non-zero, and $\theta\in [0,\pi]$ is defined by the relation $\langle {\rm Re}\,\nu,{\rm Im}\,\nu\rangle=\| {\rm Re}\,\nu\|\|{\rm Im}\,\nu\|\cos\theta$, then
\begin{equation}\label{xi-xi}
\langle \nu,\nu\rangle_{\mathbb{C}}=\|{\rm Re}\, \nu\|^2-\|{\rm Im}\, \nu\|^2+2i\| {\rm Re}\,\nu\|\|{\rm Im}\,\nu\|\cos\theta.
\end{equation}
We also set $\mathcal{N}(\nu)=\langle \nu,\nu\rangle_{\mathbb{C}}^{1/2}$, where we fix the standard branch of the square root.

We claim that
\begin{equation}
\label{alt-sphere}
\widehat{\mathrm{d}\omega}(\nu)=|\omega_{d-2}|\int_0^{\pi}e^{\mathcal{N}(\nu)\cos\varphi}(\sin\varphi)^{d-2}\,\mathrm{d}\varphi,
\end{equation}
extending \eqref{omega-hat1} to complex $\nu\in(\mathbb{C}^d)^{\ast}$. It is clear that $\widehat{\mathrm{d}\omega}(\nu)$ is an even entire function of $\nu\in(\mathbb{C}^d)^{\ast}$. Moreover, the above argument by rotational invariance shows that, for $\nu\in(\mathbb{R}^d)^{\ast}$,
\begin{equation}\label{omega-hat2}
 \widehat{\mathrm{d}\omega}(\nu)=(2\pi)^{d/2}(i\mathcal{N}(\nu))^{-d/2+1}J_{d/2-1}(i\mathcal{N}(\nu)).
 \end{equation}
The right-hand side, understood as a function of $z=\mathcal{N}(\nu)$, is an \emph{even}, \emph{entire} function of $z$, so in turn it defines (for example via its series expansion in powers of $z^2=\mathcal{N}(\nu)^2$) an analytic function of $\nu\in(\mathbb{C}^d)^{\ast}$. Thus the two sides of \eqref{omega-hat2} are analytic functions of $\nu\in(\mathbb{C}^d)^{\ast}$ that agree along $(\mathbb{R}^d)^{\ast}$. The claim \eqref{alt-sphere} follows by analytic continuation.

In particular, using the asymptotic expansion for the $J$-Bessel function~\cite[(10.17.3)]{NIST2010},
\begin{equation}
\label{AsymptoticsFourierTransformSphere}
\widehat{\mathrm{d}\omega}(\nu)=
\left(\frac{2\pi}{|\mathcal{N}(\nu)|}\right)^{(d-1)/2}\big(c_{+}e^{\mathcal{N}(\nu)}+c_{-}e^{-\mathcal{N}(\nu)}\big)+\mathrm{O}\bigg(\frac{e^{|\mathrm{Re}\,\mathcal{N}(\nu)|}}{|\mathcal{N}(\nu)|^{(d+1)/2}}\bigg),
\end{equation}
where $c_{\pm}^2=(\pm\mathcal{N}(\nu)/|\mathcal{N}(\nu)|)^{-d+1}$, as well as $\widehat{\mathrm{d}\omega}(\nu)\ll 1$ for $\mathcal{N}(\nu)=\mathrm{O}(1)$. The main term in \eqref{AsymptoticsFourierTransformSphere} arises from the stationary points at $\varphi=0,\pi$.

The expression \eqref{alt-sphere} and the corresponding asymptotics \eqref{AsymptoticsFourierTransformSphere} point to the importance of the complex parameter $\mathcal{N}(\nu)$, which can be thought of as encoding the relative sizes and positioning of $\mathrm{Re}\,\nu,\mathrm{Im}\,\nu\in\mathfrak{h}_M^{\ast}$ via \eqref{xi-xi}.

\subsection{Fourier transform of a bump function}

Let $b\in C^\infty_c(\R)$ be the bump function equal to $e^{-1/(1-x^2)}$ in $[-1,1]$ and vanishing outside of this interval. Define $g\in C^\infty_c(\R^d)$ radially by the formula $g(H)=b(\|H\|)$. The principal result of this section is the following technical Lemma~\ref{lem:SPL}, which gives asymptotics for the size of the Fourier transform $\hat{g}(\nu)$ for $\nu\in \C^d$.

\begin{lemma}\label{lem:SPL}
For $|\mathcal{N}(\nu)|\gg 1$ we have
\begin{equation}
\label{SPL-asymptotics}
\hat{g}(\nu)=\frac{(2\pi)^{d/2}}{\root{4}\of{8e}}\sum_{\pm}\frac{e^{\pm \mathcal{N}(\nu)-\sqrt{\pm 2\mathcal{N}(\nu)}}}{(\pm \mathcal{N}(\nu))^{d/2+1/4}}\big(1+\mathrm{O}(|\mathcal{N}(\nu)|^{-1/2})\big),
\end{equation}
as well as $\hat{g}(\nu)=\mathrm{O}(1)$ when $\mathcal{N}(\nu)=\mathrm{O}(1)$. 
\end{lemma}

\begin{remark}
Here and below, we use the standard branch of the square root (and all fractional powers); this does not cause confusion with the following convention in force. If $|\arg\mathcal{N}(\nu)\pm\frac{\pi}2|>\frac{\pi}6$, say, then one of the summands in \eqref{SPL-asymptotics} is clearly dominated by the error term in the other regardless of the choice of the branch, and we interpret the asymptotics \eqref{SPL-asymptotics} as simply the dominant summand; otherwise both summands appear (although one will still dominate unless $\arg\mathcal{N}(\nu)$ is very close to $\pm\frac{\pi}2$). In either case, the square root and each fractional power that appears is the standard branch, which is then only taken on $\{z\in\mathbb{C}:|\arg z|\leqslant\frac23\pi\}$.
\end{remark}

\begin{remark}
The proof of Lemma~\ref{lem:SPL} relies on the method of stationary phase (for complex phases). While the asymptotics in \eqref{SPL-asymptotics} resemble the asymptotics of the Fourier transform of the sphere for general complex arguments \eqref{AsymptoticsFourierTransformSphere}, the smooth cut-off nature of the test function $g$ (and in fact its particular choice) is reflected in new features in the proof and the delicate subexponential behavior within $\hat{g}$. A rough main term analysis for the asymptotic behavior of $\hat{g}$, in dimension 1 and for $\nu$ imaginary, as well as a numerical verification of constants, can be found in the note \cite{Johnson2015}.
\end{remark}

\begin{proof}
The second claim is clear, so we may assume that $|\mathcal{N}(\nu)|\geqslant 100$, say. It follows from \eqref{alt-sphere} that $\widehat{\mathrm{d}\omega}(\nu)$ depends only on $\mathcal{N}(\nu)$. Thus $\widehat{\mathrm{d}\omega}(\nu)=\widehat{\mathrm{d}\omega}(\mathcal{N}(\nu)\mathbf{e}_1)$ and, for $d\geqslant 2$,
\begin{align}
\hat{g}(\nu)&=\int_{\{\mathbf{x}\in\mathbb{R}^d:\|\mathbf{x}\|\leqslant 1\}}e^{\langle \mathcal{N}(\nu)\mathbf{e}_1,\mathbf{x}\rangle-1/(1-\|\mathbf{x}\|^2)}\,|\mathrm{d}\mathbf{x}|\nonumber\\
&=|\omega_{d-2}|\int_{-1}^1e^{\mathcal{N}(\nu)x}\int_0^{\sqrt{1-x^2}}e^{-1/(1-x^2-\rho^2)}\rho^{d-2}\,\mathrm{d}\rho\,\mathrm{d}x\nonumber\\
\label{hxi}
&=|\omega_{d-2}|\int_0^1\int_{-1}^1e^{\mathcal{N}(\nu)x-u/(1-x^2)}(1-x^2)^{(d-1)/2}\,\mathrm{d}x\,t^{d-2}\mathrm{d}t,
\end{align}
where $u=u_t=1/(1-t^2)$. The heart of matter is in the asymptotic evaluation of the inner integral $I(\mathcal{N}(\nu),u)$; we state this result as Lemma~\ref{SPL-inner}.

\begin{lemma}
\label{SPL-inner}
For $s\in\mathbb{C}$ and $u\geqslant 1$, let
\[ I(s,u):=\int_{-1}^1e^{\phi_{s,u}(x)}p_d(x)\,\mathrm{d}x,\qquad \phi_{s,u}(x)=sx-u/(1-x^2),\quad p_d(x)=(1-x^2)^{(d-1)/2}. \]
Then, for $u\ll |s|$,
\begin{equation}
\label{Isu-to prove}
 I(s,u)=\sum_{\pm}\frac{\sqrt{\pi}}{2u^{1/2}}\left(\frac{2u}{\pm s}\right)^{(d+2)/4}e^{\pm s-\sqrt{\pm 2us}-\frac14u+\mathrm{O}(u^{3/2}/|s|^{1/2})},
\end{equation}
as well as $I(s,u)\ll e^{|\mathrm{Re}\,s|-c|s|}$ for $u\geqslant c|s|$.
\end{lemma}

\begin{remark}
The proof of Lemma~\ref{SPL-inner} shows that, more precisely,
\[ I(s,u)=\sum_{\pm}\frac{\sqrt{\pi}}{2u^{1/2}}\left(\frac{2u}{\pm s}\right)^{(d+2)/4}\tilde{p}_d(\sqrt{\pm u/s})e^{\pm s\tilde{\phi}(\sqrt{\pm u/s})}\big(1+\mathrm{O}((u|s|)^{-1/2})\big)\quad (u\ll |s|), \]
where $\tilde{p}_d(t),\tilde{\phi}(t)\in 1+t\mathbb{C}[[t]]$ are explicit analytic functions around 0. Computing the first few terms explicitly, we find that $\tilde{\phi}(t)=1-\sqrt2t-\frac14t^2+\mathrm{O}(t^3)$, say, recovering \eqref{Isu-to prove}.
\end{remark}

We postpone the proof of Lemma~\ref{SPL-inner}, which is by the method of steepest descent, for later in this section, and return to \eqref{hxi}. Denoting by $I_0(\pm s,u)$ the summand in \eqref{Isu-to prove}, we find that
\begin{equation}
\label{hatg}
\hat{g}(\nu)=|\omega_{d-2}|\sum_{\pm}\int_0^1I_0(\pm s,u_t)t^{d-2}\,\mathrm{d}t,
\end{equation}
where for brevity we set $s=\mathcal{N}(\nu)$. By symmetry, it suffices to asymptotically evaluate the first term, and it also suffices to consider the case when $|\arg s|\leqslant\frac{2\pi}3$, say, as otherwise the corresponding summand will be seen to be clearly dominated.

For $1\leqslant u\leqslant c|s|$, the expression $-\mathrm{Re}\sqrt{2us}-\frac14u+\frac d4\log u+Cu^{3/2}/|s|^{1/2}$ has derivative $\leqslant -\frac{1}{2\sqrt{c}}+\frac{(d-1)}4+\frac32Cc^{1/2}$ and is thus a decreasing function of $1\leqslant u\leqslant c|s|$ for sufficiently small $c$. Therefore,
\eqref{Isu-to prove} along with the complementary bound for the range $|u|\geqslant c|s|$ shows that
\begin{equation}
\label{hatg-error}
\int_{|s|^{-1/8}}^1I_0(s,u_t)t^{d-2}\,\mathrm{d}t\ll\frac1{|s|^{(d+2)/4}}e^{\mathrm{Re}(s-\sqrt{2s})-\frac{\sqrt2}4|s|^{1/4}},
\end{equation}
which will momentarily be seen as negligible. In the remaining, principal range, we can write
\begin{equation}
\label{hatg-main}
\begin{aligned}
&\int_0^{|s|^{-1/8}}I_0(s,u_t)t^{d-2}\,\mathrm{d}t\\
&\quad=\frac{\sqrt{\pi}}2\left(\frac 2s\right)^{(d+2)/4}e^{s-\sqrt{2s}-\frac14}\int_0^{|s|^{-1/8}}e^{-\frac12\sqrt{2s}\cdot t^2}t^{d-2}\big(1+\mathrm{O}(t^2+|s|^{1/2}t^4+|s|^{-1/2})\big)\,\mathrm{d}t\\
&\quad=\frac{\sqrt{\pi}}{4\root{4}\of{e}}\left(\frac 2s\right)^{(2d+1)/4}\Gamma\big(\tfrac12(d-1)\big)e^{s-\sqrt{2s}}\big(1+\mathrm{O}(|s|^{-1/2})\big).
\end{aligned}
\end{equation}
Combining \eqref{hatg}, \eqref{hatg-error}, \eqref{hatg-main}, and $|\omega_{d-2}|=2\pi^{(d-1)/2}/\Gamma\big(\frac12(d-1)\big)$, with $s=\mathcal{N}(\nu)$, we obtain \eqref{SPL-asymptotics} for $d\geqslant 2$. In the case $d=1$, \eqref{SPL-asymptotics} follows by simply applying Lemma~\ref{SPL-inner} directly to the integral in the first line of \eqref{hxi}.
\end{proof}

Having proved Lemma~\ref{lem:SPL}, we now turn to the unfinished business of Lemma~\ref{SPL-inner}.\par

\begin{proof}[Proof of Lemma~\ref{SPL-inner}]
The complementary bound for $u\geqslant c|s|$ is obvious. Note that we may thus assume that in fact $u\ll |s|$ with a sufficiently small implied constant in \eqref{Isu-to prove}, since for $u\asymp |s|$ \eqref{Isu-to prove} holds vacuously (with an adjusted implied constant) in light of the complementary bound.

The exponential integral $I(s,u)$ will be analyzed using the method of stationary phase. The phase $\phi_{s,u}$ is holomorphic in $\mathbb{C}\setminus\{\pm 1\}$. Consider first the range when $|\arg s|\leqslant\frac{3\pi}5$, say; we may assume $0\leqslant\arg s\leqslant\frac{3\pi}5$ by symmetry. Inside the disk $\Omega_1=\{x\in\mathbb{C}:0<|x-1|\leqslant\frac74\}$, we have
\[ \phi_{s,u}(x)=sx-\tfrac12u/(1-x)+u\mathrm{O}(1), \]
where $\mathrm{O}(1)$ is an explicit analytic function (independent of $u$ or $s$).

There are four total stationary points: two close to $1$ and two close to $-1$. To describe them, it will be convenient to denote by $\iota_{\vartheta}(\varrho)$, a real analytic function defined on a small neighborhood of 0 (uniform in $\vartheta$), which satisfies $\iota_{\vartheta}(\varrho)=1+\mathrm{O}(\varrho)$, $(\partial\iota_{\vartheta}/\partial\vartheta)(\varrho)=\mathrm{O}(\varrho)$, and $(\partial\iota_{\vartheta}/\partial\varrho)(\varrho)=\mathrm{O}(\varrho^2)$ uniformly in $\vartheta$, and which may be different from one line to another. Using this notation, one finds two points $x_0^{\pm}=1\pm\varrho_0e^{i\vartheta_0}\in\Omega_1$ at which $(\mathrm{d}/\mathrm{d}x)|_{x=x_0^{\pm}}\phi_{s,u}=0$; they satisfy
\[ s-\tfrac12ue^{-2i\vartheta_0}/\varrho^2_0\iota_{\vartheta_0}(\varrho_0)=0 \]
with $\vartheta_0=-\frac12\arg s+\mathrm{O}(\varrho_0)$. In fact, $x_0^{\pm}=1-w$, where $w$ is one of the two solutions close to $0$ of the equation $w^2(1-\frac12w)^2/(1-w)=\frac12u/s$, which can be solved as an implicit function problem (with the choice of $\pm$ as just specified).

We now isolate a critical point of interest, say $x_0=x_0^{-}\in\Omega_1$. Writing $\tilde{\phi}(x_0)=\phi_{s,u}(x_0)$, we find
\begin{align*}
\tilde{\phi}(x_0)&=s-\sqrt{2us}-\tfrac14u+\mathrm{O}(u^{3/2}/|s|^{1/2}),\\
\phi''_{s,u}(x_0)&=-ue^{-3i\vartheta_0}/\varrho_0^3+\mathrm{O}(u),
\end{align*}
where in particular $\phi''_{s,u}(x_0)=(-u/w^3)(1+\frac18w^3)/(1-\frac12w)^3=-u/w^3+\mathrm{O}(u)$.
With this we may evaluate the corresponding principal contribution to the asymptotics of $I(s,u)$.

The segment
\[ I_{1,0}=\{x_t=x_0+te^{-i\arg(-\phi''_{s,u}(x_0))/2},\; |t|\leqslant t_0\} \]
is the best linear approximation to a curve of steepest descent. Let
\[
\varphi_k=u(u/|s|)^{-(k+1)/2},\quad \psi_k=p_d(x_0)(u/|s|)^{-k/2},\quad\text{ and }\quad t_0=u^{1/3}/|s|^{2/3}\asymp\varrho_0/(u|s|)^{1/6}. 
\]
Then along $I_{1,0}$ we have $|1-x_t|\asymp (u/|s|)^{1/2}$, $|\phi_{s,u}^{(k)}(x_t)|\asymp \varphi_k$,
\[ \phi_{s,u}(x_t)=\phi_{s,u}(x_0)-\tfrac12|\phi''_{s,u}(x_0)|t^2+\tfrac16\epsilon\phi'''_{s,u}(x_0)t^3+\mathrm{O}(\varphi_4t^4) \]
with a fixed $|\epsilon|=1$, while $|p_d^{(k)}(x_t)| \asymp\psi_k$. Thus by standard estimates (using also that contributions of the odd higher order terms vanish)
\begin{align*}
&\int_{I_{1,0}}e^{\phi_{s,u}(x)}p_d(x)\,\mathrm{d}x\\
&\qquad =\frac{(2\pi)^{1/2}}{(-\phi''_{s,u}(x_0))^{1/2}}{e^{\phi_{s,u}(x_0)}}p_d(x_0)\left(1+\mathrm{O}\left(\frac{\varphi_3^2}{\varphi_2^3}+\frac{\varphi_4}{\varphi_2^2}+\frac{\varphi_3\psi_1}{\varphi_2^2\psi_0}+\frac{\psi_2}{\varphi_2\psi_0}+e^{-\frac12(u|s|)^{1/6}}\right)\right)\\
&\qquad =\frac{\sqrt{\pi}}{2u^{1/2}}\left(\frac{2u}{s}\right)^{(d+2)/4}e^{s-\sqrt{2us}-\frac14u+\mathrm{O}(u^{3/2}/|s|^{1/2})}\big(1+\mathrm{O}((u|s|)^{-1/2})\big).
\end{align*}
This recovers one of the principal terms in \eqref{Isu-to prove}, with acceptable error terms.

\begin{figure}
\centering
\begin{overpic}[width=0.4\linewidth]{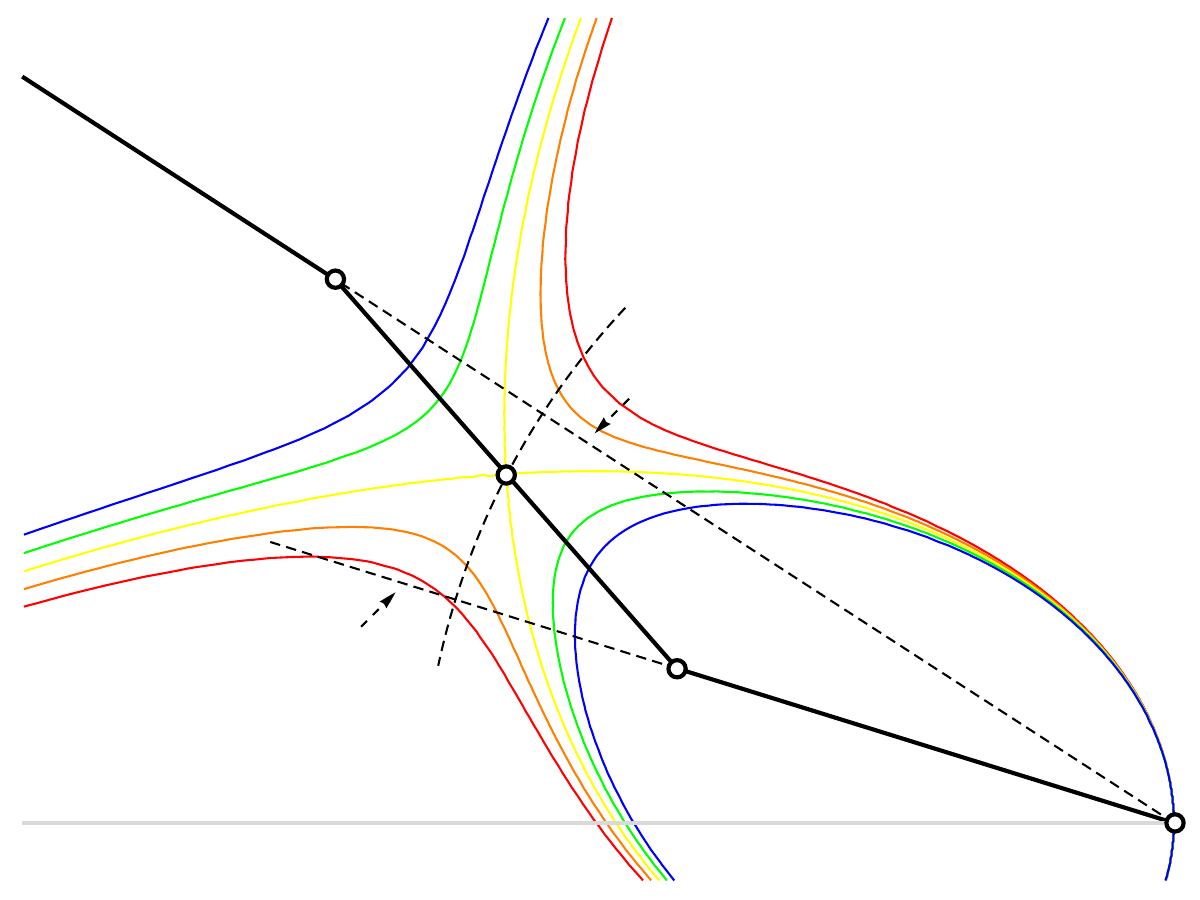}
\put(28,54){$\scriptstyle x_{-t_0}$}
\put(36,35){$\scriptstyle x_0$}
\put(53,16){$\scriptstyle x_{t_0}$}
\put(67,11){$\scriptstyle I_{1,+}$}
\put(52,26){$\scriptstyle I_{1,0}$}
\put(11,56){$\scriptstyle I_{1,-}$}
\put(100,5){$\scriptstyle 1$}
\put(50,52){$\scriptstyle\varrho_0$}
\put(54,42){$\scriptstyle\varrho_{\vartheta_{-t_0}}$}
\put(26,20){$\scriptstyle\varrho_{\vartheta_{t_0}}$}
\end{overpic}
\caption{\footnotesize Contour shift in Lemma~\ref{SPL-inner}, showing several level curves for $\mathrm{Re}\,\phi_{s,u}$, with warmer colors indicating higher values.}
\label{ContourShiftFigure}
\end{figure}

We will shift the contour of integration in $I(s,u)$ from $[-1,1]$, beginning with the part within $\Omega_1$, to a polygonal contour consisting of $I_{1,0}$ (where a contribution to the main term arises) and two segments along rays $I_{1,\pm}$ originating at $x=1$; see Figure~\ref{ContourShiftFigure}. Keeping in mind that we are working with $\sigma=\arg s$ satisfying (by symmetry) $0\leqslant\sigma\leqslant\frac{3\pi}5$, we will analyze the phase along such rays $x=1-\varrho e^{i\vartheta}$, with the goal of showing that the exponential term $e^{\phi_{s,u}(x)}$ decreases in size along $I_{1,+}$ and $I_{1,-}$ away from $x_{\pm t_0}$ and then concluding that the corresponding integrals are negligible.
For any fixed $\vartheta$ with $|\vartheta|,|\vartheta+\sigma|\leqslant\frac{2\pi}5$ (say), the critical point condition $(\partial/\partial\varrho)\mathrm{Re}\,\phi_{s,u}=0$ only has solutions close to 1 and may be written as
\[ -\mathrm{Re}(se^{i\vartheta})+\tfrac12u\mathrm{Re}(e^{-i\vartheta})/\tilde{\varrho}^2=0 \]
for some explicit $\tilde{\varrho}=\varrho\iota_{\vartheta}(\varrho)$. This equation has exactly two roots $\tilde{\varrho}=\pm\tilde{\varrho}_{\vartheta}$, where
\[ \tilde{\varrho}_{\vartheta}=(\cos\vartheta/\cos(\sigma+\vartheta))^{1/2}(\tfrac12u/|s|)^{1/2}, \]
with a global maximum of $\mathrm{Re}\,\phi_{s,u}$ at the corresponding $1-\varrho_{\vartheta}e^{i\vartheta}$. As $\vartheta$ varies within the same range, $\mathrm{d}\tilde{\varrho}_{\vartheta}/\mathrm{d}\vartheta=\frac14(u/|s|\tilde{\varrho}_{\vartheta})\sin\sigma/\cos^2(\sigma+\vartheta)\geqslant 0$, and consequently $\mathrm{d}\varrho_{\vartheta}/\mathrm{d}\vartheta\geqslant -\mathrm{O}(\varrho_{\vartheta}^2)$.

In particular, $\varrho_{\vartheta_0}=\varrho_0$. At the endpoints of $I_{1,0}$, we have
\[ \phi_{s,u}(x_{\pm t_0})=\phi_{s,u}(x_0)-\tfrac12(u|s|)^{1/6}+\mathrm{O}(1). \]
Writing $x_t=1-\varrho_te^{i\vartheta_t}$, in light of
\[
|\tfrac12\arg(-\phi''_{s,u}(x_0))-\vartheta_0|=\tfrac12|\vartheta_0|+\mathrm{O}((u/|s|)^{3/2})<\tfrac{2\pi}5,
\]
we have $\varrho_{\vartheta_0}-\varrho_{t_0}\gg\varrho_{\vartheta_0}|\vartheta_{t_0}-\vartheta_0|$ while $\varrho_{\vartheta_{t_0}}>\varrho_{\vartheta_0}+\mathrm{O}(\varrho_{\vartheta_0}^2|\vartheta_{t_0}-\vartheta_0|)$; this shows that $\varrho_{t_0}<\varrho_{\vartheta_{t_0}}$ (since $\varrho_{\vartheta_0}\asymp(u/|s|)^{1/2}$ is sufficiently small). Analogously, $\varrho_{-t_0}>\varrho_{\vartheta_{-t_0}}$. Denoting
\[ I_{1,+}=\{1-\lambda\varrho_{t_0}e^{i\vartheta_{t_0}}:\lambda\in [0,1]\},\quad I_{1,-}=\{1-\lambda\varrho_{-t_0}e^{i\vartheta_{-t_0}}:\lambda\geqslant 1\}\cap\Omega_1, \]
the term $e^{\phi_{s,u}(x)}$ decreases in size along $I_{1,+}$ and $I_{1,-}$ away from $x_{\pm t_0}$ and thus the integrals over $I_{1,\pm}$ of $e^{\phi_{s,u}(x)}p_d(x)\ll e^{\phi_{s,u}(x_0)-\frac12|s|^{1/6}}$ are (in light of $u\ll |s|$) absolutely dominated by the error term in the integral over $I_{1,0}$. This completes the description of the contour shift near the point $1$ (within $\Omega_1$) for $|\arg s|\leqslant\frac{3\pi}5$.

It remains to do the analogous analysis near $-1$ and put everything together. If $|\mathrm{arg}\,s-\pi|\leqslant\frac{3\pi}5$, say, the same argument applies, \textit{mutatis mutandis}, in the disc $\Omega_{-1}$ given by $|x+1|\leqslant\frac74$, with stationary points at $1\pm i\varrho_0 e^{i\vartheta_0}$, where $\varrho_0$ and $\vartheta_0$ satisfy the same asymptotics as above, and with analogous segments $I_{-1,0}$ and $I_{-1,\pm}$ such that
\[ \int_{I_{-1,0}}e^{\phi_{s,u}(x)}p(x)\,\mathrm{d}x=\frac{\sqrt{\pi}}{2u^{1/2}}\left(\frac{2u}{-s}\right)^{(d+2)/4}e^{-s-\sqrt{-2us}-\frac14u+\mathrm{O}(u^{3/2}/|s|^{1/2})}\big(1+\mathrm{O}((u|s|)^{-1/2})\big), \]
and the integrals over $I_{-1,\pm}$ are dominated by the error terms in the above.

If $\frac{2\pi}5\leqslant|\arg s|\leqslant\frac{3\pi}5$, the claim \eqref{Isu-to prove} now follows by shifting contours to the union of $I_{\pm 1,+}$, $I_{\pm 1,0}$, and sub-segments of $I_{\pm 1,-}$ formed by their intersection (which lies in $\Omega_1\cap\Omega_{-1}$ as, say, $\vartheta_{\pm t_0}=-\frac12\sigma+\mathrm{O}((u/|s|)^{1/2}+1/(u|s|)^{1/6})$). If $0\leqslant|\arg s|<\frac{2\pi}5$, we can simply use the $I_{1,+/0/-}$ contour and connect to $-1$ via a straight-line segment, and analogously for $\frac{3\pi}5<|\arg s|\leqslant\pi$.
\end{proof}

\section{Proof of Proposition~\ref{MainCountingResultGLn}}\label{sec:smooth-to-sharp-tempered}

Similarly to the existing literature on Weyl laws, our basic strategy in the proof of Proposition~\ref{MainCountingResultGLn} is to relate the sharp count $N(\q,\underline{\delta},P)$ to a corresponding smooth count; we can control the smooth count via trace formula input, as represented by Property (ELM), applied to the test function $f_R^{\delta,P}$ constructed in \S\ref{sec:approx-estimates}. 

We need to pass from a smooth to a sharp test function in two terms: the central contribution $J_{\rm cent}(\varepsilon_{K_1(\q)}\otimes f_R^{\delta,P})$ and the smooth count over the discrete spectrum $J_{\rm disc}(\varepsilon_{K_1(\q)}\otimes f_R^{\delta,P})$. The central contributions were already addressed in \S\ref{UniformUpperBoundsSubsection}. In the following paragraphs we treat the discrete spectrum.

\subsection{Upper bounds}

We must first bound the total deviations of $h_R^{\delta,P}$ from the sharp-cutoff condition of belonging to $P$ over the ``transition zones'' (where the smooth test function is transitioning between 0 and 1). This is done in Lemma~\ref{UpperBoundPartialP} below. For the statement, recall the notion of $\rho$-containment from Definition~\ref{ContainedInDefinition} and the sum $K_R(\q,\underline{\delta},P)$ from \eqref{KR-def}.

\begin{lemma}\label{UpperBoundPartialP}
Let $c_2>0$ be as in Lemma~\ref{cor-central2}, and let $1\leqslant R<c_2\log (2+\norm\q)$. For a suitably large $a>0$ and any $W(A_M)_\delta$-invariant set $S\subseteq i\mathfrak{h}_M^*$ which is $a/R$-contained in some $B\in\mathscr{B}(\underline{\delta})$, we have
\[
N(\q,\underline{\delta},S)\ll \varphi_n(\q)\int_B\beta_M^G(\delta,\nu)\, {\rm d}\nu+|J_{\rm error}(\varepsilon_{K_1(\q)}\otimes f_R^{\delta, B})|+K_R(\q,\underline{\delta},B).
\]
In particular, this holds with $S=\partial P(\rho)$ and $B=\partial P(\rho+a/R)$ where $P\in\mathscr{B}(\underline{\delta})$ and $\rho>0$.
\end{lemma}

\begin{proof}
Using Lemma~\ref{GLnTestFunctionEstimates}~\eqref{second-guy}, we fix $a>0$ (independent of $B\in\mathscr{B}(\underline{\delta})$) such that, for every $\nu\in B^{\circ}(a/R)$, $h_R^{\delta,B}(\nu)\geqslant\frac12$.

Let $B\in\mathscr{B}(\underline{\delta})$. As in \S\ref{reduction-to-test-functions-subsection}, we decompose the discrete distribution
\[
J_{\rm disc}(\varepsilon_{K_1(\q)}\otimes f_R^{\delta,B})
=J_{\rm temp}(\varepsilon_{K_1(\q)}\otimes f_R^{\delta,B})+J_{\rm comp}(\varepsilon_{K_1(\q)}\otimes f_R^{\delta,B}),
\]
according to whether or not the archimedean component of the $\pi$ contributing to $J_{\rm disc}(\varepsilon_{K_1(q)}\otimes f_R^{\delta,B})$ is tempered. Using the non-negativity of $h_R^{\delta,B}$ on $i\mathfrak{h}_M^{\ast}$ and the fact that $S$ is $(a/R)$-contained in $B$, we have that $h_R^{\delta,B}(\nu)\geqslant\frac12$ for every $\nu\in S$ and
\[ J_{\rm temp}(\varepsilon_{K_1(q)}\otimes f_R^{\delta,B})\geqslant\tfrac12 N(\q,\underline{\delta},S). \]

On the other hand, by definition, we have
\begin{equation}
\label{Jtemp-decomp}
J_{\rm temp}(\varepsilon_{K_1(\q)}\otimes f_R^{\delta,B})=-J_{\rm comp}(\varepsilon_{K_1(\q)}\otimes f_R^{\delta,B})+J_{\rm cent}(\varepsilon_{K_1(\q)}\otimes f_R^{\delta,B})+J_{\rm error}(\varepsilon_{K_1(\q)}\otimes f_R^{\delta,B}).
\end{equation}
It suffices to bound the first two terms on the right-hand side. We use Lemma~\ref{cor-central2}~\eqref{Jcent1} to bound $J_{\mathrm{cent}}(\varepsilon_{K_1(\q)}\otimes f_R^{\delta,B})$. Further, by Lemma~\ref{GLnTestFunctionEstimates} we have that
\[ h_R^{\delta,B}(\nu)\ll_N e^{R\|{\rm Re}\,\nu\|}\big(1+R\cdot d(i\mIm\nu,B)\big)^{-N} \]
for every $\nu\in\mathfrak{h}_{M,\C}^*$, so that
\begin{equation}\label{comp-maj}
J_{\rm comp}(\varepsilon_{K_1(\q)}\otimes f_R^{\delta,B})\ll K_R(\q,\underline{\delta},B).
\end{equation}

The last statement of the proposition follows from Lemma~\ref{SetContainments}, where it is shown that $\partial P(\rho)$ is $(a/R)$-contained in $\partial P(\rho+a/R)$.
\end{proof}

\subsection{Proof of Proposition~\ref{MainCountingResultGLn}}\label{SmoothToSharpProofSubsection}

With $C$ and $\theta$ as in Property (ELM) and $c_2$ as in Lemma \ref{cor-central2}, let $c=\min(\theta/(2C),c_2)$. Assume that $1\leqslant R<c\log (2+\norm\q)$. Then, according to \eqref{averageELM}, we have for every $\underline{\delta}\in\mathcal{D}$ and $P\in\mathscr{B}(\underline{\delta})$,
\begin{equation}
\label{ErrorTermELM}
J_{{\rm error}}(\varepsilon_{K_1(\q)}\otimes f_R^{\delta,P})\ll\mathbf{N}\q^{n-\theta'}\int_P\beta_M^G(\delta,\nu)\,\mathrm{d}\nu,
\end{equation}
with $\theta'=\theta-Cc>0$.

We begin by applying the decomposition \eqref{Jtemp-decomp}. We use the bound \eqref{comp-maj} and then Proposition \ref{KR} to bound the complementary term by $\varphi_n(\q){\rm vol}_R^\star(\delta,P)$. We bound the $J_{\rm error}$ term as in \eqref{ErrorTermELM}. Finally, when $\norm\q> C_2$, where $C_2$ is as in Lemma \ref{cor-central2}~\eqref{Jcent3}, that lemma allows us to conclude that
\begin{equation}\label{temp-is-main}
\begin{aligned}
J_{\rm temp}(\varepsilon_{K_1(q)}\otimes f_R^{\delta,P})
&=D_F^{n^2/2}\Delta_F^*(1)\varphi_n(\q)\deg(\delta)\int_P\mu_M^G(\delta,\nu)\, {\rm d}\nu\\
&+\text{O}\left(\varphi_n(\q)\big(\partial\vol_R(\delta,P)+{\rm vol}_R^\star(\delta,P)\big)+\norm\q^{n-\theta'}\int_P\beta_M^G(\delta,\nu)\,\mathrm{d}\nu\right).
\end{aligned}
\end{equation}
The case $\norm\q\leqslant C_2$ can be treated similarly, yielding an upper bound, using Lemma \ref{cor-central2}~\eqref{Jcent1}.

We now decompose the sum in $J_{\rm temp}(\varepsilon_{K_1(\q)}\otimes f_R^{\delta,P})$ according to \eqref{P-decomp}. Note that
\[
N(\q,\underline{\delta},P)=\sum_{\substack{\pi\in\Pi_{\rm disc}(\bm{G}(\A_F)^1)_{\underline{\delta}}\\ \nu_\pi\in P}}\dim V_{\pi_f}^{K_1(\q)}.
\]
Indeed, this follows from the description of the discrete spectrum by M{\oe}glin--Waldspurger \cite{MoeglinWaldspurger1989}, where it is shown that any $\pi\in\Pi_{\rm disc}(\bm{G}(\A_F)^1)$ such that $\pi_\infty$ is tempered is necessarily cuspidal. Using this and Lemma~\ref{GLnTestFunctionEstimates}, we find that $J_{\rm temp}(\varepsilon_{K_1(q)}\otimes f_R^{\delta,P})$ equals
\begin{align*}
&N(\q,\underline{\delta},P)+\text{O}\big(N(\q,\underline{\delta},\partial P(1/R))\big)
+\text{O}\bigg(\sum_{\ell=1}^{\infty}\ell^{-N}\sum_{\bullet/\circ}N\big(\q,\underline{\delta},P^{\bullet/\circ}(\ell/R,(\ell+1)/R)\big)\bigg)\\
&\qquad =N(\q,\underline{\delta},P)+\text{O}\bigg(\sum_{\ell=1}^{\infty}\ell^{-N-1}N\big(\q,\underline{\delta},\partial P(\ell/R)\big)\bigg).
\end{align*}
It follows from this and Lemma~\ref{UpperBoundPartialP} that $J_{\rm temp}(\varepsilon_{K_1(q)}\otimes f_R^{\delta,P})-N(\q,\underline{\delta},P)$ is majorized by
\begin{equation}\label{smoothing-error}
\begin{aligned}
\sum_{\ell=1}^{\infty}\ell^{-N-1}&\Big(\varphi_n(\q)\int_{\partial P((\ell+a)/R)}\beta_M^G(\delta,\nu)\,{\rm d}\nu\\
& +K_{R,N+1}(\q,\underline{\delta},\partial P\big((\ell+a)/R))+\big|J_{\rm error}(\varepsilon_{K_1(\q)}\otimes f_R^{\delta,\partial P((\ell+a)/R)})\big|\Big).
\end{aligned}
\end{equation}

The first of the error terms in \eqref{smoothing-error} is at most $\varphi_n(\q)\partial\vol_R(\delta,P)$. The second of the error terms in \eqref{smoothing-error} needs only elementary manipulation to be put into the required form. Indeed, expanding into a double sum, we have 
\begin{equation*}
\begin{aligned}
\sum_{\ell=1}^{\infty}\ell^{-N-1}K_{R,N+1}(\q,\underline{\delta},\partial P\big((\ell+a)/R))
&\leqslant\sum_{\ell=1}^{\infty}\sum_{m=1}^{\infty}\ell^{-N-1}m^{-N-1}\!\!\!\!\!\!\!\!\!\sum_{i\mIm \nu_\pi\in \partial P((\ell+m+a)/R)}\!\!\!\!\!\!\!\!\!\!\dim V_{\pi_f}^{K_1(\q)}e^{R\|{\rm Re}\, \nu_{\pi}\|}\\
&\ll{}\sum_{\ell=2}^{\infty}\sum_{i\mIm\nu_\pi\in \partial P(\ell/R)}\dim V_{\pi_f}^{K_1(\q)}e^{R\|{\rm Re}\, \nu_{\pi}\|} \sum_{\ell=k_1+k_2}(k_1k_2)^{-N-1}
\\
&\ll{}\sum_{\ell=1}^{\infty}\ell^{-N}\sum_{i\mIm\nu_\pi\in \partial P(\ell/R)}\dim V_{\pi_f}^{K_1(\q)}e^{R\|{\rm Re}\, \nu_{\pi}\|},
\end{aligned}
\end{equation*}
where the automorphic sums run over $\pi\in \Pi_{\rm disc}(\bm{G}(\A_F)^1)_{\underline{\delta}}$ with $\nu_{\pi}\not\in i\mathfrak{h}_M^{\ast}$. The last of these is at most $K_{R,N}(\q,\underline{\delta}, P)$, which by Proposition \ref{KR} we can bound by $\varphi_n(\q) {\rm vol}_R^\star(\delta,P)$.

It only remains to treat the $J_{{\rm error}}$ terms in \eqref{smoothing-error}. Indeed, applying \eqref{ErrorTermELM} as was done in \eqref{temp-is-main}, their total contribution is bounded by
\[
\pushQED{\qed}
\norm\q^{n-\theta'}\sum_{\ell=1}^{\infty}\ell^{-N-1}\int_{\partial P(\ell/R)}\beta_M^G(\delta,\nu)\,{\rm d}\nu\ll\norm\q^{n-\theta'}\overline{\vol}_R(\delta,P).
\qedhere
\popQED
\]

\section{Deducing Theorem \ref{main-implication}}\label{sec:temp-cor}

To pass from Proposition \ref{MainCountingResultGLn} to Theorem \ref{main-implication}, we need to sum over the various discrete data coming from the decomposition of $|\mathfrak{F}(Q)|$ in \eqref{univ-decomp}. In this section, we package together the terms arising in Proposition \ref{MainCountingResultGLn} after executing this summation, evaluate the main term, and bound the boundary errors. The main results are Lemma \ref{temp-cor} and Proposition \ref{temp-err}, from which we deduce, in Corollary \ref{cor:ELM-implies-Conj}, the statement of Theorem \ref{main-implication}.

\subsection{Summing over discrete data}

Recalling the set $\Omega_{\underline{\delta},X}$ of \eqref{DefinitionOmegaX}, we put
\[
P_{\underline{\delta},X}=\Omega_{\underline{\delta},X}\cap i\h_M^*\quad\text{and}\quad P_X=\bigcup_{\underline{\delta}\in\mathcal{D}}P_{\underline{\delta},X}.
\]
Then $P_X$ can be identified with $\Omega_X\cap\Pi_{\rm temp}(G^1_\infty)$, where $\Omega_X=\{\pi\in\Pi(G_\infty^1): q(\pi)\leqslant X\}$ and $\Pi_{\rm temp}$ indicates the tempered unitary dual. We use this notation to extend the definition of the boundary terms of \eqref{BoundaryTermsDefinitions1} and \eqref{BoundaryTermsDefinitions2} by writing
\begin{equation}\label{Omega-vol}
\overline{\vol}_R(P_X)=\sum_{\underline{\delta}\in\mathcal{D}}\overline{\vol}_R(\delta,P_{\underline{\delta},X}),
\end{equation}
and similarly for $\partial\vol_R(P_X)$ and $\vol_R^{\star}(P_X)$. We shall also need to introduce the notation
\[
\partial P_X(r)=\bigcup_{\underline{\delta}\in\mathcal{D}}\{\pi_{\delta,\nu}: \nu\in \partial P_{\underline{\delta},X}(r)\}.
\]

For a sequence $\mathscr{R}=(R(\mathfrak{n}))_{\mathfrak{n}\subseteq\mathcal{O}_F}$ of real numbers indexed by integral ideals $\mathfrak{n}$, and $\theta>0$, let 
\[
\overline{\vol}{}_{\mathscr{R}}^{\theta}(Q)=\sum_{\norm\q\leqslant Q}\sum_{\dd\mid\q}|\lambda_n(\q/\dd)|\norm\dd^{n-\theta}\overline{\vol}_{R(\dd)}(P_{Q/\norm\q}).
\]
Analogously notation holds for $\partial\vol_{\mathscr{R}}(Q)$ and $\vol_{\mathscr{R}}^{\star}(Q)$ with $\norm\dd^{n-\theta}$ replaced by $\varphi_n(\dd)$. 

\begin{defn}
We shall say that $\mathscr{R}=(R(\mathfrak{n}))_{\mathfrak{n}\subseteq\mathcal{O}_F}$ is {\it admissible} if $1\leqslant R(\mathfrak{n})<c\log (2+\norm\mathfrak{n})$, where $c$ is the constant in Proposition \ref{MainCountingResultGLn}.
\end{defn}

The following result reduces the remaining work to an estimation of error terms.

\begin{lemma}\label{temp-cor}
Assume Property {\rm (ELM)}. There is $\theta>0$ such that for any admissible sequence $\mathscr{R}=(R(\mathfrak{n}))$ we have 
\[
|\mathfrak{F}(Q)|=\mathscr{C}(\mathfrak{F}) Q^{n+1}+{\rm O}\big(Q^{n+1-\theta}+\partial\vol_{\mathscr{R}}(Q)+\vol_{\mathscr{R}}^{\star}(Q)+\overline{\vol}{}_{\mathscr{R}}^{\theta}(Q)).
\]
\end{lemma}
\begin{proof}
We begin by decomposing $\mathfrak{F}(Q)=\mathfrak{F}_{\rm temp}(Q)\bigcup \mathfrak{F}_{\rm comp}(Q)$, according to whether the archi\-mede\-an component $\pi_\infty$ of the cusp form $\pi\in\mathfrak{F}(Q)$ is tempered or not. 

We first dispatch with the contribution from $|\mathfrak{F}_{\rm comp}(Q)|$. Using \eqref{univ-decomp} along with the inequality
\[
\sum_{\substack{\pi\in\Pi_{\rm cusp}(\bm{G}(\A_F)^1)_{\underline{\delta}}\\ \nu_\pi\in \Omega_{\underline{\delta},X}\setminus P_{\underline{\delta},X}}}\!\!\!\!\!\dim V_{\pi_f}^{K_1(\dd)}\leqslant \sum_{\substack{\pi\in\Pi_{\rm cusp}(\bm{G}(\A_F)^1)_{\underline{\delta}}\\ i\mIm\nu_\pi\in P_{\underline{\delta},X}}}\!\!\!\!\!\dim V_{\pi_f}^{K_1(\dd)}\leqslant \sum_{\substack{\pi\in\Pi_{\rm disc}(\bm{G}(\A_F)^1)_{\underline{\delta}}\\ i\mIm\nu_\pi\in P_{\underline{\delta},X}}}\!\!\!\!\!\dim V_{\pi_f}^{K_1(\dd)}\leqslant K_{R(\dd)}(\dd,P_{\underline{\delta},X}),
\]
valid for any $R(\dd)>0$, we find that
\[
|\mathfrak{F}(Q)_{\rm comp}|\leqslant \sum_{1\leqslant\norm\q\leqslant Q}\; \sum_{\dd|\q}
|\lambda_n(\q/\dd)|K_{R(\dd)}(\dd,P_{Q/\norm\q}).
\]
Proposition \ref{KR} then bounds the latter sum by $\vol_{\mathscr{R}}^{\star}(Q)$.

For the tempered contribution, we again apply \eqref{univ-decomp} to get
\[
|\mathfrak{F}(Q)_{\rm temp}|=\sum_{1\leqslant\norm\q\leqslant Q}\; \sum_{\dd|\q}
\lambda_n(\q/\dd)N(\dd,P_{Q/\norm\q}).
\]
Proposition \ref{MainCountingResultGLn} show that $N(\q,P_X)$ is
\begin{equation}\label{applyS2S}
D_F^{n^2/2}\Delta_F^*(1)\varphi_n(\q)\int_{P_X}{\rm d}\widehat{\omega}_{\infty}^{\pl}+{\rm O}\big(\varphi_n(\q)(\partial\vol_R(P_X)+\vol_R^{\star}(P_X))+\norm\q^{n-\theta}\overline{\vol}_R(P_X)\big),
\end{equation}
for $1\leqslant R<c\log (2+\norm\q)$ and $\norm\q\geqslant C$. Summing the error terms in \eqref{applyS2S} over all $\q$ and $\dd\mid\q$, we recover the three boundary error terms in the lemma.

The second part of Proposition \ref{MainCountingResultGLn} furthermore shows that $N(\q,P_X)\ll\varphi_n(\q)\int_{P_X}{\rm d}\widehat{\omega}_{\infty}^{\pl}$ for $\norm\q\leqslant C$. From the trivial estimate
\begin{equation}\label{triv-bd}
\sum_{\substack{\dd\mid\q\\\norm\dd\leqslant X}}|\lambda_n(\q/\dd)|\varphi_n(\dd)\ll_\epsilon X^n\norm\q^{\epsilon}
\end{equation}
and Corollary~\ref{Plemma2}, we easily deduce
\[
\sum_{\norm\q\leqslant Q}\sum_{\substack{\dd\mid\q\\\norm\dd\leqslant C}}|\lambda_n(\q/\dd)|\varphi_n(\dd) \int_{P_{Q/\norm\q}}\mathrm{d}\widehat{\omega}_{\infty}^{\pl}\ll_\epsilon Q^{n-1/d+\epsilon},
\]
which is clearly admissible as an error term.

To obtain the contribution of $\mathscr{C}(\mathfrak{F}) Q^{n+1}$, first note that the summation on the main term in \eqref{applyS2S} can be extended back over all $\dd$ since the contribution from divisors with $\norm\dd\leqslant C$ is estimated as above. We conclude by an application of Corollary \ref{MT-is-vol}.
\end{proof}

We must now prove satisfactory bounds on the error terms appearing in Lemma \ref{temp-cor}. The remaining sections will be dedicated to proving the following

\begin{prop}\label{temp-err}
There is $\theta>0$ and a choice of admissible sequence $\mathscr{R}$ such that
\[
\partial\vol_{\mathscr{R}}(Q)\ll Q^{n+1}/\log Q,\quad \vol^{\star}_{\mathscr{R}}(Q)\ll Q^{n+1}/\log^3Q,\quad
\overline{\vol}{}^{\theta}_{\mathscr{R}}(Q)\ll Q^{n+1-\theta}.
\]
\end{prop}

Granting ourselves the proposition, we deduce the following

\begin{cor}\label{cor:ELM-implies-Conj}
Property $(\mathrm{ELM})$ implies Conjecture \ref{weyl-schanuel-conj}, in the effective form of Theorem \ref{main-implication}.
\end{cor}

\begin{proof}
This follows immediately from Lemma \ref{temp-cor} and Proposition \ref{temp-err}.
\end{proof}

\subsection{A preparatory lemma}

The following lemma will go a long way towards proving Proposition \ref{temp-err}. It bounds very similar quantities to $\partial\vol_{\mathscr{R}}(Q)$, $\vol_{\mathscr{R}}^{\star}(Q)$, and $\overline{\vol}_{\mathscr{R}}(Q)$, but with the arithmetic weights in the average replaced by powers of the norm, and with the admissible sequence $\mathscr{R}$ taken to be constantly equal to the real number $R$. Dealing with these arithmetic weights and choosing an appropriate admissible sequence to prove Proposition \ref{temp-err} will be done in \S\ref{temp-err-proof}.

\begin{lemma}\label{boundary}
Let $0<\theta\leqslant 2/(d+1)$ and $\sigma>n-1/d-1+\theta$. For $R\gg 1$, $Q\geqslant 1$, we have
\begin{align*}
\sum_{1\leqslant\norm\q\leqslant Q}\norm\q^{\sigma}\partial\vol_R(P_{Q/\norm\q})
&={\rm O}_{\sigma,\theta}\big(R^{-1}Q^{\sigma+1}+Q^{\sigma+1-\theta}\big),\\
\sum_{1\leqslant\norm\q\leqslant Q}\norm\q^{\sigma}\overline{\vol}_R(P_{Q/\norm\q})&=\mathrm{O}_\sigma(Q^{\sigma+1}),\\
\sum_{1\leqslant\norm\q\leqslant Q}\norm\q^{\sigma}\vol^\star_R(P_{Q/\norm\q})
&=\mathrm{O}_\sigma(R^{-3}Q^{\sigma+1}).
\end{align*}
\end{lemma}

We remark that this lemma is one of the places which put requirements on the integer $N$ implicit in the volume factors; for example, for purposes of this lemma, $N\geqslant 3+d(\sigma+1)$ suffices.

\begin{proof}
The first sum equals
\[
\sum_{1\leqslant\norm\q\leqslant Q}\norm\q^{\sigma}\partial\vol_R(P_{Q/\norm\q})
=\sum_{\ell=1}^{\infty}\ell^{-N}\sum_{1\leqslant\norm\q\leqslant Q}\norm\q^{\sigma}\int_{\partial P_{Q/\norm\q}(\ell/R)}\beta(\pi)\,\mathrm{d}\pi,
\]
where the measure $\beta(\pi)\,\mathrm{d}\pi$ is as in \eqref{Pl-maj}, and we use the definition \eqref{BoundaryTermsDefinitions1} directly (rather than Lemma~\ref{boundary-term-reduction}) to emphasize steps that apply also for the third sum. Now for any $r>0$ we have
\[
\sum_{1\leqslant\norm\q\leqslant Q}\norm\q^{\sigma}\int_{\partial P_{Q/\norm\q}(r)}\beta(\pi)\,\mathrm{d}\pi=\int_{\Pi(G_{\infty}^1)}\bigg(\sum_{\substack{1\leqslant\norm\q\leqslant Q\\ \pi\in\partial P_{Q/\norm\q}(r)}}\norm\q^{\sigma}\bigg)\, \beta(\pi)\,\mathrm{d} \pi,
\]
upon exchanging the order of the sum and integral. From Lemma~\ref{SigmaBounds} below we deduce that the right-hand side is majorized by
\[
r\big(1+r^{d(\sigma+1)}\big) Q^{\sigma+1}\int_{\Pi(G_\infty^1)}q(\pi)^{-\sigma-1}\beta(\pi)\,\mathrm{d} \pi+(1+r)^{d(\sigma+1-\theta)}Q^{\sigma+1-\theta}\int_{\Pi(G_\infty^1)}q(\pi)^{-\sigma-1+\theta}\beta(\pi)\,\mathrm{d} \pi.
\]
In view of Lemma~\ref{Plemma}, both integrals converge, yielding
\[ \sum_{1\leqslant\norm\q\leqslant Q}\norm\q^\sigma\, \int_{\partial\, P_{Q/\norm\q}(r)}\beta(\pi)\,\mathrm{d} \pi
\ll_{\sigma,\theta}r\big(1+r^{d(\sigma+1)}\big) Q^{\sigma+1}+(1+r)^{d(\sigma+1-\theta)}Q^{\sigma+1-\theta}. \]
Applying the above estimate with $r=\ell/R>0$ and executing the sum over $\ell\geqslant1$, we get
\begin{align*}
\sum_{1\leqslant\norm\q\leqslant Q}\norm\q^{\sigma}\partial\vol_R(P_{Q/\norm\q})
&\ll\sum_{\ell=1}^{\infty}\ell^{-N}\bigg(\frac{\ell}{R}\bigg(1+\left(\frac{\ell}{R}\right)^{r(\sigma+1)}\bigg)Q^{\sigma+1}+\bigg(1+\frac{\ell}{R}\bigg)^{d(\sigma+1-\theta)} Q^{\sigma+1-\theta}\bigg)\\
&\ll R^{-1}\big(1+R^{-d(\sigma+1)}\big) Q^{\sigma+1}+\big(1+R^{-d(\sigma+1-\theta)}\big) Q^{\sigma+1-\theta}.
\end{align*}
The last expression is majorized by $R^{-1}Q^{\sigma+1}+Q^{\sigma+1-\theta}$. 

For the remaining two estimates, recalling the definitions in \S\ref{sec:temp-bd}, it is enough to show that for any standard Levi $L\in\mathcal{L}_\infty(M)$ we have
\[
\sum_{1\leqslant\norm\q\leqslant Q}\norm\q^{\sigma}\overline{\vol}_{R,L}(P_{Q/\norm\q})
=\mathrm{O}_\sigma(R^{-{\rm codim}_{\h_M}(\h_L)-2|W(M,L)_\delta|}Q^{\sigma+1}),
\]
where $|W(M,L)_\delta|$ is as in \S\ref{sec:herm}. When $L=M$ (corresponding to the second sum of the lemma, over volume terms $\overline{\vol}_R(P_{Q/\norm\q})$), the above estimate follows from
\begin{align*}
\sum_{1\leqslant\norm\q\leqslant Q}\norm\q^{\sigma}\int_{P_{Q/\norm\q}}\beta(\pi)\,\mathrm{d}\pi
&=\int_{\Pi(G_{\infty}^1)}\bigg(\sum_{1\leqslant\norm\q\leqslant Q/q(\pi)}\norm\q^{\sigma} \bigg)\beta(\pi)\,\mathrm{d}\pi\\
&\ll_\sigma Q^{\sigma+1}\int_{\Pi(G_{\infty}^1)}q(\pi)^{-\sigma-1}\beta(\pi)\,\mathrm{d}\pi,
\end{align*}
the last integral converging in view of Lemma \ref{Plemma} and our assumptions on $\sigma$.

For $L$ strictly larger than $M$ (corresponding to the third sum of the lemma, over volume terms $\vol_R^\star (P_{Q/\norm\q}$)), similar manipulations apply (in light of Lemma~\ref{L1NormEstimateLemma}), and one has only to verify the convergence of the remaining integral, but with integration of the Plancherel majorizer $\beta_M^G(\delta,\nu)$ taking place over the planar sections $i\h_L^{\ast}\cap P_{\delta,X}$. From Lemma \ref{prop:density}\eqref{beta2} it follows that the integral of $\beta_M^G(\delta,\nu)$ over such a planar section is dominated by that over its $\mathrm{O}(1)$-tubular fattening. Lemma \ref{Plemma} can then be applied to complete the proof. \end{proof}

We now establish the following result, which was used in the proof of Lemma \ref{boundary}.

\begin{lemma}
\label{SigmaBounds}
Let $\pi\in\Pi_{\rm temp}(G_\infty^1)$. Let $r>0$, $0<\theta\leqslant 2/(d+1)$, and $\sigma\geqslant -1+\theta$. Then
\[
\sum_{\substack{1\leqslant\norm\q\leqslant Q\\ \pi\in\partial P_{Q/\norm\q}(r)}}\norm\q^{\sigma}\ll_{\sigma,\theta}r\big(1+r^{d(\sigma+1)}\big)\left(\frac{Q}{q(\pi)}\right)^{\sigma+1}+(1+r)^{d(\sigma+1-\theta)}\left(\frac{Q}{q(\pi)}\right)^{\sigma+1-\theta}.
\]
\end{lemma}

\begin{proof}
We first convert the condition $\pi\in\partial P_{Q/\norm\q}(r)$ on the ideal $\q$ to a more amenable condition on the norm of $\q$. We may assume that $\pi=\pi_{\delta,\nu}$ for some $\delta\in\mathscr{E}^2(M^1)$ and $\nu\in i\mathfrak{h}_M^*$, where $M$ is a standard cuspidal Levi subgroup.

For parameters $r,X>0$ we have $\pi\in\partial P_X(r)$ precisely when there is $\mu\in i\mathfrak{h}_M^*$ with $\|\mu-\nu\|<r$ and $q(\pi_{\delta,\mu})=X$. Letting $M_r(\pi)$ (resp., $m_r(\pi)$) denote the maximum (resp. minimum) value of $q(\pi_{\delta,\mu})$ as $\mu$ varies over $i\mathfrak{h}_M^*$ with $\|\mu-\nu\|<r$, we see that $\pi\in\partial P_{Q/\norm\q}(r)$ implies $m_r(\pi)\leqslant Q/\norm\q\leqslant M_r(\pi)$, so that
\[
\sum_{\substack{1\leqslant\norm\q\leqslant Q\\ \pi\in\partial P_{Q/\norm\q}^0(r)}}\norm\q^{\sigma}\leqslant \sum_{Q/M_r(\pi)\leqslant\norm\q\leqslant Q/m_r(\pi)}\norm\q^{\sigma}.
\]
Using the asymptotic~\eqref{SigmaIdealCountingAsymptotics} we see that this is bounded by
\begin{equation}\label{reduction2m}
\frac{\zeta_F^{\ast}(1)}{\sigma+1}Q^{\sigma+1}\bigg(\frac1{m_r(\pi)^{\sigma+1}}-\frac1{M_r(\pi)^{\sigma+1}}\bigg)+\text{O}\bigg(\frac{Q^{\sigma+1-\theta}}{m_r(\pi)^{\sigma+1-\theta}}\bigg).
\end{equation}
To conclude the proof of the lemma we shall need to relate $m_r(\pi)$ and $M_r(\pi)$ to expressions involving $(1+r)$ and $q(\pi)$. This will require some basic analytic properties of the archimedean conductor $q(\pi_{\delta,\nu})$ as $\nu$ varies.

For $\nu\in i\mathfrak{h}^*_M$ and $r>0$, let $\nu(r)\in i\mathfrak{h}^*_M$ denote the translation $\nu+r\nu_0$, for some fixed $\nu_0\in i\mathfrak{h}_M^*$ in the positive chamber, and write $\pi(r)=\pi_{\delta,\nu(r)}$. Since $q(\pi_{\delta,\nu})$ is monotonically increasing in $\nu$, it follows that for $\nu_0$ large enough we have $M_r(\pi)\leqslant q(\pi(r))$. Similarly, $q(\pi(-r))\leqslant m_r(\pi)$, for $r\leqslant\frac12\|\nu\|$, say. In this interval we have $q(\pi(-r))\asymp q(\pi)$, while, if $r>\frac12\|\nu\|$, we have
\[
q(\pi(-r))\geqslant q(\pi_{\delta,0})\gg (1+\|\nu\|)^{-[F_v:\R]}q(\pi)\gg (1+r)^{-d}q(\pi).
\]
Thus, in either case, $q(\pi(-r))\gg (1+r)^{-d}q(\pi)$, proving $m_r(\pi)\gg (1+r)^{-d}q(\pi)$. When inserted into the second term in \eqref{reduction2m} we obtain the second term of the lemma.

Now let $s\mapsto \pi(s)$ be a unit length parametrization of the line between $\nu(-r)$ and $\nu(r)$ in $i\mathfrak{h}_M^*$. Since $s\mapsto q(\pi(s))$ is a real-valued differentiable map on the interval $[0,1]$, we have
\[
q(\pi(-r))^{\sigma}-q(\pi(r))^{\sigma}
=\int_0^1\frac{\text{d}}{\text{d}s}q(\pi(s))^{\sigma}\,\text{d}s=\int_0^1\sigma q(\pi(s))^{\sigma-1}\frac{{\rm d} q(\pi(s))}{{\rm d}\pi(s)}\frac{\text{d}\pi(s)}{\text{d}s}\,\text{d}s.
\]
Since $\frac{{\rm d} q(\pi(s))}{{\rm d}\pi(s)}\ll q(\pi(s))$ and $\frac{\text{d}\pi(s)}{\text{d}s}\ll r$, the latter integral is bounded by
\[
\ll r|\sigma|\int_0^1q(\pi(s))^{\sigma}\,\text{d}s\ll_\sigma rq(\pi(-r))^\sigma\ll_\sigma r(1+r)^{-d\sigma}q(\pi)^{\sigma}.
\]
Since $\sigma\leqslant 0$ we have $(1+r)^{-d\sigma}\asymp_{\sigma}(1+r^{-d\sigma})$, proving $m_r(\pi)^{\sigma}-M_r(\pi)^{\sigma}\ll_{\sigma}r\big(1+r^{-d\sigma}\big)q(\pi)^{\sigma}$. Inserting this into the first term of \eqref{reduction2m} then completes the proof of the lemma.
\end{proof}

\subsection{End of proof}\label{temp-err-proof}

We now return to the proof of Proposition \ref{temp-err}.

We first choose the sequence $\mathscr{R}=(R(\mathfrak{n}))_{\mathfrak{n}\subseteq\mathcal{O}_F}$ of the form
\begin{equation}\label{admR}
R(\mathfrak{n})=
\begin{cases}
R_1,& \text{ if } Q^{1/2}<\norm\mathfrak{n}\leqslant Q,\\
R_2,& \text{ if } \norm\mathfrak{n}\leqslant Q^{1/2},
\end{cases}
\end{equation}
where $R_1,R_2>0$ will be chosen shortly (and $Q^\sigma$ works as a cutoff for any $\sigma\in(0,1]$). With the above choice of $\mathscr{R}$ the term $\partial\vol_{\mathscr{R}}(Q)$ is equal to
\[
\sum_{\norm\q\leqslant Q}\sum_{\substack{\dd\mid\q\\ Q^{1/2}<\norm\dd\leqslant Q}}|\lambda_n(\q/\dd)|\varphi_n(\dd)\partial\vol_{R_1}(P_{Q/\norm\q})
+\sum_{\norm\q\leqslant Q}\sum_{\substack{\dd\mid\q\\ \norm\dd\leqslant Q^{1/2}}}|\lambda_n(\q/\dd)|\varphi_n(\dd)\partial\vol_{R_2}(P_{Q/\norm\q}).
\]
Bounding the first term using
\[
\sum_{\dd\mid\q}|\lambda_n(\q/\dd)|\varphi_n(\dd)
\leqslant\norm\q^n\prod_{\pp\mid\q}(1-\norm\pp^{-n})\big(1+n\norm\pp^{-n}+\tbinom n2\norm\pp^{-2n}+\cdots\big)\asymp\norm\q^n
\]
and second term using \eqref{triv-bd}, the above expression can be majorized by
\[
\sum_{\norm\q\leqslant Q}\norm\q^n\partial\vol_{R_1}(P_{Q/\norm\q})+Q^{n/2}\sum_{\norm\q\leqslant Q}\norm\q^\epsilon\partial\vol_{R_2}(P_{Q/\norm\q}).
\]
Combining this with Lemma~\ref{boundary} shows that $\partial\vol_{\mathscr{R}}(Q)\ll_\epsilon R_1^{-1}Q^{n+1}+ R_2^{-1}Q^{\frac{n}{2}+1+\epsilon}+Q^{n+1-\theta}$. Similarly one obtains $\vol_{\mathscr{R}}^\star (Q)\ll_\epsilon R_1^{-3}Q^{n+1}+ R_2^{-3}Q^{\frac{n}{2}+1+\epsilon}$ and $\overline{\vol}{}_{\mathscr{R}}^{\theta}(Q)\ll Q^{n+1-\theta}$.

Now let $c>0$ be as in Proposition \ref{MainCountingResultGLn}. Then taking $R_1=\frac{c}{2}\log Q$ and $R_2=c$ in the definition of $\mathscr{R}$ in \eqref{admR} yields an admissible sequence according to the definition preceding Corollary \ref{temp-cor}. Inserting these values establishes the stated bounds of Proposition \ref{temp-err}.
\qed

\part{Proof of Theorem \ref{master}}\label{part3}

Theorem \ref{master} is proved using the Arthur trace formula, which we recalled in Section \ref{sec:ATF}. In this part, whose overview may be found in \S\ref{blurb}, we construct suitable test functions and estimate the geometric and spectral contributions in the trace formula, culminating in the proof of Theorem \ref{master} in Theorem \ref{money-cor}.

\section{Bounding the non-central geometric contributions}\label{sec:geom1}

Arthur defines a distribution $J_{\rm geom}$ on $\mathcal{H}(\bm{G}(\A_F)^1)$ related to geometric invariants of $\bm{G}$; we described its general shape in \S\ref{fine-geom}. This distribution $J_{\rm geom}$ admits an expansion along semisimple conjugacy classes of $\bm{G}(F)$, and our task in this section is to bound all but the most singular terms (the central contributions, shown explicitly in \eqref{eq:def:explicit-Jcentral}) appearing in this expansion. We must do so uniformly with respect to the level and the support of the test functions at infinity.

\begin{theorem}\label{geom-side}
Let $n\geqslant 1$. There exists $\theta>0$ and $c>0$ satisfying the following property. For any integral ideal $\q$, $R>0$, and $f\in\mathcal{H}(G_\infty^1)_R$ we have
\[
J_{\rm geom}(\varepsilon_{K_1(\q)}\otimes f)-D_F^{n^2/2}\Delta_F^*(1)\varphi_n(\q)\sum_{\gamma\in Z(F)\cap K_1(\q)}  f(\gamma)\ll e^{cR}\norm\q^{n-\theta}\|f\|_\infty.
\]
The implied constant depends on $F$ and $n$.
\end{theorem}

Our presentation is by and large based on the papers \cite{FinisLapid2018}, \cite{Matz2017}, and \cite{MatzTemplier2021}. Many aspects of our argument are simplified by the absence of Hecke operators in our context. On the other hand, we have to make explicit (in various places) the dependence in $R$. The proof of Theorem 13.1 spans Sections \ref{sec:geom1} and \ref{sec:geom2}. In the first of these, we bound the number of contributing terms as well as the global coefficients. This reduces the problem to a local one, which we treat in Section \ref{sec:geom2}, on bounding weighted orbital integrals. In the course of the proof, the value of $c$ can change from instance to another.

\subsection{Contributing classes}

We now wish to bound the number of equivalence classes $\mathfrak{o}\in\mathfrak{O}$ contributing to the coarse geometric expansion \eqref{eq:J-geom-def} of $J_{\rm geom}(\varepsilon_{K_1(\q)}\otimes f)$, for $f\in\mathcal{H}(G_\infty^1)_R$. For this it clearly suffices to take $\q$ trivial.

\begin{defn}
For $R>0$ let $\mathfrak{O}_R$ denote the set of $\mathfrak{o}\in\mathfrak{O}$ for which there is $f\in\mathcal{H}(G_\infty^1)_R$ with $J_\mathfrak{o}({\bf 1}_{\bfk_{\rm fin}}\otimes f)\neq 0$.
\end{defn}

Our main result in this subsection is the following estimate. The argument is based largely on \cite[Lemma 6.10]{Matz2017}. 

\begin{prop}\label{task1}
There is $c>0$, depending on $n$ and $[F:\Q]$, such that $|\mathfrak{O}_R|\ll e^{cR}$.
\end{prop}

\begin{proof}
Let $\mathfrak{o}$ be a semisimple $\bm{G}(F)$-conjugacy class represented by some semisimple element $\sigma\in \bm{G}(F)$. Let $\chi_\mathfrak{o}$ denote the characteristic polynomial of $\sigma$. It is a monic polynomial of degree $n$ with coefficients in $F$, independent of the choice of $\sigma$, and the map $\mathfrak{o}\mapsto\chi_\mathfrak{o}$ is a bijection onto such polynomials. We shall count the $\mathfrak{o}$ appearing in $\mathfrak{O}_R$ by counting the corresponding $\chi_\mathfrak{o}$.

As in \S\ref{sec-review}, let $\mathcal{U}_{\bm{G}_\sigma}$ denote the algebraic variety of unipotent elements in the centralizer $\bm{G}_\sigma$. The condition $\mathfrak{o}\in\mathfrak{O}_R$ is equivalent to the following collection of local conditions at every place $v$:
\begin{enumerate}
\item for every $v\nmid\infty$ there is $\nu_v\in\mathcal{U}_{G_\sigma}(F_v)$ such that the $G_v$-conjugacy class of $\sigma\nu_v$ meets $\bfk_v$;
\item there is $\nu_\infty\in\mathcal{U}_{\bm{G}_\sigma}(F_\infty)$ such that the $G_\infty$-conjugacy class of $\sigma\nu_\infty$ meets $G_{\infty,\leqslant R}^1$.
\end{enumerate}
Note that any element $\widetilde{\gamma}\in \bm{G}(\A_F)$ lying in the $\bm{G}(\A_F)$-conjugacy class of $\mathfrak{o}$ has characteristic polynomial equal to $\chi_\mathfrak{o}$. From the above local conditions we deduce that the coefficients of $\chi_\mathfrak{o}$ (for $\mathfrak{o}\in\mathfrak{O}_R$) are $v$-integral for all finite $v$, and so lie in $\mathcal{O}_F$. Moreover, their archimedean absolute value is bounded by $e^{cR}$ for some constant $c>0$. Each coefficient of $\chi_\mathfrak{o}$ for contributing classes $\mathfrak{o}$ then lies in the intersection of $\mathcal{O}_F\subset F$ with $\prod_{v\mid\infty} [-X,X]\subset F_\infty$. As there are at most ${\rm O}(X^{[F : \Q]})$ such lattice points, the proposition follows.
\end{proof}

For a contributing class $\mathfrak{o}$, and $S_\mathfrak{o}$ is defined in \eqref{def-S-o}, we will now apply the fine geometric expansion \eqref{eq:J-geom-fine} in the case where
\begin{equation}\label{S-phi}
S=S_\mathfrak{o} \cup S_\q\cup S_\infty\qquad\text{and}\qquad\phi_S=\prod_{v\in S_\mathfrak{o}, v\notin S_\q}{\bf 1}_{\bfk_v}\otimes\prod_{v\in S_\q}\varepsilon_{K_{1,v}(\q)}\otimes f,
\end{equation}
for the $f\in\mathcal{H}(G_\infty^1)_R$ appearing in the statement of Theorem \ref{geom-side}.

\subsection{Bounding global coefficients}\label{sec:global-coeffs}

Next we bound the coefficients $a^{\bm{M}}(\gamma,S_\mathfrak{o})$, for $\gamma$ lying in a contributing class $\mathfrak{o}$. Once again, we are free to assume that $\q$ is trivial, so that $\mathfrak{o}\in\mathfrak{O}_R$.

For any finite set of finite places $T$ we put
\[
q_T=\prod_{v\in T}q_v.
\]
We begin with the following useful result.

\begin{lemma}\label{prod-form}
For $\mathfrak{o}\in\mathfrak{O}_R$ we have $q_{S_\mathfrak{o}}\ll e^{cR}$. In particular, 
\begin{equation}\label{So}
|S_\mathfrak{o}|\ll R
\end{equation}
for $\mathfrak{o}\in\mathfrak{O}_R$.
\end{lemma}

\begin{proof}
If $\mathfrak{o}\in\mathfrak{O}_R$ and $\sigma\in \bm{G}(F)$ is a semisimple element representing $\mathfrak{o}$, then there is $y\in \bm{G}(F)$ such that $y^{-1}\sigma\mathcal{U}_{\bm{G}_\sigma}(F)y\cap \bfk_f\neq\emptyset$. In other words there are $y\in \bm{G}(F)$ and $u\in\mathcal{U}_{\bm{G}_\sigma}(F)$ such $y^{-1}\sigma uy\in \bfk_f$. Thus, for every finite place $v$ we have $D_v(\sigma)=D_v(y^{-1}\sigma uy)\leqslant 1$. Moreover, it follows from \cite[Lemma 4.4]{MatzTemplier2021} that (under the same assumptions on $\mathfrak{o}$ and $\sigma$) 
\begin{equation}\label{WDisc}
D_\infty(\sigma)\ll e^{cR}.
\end{equation}
Thus for $\mathfrak{o}\in\mathfrak{O}_R$ we in fact have
\[
S_\mathfrak{o}=S_{\rm wild}\cup \{v<\infty: D_v(\gamma)< 1\}=S_{\rm wild}\cup \{v<\infty: D_v(\gamma)\leqslant q_v^{-1}\}.
\]
By the product formula, we deduce that
\[
1=\prod_{v\notin S_\mathfrak{o}\cup S_\infty} D_v(\gamma)\prod_{v\in S_\mathfrak{o}} D_v(\gamma)\prod_{v\in S_\infty} D_v(\gamma)\ll q_{S_{\rm wild}}q_{S_\mathfrak{o}}^{-1}e^{cR}\ll_n q_{S_\mathfrak{o}}^{-1}e^{cR},
\]
as desired.

To deduce \eqref{So} from this we note $\sum_{v\in S_\mathfrak{o}}1\leqslant \sum_{v\in S_\mathfrak{o}}\log q_v=\log q_{S_\mathfrak{o}}\ll R$.
\end{proof}

For the next estimate, we invoke the main result of \cite{Matz2015}, a corollary of which is the following. Let $\sigma$ be elliptic semisimple in $\bm{M}(F)$. Then $\sigma$ is conjugate in $\bm{M}(\C)$ to a diagonal matrix ${\rm diag}(\zeta_1,\ldots ,\zeta_n)$. Let
\[
\Delta^{\bm{M}}(\sigma)=\norm\left(\prod_{i<j:\zeta_i\neq\zeta_j}(\zeta_i-\zeta_j)^2\right),
\]
where $\norm$ is the norm from $F$ to $\Q$, and the product is taken over indices $i<j$ such that the string $\alpha_i+\cdots+\alpha_j$ lies in the set of positive roots $\Phi^{\bm{M},+}=\Phi^{\bm{G},+}\cap\Phi^{\bm{M}}$ for $\bm{M}$. Then it follows from \cite{Matz2015} (see also \cite[(22)]{Matz2017}) that there is $\kappa>0$ such that for any finite set of places $S$ containing $S_\infty$, if $\gamma=\sigma\nu$ is the Jordan decomposition of $\gamma$, we have\footnote{The measure normalization in \cite{Matz2015} differs from ours by a power of the discriminant of $F$. This is not an issue, since our implied constants are allowed to depend on $F$ and $n$.}
\begin{equation}\label{Matz-bound}
a^{\bm{M}} (\gamma,S)\ll |S|^n\Delta^{\bm{M}}(\sigma)^\kappa \bigg(\prod_{v\in S}\log q_v\bigg)^n
\end{equation}
for $\sigma$ elliptic in $\bm{M}(F)$ and $a^{\bm{M}} (\gamma,S)=0$ otherwise.

\begin{prop}\label{task2}
There is $c>0$ such that for any $\bm{M}\in\mathcal{L}$, any $\mathfrak{o}\in\mathfrak{O}_R$ meeting $\bm{M}$, and any $\gamma\in\mathfrak{o}$, we have $a^{\bm{M}}(\gamma,S_\mathfrak{o})=\mathrm{O}(e^{cR})$.
\end{prop}

\begin{proof}
It follows from \eqref{So} that the factor of $|S_\mathfrak{o}|^n$ in the upper bound \eqref{Matz-bound} is at most $\mathrm{O}(R^n)$. To bound the factor $\Delta^M(\sigma)$ we follow the argument of \cite[Lemma 6.10, (iv)]{Matz2017}. The eigenvalues $\zeta_1,\ldots ,\zeta_n$ are the roots of the characteristic polynomial of $\mathfrak{o}$. As in the proof of Proposition \ref{task1}, for $\mathfrak{o}\in\mathfrak{O}_R$, these coefficients have $v$-adic (for archimedean $v$) absolute value bounded by $\mathrm{O}(e^{cR})$. An application of Rouch\'e's theorem shows that each $\zeta_i$ has complex absolute value bound by $\mathrm{O}(e^{cR})$, from which it follows that $\Delta^{\bm{M}}(\sigma)\ll e^{cR}$. To bound the last factor in \eqref{Matz-bound} we apply the first part of Lemma \ref{prod-form}.
\end{proof}

\subsection{Reduction to local estimates}

Now we return to the setting where $S$ is an admissible set of places of $F$, as in \S\ref{fine-geom}.

We first recall that for factorizable test functions $\phi_S=\otimes_{v\in S}\phi_v$ and $\gamma\in \bm{M}(F)$ one has a splitting formula which reduces $J_{\bm{M}}(\gamma,\phi_S)$ to a sum of products of local distributions. Write $M_v=\bm{M}(F_v)$ and put $M=\prod_{v\in S}M_v\in\mathcal{L}_S$. More precisely (see \cite[Lemma 6.11]{Matz2017} or \cite[(11.4)]{MatzTemplier2021}), there are real numbers $\{d_{\bm{M}}(L_S)\}$, indexed by Levi subgroups $L_S=(L_v)_{v\in S}\in\mathcal{L}_S(M)$, such that
\[
J_{\bm{M}}(\gamma,\phi_S)=\sum_{L_S\in\mathcal{L}_S(M)}d_{\bm{M}}(L_S)\prod_{v\in S} J_{M_v}^{L_v}(\gamma_v,\phi_v^{(Q_v)}).
\]
Here, we are using an assignment $\mathcal{L}_v\ni L_v\mapsto Q_v\in\mathcal{P}_v(L_v)$ which is independent of $S$, and for every $v\in S$ the element $\gamma_v\in M_v$ is taken to be $M_v$-conjugate to $\gamma$. The properties of interest for us on the coefficients $d_{\bm{M}}(L_S)$ are the following, proved in \cite[Lemma 6.11]{Matz2017}:
\begin{enumerate}
\item as $L_S$ varies, the coefficients $d_{\bm{M}}(L_S)$ can attain only a finite number of values; these values depend only on $n$.
\item the number of contributing Levi subgroups $L_S$ can be bounded as
\[
|\{L_S: d_{\bm{M}}(L_S)\neq 0\}|\ll |S|^{n-1}.
\]
\end{enumerate}

From these properties, it follows immediately that for any $\mathfrak{o}\in\mathfrak{O}$, $\gamma\in\mathfrak{o}$, admissible $S$, and factorisable $\phi_S=\otimes_{v\in S}\phi_v\in \mathcal{H}(\bm{G}_S)$ we have
\[
J_{\bm{M}}(\gamma,\phi_S)\ll |S|^{n-1}\max_{L_S\in\mathcal{L}_S(M)}\prod_{v\in S} |J_{M_v}^{L_v}(\gamma_v,\phi_v^{(Q_v)})|.
\]
Thus, if $\mathfrak{o}\in\mathfrak{O}_R$, $f\in \mathcal{H}(G_\infty^1)_R$, and $S$ and $\phi_S$ are taken as in \eqref{S-phi} we obtain
\begin{equation}\label{red-2-loc}
\begin{aligned}
J_{\bm{M}}(\gamma,\varepsilon_{K_1(\q)}\otimes f)\ll R^{n-1}&\max_{L_S\in\mathcal{L}_S(M)}\prod_{v\mid\infty}|J_{M_v}^{L_v}(\gamma_v,f_v^{(Q_v)})|\\
&\times \prod_{v\in S_\q} |J_{M_v}^{L_v}(\gamma_v,\varepsilon_{K_{1,v}(\q)}^{(Q_v)})|\prod_{v\in S_\mathfrak{o}, v\notin S_\q}|J_{M_v}^{L_v}(\gamma_v,{\bf 1}_{\bfk_v^{L_v}})|,
\end{aligned}
\end{equation}
where we have used \eqref{So} as well as the fact (see, for example, \cite[\S 7.5]{Matz2017} or \cite[Lemma 6.2]{ShinTemplier2016}) that ${\bf 1}_{\bfk_v}^{(Q_v)}={\bf 1}_{\bfk_v^{L_v}}$.

\section{Estimates on local weighted orbital integrals}\label{sec:geom2}

It remains to bound the local weighted orbital integrals appearing in \eqref{red-2-loc}. In this section, we provide (or recall) such bounds at every place, and show how they suffice to establish Theorem \ref{geom-side}. As we mentioned in Remark \ref{rem:Weyl-disc}, the dependence in our bounds on $\gamma$ will be expressed in terms of a Weyl dscriminant. Specifically, when $M\in\mathcal{L}_v$ and $\gamma=\sigma\nu\in M$ we put
\[
D^M_v(\gamma)=|\det (1-{\rm Ad}(\sigma)_{|\mathfrak{m}_v/\mathfrak{m}_{\sigma,v}})|_v=\prod_{\substack{\alpha\in\Phi_v^M\\\alpha(\sigma)\neq 1}}|1-\alpha(\sigma)|_v.
\]
When $M=G$, this recovers the definition of $D^G_v(\gamma)$ from \S\ref{subsec:Weyl-disc}.

For $v$ dividing $\q$, we offer the following proposition, the proof of which is based heavily on works of Finis--Lapid \cite{FinisLapid2018}, Matz \cite{Matz2017}, and Shin--Templier \cite{ShinTemplier2016}.

\begin{prop}\label{v-div-q}
There are constants $B,C,\theta>0$ such that the following holds. Let $v$ be a finite place. Let $M\in\mathcal{L}_v$, $L\in\mathcal{L}_v(M)$, and $Q\in\mathcal{P}_v(L)$. Then for any $r\geqslant 0$ and $\gamma=\sigma\nu\in \sigma\mathcal{U}_{G_\sigma}\cap M$ and $r\geqslant 0$ we have
\[
J_M^L(\gamma,{\bf 1}_{K_{1,v}(\p_v^r)}^{(Q)})\ll q_v^{aB-\theta r} D^L_v(\sigma)^{-C},
\]
where $a=0$ whenever the residue characteristic of $F_v$ is larger than $n!$ and $v\notin S_\mathfrak{o}$, and $a=1$ otherwise.
\end{prop}

A great deal of work has been recently done by Matz \cite{Matz2017} and Matz--Templier \cite{MatzTemplier2021} in bounding archimedean weighted orbital integal for $\GL_n$. Their bounds are almost sufficient for our purposes, except for the dependency in the support of $R$. By simply explicating this dependence in their proofs, we obtain the following result.

\begin{prop}\label{v-arch}
There are constants $c,C>0$ satisfying the following property. Let $v\mid\infty$. Let $M\in\mathcal{L}_v$, $L\in\mathcal{L}_v(M)$, and $Q\in\mathcal{P}_v(L)$. Then for any $R>0$ and $\gamma=\sigma\nu\in \sigma\mathcal{U}_{G_\sigma}\cap M$ and $f\in \mathcal{H}(G_v^1)_R$ we have
\[
J_M^L(\gamma,f^{(Q)})\ll e^{cR} D^L_v(\sigma)^{-C}\|f\|_\infty.
\]
\end{prop}

\subsection{Deduction of Theorem \ref{geom-side}}

We now show how the above results imply Theorem \ref{geom-side}. We will need an additional result for places $v\in S_\mathfrak{o}, v\notin S_\q$ (as in the last factor of \eqref{red-2-loc}). Namely, it is proved in \cite[Corollary 10.13]{Matz2017} that there are constants $B,C>0$ such that for any finite place $v$, and any $M\in\mathcal{L}_v$, $L\in\mathcal{L}_v(M)$, and $\gamma\in L$, one has
\begin{equation}\label{notin-q}
J_M^L(\gamma, {\bf 1}_{\bfk^L_v})\ll q_v^B D^L_v(\sigma)^{-C}.
\end{equation}

Returning to the global situation of Theorem \ref{geom-side}, we let $\mathfrak{o}\in\mathfrak{O}$ be such that $\mathfrak{o}\cap \bm{M}(F)$ is non-empty, and let $\sigma\in \bm{M}(F)$ be a semisimple element representing $\mathfrak{o}$. We may assume that $\mathfrak{o}\in\mathfrak{O}_R$, for otherwise $J_{\bm{M}}(\gamma, \varepsilon_{K_1(\q)}\otimes f)=0$. We apply \eqref{red-2-loc} to reduce to a product of local factors. Then, using Proposition \ref{v-div-q} (at finite places $v\in S_\q$), display \eqref{notin-q} (at finite places $v\in S_\mathfrak{o}$, $v\notin S_\q$), and Proposition \ref{v-arch} (at $v\in S_\infty$), we deduce that for $\gamma=\sigma\nu\in\sigma\mathcal{U}_{\bm{G}_\sigma}(F)\cap \bm{M}(F)$:
\[
J_{\bm{M}}(\gamma,\varepsilon_{K_1(\q)}\otimes f)\ll e^{cR}\norm\q^{n-\theta}q_{S_\mathfrak{o}}^B \max_{L_S\in\mathcal{L}_S(\bm{M}_S)}\prod_{v\in S_\mathfrak{o}\cup S_\infty}D^{L_v}_v(\sigma)^{-C_v} .
\]
Here we have incorporated the $(\prod_{{\rm char}(F_v)\leqslant n!}q_v)^B$ into the implied constant, which is allowed to depend on the number field $F$ and $n$. We may furthermore apply Lemma \ref{prod-form} to absorb $q_{S_\mathfrak{o}}^B$ into the exponential factor $e^{cR}$ (at the cost of a larger value of $c$).

To treat the product of Weyl discriminants, we argue as in the proof of Lemma \ref{prod-form}. Observe that for $M\in\mathcal{L}_v$ we have $D^G_v(\gamma)=D^G_M(\gamma)D^M_v(\gamma)$, where
\begin{equation}\label{DML}
D^G_M(\gamma)=\prod_{\substack{\alpha\in\Phi^G-\Phi^M \\\alpha(\sigma)\neq 1}}|1-\alpha(\sigma)|_v.
\end{equation}
We first note that for $\mathfrak{o}\in\mathfrak{O}_R$, represented by a semisimple element $\sigma\in \bm{M}(F)$, we have $D_{L_v}^G(\sigma)\leqslant 1$ for finite $v$ and $D_{L_v}^G(\sigma)\ll e^{cR}$ for archimedean $v$. We may therefore replace $D_v^{L_v}(\sigma)$ by $D_v(\sigma)=D_v^G(\sigma)$ in the statements of Propositions \ref{v-div-q} and \ref{v-arch}, as well as in display \eqref{notin-q}. Moreover, since $D_v(\sigma)\leqslant 1$ for every $v<\infty$, we may increase the value of $C$ in Proposition \ref{v-div-q} and display \eqref{notin-q} at the cost of a worse bound. Let $C_v$ denote the value of $C$ at each place $v\in S_\mathfrak{o}\cup S_\infty$ and put $C=\max_{v\in S_\mathfrak{o}\cup S_\infty} C_v$. An application of the product formula yields
\[
\prod_{v\in S_\mathfrak{o}\cup S_\infty}D_v(\sigma)^{-C_v}\leqslant \prod_{v\in S_\mathfrak{o}}D_v(\sigma)^{-C}\prod_{v\in S_\infty}D_v(\sigma)^{-C_v}=\prod_{v\in S_\infty}D_v(\sigma)^{C-C_v}.
\]
Since $C-C_v\geqslant 0$ we deduce from \eqref{WDisc} that $\prod_{v\in S_\mathfrak{o}\cup S_\infty}D_v(\sigma)^{-C_v}\ll e^{cR}$, which completes the proof of Theorem \ref{geom-side}.

\subsection{Proof of Proposition \ref{v-div-q}}\label{sec:steps}

In this paragraph we let $v$ denote a finite place. Where possible, we will drop the subscript $v$. So, for example, $G=G_v$, $G_\sigma=\bm{G}_{\sigma,v}$, $\mathcal{U}_{G_\sigma}=\mathcal{U}_{\bm{G}_\sigma}(F_v)$, $\bfk=\bfk_v$, $\p=\p_v$, $\varpi=\varpi_v$, $K_1(\p^r)=K_{1,v}(\p_v^r)$, $q=q_v$, and $D(\sigma)=D_v(\sigma)$.

The basic idea of the proof of Proposition \ref{v-div-q} is to show that the semisimple conjugacy class $\mathfrak{o}$ has small intersection with $K_1(\p^r)$. One has to do this in the framework of the definition of the general weighted orbital integrals \eqref{woi}, which involve various weight functions. 
(No confusion should arise between the parabolic subgroups $R\in\mathcal{F}^{G_{\sigma}}(M_{\sigma})$ in \eqref{woi} and the archimedean parameter $R>0$.) We shall divide the proof into three steps as follows:

\bigskip

\noindent {\it Step 1.} Reduce to the case that $L=G$. We do this by showing that whenever $L$ is a proper Levi subgroup of $G$ we can get savings in the level by means of the constant term alone. 

\bigskip

\noindent {\it Step 2.} Reduce to the case of $M=G$ and $\gamma$ semisimple non-central. This involves bounding the parenthetical expression in \eqref{woi}, as a function of $y\in G_\sigma\backslash G$, $\gamma$, and the level $\p^r$.

\medskip

$\bullet$ If $\gamma$ is not semisimple, then for every $R\in\mathcal{F}^{G_\sigma}(M_\sigma)$ we get savings in the level for the weighted unipotent integrals $J^{M_R}_{M_\sigma}(\nu,\Phi_{R,y})$ of \eqref{J-sigma} by bounding the intersection of unipotent conjugacy classes (in the centralizer of $\sigma$) with congruence subgroups (which depend on $y$).

\medskip

$\bullet$ If $\gamma$ is semisimple non-central, but $M\neq G$, then the same argument as above applies to all terms except the one associated with $R=M_\sigma$, since in that case the unipotent integral collapses and one has simply $J^{M_\sigma}_{M_\sigma}(1,\Phi_{M_\sigma,y})={\bf 1}_{K_1(\p^r)}(y^{-1}\sigma y)v_{M_\sigma}'(y)$.

\bigskip

\noindent {\it Step 3.} Bound the invariant orbital integrals of ${\bf 1}_{K_1(\p^r)}$ for semisimple non-central $\gamma$.

\bigskip

\noindent In all cases, the central ingredient to bounding intersections of conjugacy classes with open compact subgroups is the powerful work of Finis--Lapid \cite{FinisLapid2018}. We shall give a brief overview of their results in \S\ref{sec:FL} below.

It is instructive to examine the division into Steps 2 and 3 in the case where $G_\gamma\subset M$, as the notation greatly simplifies under this assumption. As usual, let $\gamma=\sigma\nu\in\sigma\mathcal{U}_{G_\sigma}\cap M$. Let $\mathcal{V}\subset\mathcal{U}_{G_\sigma}$ be the $G_\sigma$-conjugacy class of $\nu$ in $G_\sigma$, endowed with the natural measure. Then we are to estimate the integral
\[
J_M^G(\gamma,{\bf 1}_{K_1(\p^r)})=\int_{G_\sigma\backslash G}{\rm vol}_{y^{-1}\sigma\mathcal{V}y}(y^{-1}\sigma\mathcal{V}y\cap K_1(\p^r) )v_M'(y)\, \text{d}y.
\]
We proceed differently according to whether $\nu$ is trivial or not.
\medskip

\noindent $\bullet$ If $\nu$ is trivial, then the inner $y^{-1}\sigma\mathcal{V}y$-volume is just ${\bf 1}_{K_1(\p^r)}(y^{-1}\sigma y)$. Thus Step 2 is vacuous is this case, and Step 3 then bounds
\[
\int_{G_\sigma\backslash G}{\bf 1}_{K_1(\p^r)}(y^{-1}\sigma y)v_M'(y)\, \text{d}y
\]
by estimating the intersection volume of the conjugacy class of $\sigma$ with the congruence subgroup $K_1(\p^r)$. 

\medskip

\noindent $\bullet$ If $\nu$ is non-trivial, then $\mathcal{V}$ is of positive dimension and Step 2 bounds the inner $y^{-1}\sigma\mathcal{V}y$-volume by a quantity which is roughly of the form $q^{-r}{\bf 1}_{B(t_\sigma)}(y\sigma y^{-1})$. Here, for a real parameter $t>0$, we have denoted by ${\bf 1}_{B(t)}$ the characteristic function of the ball $B(t)$ of radius $t$ about the origin, and $t_\sigma$ roughly of size $D(\sigma)^{-C}$. One may then estimate the volume of $G_\sigma\backslash B(t)$ (a compact piece of the tube of radius $t$ about $G_\sigma$) by appealing to the work of Shin--Templier \cite{ShinTemplier2016}.

\bigskip

\noindent In either case, the weight function $v_M'$ is easy to control, as it grows by a power of log with the norm of $y$, using results from \cite{Matz2017}.

The more general case, when $G_\gamma$ is not necessarily contained in $M$, is complicated by the presence of the various terms parametrized by $R\in\mathcal{F}^{G_\sigma}(M_\sigma)$ in \eqref{woi}. For example, whenever $R\neq G_\sigma$ these terms include the unipotent weight functions $w_{\mathcal{V},U}^{G_\sigma}$, which must be dealt with. Nevertheless the case $G_\gamma\subset M$ described above already contains most of the difficulties.

\subsubsection{The work of Finis--Lapid}\label{sec:FL}

It will be convenient to use the results of \cite{FinisLapid2018}. We now recall their notation (specialized to our setting) and describe two of their main theorems. As we will sometimes need the global group $\bm{G}$ alongside the local group $G_v$, where $v$ is a finite place, in this subsection we temporarily return to the notational convention of the rest of the paper, and restore all subscripts $v$.

For $r\geqslant 0$ let $K_v(\p_v^r)=\{k\in \bfk_v: k\equiv {\rm Id}\pmod{\p_v^r}\}$ be the principal congruence subgroup of level exponent $r$. Let $\mathfrak{g}_v=M_n(F_v)$ be the Lie algebra of $G$ and write $\Lambda=M_n(\mathcal{O}_v)$. Following \cite[Definition 5.2]{FinisLapid2018} for $\gamma\in G_v$ we put
\[
\lambda_v(\gamma)=\max\{r\in\Z\cup\infty: ({\rm Ad}(\gamma)-1)\Lambda\subset \varpi_v^r\Lambda\}.
\]
In other words, if we make the identification $\GL(\mathfrak{g}_v)=\GL_{n^2}(F_v)$, then $\lambda_v(\gamma)$ is the maximal $r\in\Z\cup\infty$ such that ${\rm Ad}(\gamma)$ lies in the principal congruence subgroup of $\GL_{n^2}(\mathcal{O}_v)$ of level exponent $r$, {\it c.f.} \cite[Remark 5.23]{FinisLapid2018}. The function $\lambda_v$ on $G_v$ descends to one on $\PGL_n(F_v)$, and one has $\lambda_v(\gamma)\geqslant 0$ whenever $\gamma\in \bfk_v$. 

For a twisted Levi subgroup $\bm{H}$ recall that $\bfk_v^{\bm{H}}=\bm{H}_v\cap \bfk_v$. Then $K^{\bm{H}}_v(\p_v^r)=K_v(\p_v^r)\cap \bfk_v^{\bm{H}}$ is the principal congruence subgroup $\bfk_v^{\bm{H}}$ of level exponent $r$. We define the level exponent of an arbitrary open compact subgroup $K$ of $\bfk_v^{\bm{H}}$ as the smallest non-negative integer $f$ such that $K_v^{\bm{H}}(\p_v^f)$ is contained in $K$. For example, the level exponent of $K_{1,v}(\p_v^r)$ is $r$. 

We shall make critical use of the following result, which can be deduced from Propositions 5.10 and 5.11 of \cite{FinisLapid2018}. See \cite{BrumleyMarshall2020} for more details on this deduction. Recall the measures $\mu_{\bm{H},v}$ from \S\ref{sec:measures-on-subgps}.

\begin{prop}[Finis--Lapid \cite{FinisLapid2018}]
\label{FLprop}
For every $\epsilon>0$ small enough there is $\theta>0$ such that the following holds. Let $r$ be a non-negative integer, $v$ a finite place of $F$, $\bm{H}$ a twisted Levi subgroup of $G$, and $x\in \bfk_v^{\bm{H}}$. If $\lambda_v(x)<\epsilon r$ then for any open compact subgroup $K$ of $\bfk_v^{\bm{H}}$ of level exponent $r$ we have
\[
\mu_{\bm{H},v}\{k\in \bfk_v^{\bm{H}}: k^{-1}xk\in K\}\ll q_v^{-\theta r}.
\]
\end{prop}

We shall also need the following result, which can be deduced from \cite[Lemma 5.7]{FinisLapid2018}; see also the proof of \cite[Corollary 5.28]{FinisLapid2018}.

\begin{prop}[Finis--Lapid \cite{FinisLapid2018}]
\label{FLunip}
Let $\bm{H}$ be a twisted Levi subgroup of $\bm{G}$. Let $\bm{P}$ be a proper parabolic subgroup of $\bm{H}$, with unipotent radical $\bm{U}$. Let $v$ be a finite place of $F$. Then
\[
{\rm vol}\{u\in \bm{U}_v\cap \bfk_v^{\bm{H}}: \lambda_v(u)\geqslant m\}\ll q_v^{-m},
\]
uniformly in $v$.
\end{prop}

Finally we remark that in \cite[Lemma 6.10]{BrumleyMarshall2020} it is shown that for semisimple $\sigma\in \bfk_v$ we have
\begin{equation}\label{lambda-Weyl}
q_v^{\lambda_v(\sigma)}\ll D_v(\sigma)^{-1}.
\end{equation}
This inequality will be occasionally used to convert from large values of $\lambda_v(\sigma)$ to large values of $-\log_{q_v}D_v(\sigma)$.

\subsubsection{Reduction to $L=G$}

We return to the purely local setting of Proposition \ref{v-div-q} and drop all subscripts $v$ where possible, as explained in the opening paragraph of \S\ref{sec:steps}. 

We recall the definition of the constant term. Let $M\in\mathcal{L}_S$ and $P\in \mathcal{P}_S(M)$ with Levi decomposition $P=MU$. Let $\phi\in C_c^\infty(G)$. Then the constant term of $\phi$ along $P$ is defined by
\begin{equation}\label{cst-def}
\phi^{(P)}(m)=\delta_P(m)^{1/2}\int_U\int_\bfk\phi(k^{-1}m uk)\, \text{d}k \, \text{d}u\qquad (m\in M).
\end{equation}
Then $\phi^{(P)}\in C_c^\infty(M)$.

We begin by reducing the proof of Proposition \ref{v-div-q} to the case $L=G$. The first step in this reduction is to estimate the constant terms of the functions ${\bf 1}_{K_1(\p^r)}$, uniformly in $r$ and $\gamma$. 

\begin{prop}\label{const-term}
There is $\theta>0$ such that the following holds. Let $M\in\mathcal{L}_v$, $M\neq G$, and $P\in\mathcal{P}_v(M)$. Then for $\gamma\in M$, and $r\geqslant 0$ we have
\[
{\bf 1}_{K_1(\p^r)}^{(P)}(\gamma)\ll q^{-\theta r}{\bf 1}_{\bfk^M}(\gamma).
\]
\end{prop}
\begin{proof}
It follows from \cite[Lemma 7.3 (i)]{Matz2017} or \cite[Lemma 6.2]{ShinTemplier2016} that ${\bf 1}_\bfk^{(P)}={\bf 1}_{\bfk^M}$. Now $0\leqslant f\leqslant g$ implies that $f^{(P)}\leqslant g^{(P)}$. From this we deduce that ${\bf 1}_{K_1(\p^r)}^{(P)}(\gamma)=0$ unless $\gamma\in \bfk^M$. Henceforth we may and will assume that $\gamma\in \bfk^M$; note that $\delta_P=1$ on $\bfk^M$.

Recall the definition of the constant term map in \eqref{cst-def}. Note that if $u\in U$ is such that $k^{-1}\gamma uk\in\bfk$ then $u\in U\cap \bfk$. Fixing $u\in U\cap \bfk$ the inner integral is
\[
\mu_G(k\in \bfk: k^{-1}\gamma uk\in K_1(\p^r)).
\]
From Proposition \ref{FLprop}, for every $\epsilon>0$ small enough there is $\theta>0$ such that
\[
{\bf 1}_{K_1(\p^r)}^{(P)}(\gamma)\ll {\rm vol}\{u\in U\cap \bfk: \lambda(\gamma u)>\epsilon r\}+q^{-\theta r}{\rm vol}\{u\in U\cap \bfk: \lambda(\gamma u)\leqslant \epsilon r\}.
\]
We apply the trivial bound ${\rm vol}(U\cap \bfk)=1$ to the latter volume. To deal with the former, we note that $\lambda(\gamma u)\leqslant \lambda(u)$ for $\gamma\in \bfk^M$ and $u\in U\cap \bfk$ and then apply Proposition \ref{FLunip} to finish the proof.\end{proof}

We now prove Proposition \ref{v-div-q} in the the case that $L\neq G $. Let $\gamma=\sigma\nu\in\sigma\mathcal{U}_{G_\sigma}\cap M$. From Proposition \ref{const-term} we deduce
\[
J_M^L(\gamma, {\bf 1}_{K_{1,v}(\p^r)}^{(Q)})\ll
q^{bB-\theta r} \widetilde{J_M^L}(\gamma, {\bf 1}_{\bfk^L}),
\]
where $b=0$ or $1$ according to whether the residual characteristic of $F_v$ is $> n!$ or not, and $\widetilde{J_M^L}$ denotes the weighted orbital integral $J_M^L$ but with absolute values around the weight functions in \eqref{J-sigma}. If $v\in S_\mathfrak{o}$ we apply \eqref{notin-q} to the latter integral (which is valid with $J_M^G$ replaced by $\widetilde{J_M^G}$). Otherwise, if the finite place $v$ is not in $S_\mathfrak{o}$, then it follows from \cite[Lemma 10.12]{Matz2017} (which, again, is valid with $J_M^G$ replaced by $\widetilde{J_M^G}$) and the identity $J_L^L(\sigma,{\bf 1}_{\bfk^L})=1$ (for semisimple $\sigma$) that
\[
\widetilde{J_M^L}(\gamma,{\bf 1}_{\bfk^L})\ll q^{bB},
\]
with the same convention on $b$ as before. This yields the desired estimate in both cases.

\subsubsection{Bounding the weighted unipotent orbital integrals on descent functions}

We shall now bound the parenthetical expression in \eqref{woi}, with $\phi$ the characteristic function of $K_1(\p^r)$. Before doing so, we shall need to introduce slightly more notation.

If $y=k_1{\rm diag}(\varpi^{m_1},\ldots ,\varpi^{m_n})k_2\in G$, where $k_1,k_2\in \bfk$ and $m_1\geqslant \cdots \geqslant m_n$ are integers, then we write
\[
\|y\|=q^{\max\{|m_1|,|m_n|\}}.
\]
For $t>0$ we write $B(t)=\{g\in G: \|g\|\leqslant t\}$ for the ball of radius $t$ about the origin in $G$. Then ${\bf 1}_{B(t)}$ is the characteristic function of $B(t)$.

\begin{lemma}
\label{hard-lemma}
There are constants $B,C,\theta>0$ such that the following holds. Let $\gamma=\sigma\nu\in \sigma\mathcal{U}_{G_\sigma}\cap M$ be non-central. Let $b=0$ or $1$ according to whether the residual characteristic of $F_v$ is $> n!$ or not, and put $t_\sigma=D(\sigma)^{-C}q^{bB}$. Let $r\geqslant 0$ be an integer. Then there is a set of representatives $y\in G_\sigma\backslash G$ such that the expression
\[
\sum_{R\in\mathcal{F}^{G_\sigma}(M_\sigma)}J^{M_R}_{M_\sigma}(\nu,\Phi_{R,y}),
\]
where the descent functions $\Phi_{R,y}$ are associated with ${\bf 1}_{K_1(\p^r)}$, is majorized by
\[
\left(1+\log t_\sigma\right)^{n-1}{\bf 1}_{K_1(\p^r)}(y^{-1}\sigma y)+q^{bB-\theta r}D(\sigma)^{-C}{\bf 1}_{B(t_\sigma)}(y^{-1}\sigma y).
\]
\end{lemma}

\begin{proof}
Let $(y,u,k)\in G\times \mathcal{U}_{G_\sigma}\times \bfk^{G_\sigma}$ be such that $y^{-1}\sigma k^{-1}uky\in \bfk$. From \cite[Corollary 8.4]{Matz2017}, there are constants $B,C>0$, and for a triplet as above there is $g\in G_\sigma$ (which can be taken to be independent of $u$) such that 
\begin{align}
\|gy\|&\leqslant  D(\sigma)^{-C}q^{bB},\label{restr-supp1}\\
\|gug^{-1}\|&\leqslant  D(\sigma)^{-C}q^{bB},\label{restr-supp2}
\end{align}
where the convention on $b$ is as in the lemma. Henceforth we take a set of representatives $y\in G_\sigma\backslash G$ whose norm is bounded by the right-hand side of \eqref{restr-supp1}, in which case it can be assumed that the norm of $u$ is bounded by the right-hand side of \eqref{restr-supp2}.

We furthermore recall the bound on the weight function $|v'_Q(x)|\ll (1+\log \|x\|)^{n-1}$ established in \cite[Corollary 10.9]{Matz2017}, valid for any parabolic $Q\in\mathcal{F}(M)$. Thus, using \eqref{def:weight}, we deduce that for any $y$ as above and any $k\in \bfk^{G_\sigma}$ we have
\[
|v'_R(ky)|\ll \left(1+\log D(\sigma)^{-C}q^{bB}\right)^{n-1}.
\]
In particular, if $\nu=1$ and $R=M_\sigma$ then we may go ahead and bound the integral $J_{M_\sigma}^{M_\sigma}(1,\Phi_{R,y})$ appearing in \eqref{J-sigma}. Indeed, the unipotent subgroup $U$ of that formula is reduced to the identity in this case, so that the $U$ integral collapses and one has 
\[
J_{M_\sigma}^{M_\sigma}(1,\Phi_{M_\sigma,y})={\bf 1}_{K_1(\p^r)}(y^{-1}\sigma y)v_{M_\sigma}'(y).
\]
Using the above bound on the weight factor, we obtain the first term of the majorization of the lemma.

Next, for $\rho\in\R$ let ${\rm Den}(\rho)$ denote the set of matrices in $M_n(F_v)$ all of whose coefficients have valuation at least $-\rho$. Note that if $g\in G$ is such that $\|g\|\leqslant q^\rho$ then $g\in {\rm Den}(\rho)$; indeed it suffices to establish this for diagonal elements in the positive chamber, where it is immediate. Thus, for $u$ as above, we have $u\in U\cap {\rm Den}(\rho_\sigma)$, where $\rho_\sigma=bB-C\log_{q}D(\sigma)$. 

We return to estimation of the integral $J_{M_\sigma}^{M_R}(\nu,\Phi_{R,y})$, this time in the case where either $\gamma$ is not semisimple or $M\neq G$. In either of these cases, the $U$ appearing in \eqref{J-sigma} satisfies $\dim U\geqslant 1$. From the above discussion we deduce that $J_{M_\sigma}^{M_R}(\nu,\Phi_{R,y})$ is majorized by
\begin{equation}\label{weightedCT}
\left(1+\log D(\sigma)^{-C}q^{bB}\right)^{n-1}\int_{U\cap {\rm Den}(\rho_\sigma)}|w_{M_\sigma,U}^{M_R}(u)|\int_{\bfk^{G_\sigma}}{\bf 1}_{K_1(\p^r)}(y^{-1}\sigma k^{-1}uky)\, \text{d}k\, \text{d}u.
\end{equation}

For $y\in G$ and $r\geqslant 0$ let us put $K^\sigma(y,r)=yK_1(\p^r)y^{-1}\cap G_\sigma$. In the special case when $r=0$ we shall simply write $K^\sigma(y)=K^\sigma(y,0)$. With this notation, the inner integral in \eqref{weightedCT} is $\int_{\bfk^{G_\sigma}}{\bf 1}_{K^\sigma(y,r)}(k^{-1}\sigma vk)\,  \text{d}k$. After an application of the Cauchy--Schwarz inequality, we see that the double integral in \eqref{weightedCT} is bounded by
\begin{equation}\label{separated}
\left(\int_{\bfk^{G_\sigma}}\int_U {\bf 1}_{K^\sigma(y,r)}(k^{-1}\sigma uk)\, \text{d}u\, \text{d}k\right)^{1/2}\left(\int_{U\cap {\rm Den}(\rho_\sigma)}|w_{M_\sigma,U}^{M_R}(u)|^2 \, \text{d}u\right)^{1/2}.
\end{equation}
Using \cite[Lemma 10.5]{Matz2017}, we see that the second factor in \eqref{separated} is $\mathrm{O}(D(\sigma)^{-C}q^{bB})$, with the same convention on $b$. (The aforementioned result in fact bounds the integral $\int_{U\cap {\rm Den}(\rho_\sigma)}|w_{M_\sigma,U}^{M_R}(u)| \, \text{d}u$, but the same proof applies with $|w_{M_\sigma,U}^{M_R}(u)|^2$ as integrand, simply by replacing the polynomial $q$ in \cite[Lemma 10.4]{Matz2017} by its square.)

Next, we treat the first factor in \eqref{separated}. We follow closely the presentation in \cite[Corollary 5.28]{FinisLapid2018}, explicating several small differences. We first write the double integral as
\[
\int_{\bfk^{G_\sigma}}{\rm vol}_U(U\cap \sigma^{-1} K^\sigma(ky,r))\,\text{d}k.
\]
We may suppose that $U\cap \sigma^{-1} K^\sigma(ky,r)$ is non-empty, in which case, fixing any $u_0$ in this intersection, we have $U\cap \sigma^{-1} K^\sigma(ky,r)=u_0(U\cap K^\sigma(ky,r))$. By invariance of the Haar measure on $U$ we obtain in this case
\[
{\rm vol}_U(U\cap \sigma^{-1} K^\sigma(ky,r))={\rm vol}_U(U\cap K^\sigma(ky,r)).
\]

We claim that we can reduce to the case where $K^\sigma(y,r)$ is replaced by $K^\sigma(y,r)\cap \bfk^{G_\sigma}=yK_1(\p^r)y^{-1}\cap \bfk^{G_\sigma}$. Indeed, this double integral is
\begin{equation}\label{set-up}
\int_{\bfk^{G_\sigma}}{\rm vol}_U(U\cap K^\sigma(ky,r))\,\text{d}k=\int_{\bfk^{G_\sigma}}i(k)\, {\rm vol}_U(U\cap K^\sigma(ky,r)\cap\bfk^{G_\sigma})\,\text{d}k,
\end{equation}
where
\begin{align*}
i(k)&=\big[U\cap K^\sigma(ky,r):U\cap K^\sigma(ky,r)\cap\bfk^{G_\sigma}\big]\leqslant [K^\sigma(ky,r):K^\sigma(ky,r)\cap\bfk^{G_\sigma}]\\
&=[K^\sigma(y,r):K^\sigma(y,r)\cap\bfk^{G_\sigma}]\leqslant [K^\sigma(y):K^\sigma(y)\cap\bfk^{G_\sigma}].
\end{align*}
Therefore the expression in \eqref{set-up} bounded by
\begin{equation}\label{applyFL}
[K^\sigma(y):K^\sigma(y)\cap\bfk^{G_\sigma}]\int_U\int_{\bfk^{G_\sigma}}{\bf 1}_{yK_1(\p^r)y^{-1}\cap \bfk^{G_\sigma}}(k^{-1}uk)\,  \text{d}k\, \text{d}u.
\end{equation}
Continuing, it now follows from \eqref{restr-supp1} that
\[
[K^\sigma(y):K^\sigma(y)\cap\bfk^{G_\sigma}]\ll  D(\sigma)^{-C}q^{bB}.
\]
It remains to bound the double integral in \eqref{applyFL}.

We first note that the level exponent of $yK_1(\p^r)y^{-1}\cap \bfk^{G_\sigma}$ is at least $r$. That is to say, $yK_1(\p^r)y^{-1}\cap \bfk^{G_\sigma}$ cannot contain $K^{G_\sigma}(\p^{r-1})$. This follows, for example, from the fact that the central element $1+\varpi^{r-1}$ lies in $K^{G_\sigma}(\p^{r-1})$ but not in $y^{-1}K_1(\p^r)y\cap \bfk^{G_\sigma}$. In light of this, we may apply Proposition \ref{FLprop}, with $H=G_\sigma$, to find $\theta,\epsilon>0$ such that
\[
\int_{\bfk^{G_\sigma}}{\bf 1}_{yK_1(\p^r)y^{-1}\cap \bfk^{G_\sigma}}(k^{-1} uk)\,  \text{d}k\ll q^{-\theta r}
\]
whenever $\lambda(u)<\epsilon r$. The double integral in \eqref{applyFL} is therefore bounded by
\[
{\rm vol}\{u\in U\cap \bfk^{G_\sigma}:\lambda( u)> \epsilon r\}+q^{-\theta r}{\rm vol}\{u\in U\cap \bfk^{G_\sigma}:\lambda( u)\leqslant \epsilon r\},
\]
We bound the second volume factor trivially by ${\rm vol}\{u\in U\cap \bfk^{G_\sigma}\}=1$. Finally, an application of Proposition \ref{FLunip} (with $H=G_\sigma$) shows that the first volume factor is majorized by $q^{-\epsilon r}$, finishing the proof.
\end{proof}

\subsubsection{Invariant orbital integrals}

In view of Lemma \ref{hard-lemma}, it now remains to establish good bounds on the invariant orbital integrals of ${\bf 1}_{K_1(\p^r)}$  and ${\bf 1}_{B(t_\sigma)}$. It suffices to estimate the unnormalized orbital integral
\[
O_\sigma(\phi)=\int_{G_\sigma\backslash G}\phi(x^{-1}\gamma x){\rm d}\mu_{\sigma}(x),
\]
since $J_G^G(\sigma,\phi)=D(\sigma)^{1/2}O_\sigma(\phi)$.

We first handle the invariant orbital integral for ${\bf 1}_{B(t_\sigma)}$. If $\sigma\notin S_\mathfrak{o}$ and the residue characteristic of $F_v$ is $>n!$ then $t_\sigma\ll 1$, and we may apply \cite[Theorem A.1]{ShinTemplier2016} to deduce that $O_\sigma({\bf 1}_{B(t_\sigma)})\ll 1$ in this case. If either the residue characteristic of $F_v$ is $\leqslant n!$ or $\gamma\in S_\mathfrak{o}$ then we proceed as follows. For every $\lambda\in X_*^+(T_0)$ let $\tau(\lambda)$ denote the associated Hecke operator, namely the characteristic function of $\bfk\lambda(\varpi)\bfk$. Then
\[
O_\sigma({\bf 1}_{B(t_\sigma)})=\sum_{\substack{\lambda\in X_*^+(T_0)\\ \|\lambda\|\leqslant \log_{q} t_\sigma}} O_\sigma(\tau_\lambda).
\]
Note that there are only ${\rm O}(\log_{q}^n t_\sigma)$ cocharacters $\lambda$ satisfying the bound in the sum. For each of the above orbital integrals, it follows from \cite[Theorem 7.3]{ShinTemplier2016} (see also \cite[Theorem B.1]{ShinTemplier2016} for a stronger result) that there are constants $B,C\geqslant 0$ such that $O_\sigma (\tau_\lambda)\ll t_\sigma^B D(\sigma)^{-C}$. Inserting this into the above expression (and recalling the definition of $t_\sigma$ from Lemma \ref{hard-lemma}) we deduce the bound $O_\sigma({\bf 1}_{B(t_\sigma)})\ll q^B D(\sigma)^{-C}$ in this case. We conclude that in all cases we have
\[
O_\sigma({\bf 1}_{B(t_\sigma)})\ll q^{bB}D(\sigma)^{-C},
\]
where $b=0$ or $1$ according to whether the residue characteristic of $F_v$ is $>n!$ or not. 

It remains then to estimate the invariant orbital integral $O_\sigma ({\bf 1}_{K_1(\p^r)})$ uniformly in the level $\p^r$ and the semisimple element $\sigma$. We accomplish this in the next lemma; our presentation follows closely that of \cite[Proposition 6.2]{BrumleyMarshall2020}.

\begin{lemma}\label{level-aspect}
There are constants $B,C,\theta>0$ such that the following holds. Let $v$ be a finite place. For any $r\geqslant 0$ and semisimple $\sigma\in G$, $\sigma\notin Z$, we have
\[
O_\sigma({\bf 1}_{K_1(\p^r)})\ll q^{aB- \theta r}D(\sigma)^{-C},
\]
where $a=1$ or $0$ according to whether $v\in S_\mathfrak{o}$ or not.
\end{lemma}

\begin{proof}
Letting $C_{\sigma,G}$ denote the conjugacy class of $\sigma$, we have
\[
O_\sigma({\bf 1}_{K_1(\p^r)})=\mu_{\sigma}(C_{\sigma,G}\cap K_1(\p^r)),
\]
where $\mu_{\sigma}=\mu_G/\mu_{G_x}$. Now $C_{\sigma,G}$ is closed since $\sigma$ is semisimple. The compact set $C_{\sigma,G}\cap K_1(\p^r)$ is then a disjoint union of finitely many (open) $\bfk$-conjugacy classes $C_{x_i,\bfk}$ meeting $K_1(\p^r)$. This gives
\[
O_\sigma({\bf 1}_{K_1(\p^r)})=\sum_{i=1}^t \frac{\mu_{\sigma}(C_{x_i,\bfk}\cap K_1(\p^r))}{\mu_{\sigma}(C_{x_i,\bfk})}\mu_{\sigma}(C_{x_i,\bfk}).
\]
From the definition of the quotient measure, for any $x\in \bfk$, we have
\[
\mu_{\sigma}(C_{x,\bfk}\cap K_1(\p^r))=\frac{\mu_{G}(k\in \bfk: k^{-1}xk\in K_1(\p^r))}{\mu_{G_{x}}(G_{x}\cap \bfk)}.
\]

Using \eqref{lambda-Weyl}, we may deduce from Proposition \ref{FLprop} that if $x\in \bfk$ is semisimple and non-central, and $D(x)\gg q^{-\epsilon r}$ for some $\epsilon>0$, then
\[
\mu_{G}(k\in \bfk: k^{-1}xk\in K_1(\p^r))\ll q^{-(1-\epsilon) r}.
\]
Thus for every $\epsilon>0$ there is $\theta>0$ such that if $D(\sigma)=D(x_i)\gg q^{-\epsilon r}$ then
\[
\frac{\mu_{\sigma}(C_{x,\bfk}\cap K_1(\p^r))}{\mu_{\sigma}(C_{x,\bfk})}=\mu_{G}(k\in \bfk: k^{-1}xk\in K_1(\p^r))\ll q^{-\theta r}.
\]
In this case we obtain
\[
O_\sigma({\bf 1}_{K})\ll q^{-\theta r}\sum_{i=1}^t\mu_{\sigma}(C_{x_i,\bfk})=q^{-\theta r}O_\sigma({\bf 1}_{\bfk}),
\]
since
\[
\sum_{i=1}^t\mu_{\sigma}(C_{x_i,\bfk})=\sum_{i=1}^t\mu_{\sigma}(C_{x_i,G}\cap \bfk)=\mu_{\sigma}(C_{\sigma}\cap \bfk)=O_\sigma({\bf 1}_{\bfk}).
\]
If, on the other hand, $D(\sigma)\ll q^{-\epsilon r}$ (so that $1\ll q^{-\theta r}D(\sigma)^{-\theta/\epsilon} $) then we may apply the trivial bound $O_\sigma({\bf 1}_{K_1(\p^r)})\leqslant O_\sigma({\bf 1}_{\bfk})$ to obtain
\[
O_\sigma({\bf 1}_{K_1(\p^r)})\ll q^{-\theta r}D(\sigma)^{-\theta/\epsilon}O_\sigma({\bf 1}_{\bfk}).
\]
If $v\notin S_\mathfrak{o}$ then $O_\sigma({\bf 1}_{\bfk})=1$. If $v\in S_\mathfrak{o}$ we apply the bound $O_\sigma({\bf 1}_{\bfk})\ll q^B D(\sigma)^{-C}$ of \cite[Theorems 7.3 and B.1]{ShinTemplier2016}. This proves the desired estimate in either case.
\end{proof}

\subsection{Proof of Proposition \ref{v-arch}}

The statement of Proposition \ref{v-arch}, without the explicit dependency in $R$, follows from the proof of \cite[Theorem 1.12]{MatzTemplier2021} (in the case of $v$ real) and \cite{Matz2017} (in the case of $v$ complex). To prove Proposition \ref{v-arch} it therefore suffices to make explicit the dependence in $R$ in these works. For simplicity, we shall concentrate on the real case here. Once again, we drop all $v$ subscripts from the notation, so in particular $G=G_v$.

It suffices to take $L=G$, since on one hand the constant term map $f\mapsto f^{(Q)}$ takes $C_c^\infty(G^1)_R$ to $C_c^\infty(L^1)_{cR}$ (see, for example, \cite[Lemma 7.1 (iii)]{Matz2017}), and on the other the factor $\delta_Q^{1/2}$ is bounded by $\mathrm{O}(e^{c'R})$ on $G_{\leqslant R}^1\cap L$. Here $c, c'>0$ are constants depending only on $n$.

We would like to make explicit the dependency in $C(f_1)$ in \cite[Theorem 1.12]{MatzTemplier2021} on both $\|f_1\|_\infty$ and the support of $f_1$. (Note that, since we do not seek any savings in the spectral parameter, our interest is in $\eta=0$.) It is clearly enough to bound the modified weighted orbital integral $\widetilde{J^G_M}(\gamma,f)$, where the weight functions are replaced by their absolute values. The dependency on $\|f_1\|_\infty$ is easy enough to make explicit, for in the proof of \cite[Theorem 1.12]{MatzTemplier2021} (see Proposition 7.4, Corollaries 8.3--8.4, and Proposition 8.5 of \cite{MatzTemplier2021}), one replaces the function $f_1$ with a majorizer of the characteristic function of its support. We can thus assume that $f$ is the characteristic function of $G_{\leqslant R}^1$. 

We now supplement a few of the lemmas and propositions leading up to the proof of \cite[Theorem 1.12]{MatzTemplier2021}, pointing out how the dependency in $R$ can be made explicit.
\begin{itemize}

\item[$\bullet$] The constant $c$ in \cite[Lemma 7.3]{MatzTemplier2021} can be taken to be of the form $\mathrm{O}(R)$, with an implied constant depending only on $n$. To see this, first note that the constants in Lemmata 4.6 and 4.7 depend only on $n$. It can then readily be seen that each of the constants $a_i$ in the proof of Lemma 7.3 can be taken to be of the form $e^{\kappa_i R}$, for $\kappa_i=\kappa_i(n)$. (For example, $a_1=ce^{c_1R}$, where $c=c(n)>0$ and $c_1=c_1(n)>0$ are given in Lemma 4.6.)

\medskip

\item[$\bullet$] Inspecting then the proof of \cite[Proposition 7.4]{MatzTemplier2021}, one then applies Lemma 7.5 with $c=\mathrm{O}(R)$ and $c_1=\mathrm{O}(1)$, and Lemma 7.6 with $s=c\Delta^-(\gamma_s)^{\mathrm{O}(1)}$ and the same $c=\mathrm{O}(R)$. These dependencies are admissible, as the integral in the statement of \cite[Lemma 7.5]{MatzTemplier2021} is bounded exponentially in $r(\gamma_s)=c+c_1\log (\Delta^-(\gamma_s))$, as is remarked in the proof, and the statement of Lemma 7.6 is polynomial in $s$.

\medskip

\item[$\bullet$] We now consider the proof of \cite[Proposition 7.7]{MatzTemplier2021}. One is led to consider integrals of the form $\int_V{\bf 1}_{G_{\leqslant R}^1}(v)|\log |p(v)||^k\, {\rm d}v$, where $V$ is the unipotent radical of a proper parabolic of $G$, $k\geqslant 1$ is an integer, and $p:V\rightarrow\R$ is a polynomial function on the coordinates. (Once again, recall that $\eta=0$.) This can be bounded by $aR^b\int_V{\bf 1}_{G_{\leqslant R}^1}(v)\, {\rm d}v$, where $a>0$ and $b\geqslant 0$ depend on $p$, $k$, and $n$. The latter integral is $\mathrm{O}(e^{cR})$.
\end{itemize}

\section{Construction of test functions}\label{PW}

In this section, we construct (in certain cases) explicit realizations of test functions $f\in \mathcal{H}(G^1)_{cR,\underline{\delta}}$ having prescribed spectral transform $h\in\mathcal{PW}_{R,\underline{\delta}}$. (We are again dropping the subscript $\infty$ in the notation; thus $G=G_\infty$, $G^1=G_\infty^1$, $\bfk=\bfk_\infty$, $\mathcal{L}=\mathcal{L}_\infty$, etc.) We also provide bounds on these which will be important in the estimation of associated orbital integrals.

Let $G'=\prod_{v\mid \R}\SL_n^\pm(\R)\prod_{v\mid\C}\SL_n(\C)$, where $\SL_n^\pm(\R)=\{g\in\GL_n(\R): \det (g)=\pm 1\}$, and set $Z^1=Z\cap G^1$. The construction of $G'$ guarantees that we may write every $g\in G^1$ as $g=z_0g_1$, where $g_1\in G'$ and $z_0\in Z^1$. Let $Z^1_n=Z^1\cap G'$, and write $\Delta_n^1=\{(u,u^{-1}): u\in Z^1_n\}\subset Z^1\times G'$. There is a short exact sequence $1\to\Delta_n^1\to Z^1\times G'\to G^1\to 1$, and so if $f_0\in\mathcal{H}(Z^1)$ and $f_1\in\mathcal{H}(G')$ are such that $f_0\times f_1$ on $Z^1\times G'$ is invariant under $\Delta_n^1$ then one can construct $f\in\mathcal{H}(G^1)$ by setting $f(g)=f_0(z_0)f_1(g_1)$. In particular, in either of the two situations \eqref{hyp-n2}--\eqref{hyp-sph} in Proposition~\ref{test-fn-bd} below, one immediately verifies that every element of $Z_n^1$ acts by scalar multiplication on the representation space of $\delta\in\mathscr{E}^2(M^1)$, giving rise to a character we denote (via a small abuse) by $\delta|_{Z_n^1}$; then, the previous invariance condition is satisfied if both $f_0$ and $f_1$ transform by the complex conjugate $\bar{\delta}|_{Z_n^1}$ under the regular action by $Z_n^1$.

Note that $Z^1=(Z\cap\mathbf{K}) \exp(\mathfrak{h}_G)$. For $R>0$ we denote by $\mathcal{H}(Z^1)_R$ the space of functions in $\mathcal{H}(Z^1)$ whose support is contained in $(Z\cap\mathbf{K})\exp (B(0,R))$. Analogously, writing $\mathbf{K}'=\mathbf{K}\cap G'=\prod_{v\mid\infty}\mathbf{K}'_v$ with $\mathbf{K}'_v$ equal to $\mathrm{O}_n(\mathbb{R})$ for $v\mid\mathbb{R}$ and $\mathrm{SU}(n)$ for $v\mid\mathbb{C}$, we denote by $\mathcal{H}(G')_R$ the space of functions in $\mathcal{H}(G')$ with support contained in $\mathbf{K}'\exp(B(0,R))\mathbf{K}'$, and by $\mathcal{H}(G'_v)_R$ the space of functions in $\mathcal{H}(G'_v)$ with support contained in $\mathbf{K}'_v\exp(B(0,R))\mathbf{K}'_v$.

For $M\in\mathcal{L}$ recall the Lie algebra decompositions $\mathfrak{h}_M=\mathfrak{h}_G\oplus\h_M^G$ and $\h_M^G=\bigoplus_{v\mid\infty}\aa_{M_v}^{G_v}$ from \S\ref{sec:hM-decomp}. Recall furthermore the notation $\mathcal{H}(G^1)_{\underline{\delta}}$ from \S\ref{disc-aut} and $\mathcal{PW}_{\underline{\delta}}$ from \S\ref{PWCD}.

\begin{prop}\label{test-fn-bd}
Let $n\geqslant 1$. Let $\underline{\delta}\in\mathcal{D}$ have standard representative $(M,\delta)$. Suppose that either
\begin{enumerate}
\item\label{hyp-n2} $n\leqslant 2$, or
\item\label{hyp-sph} $n\geqslant 1$ is arbitrary, but $M=T_{0,\infty}$ and $\delta$ is the trivial character of $T_{0,\infty}^1$. 
\end{enumerate}

Let $h\in\mathcal{PW}_{\underline{\delta}}$ factorize as
\[ h=h_0h_1,\quad h_1=\prod_{v\mid\infty}h_v,\]
for some $h_0\in\mathcal{PW}(\mathfrak{h}_{G,\C}^*)_{R_0}$ and $h_1\in\mathcal{PW}((\h_M^G)_\C^*)$ with $h_v\in\mathcal{PW}((\aa_{M_v}^{G_v})_{\C}^{\ast})_{R_v}$. Then there exist $f_0\in \mathcal{H}(Z^1)_{R_0}$ and $f_1=\prod_{v\mid\infty}f_v\in\mathcal{H}(G')$ with $f_v\in\mathcal{H}(G'_v)_{R_v}$, both transforming by $\overline{\delta}|_{Z_n^1}$ under $Z_n^1$, and such that the product $f=f_0f_1$ is the unique function in $\mathcal{H}(G^1)_{\underline{\delta}}$ for which ${\rm tr}\,\pi_{\delta,\nu}(f)=h(\nu)$. Moreover,
\begin{equation}\label{prop-bound}
\|f\|_\infty\ll \|h\|_1,
\end{equation}
the latter being taken with respect to the measure $\deg(\delta)\mu_M^G(\delta,\nu)\,\mathrm{d}\nu$ supported on $i\mathfrak{h}_M^*$.
\end{prop}

To facilitate the exposition that follows, we spell out the first stated property on $f$. Namely, for any $\underline{\sigma}\in\mathcal{D}$, with standard representative $(L,\sigma)$, and $\nu\in\mathfrak{h}_{L,\C}^*$, one has
\begin{equation}\label{take-trace}
{\rm tr}\,\pi_{\sigma,\nu}(f)=\begin{cases}
h(w^{-1}.\nu), & \sigma=w.\delta,\; w\in W(A_M);\\
0, & \textrm{otherwise}.
\end{cases}
\end{equation}
Both sides of \eqref{take-trace} depend only on the $W(A_L)$-orbit $[\sigma,\nu]$ of the pair $(\sigma,\nu)$.

\subsection{The operator version of Proposition~\ref{test-fn-bd}}\label{rem:operator-version}

In this subsection, we state an operator version of Proposition~\ref{test-fn-bd} which yields the scalar version upon taking traces.

The key to picking out a particular $\underline{\delta}$ is a sort of M\"obius inversion process on $\bfk'$-types, which we now explain. For $v\mid\infty$ and $\pi_v\in\Pi(G'_v)$, recall that the restriction of $\pi_v$ to $\mathbf{K}'_v$ decomposes as the orthogonal direct sum over $\Pi(\mathbf{K}'_v)$, with each appearing with multiplicity at most one and exactly one of them being the lowest $\mathbf{K}'_v$-type $\tau(\pi_v)$ (relative to the norm $\|\tau_v\|$ as in \S\ref{rep-notation}). Given any $\tau_v\in\Pi(\mathbf{K}'_v)$ we say that $\tau_v^{-}\in\Pi(\bfk_v')$ is an ``immediate predecessor'' of $\tau_v$ if, for every $\pi_v$ in which $\tau_v$ appears without being its lowest $\mathbf{K}'_v$-type, then $\tau_v^-$ also appears in $\pi_v$. It is easily verified that, when $n\leqslant 2$, only the characters of $\bfk_v'$ admit no immediate predecessors, and that immediate $\tau_v^-$ is uniquely determined for the others. For $\tau=\prod_{v\mid\infty}\tau_v,\tau'=\prod_{v\mid\infty}\tau'_v\in\Pi(\mathbf{K}')$, write
\[\mu(\tau:\tau')=\begin{cases} (-1)^{\#\{v\mid\infty:\tau'_v=\tau_v^{-}\}},& \textrm{if } \tau'_v\in\{\tau_v,\tau_v^{-}\} \textrm{ for every } v\mid\infty,\\
 0&\textrm{ otherwise.}
 \end{cases}
 \]
For general $n\geqslant 1$ and $\tau=\tau_0$ the trivial $\mathbf{K}'$-type, we naturally extend this definition simply by $\mu(\tau_0:\tau')=\mathbf{1}_{\tau'=\tau_0}$.

Now, let $(M,\delta)$ and $h$ be as in the statement of Proposition \ref{test-fn-bd}. Let $\tau(\pi_\delta)\in\Pi(\bfk')$ be defined as in \S\ref{rep-notation}. To prove Proposition \ref{test-fn-bd} we shall construct an $f\in \mathcal{H}(G^1)$, verifying the stated factorization and transformation properties, as well as a function $H$ defined on pairs $(\sigma,\nu)\in\mathscr{E}^2(T_{0,\infty}^1)\times\h_{0,\C}^{\ast}$, such that
\begin{enumerate}
\item\label{HM-support} $H(\sigma,\nu)$ is supported on $\sigma\prec\delta$ and verifies
\begin{equation}\label{prop-trace}
\pi(f)=H(\sigma,\nu)\sum_{\tau\in\Pi(\bfk')}\mu(\tau(\pi_{\delta}):\tau)\frac{1}{\dim\tau}\Pi_{\tau}\qquad (\pi\in\{\pi_{\sigma,\nu},\pi(\sigma,\nu)\}),
\end{equation}
where $\Pi_{\tau}$ denotes the orthogonal projection onto the $\tau$-isotypic component of $\pi$;
\item\label{HM-value} $H(\sigma,\nu)=h(w^{-1}.\nu)$ when $\sigma=w.\delta$;
\item\label{HM-components} when $n=2$, we have $H(\sigma,\nu)=h_0(\nu)\prod_{v\mid\infty}H_{1,v}(\sigma_v,\nu_v)$, where $H_{1,v}$ is given by the expressions in \eqref{defn-hpnu} and \eqref{h-eta-s} at complex and real places, respectively.
\end{enumerate}

The identity \eqref{prop-trace}, in the case $\pi=\pi_{\sigma,\nu}$, along with property \eqref{HM-value}, together imply \eqref{take-trace} upon taking the trace. We will use \eqref{prop-trace} when $\pi=\pi(\sigma,\nu)$ in Section~\ref{EisensteinSection} in the course of proving Theorem \ref{spec-est}.

\begin{remark}\label{CD-support}
The existence of $f\in\mathcal{H}(G^1)$ satisfying \eqref{take-trace} is a consequence of the Paley--Wiener theorem of Clozel--Delorme \cite{ClozelDelorme1984}, which is of course valid without assuming either \eqref{hyp-n2} or \eqref{hyp-sph}. In general, it does not make sense to ask for the bound \eqref{prop-bound} to hold, since, given $h$, the function $f$ is unique only up to addition by functions whose orbital integrals vanish identically. 
\end{remark}

\subsection{Reduction to semisimple component}

We now prepare the proof of Proposition \ref{test-fn-bd} by reducing it to the corresponding statement for the function $h_1$. For $\pi$ a finite length admissible representation of $G^1$ with central character $\omega_\pi$, the restriction $\pi'=\pi|_{G'}$ is again of finite length as a representation of $G'$ and we have
\[
\pi (f)= \widehat{f_0}(\overline{\omega_\pi})\pi' (f_1).
\]
For $\delta\in\mathscr{E}^2(M^1)$ we let $\delta_0$ denote the restriction of $\delta$ to $Z\cap\mathbf{K}\subset M^1$. If $\pi\in {\rm Rep}(G^1)_\delta$ then $\omega_\pi|_{Z\cap\mathbf{K}}=\delta_0$. We write $\delta'$ for the restriction of $\delta$ to $M^1\cap G'$. Let $\mathcal{D}'$ denote the corresponding set of conjugacy classes $\underline{\delta'}$ on $G'$. If we set $M'=M\cap G'$, then we may identify $\h_M^G=\h_{M'}^{G'}=\h_{M'}$; indeed $\h_{G'}=0$ since $G'$ has finite center.

From the product description in \S\ref{sec-arch-loc-int}, $\mu_M^G(\delta,\nu)$ is plainly $i\h_G^{\ast}$-translation invariant in $\nu$, and so
the measure $\deg(\delta)\mu_M^G(\delta,\nu)\rm{d}\nu$ on $i\mathfrak{h}_M^*=i\mathfrak{h}_G^*\oplus i(\h_M^G)^*$ factorizes as the Lebesgue measure on $i\mathfrak{h}_G^*$ times $\deg(\delta')\mu_M^G(\delta',\nu')\rm{d}\nu'$ on $i(\h_M^G)^*$. Accordingly, we have $\|h\|_1=[W(A_M):W(A_M)_{\delta}]\cdot\|h_0\|_1\|h_1\|_1$, the Weyl group index coming from the $\sigma\in W(A_M).\delta$.

To prove Proposition \ref{test-fn-bd} it therefore suffices to construct test functions $f_0\in \mathcal{H}(Z^1)_{R_0}$ and $f_1=\prod_{v\mid\infty}f_v\in \mathcal{H}(G')$ with $f_v\in\mathcal{H}(G'_v)_{R_v}$, both transforming by $\overline{\delta}|_{Z_n^1}$ under $Z_n^1$, such that
\begin{enumerate}
\item  for $\chi_0e^{\nu}\in\widehat{Z^1}$, where $\chi_0$ is a character of $Z\cap\mathbf{K}$ and $\nu\in \mathfrak{h}_{G,\C}^*$, we have
\[
\widehat{f_0}(\chi_0e^{\nu})=\mathbf{1}_{\overline{\chi}_0=\delta_0}h_0(\nu)
\]
and $\|f_0\|_\infty\leqslant \|h_0\|_1$;
\item for any $\underline{\sigma}\in\mathcal{D}$, with standard representative $(L',\sigma)$, and $\nu\in \h_{L',\C}^*$, we have
\begin{equation}\label{take-trace'}
{\rm tr}\,\pi_{\sigma,\nu}(f_1)=\begin{cases}h_1(w^{-1}.\nu),&\sigma=w.\delta',\,w\in W(A_{M'});\\ 0,&\text{otherwise,}\end{cases}
\end{equation}
and
\begin{equation}\label{prop-bound'}
\|f_1\|_\infty\ll \|h_1\|_1.
\end{equation}
\end{enumerate}

The first point is simple: let $g_0\in C^\infty_c(\mathfrak{h}_G)_R$ denote the (inverse) Fourier transform of $h_0$. For $u\in Z\cap\mathbf{K}$ and $H\in \mathfrak{h}_G$, we set $f_0(ue^H)=\overline{\delta_0}(u)g_0(H)$. The inequality $\|f_0\|_\infty\leqslant \|h_0\|_1$ follows from Fourier inversion.

It remains to carry out the second point, as announced. The main idea is that, for $\tau\in\Pi(\mathbf{K}')$ and $f$ in the $\tau$-isotypic Hecke algebra $\mathcal{H}(G',\tau)$, the transform $f\mapsto\tr\pi(f)$ is described by the spherical transform of type $\tau$, with an inversion formula when $\mathcal{H}(G',\tau)$ is commutative. We review this abstract theory in \S\ref{sec:sph} and then in \S\ref{sec:subq} its reformulation in terms of (non-unitary) principal series representations $\pi(\eta,\xi)$ and the corresponding $\tau$-spherical transform $H(\eta,\xi)=\mathscr{H}^{\tau}(\eta,\xi)$. For the actual construction of the test function in \S\S\ref{sec:construction-sph}--\ref{sec:construction-n2}, we first rewrite the target function on the right-hand side of \eqref{take-trace'} as a linear combination of finitely many eligible $H(\eta,\xi)$, for which we rely on the Paley--Wiener theorem for the $\tau$-spherical transform. We then explicitly describe the types $\underline{\delta'}$ appearing in the corresponding inversion formula \eqref{sph-inv2} and show that this inversion indeed provides a test function satisfying \eqref{take-trace'}. Finally, we prove \eqref{prop-bound'} by estimating all the terms in \eqref{sph-inv2}. We review the theory of spherical transforms over the entire group $G'$ and only pass to factors $G'_v$ at the end, but the reader will notice that the entire setup for $G'$ in fact factorizes over all archimedean places.

\subsection{Spherical transform of type $\tau$: abstract theory}\label{sec:sph} 

We begin by recalling the spherical functions (and trace spherical functions) of a given $\bfk'$-type $\tau$ on a semisimple Lie group with finite center. These will then be used to define the associated spherical transform on the $\tau$-isotypic Hecke algebra. For these definitions, see, for example, \cite[\S 6.1]{Warner1972} and \cite{Camporesi1997}.

Fix $\tau\in\Pi(\bfk')$. For $\pi\in\Pi (G')$, acting on the space $V_\pi$, let $\Pi_\tau$ be the canonical projection onto the $\tau$-isotypic subspace $V_\pi^\tau$. Then the {\it spherical function of type $\tau$ for $\pi$} is defined by
\begin{equation}\label{sph-fn-pi}
\Phi_\pi^\tau(g)=\Pi_\tau \circ\pi(g)\circ\Pi_\tau.
\end{equation}
Note that $\Phi_\pi^\tau(g)$ is an endomorphism of the finite dimensional space $V_\pi^\tau$, which is zero if $\tau$ is not a $\bfk'$-type of $\pi$. Similarly, we may define the {\it spherical trace function of type $\tau$ for $\pi$} by
\begin{equation}\label{sph-tr-fn-pi}
\varphi_\pi^{\,\tau}(g)={\rm tr}\,\Phi_\pi^{\,\tau}(g).
\end{equation}
Note that $\Phi_\pi^{\,\tau}(e)=\Pi_\tau$ and $\varphi_\pi^{\,\tau}(e)=\dim V_\pi^\tau$. Furthermore, since $\pi$ is unitary, we have
\begin{equation}\label{sph1}
|\varphi_\pi^\tau (g)|\leqslant \dim V_\pi^\tau \|\Phi_\pi^\tau(g)\|\leqslant \dim V_\pi^\tau.
\end{equation}
Moreover, both $\Phi_\pi^{\,\tau}$ and $\varphi_\pi^{\,\tau}$ transform under $Z'$ by the central character of $\pi$.

For $\tau\in\Pi(\bfk')$ let  $\xi_\tau$ denote the character of $\tau$ and write $\chi_\tau=(\dim\tau)\,  \xi_\tau$. We then let $\mathcal{H}(G',\tau)$ denote the space of functions $f\in C^\infty_c(G')$ such that
\begin{enumerate}
\item $f(kgk^{-1})=f(g)$ for all $g\in G'$ and $k\in \bfk'$,
\item $\overline{\chi_\tau}*f=f=f*\overline{\chi_\tau}$.
\end{enumerate}
Then for any $f\in\mathcal{H}(G',\tau)$ and any $\pi\in\Pi(G')$ we have $\Pi_\tau\circ \pi(f)\circ\Pi_\tau=\pi(f)$; see, for example, \cite[Prop.~3.2]{Camporesi1997}. In particular $\pi(f)=0$ on $\mathcal{H}(G',\tau)$ unless $\tau$ is a $\bfk'$-type of $\pi$. Note as well that any $f\in \mathcal{H}(G',\tau)$ transforms under $Z'$ by the conjugate of the central character of $\tau$.

We define the spherical transform of type $\tau$ of a function $f\in \mathcal{H}(G',\tau)$ by
\[
\mathscr{H}^{\, \tau}(f)(\pi)=\int_{G'}f(g)\, \varphi_\pi^\tau(g)\, {\rm d}g.
\]
It follows from the definitions (see \cite[(14)]{Camporesi1997}) that, for $f\in \mathcal{H}(G',\tau)$ we have
\begin{equation}\label{sph-proj}
\pi(f)=\mathscr{H}^{\, \tau}(f)(\pi)\cdot \frac{1}{\dim V_\pi^\tau} \Pi_\tau
\end{equation}
and hence
\begin{equation}\label{traceH}
{\rm tr}\, \pi(f)=\mathscr{H}^{\, \tau}(f)(\pi).
\end{equation}

The convolution algebra $\mathcal{H}(G',\tau)$ is commutative if and only if $\tau$ appears with multiplicity at most $1$ in every irreducible admissible representation $\pi$ of $G'$ \cite[Proposition 6.1.1.6]{Warner1972}. This is the case, for example, for arbitrary $\bfk'$-types of archimedean $\GL_2$, and for the trivial $\bfk'$-type for archimedean $\GL_n$; these are the two cases described in the hypotheses \eqref{hyp-n2} and \eqref{hyp-sph} of Proposition \ref{test-fn-bd}. Whenever $\mathcal{H}(G',\tau)$ is commutative, we may invert the spherical transform of type $\tau$. Indeed, it is shown in \cite[p.43]{Camporesi1997} that in this case one has the inversion formula
\[
f(g)=\frac{1}{\dim\tau}\int_{\Pi(G')}\mathscr{H}^{\,\tau}(f)(\pi)\, \varphi_\pi^\tau(g^{-1})\, {\rm d}\hat{\omega}^{\rm pl}(\pi)
\]
for all $f\in\mathcal{H}(G',\tau)$. In such situations we see, using \eqref{sph1} and the equality $\dim V_\pi^\tau=\dim\tau$ valid in this case, that
\begin{equation}\label{inftyL1}
\|f\|_\infty\leqslant \|\mathscr{H}^{\,\tau}(f)\|_{L^1(\hat{\omega}^{\rm pl})}.
\end{equation}

\subsection{Reformulation using the subquotient theorem}\label{sec:subq}

Using the Harish-Chandra subquotient theorem, we may complement the abstract theory of the previous section to give an explicit integral representation of $\tau$-spherical functions, and explicit formulae for the associated $\tau$-spherical transforms and (in the commutative case) their inversions.

We begin by extending the definitions \eqref{sph-fn-pi} and \eqref{sph-tr-fn-pi} to the principal series representations $\pi(\eta,\xi)$, where $\eta\in\mathscr{E}^2(T_{0,\infty}')$ and $\xi\in (\h_0^{G'})_\C^*$, which are not necessarily unitary nor irreducible, but admit a central character $\eta|_{Z'}$. For $\tau\in\Pi(\bfk')$, we write
\[
\Phi^{\, \tau}_{\eta,\xi}(g)=\Pi_\tau \circ \pi(\eta,\xi)(g)\circ \Pi_\tau,
\]
where $\Pi_\tau$ is the projection of $I(\eta,\xi)$ onto its $\tau$-isotypic component $I(\eta,\xi)^\tau$. Here, $I(\eta,\xi)$ is the space on which $\pi(\eta,\xi)$ acts. Similarly, we put
\[
\varphi_{\eta,\xi}^{\, \tau}(g)={\rm tr}\,\Phi^{\, \tau}_{\eta,\xi}(g).
\]
Both $\Phi^{\, \tau}_{\eta,\xi}$ and $\varphi_{\eta,\xi}^{\, \tau}$ transform under $Z'$ by $\eta|_{Z'}$. We have the integral representation of Harish-Chandra (see \cite[Corollary 6.2.2.3]{Warner1972})
\begin{equation}
\label{integral-representation}
\varphi_{\eta,\xi}^\tau(g)=\int_{\bfk'}(\chi_\tau * \eta)(\kappa(k^{-1}gk))e^{\langle \xi-\rho,H(gk)\rangle}\,\text{d}k,
\end{equation}
where $\text{d}k$ is the probability Haar measure on $\bfk'$, $\rho$ is the half-sum of positive roots, and $g=\kappa(g)\exp(H(g))n(g)$ in the Iwasawa decomposition $G'=\mathbf{K}'A_0'N'$. Moreover, if $\tau\in\Pi(\bfk')$ appears as a $\bfk'$-type of $\pi(\eta,\xi)$, we associate with a function $f\in\mathcal{H}(G',\tau)$ the transform
\begin{equation}
\label{transform-fullyinduced}
\mathscr{H}^{\, \tau}(f)(\eta,\xi)=\int_{G'}f(g)\varphi_{\eta,\xi}^{\, \tau}(g)\,{\rm d}g.
\end{equation}

The relation between $\varphi_\pi^\tau$ (resp., $\mathscr{H}^{\tau}(f)(\pi)$), defined on unitary representations $\pi\in\Pi(G')$, and $\varphi_{\eta,\xi}^\tau$ (resp., $\mathscr{H}^\tau(f)(\eta,\xi)$) is given by the Harish-Chandra subquotient theorem \cite[Theorem 5.5.1.5]{Warner1972a}. This theorem states (in particular) that for any $\pi\in\Pi(G')$ there is $\eta=\eta(\pi)\in\mathscr{E}^2(T_{0,\infty}')$ and $\xi=\xi(\pi)\in (\h_0^{G'})_\C^*$ such that $\pi$ is infinitesimally equivalent to a subquotient of the principal series representation $\pi(\eta,\xi)$. Thus for $\pi\in\Pi(G')$ appearing as a subquotient of $\pi(\eta,\nu)$, and for any $\bfk'$-type $\tau$ of $\pi$, we have
\begin{equation}\label{sph-fn2}
\varphi_\pi^{\,\tau}=\varphi_{\eta,\xi}^{\,\tau}\qquad\text{and}\qquad \mathscr{H}^{\, \tau}(f)(\pi)=\mathscr{H}^{\, \tau}(f)(\eta,\xi).
\end{equation}
Note the importance of the assumption that $\tau\in\Pi(\bfk')$ appearing as a $\bfk'$-type of $\pi$ for this formula to hold: if $\tau$ is not a $\bfk'$-type of $\pi$, then both $\varphi_\pi^{\,\tau}$ and $\mathscr{H}^{\, \tau}(f)(\pi)$ are zero, whereas this is not necessarily the case for $\varphi_{\eta,\xi}^{\,\tau}$ and $\mathscr{H}^{\, \tau}(f)(\eta,\xi)$.

\textit{Now assume $\mathcal{H}(G',\tau)$ commutative.} Following \cite[(46)]{Camporesi1997}, we write the inverse spherical transform of type $\tau$ more explicitly. Let $\mathcal{D}_\tau'$ denote the subset of $\underline{\delta'}\in\mathcal{D}'$ having standard representative $(M',\delta')$ for which $\tau$ appears as a $\bfk'$-type in $\pi(\delta',0)$. For any $f\in\mathcal{H}(G',\tau)$ we have
\begin{equation}\label{sph-inv2}
f(g)=\frac{1}{\dim\tau}\sum_{\underline{\delta'}\in\mathcal{D}_\tau'} c_{M'}\deg(\delta')\int_{i\h_{M'}^*}\mathscr{H}^{\,\tau}(f)(\eta,\xi)\varphi_{\eta,\xi}^{\, \tau}(g^{-1})\mu_{M'}^{G'}(\delta',\nu)\, {\rm d}\nu,
\end{equation}
where the $c_{M'}>0$ are constants, and $\eta=\eta(\pi_{\delta',\mu}), \xi=\xi(\pi_{\delta',\nu})$. We remark that for all $\underline{\delta'}\in\mathcal{D}_\tau'$ we have $\delta'|_{Z'}=\tau|_{Z'}=\eta|_{Z'}$, and thus each $\varphi_{\eta,\xi}^\tau(g^{-1})$ appearing on the right-hand side of \eqref{sph-inv2} transforms under $Z'$ by $\overline{\delta'}|_{Z'}$.

\subsection{Proof of Proposition \ref{test-fn-bd} in the spherical case} \label{sec:construction-sph}

For the trivial $\bfk'$-type $\tau_0\in\Pi(\bfk')$, and the trivial character $\eta_0\in\mathscr{E}^2(T_{0,\infty}')$ we write $\varphi_\nu=\varphi_{\eta_0,\nu}^{\tau_0}$ for the associated trace spherical function. Moreover, for $f_1\in\mathcal{H}(G',\tau_0)$, we write $h_1(\nu)=\mathscr{H}^{\tau_0}(f_1)(\eta_0,\nu)$ for the associated spherical transform. Then the inversion formula \eqref{sph-inv2} becomes
\begin{equation}\label{classical-inversion}
f_1(g)=\int_{i(\h_0^{G'})^*}h_1(\nu)\varphi_{-\nu}(g) |c(\nu)c(\rho)^{-1}|^{-2}\, {\rm d}\nu,
\end{equation}
where $c(\nu)$ is the Harish-Chandra $c$-function whose definition in our setting was recalled in \eqref{eq:H-C-c-function}. In this case, $(f_1,h_1)$ forms the familiar spherical transform pair \cite[IV,\S7 Equations (3),(14)]{Helgason2000}, which we recall is a bijection from $\mathcal{H}(G',\tau_0)$ onto the space of $W(A_0')$-invariant functions in $\mathcal{PW}((\h_0^{G'})_\C^*)$ \cite[IV, \S7 Theorem 7.1]{Helgason2000}.

We now turn to the proof of Proposition \ref{test-fn-bd} in the spherical case \eqref{hyp-sph}. From the discussion directly preceding \S\ref{sec:sph}, it only remains to consider
$h_1\in \mathcal{PW}((\h_0^{G'})_\C^*)$. We use \eqref{classical-inversion} to define $f_1$, that is, we let $f_1\in\mathcal{H}(G',\tau_0)$ be the inverse $\tau_0$-spherical transform of $h_1$, as shown in \eqref{classical-inversion}. Moreover, the factorization $h_1=\prod_{v\mid\infty}h_v$ gives rise to $f_1=\prod_{v\mid\infty}f_v$ in the obvious way, with $(f_v,h_v)$ forming a spherical transform pair over $G'_v$. The support conditions that $f_v\in\mathcal{H}(G'_v,\tau_0|_{G'_v})_{R_v}$ are contained in \cite[IV, \S7 Theorems 7.1 and 7.3]{Helgason2000}. Since $f_1\in\mathcal{H}(G',\tau_0)$, the operator $\pi(f_1)$ acts on $V_\pi$ as zero whenever $\pi$ does not contain the trivial $\bfk'$-type. Moreover, writing $\pi_\nu=\pi_{\eta_0,\nu}$, we deduce from \eqref{sph-proj} that $\pi_\nu(f_1)=h_1(\nu) \Pi_{\tau_0}$. These two observations together imply \eqref{prop-trace}. The bound \eqref{prop-bound'} follows simply from \eqref{inftyL1}.

\subsection{Proof of Proposition \ref{test-fn-bd} for $n\leqslant 2$}\label{sec:construction-n2}

We now remove the spherical hypothesis of the preceding paragraph, but restrict to $n\leqslant 2$. Let $\underline{\delta}'$ be represented by $\delta'=\prod_{v\mid\infty}\delta'_v$ for some $\delta'_v\in\mathscr{E}^2(M'_v)$, with $M'_v\in\{T'_{0,v},G'_v\}$ for every $v\mid\infty$. In the following three subsections, we will construct functions $f_v\in\mathcal{H}(G'_v)_{R_v}$ that verify the local versions of \eqref{prop-trace} (and along with it \eqref{take-trace'}) as well as \eqref{prop-bound'}; the corresponding statements for $f=\prod_{v\mid\infty}f_v\in\mathcal{H}(G')$ follow immediately.

In our construction, we rely on the explicit $\tau$-spherical transform and its inversion in \S\S\ref{sec:sph}--\ref{sec:subq}, as well as their local versions for $\tau_v$-spherical transform on $G'_v$. Here, we denote by $\tau=\prod_{v\mid\infty}\tau_v$ the minimal $\mathbf{K}'$-type of $\pi_{\delta'}$, so that each $\tau_v$ is the minimal $\mathbf{K}'_v$-type of $\pi_{\delta'_v}$, and $\mathcal{D}_\tau=\prod_{v\mid\infty}\mathcal{D}_{\tau_v}$.

\textit{From now on, we refresh the notation and drop the prime decorations and ``local'' notations. In particular, we shall write $f,h,\delta,\tau,G,R$ for $f_v,h_v,\delta'_v,\tau_v,G'_v,R_v$, and $\h_0$ for $\h_0^{G'}$}.

\subsubsection{The case of $M_v=T_0$ and $v$ real}\label{hT-subsec}

We begin by parametrizing the irreducible dual of ${\rm O}(2)$ as follows. For $k\geqslant 1$ we put $\tau_k={\rm Ind}_{\SO(2)}^{{\rm O}(2)}(e^{ik\theta})$, a two dimensional representation. Then $\Pi({\rm O}(2))=\{\tau_0, \det\}\cup\{\tau_k\}_{k\geqslant 1}$. We let $\eta_0,\eta_-$, and $\eta_1$ denote the characters $(1,1)$, $(\mathrm{sgn},\mathrm{sgn})$, and $(1,\mathrm{sgn})$ on $T_0$, respectively. Thus in this case $\tau$ is either $\tau_0$ (when $\delta=\eta_0$), $\det$ (when $\delta=\eta_-$), or $\tau_1$ (when $\delta=\eta_1$).

We also note that no elements of $\mathscr{E}^2(\SL_2^{\pm}(\R))$ appear in $\mathcal{D}_{\tau}$. Indeed, $\mathscr{E}^2(\SL_2^\pm(\R))= \{D_k: \, k\geqslant 2\}$ where $D_k$ is the weight $k$ discrete series representation for $\SL_2^\pm (\R)$, and the lowest ${\rm O}(2)$-type of $D_k$ is $\tau_k$ (see \eqref{KtypeDk} below) with $k\geqslant 2$ (thus type $\tau$ does not appear in $D_k$). Hence, in fact, $\mathcal{D}_{\tau}=\{\delta\}$, and this $\delta\in\mathscr{E}^2(T_0)$ is the only term contributing non-trivially to the summation in the inversion formula \eqref{sph-inv2}.

According to Ehrenpreis--Mautner\footnote{In fact, Ehrenpreis--Mautner treat the group $\SL_2(\R)$ rather than $\SL_2^\pm(\R)$, but their methods extend to this case.} \cite[Theorem 2.1]{EhrenpreisMautner1957}, the $\tau$-spherical (or Fourier) transform is a bijection from $\mathcal{H}(G,\tau)$ onto the space of $W(A_0)_{\delta}$-invariant functions in $\mathcal{PW}(\h_{0,\C}^*)$. We define $f$ as the inverse $\tau$-spherical transform \eqref{sph-inv2} of $h$. Moreover, noting that for any $\tilde{\delta}\in\mathscr{E}^2(T_0)$ we have $\eta=\eta(\pi_{\tilde{\delta},\mu})=\tilde{\delta}$ and $\xi=\xi(\pi_{\tilde{\delta},\nu})=\nu$, the formula \eqref{sph-inv2} reads as
\[
f(g)=\frac{c}{\dim\tau}\int_{i\h_0^*}h(\nu)\varphi_{\delta,\nu}^{\, \tau}(g^{-1})\mu_0^G(\delta,\nu)\, {\rm d}\nu.
\]

Recalling that $h\in \mathcal{PW}(\h_{0,\C}^*)_R$, the condition that $f\in\mathcal{H}(G)_R$ now follows from \cite[Theorem 2.1, Proposition 2.1]{EhrenpreisMautner1957}.
The verification of \eqref{prop-trace} (and hence \eqref{take-trace'}) follows from \eqref{sph-proj} and \eqref{sph-fn2}. As in the spherical case, the bound \eqref{prop-bound'} comes directly from the above explicit realization of \eqref{sph-inv2} and \eqref{inftyL1}.

\subsubsection{The case of $M_v=T_0$ and $v$ complex}\label{sec:hv-complex}

For $\nu\in\C$ and $p\in\frac12\mathbb{Z}$, let the character $\chi_{\nu,p}\in T_0^{\ast}$ be given by $\chi_{\nu,\ell}(\mathrm{diag}(z,z^{-1}))=|z|^{2i\nu}(z/|z|)^{2p}$; we also write $\delta_p=\chi_{0,p}\in\mathscr{E}^2(T_0)$. For every $\ell\in\frac12\mathbb{Z}_{\geqslant 0}$, $\tau_{\ell}=\mathop{\mathrm{sym}}^{2\ell}(\C^2)$ is an irreducible $(2\ell+1)$-dimensional representation of $\mathrm{SU}(2)$, and in fact $\Pi(\mathrm{SU}(2))=\{\tau_{\ell}\}_{\ell\in\frac12\mathbb{Z}_{\geqslant 0}}$. For $\nu\in i\mathbb{R}$ and $p\in\frac12\mathbb{Z}$, the corresponding irreducible representation $\pi(\delta_p,\nu)=\pi_{\delta_p,\nu}=\pi_{p,\nu}$ satisfies $\eta(\pi_{\delta_p,\nu})=\delta_p$ and $\xi(\pi_{\delta_p,\nu})=\nu$ and decomposes according to the $\mathrm{SU}(2)$-action as
\[ \pi_{p,\nu}\simeq\bigoplus_{\substack{\ell\in\frac12\mathbb{Z},\,\,\ell\geqslant |p|,\\ \ell\equiv p\bmod 1}}\tau_{\ell}; \]
in particular, $\tau_{\ell}$ with $\ell=|p|$ is the lowest $\mathrm{SU}(2)$-type appearing in $\pi_{p,\nu}$, and each $\tau_{\ell}$ with $\ell\geqslant |p|$, $\ell\equiv p\pmod 1$ appears with multiplicity one. Put another way, $\mathcal{D}_{\tau_{\ell}}=[[\ell]]:=\{p\in\frac12\mathbb{Z}: p\leqslant|\ell|,\,\, p\equiv\ell\pmod 1\}$.

The formulas \eqref{transform-fullyinduced} and \eqref{sph-inv2} for the $\tau_{\ell}$-spherical transform are then explicitly realized as the transform from $f\in\mathcal{H}(G,\tau_{\ell})$ to the system of functions $(H(p,\nu))_{|p|\in[[\ell]]}$, namely $H(p,\nu)=\mathscr{H}^{\tau_{\ell}}(f)(\delta_p,\nu)$ via
\begin{equation}
\label{tauell-expl}
H(p,\nu)=\int_Gf(g)\varphi_{\nu,\delta_p}^{\tau_{\ell}}(g)\,\mathrm{d} g,\quad f(g)=\frac{c}{2\ell+1}\sum_{p\in[[\ell]]}\int_{i\h_0^{\ast}}H(p,\nu)\varphi_{\nu,\delta_p}^{\tau_{\ell}}(g^{-1})
(|\nu|^2+p^2)\,\mathrm{d}\nu,
\end{equation}
with an absolute $c>0$. Note that $\mathscr{H}^{\tau_{\ell}}(f)(\delta_p,\nu)$ vanishes for $p\not\in[[\ell]]$. This explicit description will be particularly handy when verifying properties \eqref{prop-trace} and \eqref{prop-bound'}.

The following Paley--Wiener theorem of Wang~\cite[Proposition 4.5, Lemma 4.4]{Wang1974} fully characterizes the image of $\mathcal{H}(G,\tau_{\ell})$ under the $\tau_{\ell}$-spherical transform $\mathscr{H}^{\tau_{\ell}}$: for $R>0$, a system of functions $(H(p,\nu))_{p\in[[\ell]]}$ lies in the image of $\mathcal{H}(G,\tau_{\ell})_R$ if and only if it satisfies the following conditions:
\begin{enumerate}
\item\label{req0} for every $p\in[[\ell]]$, $H(p,\cdot)\in\mathcal{PW}(\h_{0,\C}^{\ast})_R$;
\item\label{req2} $H(p,\nu)=H(-p,-\nu)$ for every $|p|\leqslant\ell$, $\nu\in\h_{0,\C}^{\ast}$;
\item\label{req3} $H(p,\nu)=H(\nu,p)$ if $\nu\in\frac12\mathbb{Z}$ and $|p|,|\nu|\leqslant\ell$.
\end{enumerate}
We denote the set of such systems by $\mathcal{PW}_{[[\ell]],R}$. It will be convenient to denote a system $(H(p,\nu))_{p\in[[\ell]]}$ in $\mathcal{PW}_{[[\ell]],R}$ by $H_{[[\ell]]}$, and, if $\ell\geqslant 1$, by $H_{[[\ell-1]]}=(H(p,\nu))_{p\in[[\ell]-1]}$ its restriction to $|p|\leqslant\ell-1$.

The (perhaps initially counterintuitive) ``additional symmetry'' \eqref{req3} (a case of the so-called Arthur--Campoli relations) corresponds to a discrete set of (for $(\nu,p)\neq(0,0)$) non-unitary representations and can be understood both in terms of the infinitesimal characters~\cite[Corollary 2 and its proof]{Wang1974} or in terms of irreducible subquotients of the fully induced representations $\pi(\nu,p)$ \cite[(2.7)]{BlomerHarcosMagaMilicevic2022}. It makes analytically interesting (and nontrivial) the problem of extending from given target functions $H(\pm\ell,\nu)=h_{\pm\ell}(\nu)\in\mathcal{PW}(\h_{0,\C}^{\ast})_R$ to the full system $(H(p,\nu)_{p\in[[\ell]]}\in\mathcal{PW}_{[[\ell]],R}$ without increasing the Paley--Wiener support parameter $R$ or (essentially) the Plancherel $L^1$ norm, which features, for example, in our announced bound \eqref{prop-bound'}. We record our solution to this extension problem as a formal lemma for its independent interest.

\begin{lemma}
Let $\ell\in\frac12\mathbb{Z}_{\geqslant 0}$, $R>0$, and a pair of functions $h_{\pm\ell}\in\mathcal{PW}(\h_{0,\C}^{\ast})_R$ be given, satisfying $h_{-\ell}(\nu)=h_{\ell}(-\nu)$ for all $\nu\in\h_{0,\C}^{\ast}$. Then, the system $H_{[[\ell]]}$, given explicitly by \eqref{defn-hpnu} below for $\ell>0$ (and simply by $H(0,\nu):=h_0(\nu)$ if $\ell=0$) satisfies
\[ H_{[[\ell]]}\in\mathcal{PW}_{[[\ell]],R},\quad H(\pm\ell,\nu)=h_{\pm\ell}(\nu)\quad (\nu\in\h_{0,\C}^{\ast}), \]
as well as $H_{[[\ell-1]]}\in\mathcal{PW}_{[[\ell-1]],R}$ for $\ell\geqslant 1$.
Moreover, denoting
\[ f=(\mathscr{H}^{\tau_{\ell}})^{-1}(H_{[[\ell]]})-\mathbf{1}_{\ell\geqslant 1}(\mathscr{H}^{\tau_{\ell-1}})^{-1}(H_{[[\ell-1]]}), \]
we have that $f\in C_c^{\infty}(G_{\leqslant R})$,
\begin{equation}
\label{req4}
\tr\pi_{\delta,\nu}(f)=\mathbf{1}_{\delta=\delta_{\pm\ell}}h_{\pm\ell}(\nu)\quad\text{and}\quad \|f\|_{\infty}\ll\|h_{\ell}\|_1,
\end{equation}
with the latter $1$-norm being taken with respect to the standard Plancherel measure $(|\nu|^2+\ell^2)\,\mathrm{d}\nu$, and the implied constant is independent of $\ell$, $R$, and $h_{\pm\ell}$.
\end{lemma}

\begin{proof}
For $\ell\in\{0,\frac12\}$, there is virtually nothing to do, so we assume from now on that $\ell\geqslant 1$. We may identify $i\h_0^{\ast}\simeq i\mathbb{R}$ (and thus $\h_{0,\C}^{\ast}\simeq\mathbb{C}$) so that, as in \S\ref{PWCD}, the fact that $h_{\ell}\in\mathcal{PW}(\h_{0,\C}^{\ast})_R$ means equivalently that the standard inverse Fourier transform $\widecheck{h_{\ell}}\in C_c^{\infty}([-R,R])$, the latter normalized so that $h_{\ell}(\nu)=\int_{\mathbb{R}}\widecheck{h_{\ell}}(x)e^{-\nu x}\,\mathrm{d}x$. We define functions $(H(p,\nu))_{p\in[[\ell]]}$ in $\mathcal{PW}(\h_{0,\C}^{\ast})_R$ by
\begin{equation}
\label{defn-hpnu}
\begin{aligned}
H(p,\nu):=&\int_{\mathbb{R}}e^{(-\nu-p)x}\frac{-e^{-\ell x}\widecheck{h_{\ell}}(x)+e^{\ell x}\widecheck{h_{\ell}}(-x)}{e^{2\ell x}-e^{-2\ell x}}\,\mathrm{d}x\\
&\qquad +\int_{\mathbb{R}}e^{(-\nu+p)x}\frac{-e^{-\ell x}\widecheck{h_{\ell}}(-x)+e^{\ell x}\widecheck{h_{\ell}}(x)}{e^{2\ell x}-e^{-2\ell x}}\,\mathrm{d}x.
\end{aligned}
\end{equation}
Noting that both fractions appearing above are even functions of $x$, one immediately verifies that the system $H_{[[\ell]]}=(H(p,\nu))_{p\in[[\ell]]}$ indeed belongs to $\mathcal{PW}_{[[\ell]],R}$ (which of course implies that $H_{[[\ell-1]]}\in\mathcal{PW}_{[[\ell-1]],R}$) and that $H(\pm\ell,\nu)=h_{\pm\ell}(\nu)$.

By the $\tau_{\ell}$- and $\tau_{\ell-1}$-Paley--Wiener theorem, the systems $H_{[[\ell]]}$ and $H_{[[\ell-1]]}$ have the pre-images $f_{+}=(\mathscr{H}^{\tau_{\ell}})^{-1}(H_{[[\ell]]})\in\mathcal{H}(\tau_{\ell})_R$ and $f_{-}=(\mathscr{H}^{\tau_{\ell-1}})^{-1}(H_{[[\ell-1]]})\in\mathcal{H}(\tau_{\ell-1})_R$, which, by \eqref{traceH} and \eqref{sph-fn2}, satisfy
\[ \tr\pi_{\delta_p,\nu}(f_{+})=\mathbf{1}_{|p|\leqslant\ell}H(p,\nu),\quad\tr\pi_{\delta_p,\nu}(f_{-})=\mathbf{1}_{|p|\leqslant\ell-1}H(p,\nu).\]
Hence $f=f_{+}-f_{-}\in C_c^{\infty}(G_{\leqslant R})$ satisfies the first claim of part \eqref{req4}.
Finally, for the second part of \eqref{req4}, denote for an integer $|u|<2\ell$ the even function
\[ H_{u,2\ell}(\nu):=\int_{\mathbb{R}}e^{-\nu x}\frac{e^{ux}-e^{-ux}}{e^{2\ell x}-e^{-2\ell x}}\,\mathrm{d}x= 
\frac{\pi\sin(\pi u/2\ell)}{2\ell [\cos(\pi\nu/2\ell)+\cos(\pi u/2\ell)]}, \]
where the evaluation comes from \cite[Table 17.23, Entry 20, renormalized]{GradshteynRyzhik2007}. We see that, for $t\in\mathbb{R}$, $0<H_{u,2\ell}(it)\ll 1/(2\ell-|u|)$, $H_{u,2\ell}(it)\ll(1/\ell)e^{-\pi|t|/2\ell}$ for $|t|\gg\ell$, and $\int_{\mathbb{R}}H_{u,2\ell}(it)\,\mathrm{d}t=u/(4\pi\ell)$. We may now write, for $p\in[[\ell-1]]$,
\begin{equation}\label{defn:Hpnu}
H(p,\nu)=\int_{\mathbb{R}}e^{-\nu x}\big(\widecheck{H_{\ell+p,2\ell}}(x)\widecheck{h_{\ell}}(x)+\widecheck{H_{\ell-p,2\ell}}(-x)\widecheck{h_{\ell}}(-x)\big)\,\mathrm{d}x\\
=H_{\ell+p,2\ell}\ast h_{\ell}+H_{\ell-p,2\ell}\ast h_{-\ell},
\end{equation}
where $\ast$ denotes the usual convolution, and so by spherical inversion \eqref{tauell-expl}
\begin{align*}
f(g)&=(\mathscr{H}^{\tau_{\ell}})^{-1}(H_{[[\ell]]})(g)-(\mathscr{H}^{\tau_{\ell-1}})^{-1}(H_{[[\ell-1]]})(g)\\
&=\sum_{\pm}\frac{c}{2\ell+1}\int_{\mathbb{R}}h_{\pm\ell}(it)\varphi_{\delta_{\pm\ell},it}^{\tau_{\ell}}(g^{-1})(t^2+\ell^2)\,\mathrm{d}t\\
&\qquad +\sum_{p\in[[\ell-1]]}\sum_{\pm}\int_{\mathbb{R}}(H_{\ell\pm p,2\ell}\ast h_{\pm\ell})(it)\Big(\frac{c}{2\ell+1}\varphi_{\delta_p,it}^{\tau_{\ell}}-\frac{c}{2\ell-1}\varphi_{\delta_p,it}^{\tau_{\ell-1}}\Big)(g^{-1})(t^2+p^2)\,\mathrm{d}t.
\end{align*}
The announced bound on $|f(g)|$ in \eqref{req4} follows using the estimates $\varphi_{\delta_{\pm\ell},it}^{\tau_{\ell}}\ll\ell$ and $\varphi_{\delta_p,it}^{\tau_{\ell}}-\varphi_{\delta_p,it}^{\tau_{\ell-1}}\ll 1$, which are immediate from \eqref{sph1} and \eqref{integral-representation} (noting that $\chi_{\tau_\ell}-\chi_{\tau_{\ell-1}}\ll 1$), and the resulting estimate
\[ \Big[\Big(\frac{c}{2\ell+1}\varphi_{\delta_p,i\cdot}^{\tau_{\ell}}-\frac{c}{2\ell-1}\varphi_{\delta_p,i\cdot}^{\tau_{\ell-1}}\Big)(g^{-1})(\cdot^2+p^2)\ast H_{\ell\pm p,2\ell}\Big](it)\ll t^2+\ell^2. \qedhere \]
\end{proof}

The verification of \eqref{prop-trace} (and hence \eqref{take-trace'}) then follows from \eqref{sph-proj} and \eqref{sph-fn2}. As in the spherical case, the bound
\eqref{prop-bound'} comes directly from \eqref{inftyL1}.

\subsubsection{The case of $M_v=G$.}\label{sssec:GL2-const}

Before coming to the proof of Proposition \ref{test-fn-bd} in this case, we shall restate the formula \eqref{sph-inv2} explicitly. Recall that $D_k$ ($k\geqslant 2$) appears as the unique irreducible subrepresentation of $I(\eta_\epsilon;(k-1)/2)$, where $\epsilon\equiv k$ mod $2$. Indeed, there is an exact sequence
\begin{equation}\label{es}
1\longrightarrow D_k \longrightarrow I(\eta_\epsilon,(k-1)/2)\longrightarrow {\rm Sym}^{k-2}\longrightarrow 1.
\end{equation}
The ${\rm O}(2)$-type decomposition of $D_k$ can therefore be deduced from that of $I(\eta_\epsilon,(k-1)/2)$ and ${\rm Sym}^{k-2}$. Namely, for any $k\geqslant 2$ we have
\begin{equation}\label{KtypeDk}
{\rm Res}_{{\rm O}(2)}(D_k)=\displaystyle\bigoplus_{\substack{n\geqslant k\\ n\equiv k\, \text{mod}\, 2}} \tau_n,
\end{equation}
while $\mathrm{Sym}^{k-2}$ consists of lower ${\rm O}(2)$-types.

Using \eqref{es} and \eqref{KtypeDk} we find that for $k\geqslant 2$ and $j\geqslant k$ with $j\equiv k\equiv\epsilon \bmod 2$, we have
\[
\varphi_{D_k}^{\tau_j}=\varphi_{\eta_{\epsilon,\frac{k-1}{2}}}^{\tau_j}.
\]
For $k\geqslant 2$ and $f\in\mathcal{H}(\SL_2^\pm(\R),\tau_k)$ we put $H(\eta,\nu)=\mathscr{H}^{\tau_k}(f)(\eta,\nu)$. Then, when $k\geqslant 2$, we have
\[ {\rm tr}\, D_j(f)=
\begin{cases}
H(\eta_{\epsilon},(j-1)/2),& 2\leqslant j\leqslant k,\; j\equiv k\equiv\epsilon\bmod 2;\\
0,& \text{else}.
\end{cases}
\]
Then \eqref{sph-inv2} states that there are constants $a,b>0$ such that the following holds: if $k\geqslant 2$ and $f\in\mathcal{H}(\SL_2^\pm(\R),\tau_k)$ then $f(g)$ is given by
\begin{equation}
\label{inversion-even}
a\!\!\!\sum_{\eta=\eta_0,\eta_-}\int_\R\varphi_{\eta,it}^{\tau_k}(g^{-1})H (\eta,it)t\tanh(\pi t/2) \, {\rm d}t+ b\!\!\sum_{\substack{2\leqslant j\leqslant k\\ j\; \text{even}}}\varphi_{\eta_0,\frac{j-1}{2}}^{\tau_k}(g^{-1})H(\eta_0,(j-1)/2)(j-1)
\end{equation}
for $k$ even, and
\begin{equation}
\label{inversion-odd}
a\int_\R\varphi_{\eta_1,it}^{\tau_k}(g^{-1})H(\eta_1,it)t\coth(\pi t/2)\, {\rm d}t+ b\sum_{\substack{3\leqslant j\leqslant k\\ j\; \text{odd}}}\varphi_{\eta_1,\frac{j-1}{2}}^{\tau_k}(g^{-1})H(\eta_1,(j-1)/2)(j-1)
\end{equation}
for $k$ odd. These inversion formulae hold also for $k\in\{0,1\}$, in which case the latter sums are empty. 

Let $\epsilon\in\{0,1\}$, with $\epsilon\equiv k\bmod 2$, record the parity of $k$. We introduce the space $\mathcal{PW}_\epsilon(\h_{0,\C}^*)_R$ consisting of functions $H(\eta,\nu)$ on $E_\epsilon\times\h_{0,\C}^*$, where $E_0=\{\eta_0,\eta_-\}$, $E_1=\{\eta_1\}$, and $\nu\mapsto H(\eta,\nu)$ is a $W(A_0)_\eta$-invariant function in $\mathcal{PW}(\h_{0,\C}^*)_R$. With this notation, the $\tau_k$-spherical transform \eqref{transform-fullyinduced} is a bijection from $\mathcal{H}(\mathrm{SL}_2^{\pm}(\mathbb{R}),\tau_k)_R$ onto $\mathcal{PW}_\epsilon(\h_{0,\C}^*)_R$, as in~\cite[Theorem 2.1, Proposition 2.1]{EhrenpreisMautner1957}. Formulas \eqref{inversion-even} and \eqref{inversion-odd} then explicitly invert this transform.

With formula \eqref{sph-inv2} explicitly restated, we now prove Proposition \ref{test-fn-bd} for the delta mass on some $D_k$, for some $k\geqslant 2$, that is, we construct an $f\in\mathcal{H}(\mathrm{SL}_2^{\pm}(\mathbb{R}))_R$ satisfying \eqref{prop-trace} (and hence \eqref{take-trace'}) and \eqref{prop-bound'}. (In fact we shall produce an $f\in\mathcal{H}(\mathrm{SL}_2^{\pm}(\mathbb{R}))_1$.)

Let $r\in C_c^\infty(\R)$ be supported in $[-1,-\frac12]$ and satisfy $\int_\R r(x)\,\textrm{d}x=1$. Let $\widehat{r}\in\mathcal{PW}(\C)$ denote the Fourier transform of $r$. Then $\widehat{r}(0)=1$ and, for $y>0$, the Paley--Wiener estimate $\widehat{r}(x+iy)\ll_N e^{-y/2} (1+y)^N (1+|x|)^{-N}$ holds. Let $\epsilon\in\{0,1\}$ be of the same parity as $k$. Then we put
\begin{equation}
\label{h-eta-s}
H(\eta,s)=\begin{cases}
\sum_\pm \widehat{r}((\pm s-(k-1)/2)/i), & \eta=\eta_\epsilon;\\
0, & \mathrm{else}.
\end{cases}
\end{equation}
Then $H(\eta,\cdot)\in\mathcal{PW}(\C)_1$ and $c_k:=H(\eta_\epsilon,(k-1)/2)=1+\mathrm{O}(e^{-k/2})$. Recalling that $\tau_k$ is the lowest ${\rm O}(2)$-type of $D_k$, we let $f_+\in\mathcal{H}(\SL_2^\pm(\R),\tau_k)_1$ be the inverse $\tau_k$-spherical transform of $H$, as in \eqref{inversion-even} and \eqref{inversion-odd}. Observe that $\tau_{k-2}$ is the ${\rm O}(2)$-type directly preceding $\tau_k$ (in the natural ordering) in the induced representation containing $D_k$. We let $f_-\in\mathcal{H}(\SL_2^\pm(\R),\tau_{k-2})_1$ be the inverse $\tau_{k-2}$-spherical transform of $H$. We then put $f_G=(f_{+}-f_{-})/c_k\in\mathcal{H}(\SL_2^\pm(\R))_1$. Recall the ordering $\prec$ on $\mathcal{D}$ introduced in \S\ref{rep-notation}. From \eqref{traceH}, \eqref{sph-fn2}, and the definition of $f_+$ and $f_-$, it follows that, for $\pi_{\sigma,\nu}$ occurring as a subquotient of $\pi(\eta,\xi)$, we have
\begin{align*}
\tr\pi_{\sigma,\nu}(f_G)
&=\mathscr{H}^{\tau_k}(f_{+}/c_k)(\pi_{\sigma,\nu})-\mathscr{H}^{\tau_{k-2}}(f_{-}/c_k)(\pi_{\sigma,\nu})\\
&=\mathbf{1}_{\eta_{\epsilon}\prec\underline{\sigma}\prec D_k}\mathscr{H}^{\tau_k}(f_{+})(\eta,\xi)/c_k-\mathbf{1}_{\eta_{\epsilon}\prec\underline{\sigma}\prec D_{k-2}}\mathscr{H}^{\tau_{k-2}}(f_{-})(\eta,\xi)/c_k\\
&=\mathbf{1}_{\underline{\sigma}=D_k}H(\eta,\xi)/c_k=\mathbf{1}_{\underline{\sigma}=D_k}.
\end{align*}
This establishes \eqref{take-trace'}. In fact, as in the previous cases, \eqref{prop-trace} follows in the same way as invoking \eqref{sph-proj} yields
\[ \pi_{\sigma,\nu}(f_G)=H(\eta,\xi)/c_k\cdot\left(\frac1{\dim\tau_k}\Pi_{\tau_k}-\frac1{\dim\tau_{k-2}}\Pi_{\tau_{k-2}}\right). \]

It remains to prove \eqref{prop-bound'} for this choice of $f_G$. The term corresponding to $j=k$ in \eqref{inversion-even} and \eqref{inversion-odd} can be bounded, using \eqref{sph1}, by $H(\eta_\epsilon, (k-1)/2)(k-1)\ll k-1=\deg(D_k)$. Unlike \eqref{sph-inv2}, it is not the case that all other terms in \eqref{inversion-even} and \eqref{inversion-odd} vanish; nevertheless, using \eqref{sph1} and the properties of $r$, we can estimate
\begin{align*}
f(g)&\ll\sum_{2\leqslant j\leqslant k} \big(\big|\widehat{r}((j-k)/2i)\big|+\big|\widehat{r}((-j-k+2)/2i))\big|\big)(j-1)+e^{-\frac{k}{4}} k^N\int_\R |x|(1+|x|)^{-N}\,\mathrm{d}x\\
&\ll \sum_{2\leqslant j\leqslant k} je^{-\frac{k-j}4}+k^{-100}\ll k.
\end{align*}
This establishes \eqref{prop-bound'}.

\section{Controlling the Eisenstein contribution}\label{EisensteinSection}

Our aim in this section is bound, for appropriate test functions $\phi\in\mathcal{H}(\bm{G}(\A_F)^1)$, the Eisenstein distribution 
\begin{equation}\label{J-Eis}
J_{\rm Eis}(\phi)=\sum_{\bm{M}\neq \bm{G}} J_{{\rm spec},\bm{M}}(\phi)
\end{equation}
from \eqref{def:J-Eis}. The main result is given in Theorem \ref{spec-est} below. We will then show how to deduce from Theorem \ref{spec-est}, along with the results of previous sections, Theorem \ref{master}. This is done in Theorem \ref{money-cor}. In particular, it will be seen that it suffices to take test archimedean functions associated with characters of the maximal torus $T_{0,\infty}$.
  
\begin{theorem}\label{spec-est}
Let $n\geqslant 1$. Let $\underline{\delta}\in\mathcal{D}$ have standard representative $(T_{0,\infty},\delta)$. Assume $\delta$ is the trivial character if $n\geqslant 3$. For $\mu\in i\mathfrak{h}_0^*$ and $R\geqslant 1$, and let $f_R^{\delta,\mu}\in\mathcal{H}(G_\infty^1)_{R,\underline{\delta}}$ be the archimedean test function associated with the spectral localizer $h_R^{\delta,\mu}$ of Definition \ref{defn-localizer} in the explicit construction of Proposition~\ref{test-fn-bd}. Let $\q$ be an integral ideal. Then there is $C>0$, depending only on $n$ and $F$, such that
\[
J_{\rm Eis} (\varepsilon_{K_1(\q)}\otimes f_R^{\delta,\mu})\ll_\epsilon e^{CR}\norm\q^{n-1+\epsilon}\beta_0(\delta,\mu),
\]
where $\beta_0(\delta,\mu)=\beta_{T_{0,\infty}}^G(\delta,\mu)$.
\end{theorem}

The proof of Theorem \ref{spec-est} in fact establishes a much stronger estimate, which, for each proper Levi subgroup $\bm{M}\in\mathcal{L}$, majorizes $J_{{\rm spec},\bm{M}}(\varepsilon_{K_1(\q)}\otimes f_R^{\delta,\mu})$ by the upper bound in Propositions \ref{thm:n=2-stronger} and \ref{thm:general-n-stronger}. The proof that the latter estimates indeed imply Theorem \ref{spec-est} is given in Remarks \ref{rem:n=2:implies} and \ref{rem:n>3:implies}.

From Theorem \ref{spec-est}, as well as the various estimates established in Sections \ref{sec:geom1}--\ref{PW}, we obtain as an important consequence the following reformulation of Theorem \ref{master}.

\begin{theorem}\label{money-cor}
Let $n\geqslant 1$. Let $\underline{\delta}\in\mathcal{D}$ have standard representative $(M,\delta)$. Assume that either
\begin{enumerate}
\item\label{hyp-n2-spec-cor} $n\leqslant 2$, or
\item\label{hyp-sph-spec-cor} $M=T_{0,\infty}$ and $\delta$ is the trivial character of $T_{0,\infty}^1$. 
\end{enumerate}
Then Property {\rm (ELM)} holds with respect to $\underline{\delta}$.
\end{theorem}

\begin{proof}
By definition, we have $J_{\rm spec}-J_{\rm Eis}=J_{\rm disc}$. Moreover, the Arthur trace formula is the distributional identity $J_{\rm spec}=J_{\rm geom}$ on $\mathcal{H}(G(\A_F)^1)$. Thus
\[
J_{\rm error}(\varepsilon_{K_1(\q)}\otimes f_R^{\delta,\mu})=J_{\rm Eis}(\varepsilon_{K_1(\q)}\otimes f_R^{\delta,\mu})+J_{\rm geom}(\varepsilon_{K_1(\q)}\otimes f_R^{\delta,\mu})-D_F^{n^2/2}\Delta_F^*(1)\varphi_n(\q)\sum_{\gamma\in Z(F)\cap K_1(\q)}  f_R^{\delta,\mu}(\gamma).
\]
Note that in case \eqref{hyp-n2-spec-cor}, if $M\neq T_{0,\infty}$ then the only rational Levi subgroup $\bm{M}\subset\bm{G}$ containing $M$ is $\bm{G}$ itself, in which case $J_{\rm Eis}(\varepsilon_{K_1(\q)}\otimes f_R^{\delta,\mu})=0$. We may therefore apply Theorem \ref{spec-est} in all cases, which shows that the Eisenstein contribution is of acceptable size. For the remaining two geometric terms, we first apply Theorem \ref{geom-side}. We then use Proposition \ref{test-fn-bd} to convert the $L^\infty$-norm of $f_R^{\delta,\mu}$ to the $L^1$-norm of $h_R^{\delta,\mu}$ (it is here where we use the conditions on $n$ and $\underline{\delta}$) and apply Lemma \ref{prop:density}\,\eqref{beta1} with $M=T_{0,\infty}$, to bound the latter by an acceptable error.
\end{proof}

\subsection{The distribution $J_{{\rm spec},\bm{M}}$}\label{sec:explicating-JMspec}

From \eqref{J-Eis} it follows that, to prove Theorem \ref{spec-est}, it will be enough to bound $J_{{\rm spec},\bm{M}}(\varepsilon_{K_1(\q)}\otimes f_R^{\delta,\mu})$ for a given $\bm{M}\in\mathcal{L}$, where $\bm{M}\neq\bm{G}$. 

We apply the expansion \eqref{JspecM} and bound each term $J_{{\rm spec},\bm{M}}(\varepsilon_{K_1(\q)}\otimes f_R^{\delta,\mu}; s,\underline{\beta})$ separately, where $s\in W_{\bm{M}}$ and $\underline{\beta}\in\mathfrak{B}_{\bm{P},\bm{L}_s}$. Recalling the definition \eqref{Jintegral} and expanding over $\Pi_{\rm disc}(\bm{M}(\A_F)^1)$, we find that the integral $J_{{\rm spec},\bm{M}}(\varepsilon_{K_1(\q)}\otimes f_R^{\delta,\mu}; s,\underline{\beta})$ is equal to
\begin{equation}\label{bound-me}
\sum_{\pi\in\Pi_{\rm disc}(\bm{M}(\mathbb{A})^1)}\int_{i(\mathfrak{a}_{\bm{L}_s}^{\bm{G}})^*}{\rm tr}\left(\Delta_{\mathcal{X}_{\bm{L}_s}(\underline{\beta})}(\bm{P},\pi,\lambda)\mathcal{M}(\bm{P},s,\pi)\rho(\bm{P},\lambda,\pi,\varepsilon_{K_1(\q)}\otimes f_R^{\delta,\mu})\right)\, {\rm d}\lambda,
\end{equation}
where $\Delta_{\mathcal{X}_{\bm{L}_s}(\underline{\beta})}(\bm{P},\pi,\lambda)$, $\mathcal{M}(\bm{P},s,\pi)$, and $\rho(\bm{P},\lambda,\pi,\phi)$ denote the restrictions of the corresponding operators to the invariant subspace $\bar{\mathcal{A}}^2_\pi(\bm{P})$. 

Recall the notation $\bfk_\infty'=\bfk_\infty\cap \SL_n^\pm(F_\infty)$ from \S\ref{rep-notation} and Section \ref{PW}. For $\tau\in\Pi(\bfk_\infty')$ we let $\Pi_{K_1(\q),\tau}$ denote the orthogonal projection of $\bar{\mathcal{A}}^2_\pi(\bm{P})$ onto $\mathcal{A}_\pi^2(\bm{P})^{K_1(\q),\tau}$. From \eqref{prop-trace} for the fully induced representation, the description of $\rho(\bm{P},\lambda,\pi,\phi)$ in \eqref{L2-induced}, and the induction by stages identity ${\rm Ind}_{\bm{P}_\infty}^{G_\infty}(\pi(\sigma,\nu)\otimes e^{\langle \lambda ,H_{\bm{M}}\rangle})=\pi(\sigma,\nu+\lambda)$, as representations of $G_\infty$, it follows that
\[
\rho(\bm{P},\lambda,\pi,\varepsilon_{K_1(\q)}\otimes f_R^{\delta,\mu})=
H(\delta_\pi,\nu_\pi+\lambda) \sum_{\tau\in\Pi(\bfk')}\mu(\tau(\pi_{\delta}):\tau)\frac{1}{\dim \tau}\Pi_{K_1(\q),\tau}\qquad (\pi=\pi_f\otimes\pi_{\delta_\pi,\nu_\pi}).
\]
Recalling the support condition on $H$ from property \eqref{HM-support} of \S\ref{rem:operator-version}, we conclude that for a given $\pi\in\Pi_{\rm disc}(\bm{M}(\mathbb{A})^1)$, the corresponding term in \eqref{bound-me} vanishes unless $\underline{\delta_\pi}\prec\underline{\delta}$, which then contributes
\begin{equation}\label{eq:Eis-to-bound}
\int_{i(\mathfrak{a}_{\bm{L}_s}^{\bm{G}})^*}\! H(\delta_\pi,\nu_\pi+\lambda) {\rm tr}\bigg(\sum_{\tau\in\Pi(\bfk')}\!\mu(\tau(\pi_{\delta}):\tau)\Delta_{\mathcal{X}_{\bm{L}_s(\underline{\beta})}}(\bm{P},\pi,\lambda)\mathcal{M}(\bm{P},s,\pi)\frac{1}{\dim\tau}\Pi_{K_1(\q),\tau}\bigg)\, {\rm d}\lambda.
\end{equation}
We now explicate and bound \eqref{eq:Eis-to-bound} according to whether $n=2$, or $n\geqslant 3$ and $\delta$ is trivial character.

\subsection{Proof of Theorem \ref{spec-est} for $n=2$}

In the $\GL_2$ case, many of the quantities in \eqref{eq:Eis-to-bound} simplify. We begin by describing these simplifications explicitly.

We have, for example, $\bm{M}=\bm{T}_0$, and the parabolic $\bm{P}$ is necessarily the standard Borel subgroup $\bm{P}_0$. The $\pi\in\Pi_{\rm disc}(\bm{M}(\mathbb{A}_F)^1)=\Pi_{\rm disc}(\bm{T}_0(\mathbb{A}_F)^1)$ are therefore pairs of unitary Hecke characters on $F^\times\backslash\A_F^1$, which we shall denote by $\chi=(\chi_1,\chi_2)$. Recall from \S\ref{sec:hM-decomp} the subspace
\begin{equation}\label{eq:explicit-h0T0}
\h_0^{\bm{T}_0}=\left\{\prod_{v\mid\infty}\begin{pmatrix}x_v & \\ & y_v\end{pmatrix} : \sum_v d_v x_v=\sum_v d_vy_v=0\right\}\qquad (d_v=[F_v:\R]).
\end{equation}
Then $\sigma_\chi\in\mathscr{E}^2(T_{0,\infty}^1)$ is the restriction of $\chi_\infty$ to $T_{0,\infty}\cap\bfk$ and $\nu_\chi$ lies in $i(\mathfrak{h}_0^{\bm{T}_0})^*$.

Moreover, the Levi subgroup $\bm{L}_s$ can be either $\bm{T}_0$ or $\bm{G}$, according to whether $s\in W_0$ is the trivial element $e$ or the non-trivial element $w$. In the former case, we may identify the complex dual $(\mathfrak{a}_0^{\bm{G}})^*_\C$ with $\C$, in the latter we have $(\mathfrak{a}_{\bm{G}}^{\bm{G}})^*_\C=0$. When $s=e$ the combinatorial data $\mathcal{X}_{\bm{T}_0}(\underline{\beta})$ is the pair $\{\langle\bm{P}_0,\overline{\bm{P}}_0\rangle\}$, and the operator $\Delta_{\mathcal{X}_{\bm{T}_0}(\underline{\beta})}(\bm{P}_0,\chi,\lambda)$ is the logarithmic derivative of the normalized intertwining operator $\mathcal{M}(\chi,\lambda)=n(\lambda,\chi)R(\chi,\lambda)=n(\lambda,\chi)\prod_v R_v(\chi_v,\lambda)$ from \cite[\S 4]{GelbartJacquet1979}, where $n(\lambda,\chi)=\Lambda(1-2\lambda,\chi_2/\chi_1)\Lambda(1+2\lambda,\chi_1/\chi_2)^{-1}$. Moreover, in this case the intertwining operator $\mathcal{M}(\bm{P}_0,e,\chi)$ is trivial. On the other hand, when $s=w$, the combinatorial data $\mathcal{X}_{\bm{G}}(\underline{\beta})$ is empty, the operator $\Delta_{\mathcal{X}_{\bm{G}}(\underline{\beta})}(\bm{P}_0,0)$ is trivial, and ${\rm tr}(\mathcal{M}(\bm{P}_0,w)\rho(\bm{P},0,\phi))=-{\rm tr}(\rho(\bm{P},0,\phi)|_{\mathcal{A}^2_0(\bm{P})})$, where $\mathcal{A}^2_0(\bm{P})$ is the subspace of invariants under left-multiplication by $\left\{{\rm diag}(t,t^{-1}), t\in \A_F^1\right\}$; for this, see \cite[p. 375]{FinisLapid2011}. In particular,  ${\rm tr}(\mathcal{M}(\bm{P}_0,w,\chi)\rho(\bm{P},0,\chi,\phi))$ is supported on characters $\chi=(\chi_1,\chi_2)$ such that $\chi_1=\chi_2$. Finally, recall from the discussion preceding \eqref{eq:Eis-to-bound} that we may assume that $\underline{\sigma}_\chi\prec\underline{\delta}$. 

All in all, we deduce from the above explicit descriptions that \eqref{eq:Eis-to-bound} is either
\begin{equation}\label{eq:s=e}
\int_{i(\mathfrak{a}_0^{\bm{G}})^*}H(\sigma_\chi,\nu_\chi+\lambda) {\rm tr}\left(\sum_{\tau\in\Pi(\bfk')}\!\mu(\tau(\pi_{\delta}):\tau)\mathcal{M}(\chi,\lambda)^{-1}\mathcal{M}'(\chi,\lambda)\frac{1}{\dim\tau}\Pi_{K_1(\q),\tau}\right) {\rm d}\lambda,
\end{equation}
when $s=e$, or
\begin{equation}\label{eq:s=w}
H(\sigma_\chi,\nu_\chi)\dim  \pi(\chi_f)^{K_1(\q)}\sum_{\|\tau\|\geqslant \|\tau(\pi_{\sigma_\chi})\|}\!\mu(\tau(\pi_{\delta}):\tau) ,
\end{equation}
when $s=w$. The condition on $\tau$ comes from the projection onto the $\tau$-isotypic subspace of $\pi(\chi_\infty)$. The latter contributes to the discrete part of the trace formula; it derives from the residual spectrum.

Having explicitly described the Eisenstein contribution for $\GL_2$, we now proceed to prove Theorem \ref{spec-est}. We shall in fact prove a stronger estimate, recorded in Proposition \ref{thm:n=2-stronger} below. We recall the decomposition $\h_0=\aa_0^{\bm{G}}\oplus\h_0^{\bm{T}_0}$ from \eqref{eq:hM-aMG}; we write $\mu_{\bm{T}_0}$ and $\mu^{\bm{T}_0}$ for the orthogonal projections of $\mu$ onto $i(\aa_0^{\bm{G}})^*$ and $i(\h_0^{\bm{T}_0})^*$, respectively. 

\begin{prop}\label{thm:n=2-stronger} Let $n=2$. Then for $R\geqslant 1$, we have
\[
J_{{\rm spec},\bm{T}_0}(\varepsilon_{K_1(\q)}\otimes f_R^{\delta,\mu})\ll_\epsilon \norm\mathfrak{q}^{1+\epsilon}\log (1+\beta_0(\delta,\mu^{\bm{T}_0})).
\]
\end{prop}

\begin{remark}\label{rem:n=2:implies}
To see that Proposition \ref{thm:n=2-stronger} implies Theorem \ref{spec-est} it suffices to observe that $\log (1+\beta_0(\delta,\mu^{\bm{T}_0}))\ll \log (1+\beta_0(\delta,\mu))$.
Let $\beta$ denote the unique positive root in $\Phi^{\bm{G},+}$. For $\alpha\in \Phi^{G_\infty, +}$ let $n_\alpha$ denote the multiplicity of the corresponding root space. Recalling Definition~\ref{BetaMajorizer}, and setting $k=\prod_v (k_{1,v}e_{1,v}^*+k_{2,v}e_{2,v}^*)\in\mathfrak{a}_{0,\C}^*$, we have that
\begin{equation}\label{eq:beta0-mu}
\log(1+\beta_0(\delta,\mu))\asymp\log(2+\max_{\alpha\in \Phi^{G_\infty, +}}|\langle k,\alpha\rangle|+\max_{\alpha\in \Phi^{G_\infty, +}}|\langle\mu,\alpha\rangle|).
\end{equation}
We put $\mu^*=\langle\mu_{\bm{T}_0},\beta\rangle$ and note that $\mu^*=\langle\mu_{\bm{T}_0},\alpha\rangle$ for all $\alpha\in\Phi^{G_\infty, +}$. Thus $\log(1+\beta_0(\delta,\mu^{\bm{T}_0}))$ verifies the same asymptotic as \eqref{eq:beta0-mu} but with $|\langle \mu,\alpha\rangle-\mu^{\ast}|$ in place of $|\langle \mu,\alpha\rangle|$. The desired conclusion is of course immediate if $\mu^{\ast}=0$. Otherwise we observe, using \eqref{eq:explicit-h0T0}, that
\[
\sum_\alpha n_\alpha(\langle \mu,\alpha\rangle-\mu^{\ast})=\sum_v d_v(\langle \mu^{\bm{T}_0},e_{1,v}^*\rangle-\langle \mu^{\bm{T}_0},e_{2,v}^*\rangle)=0.
\]
From this it follows that the maximum of $|\langle \mu,\alpha\rangle-\mu^{\ast}|$ over all $\alpha$ is unchanged (up to a multiplicative constant) if restricted to $\alpha$ for which $\langle \mu,\alpha\rangle\in \R_{\geqslant 0}\mu^\ast$. The latter maximum is at most $\max_\alpha |\langle\mu,\alpha\rangle|$, as desired.
\end{remark}

\begin{proof}
Taking the logarithmic derivative of the normalized global intertwining operator, we have
\begin{equation}\label{eq:log-der-M}
\mathcal{M}(\chi,\lambda)^{-1}\mathcal{M}'(\chi,\lambda)=\frac{n'(\lambda,\chi)}{n(\lambda,\chi)}I+\sum_u R_u^{-1}(\chi_u,\lambda)R_u'(\chi_u,\lambda)\bigotimes_{v\neq u} I_v.
\end{equation}
We treat separately the first term (involving the logarithmic derivative of $n(\lambda,\chi)$), and then the other terms according to whether $u$ is finite or archimedean. When inserted into \eqref{eq:s=e}, only those $u$ dividing $\q\infty$ will contribute: indeed, at the remaining places, $R_u(\chi_u,\cdot)$ acts as a constant on the spherical vector, so its (logarithmic) derivative is zero.

When inserted into \eqref{eq:s=e}, the first term in \eqref{eq:log-der-M} contributes
\[
\dim  \pi(\chi_f)^{K_1(\q)}\int_{i(\mathfrak{a}_0^{\bm{G}})^*}H(\sigma_\chi,\nu_\chi+\lambda) \frac{n'(\lambda,\chi)}{n(\lambda,\chi)}\sum_{\|\tau\|\geqslant \|\tau(\pi_{\sigma_\chi})\|}\!\mu(\tau(\pi_{\delta}):\tau)\, {\rm d}\lambda.
\]
The sum over $\tau$ vanishes unless $\underline{\sigma}_\chi=\underline{\delta}$, in which case it is 1. Since $H(\delta,\nu)=h_R^{\delta,\mu}(\nu)$ by Property \eqref{HM-value} of \S\ref{rem:operator-version} we have
\[
\dim\pi(\chi_f)^{K_1(\q)} \int_{i(\mathfrak{a}_0^{\bm{G}})^*}h_R^{\delta,\mu}(\nu_\chi+\lambda)n'(\lambda,\chi)/n(\lambda,\chi)\, {\rm d}\lambda.
\]
Observe that $\dim\pi(\chi_f)^{K_1(\q)}=0$ unless $\q_{\chi_1}\q_{\chi_2}|\q$, in which case $\dim\pi(\chi_f)^{K_1(\q)}\ll \log(1+ \norm\q)$. We let $\Gamma$ be a lattice in $i(\mathfrak{a}_0^{\bm{G}})^*\simeq i\R$, of unit spacing and break up the integral over $i(\mathfrak{a}_0^{\bm{G}})^*$ into a sum over unit intervals centered at $\nu_k\in\Gamma$. Inserting the rapid decay estimate on $h_R^{\delta,\mu}$ from Definition \ref{defn-localizer}\eqref{rapid-decay-item} (disregarding the savings in $R$), we deduce that, for $N$ large enough, the above expression is bounded by
\[
\log (1+\norm\q)\sum_{\nu_k\in\Gamma} \max_{w\in W(A_0)_\delta} (1+\|\nu_\chi+\nu_k-w\mu\|)^{-N}\int_{\|\lambda-\nu_k\|\leqslant 1/2}|n'(\lambda,\chi)/n(\lambda,\chi)|\, {\rm d}\lambda.
\]
From the proof of \cite[Propositions 4.5 and 5.1]{Muller2007} we deduce that the last integral is bounded by ${\rm O}(\log(1+\|\nu_k\|+ Q(\chi_1\chi_2^{-1}))$, with $Q$ denoting the analytic conductor. Now $Q(\chi_1\chi_2^{-1})\leqslant \norm\q\prod_{v\mid\infty}(1+|\delta_{v1}-\delta_{v2}+\nu_{\chi_{v1}}-\nu_{\chi_{v2}}|)^{d_v}=\norm\q\,\beta_0(\delta,\nu_\chi)$. Executing the sum over $\Gamma$ and summing over $\chi=(\chi_1,\chi_2)$, we get\footnote{In this, and the estimates which follow, the value of $N$ can change from line to line.}
\[
\log(1+\norm\q)\sum_{\substack{\chi: \underline{\sigma}_\chi=\underline{\delta}\\ \q_{\chi_1}\q_{\chi_2}\mid\q}} \max_{w\in W(A_0)_\delta} (1+\|\nu_\chi-w\mu^{\bm{T}_0}\|)^{-N}\log(1+\|\nu_\chi-w\mu^{\bm{T}_0}\|+\norm\q\beta_0(\delta,\nu_\chi)).
\]
We decompose $\nu_\chi=\nu_{\chi_1}+\nu_{\chi_2}$ with respect to the diagonal $\GL_1$ block decomposition of \eqref{eq:explicit-h0T0}. We cover $i(\mathfrak{h}_0^{\bm{T}_0})^*$ by balls of unit radius centered at points in a lattice $\Lambda=\Lambda_1\oplus\Lambda_2$ respecting \eqref{eq:explicit-h0T0}. For $\mu^{(j)}\in\Lambda$ we write $\mu^{(j)}=\mu_1^{(j)}+\mu_2^{(j)}$ accordingly. Using Lemma \ref{prop:density}\eqref{beta2}, we find that the above expression is then bounded by $(\log(1+\norm\q))^2$ times
\[
\sum_{\mu^{(j)}\in\Lambda}\max_{w\in W(A_0)_\delta} (1+\|\mu^{(j)}-w\mu^{\bm{T}_0}\|)^{-N}\log(1+\|\mu^{(j)}-w\mu^{\bm{T}_0}\|+\beta_0(\delta,\mu^{(j)})) \!\!\!\!\!\!\!\!\!\!\!\sum_{\substack{\chi: \underline{\sigma}_\chi=\underline{\delta}\\ \q_1\q_2\mid\q\\ \|\nu_{\chi_k}-\mu_k^{(j)}\|\leqslant 1,\; k=1,2}}\!\!\!\!\!\!\!\!\!1.
\]
If $\delta=\delta_1\delta_2$, with $\delta_i$ a character of $\prod_{v\mid\R}\{\pm 1\}\prod_{v\mid\C}\U(1)$, the condition $\underline{\sigma}_\chi=\underline{\delta}$ is equivalent to $\{\sigma_{\chi_1},\sigma_{\chi_2}\}= \{\delta_1,\delta_2\}$. The last sum on $\chi$ can be written as the sum of $N(\q_1,\delta_i,\mu_1^{(j)})N(\q_2,\delta_{i'},\mu_2^{(j)})$ over all $\q_1\q_2|\q$ and all $\{i,i'\}=\{1,2\}$. Applying Theorem \ref{MainCountingResultGLn} for $n=1$ (where Property (ELM) holds trivially) and the standard divisor bound yields $\norm\q^{1+\epsilon}$. The remaining sum over $\mu_j\in\Lambda$ has rapid decay away from $\mu^{(j)}\in W\mu^{\bm{T}_0}+{\rm O}(1)$, and is bounded by ${\rm O}(\log(1+ \beta_0(\delta,\mu^{\bm{T}_0})))$ as desired.

Now we address the sum on $u$ in \eqref{eq:log-der-M}. At a finite place $u\mid\q$, the sum over $\tau$ collapses as before, leaving only $\chi$ for which $\underline{\sigma}_\chi=\underline{\delta}$. Using the upper bound $\|A\|_1\leqslant \dim V\|A\|$ for any linear operator $A$ on a finite dimensional Hilbert space $V$, the contribution of such $u$ is majorized by
\[
\sum_{\chi: \underline{\sigma}_\chi=\underline{\delta}}\dim\pi(\chi_f)^{K_1(\q)} \int_{i(\mathfrak{a}_0^{\bm{G}})^*}\big|h_R^{\delta,\mu}(\nu_\chi+\lambda)| \big\|R_u'(\chi_u,\lambda)|_{\pi(\chi_u,\lambda)^{K_{1,u}(\q)}}\big\|  \, {\rm d}\lambda.
\]
Here, we have used that $R_u$ is unitary and commutes with the projection $\Pi_{K_1(\q),\tau}$. When $\q_{\chi_1}\q_{\chi_2}\mid\q$, we break up the above integral over unit intervals centered at a lattice $\Gamma\subset i(\mathfrak{a}_0^{\bm{G}})^*$, and use the rapid decay estimate of $h_R^{\delta,\mu}$ from Definition \ref{defn-localizer}\eqref{rapid-decay-item}, to get
\[
\log (1+\norm\q)\sum_{\nu_k\in\Gamma} \max_{w\in W(A_0)_\delta} (1+\|\nu_\chi+\nu_k-w\mu\|)^{-N}\int_{\|\lambda-\nu_k\|\leq 1/2}\big\|R_u'(\chi_u,\lambda)|_{\pi(\chi_u,\lambda)^{K_{1,u}(\q)}}\big\|\, {\rm d}\lambda.
\]
We insert the Bounded Degree Property of \cite[Theorem 1]{FinisLapidMuller2012} (see also the analysis in \cite[\S 5.4]{FinisLapidMuller2015}) to bound the integral by $\log (1+\norm\q)$. We then execute the sum over $\Gamma$, sum over $\chi$, and argue as in the previous paragraph to obtain a contribution of ${\rm O}(\norm\q^{1+\epsilon})$, which is acceptable.

Finally, to treat an archimedean place $u$, for a $\bfk_u'$-type $\tau_u$, we denote by $\gamma(\sigma_{\chi_u},\lambda,\tau_u)\in\C$ the scalar by which the $\bfk_u'$-intertwining operator $R_u(\sigma_{\chi_u},\lambda)$ acts on the isotypic subspace $\pi_u(\sigma_{\chi_u},\lambda)^{\tau_u}$. With this notation, we may write the contribution of $u\mid\infty$ to \eqref{eq:s=e} as
\[
\dim\pi(\chi_f)^{K_1(\q)}\int_{i(\mathfrak{a}_0^{\bm{G}})^*}H(\sigma_\chi,\nu_\chi+\lambda) \sum_{\|\tau\|\geqslant \|\tau(\pi_{\sigma_\chi})\|}\mu(\tau(\pi_{\delta}):\tau)\frac{\gamma'}{\gamma}(\sigma_{\chi_u},\nu_{\chi_u}+\lambda,\tau_u)\, {\rm d}\lambda.
\]
For $u$ real, the sum on $\tau$ again vanishes unless $\underline{\sigma}_\chi=\underline{\delta}$. Applying Stirling's formula to the classical formula for the intertwining operator \cite[Theorem 7.17]{Wallach1979} yields $\frac{\gamma'}{\gamma}(\delta_u,\lambda,\tau_u)\ll \log (1+\beta_{0,u}(\delta_u,\lambda)$). We can then argue as in the preceding cases to obtain an acceptable contribution. For $u$ complex, the sum on $\tau$ vanishes unless $\underline{\sigma}_{\chi_v}=\underline{\delta}_v$ for every archimedean $v\neq u$. Applying Stirling's formula to the formula \cite[Theorem 7.23]{Wallach1979} shows, using the parametrization of \S\ref{sec:hv-complex}, that, for $p\in [[\ell]]$, $p\neq\ell$, we have $\frac{\gamma'}{\gamma}(p,\lambda,\ell)-\frac{\gamma'}{\gamma}(p,\lambda,\ell-1)\ll \ell^{-1}$, as well as $\frac{\gamma'}{\gamma}(\ell,\lambda,\ell)\ll \log (1+\beta_{0,u}(\delta_\ell,\lambda))$. Letting $\ell_v\in\frac12\Z_{\geqslant 0}$ be such that $\tau_{\ell_v}=\tau(\delta_v)$, for every $v\mid\infty$, this produces the bound of 
\begin{equation}\label{u-complex-contribution}
\log(1+\norm\q)\left(\ell_u^{-1}\sum_{p_u\in [[\ell_u]]} +\log (1+\beta_{0,u}(\delta_\ell,\lambda))\sum_{p_u=\pm\ell_u}\right) \sum_{\substack{\chi: \, \q_{\chi_1}q_{\chi_2}|\q\\ p_{\chi_v}=\ell_v (v\neq u)\\ p_{\chi_u}=p_u}}\int_{i(\mathfrak{a}_0^{\bm{G}})^*}|H(p_\chi,\nu_\chi+\lambda)|\, {\rm d}\lambda.
\end{equation}
For the $\chi$ prescribed by the summation, we have $|H_u(p_{\chi_u},\nu_{\chi_u}+\lambda)\prod_{v\mid\infty, v\neq u} h_R^{\delta_v,\mu_v}(\nu_{\chi_v}+\lambda)|$. For the factors at archimedean places $v\neq u$, we insert the rapid decay estimate from Definition \ref{defn-localizer}\eqref{rapid-decay-item}. At $u$, we put $\Lambda_u=\min_\pm\|\lambda\pm\mu_u\|$ and use the definition \eqref{defn:Hpnu} to deduce 
\[
H_u(p_u,\lambda)\ll
\begin{cases}
(\ell_u-|p_u|)^{-1}, & \textrm{if } \Lambda_u\ll \ell_u-|p_u|;\\
\Lambda_u^{-2}(\ell_u-|p_u|), &\textrm{if } \ell_u-|p_u|\ll \Lambda_u \ll\ell_u;\\
\ell_u^{-2}(\ell_u-|p_u|)e^{-\pi \Lambda_u/ 2\ell_u}, &\textrm{if } \Lambda_u\gg\ell_u.
\end{cases}
\]
We now combine these estimates across all archimedean places and integrate over $\lambda$. To succinctly express the resulting bound, it will be convenient to introduce the following notation. Let $S_\infty^{(u)}=\{v\mid\infty: v\neq u\}$. For $\nu\in i(\h_0^{\bm{T}_0})^*$ and $v\mid\infty$ we put 
\[
X_\nu=\max\{\|\nu_{v_1} - \nu_{v_2}\|: v_1,v_2\in S_\infty^{(u)}\}, \qquad \Delta_\nu=\|\nu_u-|S_\infty^{(u)}|^{-1}\sum_{v\in S_\infty^{(u)}}\nu_v\|.
\]
Then the integral in \eqref{u-complex-contribution} is bounded by ${\rm O}_N(\max_{w\in W(A_0)_\delta} g_{\ell,N}(p_\chi,\nu_\chi-w\mu^{\bm{T}_0}))$, where, for $\xi\in i(\h_0^{\bm{T}_0})^*$, we have set
\[
g_{\ell,N}(p,\xi)=\begin{cases}
(\ell_u-|p_u|)^{-1}(1+X_\xi)^{-N}, & \text{if}\quad \Delta_\xi\ll \ell_u-|p_u|;\\
\Delta_\xi^{-2}(\ell_u-|p_u|)(1+X_\xi)^{-N}, &\text{if}\quad \ell_u-|p_u|\ll \Delta_\xi\ll \ell_u;\\
\ell_u^{-2}(\ell_u-|p_u|)(1+\ell_u^{-1}\Delta_\xi + X_\xi)^{-N}, & \text{if}\quad \Delta_\xi\gg \ell_u.
\end{cases}
\]
Observe that $\int g_{\ell,N} (p,\xi){\rm d}\xi={\rm O}(1)$. Inserting this expression into \eqref{u-complex-contribution}, we obtain
\[
\log(1+\norm\q)\log (1+\beta_{0,u}(\delta_u,\lambda))\ell_u^{-1}\sum_{p_u\in [[\ell_u]]}  \sum_{\mu^{(j)}\in\Lambda}\max_{w\in W(A_0)_\delta} g_{\ell,N}(p,\mu^{(j)}-w\mu^{\bm{T}_0})\sum_{\substack{\chi: \, \q_{\chi_1}q_{\chi_2}|\q\\ p_{\chi_v}=\ell_v (v\neq u), p_{\chi_u}=p_u\\ \|\nu_\chi-\mu^{(j)}\|\leqslant 1}}1.
\]
Inserting Lemma \ref{MainCountingResultGLn} for $n=1$, and using the fact that the sum over $\Lambda$ and the normalized sum over $p_u\in [[\ell_u]]$ are both ${\rm O}(1)$, we deduce that the contribution of complex places $u$ to \eqref{eq:s=e} is acceptable. All cases having been treated, this completes the proof of the estimation of \eqref{eq:s=e}. 

We now turn to \eqref{eq:s=w}, which we must sum over characters $\chi$ of the form $(\chi_1,\chi_1)$. As before, the alternating sum on $\bfk_\infty'$-types implies that only those $\chi$ for which $\underline{\sigma}_\chi=\underline{\delta}$ contribute non-trivially. Similarly, we may assume $\q_{\chi_1}^2\mid\q$. Summing \eqref{eq:s=w} over such $\chi$, we obtain
\[ 
\sum_{\substack{\chi=(\chi_1,\chi_1)\\ \underline{\sigma}_\chi=\underline{\delta},\; \q_{\chi_1}^2\mid\q}} h_R^{\delta,\mu}(\nu_\chi) \dim \pi(\chi_f)^{K_1(\q)}.
\]
Bounding $\dim\pi(\chi_f)^{K_1(\q)}\ll\log (1+\norm\q)$, and arguing as in the previous paragraphs, we find that that \eqref{eq:s=w} makes an acceptable contribution.
\end{proof}

\subsection{Proof of Theorem \ref{spec-est} in the spherical case}

It remains to prove Theorem \ref{spec-est} for $\bm{G}=\GL_n$, $n\geqslant 3$, when $\delta$ trivial. We first conveniently package two of the major inputs that will be necessary for the proof of Theorem \ref{spec-est} for $\delta$ trivial. Throughout this section, the reader is encouraged to regularly consult \S\ref{JMspec} for the notation related to the spectral side of the Arthur trace formula.

The first input concerns the norm of the operators $\Delta_{\mathcal{X}_{\bm{L}_s(\underline{\beta})}}(\bm{P},\lambda):\mathcal{A}^2(\bm{P})\rightarrow \mathcal{A}^2(\bm{P})$ and is encapsulated in Lemma \ref{package} below. The second, recorded in Lemma \ref{aut-2-ind}, bounds the dimension of the space of oldforms of an induced automorphic representation in terms of the corresponding dimension of the inducing data. 

We first need to introduce some more notation. As in \S\ref{JMspec} we let $\bm{M}\in\mathcal{L}$ and $\bm{P}\in\mathcal{P}(\bm{M})$. For $\pi\in\Pi_{\rm disc}(\bm{M}(\A_F)^1)$, let $\mathcal{A}^2_\pi(\bm{P})$ denote the subspace of $\mathcal{A}^2(\bm{P})$ consisting of $\varphi$ such that, for each $x\in\bm{G}(\A_F)^1$, the function $\varphi_x$ transforms under $\bm{M}(\A_F)^1$ according to $\pi$. For a compact open subgroup $K_f$ of $\bm{G}(\A_f)$ we let $\mathcal{A}_\pi(\bm{P})^{K_f,\tau_0}$ be the finite dimensional subspace of $K_f\bfk_\infty$-invariant functions. Finally, if $\pi$ is spherical, so that $\pi^{\bfk^{\bm{M}}_\infty}\neq 0$, then we denote by $\nu_\pi\in (\h_0^{\bm{M}})_\C^*$ its spectral parameter and put
\begin{equation}\label{defn:beta_M}
\beta_{\bm{M}}(\nu_\pi)=\prod_{\substack{\alpha\in \Phi^{G_\infty,+}\\ \alpha\notin\Phi^{\bm{M}_\infty,+}}} \big(1+\big|\langle \nu_\pi, \alpha\rangle\big|\big)^{n_\alpha},
\end{equation}
where $n_\alpha$ is the dimension of the corresponding root space.

\begin{lemma}[Finis--Lapid--M\"uller, Lapid, Matz]\label{package}
Let $\q$ be an integral ideal. Let $\bm{M}\in\mathcal{L}$, $\bm{M}\neq\bm{G}$, and $\bm{L}\in\mathcal{L}(\bm{M})$. Then for all $\pi\in\Pi_{\rm disc}(\bm{M}(\A_F)^1)$, with $\pi^{\bfk^{\bm{M}}_\infty}\neq 0$, and $\nu\in i(\mathfrak{a}_{\bm{L}_s}^{\bm{G}})^*$ we have
\[
\int_{B(\nu)\cap i(\mathfrak{a}_{\bm{L}_s}^{\bm{G}})^*}\|\Delta_{\mathcal{X}_{\bm{L}_s(\underline{\beta})}}(\bm{P},\lambda)_{|\mathcal{A}^2_\pi(\bm{P})^{K_1(\q),\tau_0}}\|\, {\rm d}\lambda\ll (1+ \log\norm\q+\log (1+\|\nu\|)+\log (1+\beta_{\bm{M}}(\nu_\pi)))^{2m_{\bm{L}}},
\]
where $m_{\bm{L}}=\dim\mathfrak{a}_{\bm{L}}$.
\end{lemma}

\begin{proof}
This is a slight refinement of \cite[Lemma 14.3]{Matz2017}. The proof is based on two important contributions from Finis--Lapid--M\"uller. The first is a strong form of the Tempered Winding Number property for $\GL_n$, established in \cite[Proposition 5.5]{FinisLapidMuller2015}. This is usually expressed in terms of the Casimir eigenvalue but the proof in fact yields the stronger statement with the finer invariant $\beta_{\bm{M}}(\nu_\pi)$. Indeed, if $\bm{M}\simeq\GL_{n_1}\times\cdots\times\GL_{n_m}$ and $\pi\simeq\pi_1\otimes\cdots\otimes\pi_m$, the upper bound obtained in \cite[p. 26]{FinisLapidMuller2015} is given in terms of $\prod_{i<j}q(\pi_i\times\widetilde\pi_j)$, the product of analytic conductors of Rankin--Selberg $L$-functions. Under the spherical hypothesis, the archimedean component of this product is precisely the expression $\beta_{\bm{M}}(\nu_\pi)$. The second is the Bounded Degree property; in \cite[Theorem 1]{FinisLapidMuller2012} it is shown that $\GL_n$ over $p$-adic fields satisfies this property and in the appendix to \cite{MullerSpeh2004} (see also \cite[Theorem 2]{FinisLapidMuller2012}) the same is shown for arbitrary real groups.
\end{proof}

We next relate the dimension of space of invariants of the $\bm{P}$-induced automorphic forms on $G$ to the dimension of the corresponding space of invariants of the inducing representation. 

\begin{lemma}\label{aut-2-ind}
Let $\bm{M}\simeq \GL_{n_1}\times\cdots \times \GL_{n_m}\in\mathcal{L}$ with $n_1+\cdots + n_m=n$. Let $\pi\simeq\pi_1\otimes\cdots\otimes\pi_m\in\Pi_{\rm disc}(\bm{M}(\A_F)^1)$ and write $\q_k$ for the arithmetic conductor of $\pi_k$. Then, for any integral ideal $\q$, we have $\mathcal{A}_\pi^2(\bm{P})^{K_1(\q),\tau_0}=0$ unless all $\pi_k$ are spherical at infinity and $\prod_{k=1}^m\q_k\mid\q$, in which case
\[
\dim\mathcal{A}_\pi^2(\bm{P})^{K_1(\q),\tau_0}\ll (1+\log \norm\q)^n\prod_{k=1}^m \dim V_{\pi_k}^{K_1^{\GL_{n_k}}(\q_k)},
\]
where $K_1^{\GL_{n_k}}(\q_k)$ is the Hecke congruence subgroup for $\GL_{n_k}$ from \S\ref{sec:def-cond}.
\end{lemma}
\begin{proof}
Let $\mathcal{H}_{\bm{P}}(\pi_f)$ and $\mathcal{H}_{\bm{P}}(\pi_\infty)$ be the Hilbert spaces of the (unitarily) induced representations $I_{\bm{P}(\A_f)}^{\bm{G}(\A_f)}(\pi_f)$ and $I_{\bm{P}_\infty}^{G_\infty}(\pi_\infty)$. By \cite[(3.5)]{Muller2007} and the multiplicity one theorem described in \S\ref{JMspec}, we have
\[
\dim\mathcal{A}_\pi^2(\bm{P})^{K_1(\q),\tau_0}=\dim\mathcal{H}_{\bm{P}}(\pi_f)^{K_1(\q)} \dim\mathcal{H}_{\bm{P}}(\pi_\infty)^{\tau_0}.
\]
By \cite[(3.6)]{Muller2007} we have $\dim\mathcal{H}_{\bm{P}}(\pi_\infty)^{\tau_0}\leqslant \dim V_{\pi_\infty}^{\tau_0^{\bm{M}}}$, where $\tau_0^{\bm{M}}=\tau_0|_{\bfk_\infty^{\bm{M}}}$ denotes the trivial representation of $\bfk_\infty^{\bm{M}}$. The latter dimension is $1$ or $0$ according to whether all $\pi_k$ are spherical or not.

Now let $v$ be a finite place and write $f_v$ for the conductor exponent of the irreducible tempered generic representation $\varPi_v=I_{P_v}^{G_v}(\pi_v)$ of $G_v$. From the dimension formulae of \cite{Casselman1973,Reeder1991} we have $\dim\mathcal{H}_{\bm{P}}(\pi_v)^{K_1(\p_v^r)}=\binom{n+r-f_v}{n}$ for $r\geqslant f_v$ and $\mathcal{H}_{\bm{P}}(\pi_v)^{K_1(\p_v^r)}=0$ otherwise. On the other hand, by definition of the conductor exponent of $\pi_{v,k}$, the dimension of the $K_1^{\GL_{n_k}}(\p_v^r)$-fixed vectors in $V_{\pi_{v,k}}$ is at least $1$ precisely when $r\geqslant f_{v,k}$. Now, we may use the Local Langlands Correspondence \cite{HarrisTaylor2001, Henniart2000, Scholze2013} to compare the conductors of the $\pi_{v,k}$ and $\varPi_v$. Indeed, if $\phi_v$ is the Langlands parameter of $\varPi_v$ and $\phi_{v,k}$ that of $\pi_{v,j}$ then $\phi_v=\oplus_k\phi_{v,k}$. From this it follows that $f_v=\sum_kf_{v,k}$. We deduce that $\dim\mathcal{H}_{\bm{P}}(\pi_v)^{K_1(\q_v)}=0$ precisely when $\prod_k \q_{v,k}\nmid \q_v$, and the upper bound follows from $\binom{n+\log_v \norm\q_v-f_v}{n}\leqslant\binom{n+\log_v \norm\q_v}{n}\ll_n (1+\log_v \norm\q_v)^n$.
\end{proof}

 We shall again prove a stronger estimate than that stated in Theorem \ref{spec-est}. To state it, we recall the decomposition $\h_0=\aa_{\bm{M}}^{\bm{G}}\oplus\h_0^{\bm{M}}$ from \eqref{eq:hM-aMG} to write $\mu=\mu_{\bm{M}}+\mu^{\bm{M}}$. 
\begin{prop}\label{thm:general-n-stronger}
Let $n\geqslant 1$. Let $\bm{M}\in\mathcal{L}$, $\bm{M}\neq\bm{G}$. When $\delta$ is the trivial character, we have
\[
J_{{\rm spec},\bm{M}}(\varepsilon_{K_1(\q)}\otimes f_R^{\delta,\mu})\ll_\epsilon e^{CR}\norm\mathfrak{q}^{n-1+\epsilon}(\log(1+\beta_{\bm{M}}(\mu^{\bm{M}}))^{2m_{\bm{M}}}\beta^{\bm{M}}_0(\mu^{\bm{M}}).
\]
\end{prop}

\begin{remark}\label{rem:n>3:implies}
Similarly to the $n=2$ case, we now show how the spherical case of Theorem \ref{spec-est} can be deduced from Proposition \ref{thm:general-n-stronger}. When $M=T_{0,\infty}$ we may rewrite Definition \ref{BetaMajorizer} as
\[
\beta_0(\mu)=\prod_{\alpha\in\Phi^{G_\infty,+}}\big(1+\big|\langle \mu, \alpha\rangle\big|\big)^{n_\alpha}.
\]
In this way, $\beta_0^{\bm{M}}(\mu^{\bm{M}})$ is given by the same expression as $\beta_0(\mu)$, but with $\Phi^{G_\infty,+}$ replaced by $\Phi^{\bm{M}_\infty,+}$. Recalling the definition \eqref{defn:beta_M}, it follows that $\beta_0(\mu)=\beta_0^{\bm{M}}(\mu^{\bm{M}})\beta_{\bm{M}}(\mu)$. It therefore suffices to verify that $\log(1+\beta_{\bm{M}}(\mu^{\bm{M}}))\ll \log(1+\beta_{\bm{M}}(\mu))$. For every $\beta\in \Phi^{\bm{G},+}\setminus \Phi^{\bm{M},+}$, let $\bm{M}_\beta<\bm{M}$ be the Levi subgroup on which $\beta$ is non-trivial, and $\bm{G}_\beta$ the smallest Levi subgroup properly containing $\bm{M}_\beta$. Then $\bm{M}_{\beta,\infty}$ contains $A_{0,\beta}=A_0\cap\bm{G}_{\beta,\infty}$, and we may consider $\Phi^{\bm{M}_{\beta,\infty}}=\Phi(A_{0,\beta},M_{\beta,\infty})$ as a subset of $\Phi^{\bm{G}_{\beta,\infty}}=\Phi(A_{0,\beta},\bm{G}_{\beta,\infty})$. With this notation, we may reorganize the $\alpha$ appearing in \eqref{defn:beta_M} according to the disjoint union over \textit{rational} roots $\beta\in\Phi^{\bm{G},+}$ not in $\Phi^{\bm{M},+}$ of the various complements $\Phi^{\bm{G}_{\beta,\infty},+}\setminus \Phi^{\bm{M}_{\beta,\infty},+}$ of archimedean roots. Thus
\begin{equation}\label{eq:log-beta0}
\log(1+\beta_{\bm{M}}(\mu))\asymp\log\bigg(2+\sum_{\substack{\beta\in \Phi^{\bm{G},+}\\ \beta\notin \Phi^{\bm{M},+}}}\big(\max_{\substack{\alpha\in \Phi^{\bm{G}_{\beta,\infty},+}\\ \alpha\notin\Phi^{\bm{M}_{\beta,\infty},+}}}|\langle\mu,\alpha\rangle|\big)\bigg).
\end{equation}
Moreover, $\log(1+\beta_{\bm{M}}(\mu^{\bm{M}}))$ verifies the same asymptotic as in \eqref{eq:log-beta0} but with $
\langle \mu,\alpha\rangle-\langle \mu^{\bm{M}_\beta},\beta\rangle$ in place of $\langle \mu,\alpha\rangle$. We conclude by arguing as in Remark~\ref{rem:n=2:implies} for each $\beta$ separately.
\end{remark}

\begin{proof}
The proof is by induction on $n$. For $n=1$ there is no continuous spectrum so Proposition  \ref{thm:general-n-stronger} is trivially true in that case. Now assume the result for $\GL_m$ for all $m<n$.

We adapt the discussion in \S\ref{sec:explicating-JMspec} to the present context. When $\delta$ is the trivial character of $T_{0,\infty}^1$ and $\tau(\pi_\delta)$ is the trivial $\bfk'$-type $\tau_0$, the condition $\underline{\delta_\pi}\prec\underline{\delta}$ implies $\delta_\pi=\delta$. Thus \eqref{bound-me} becomes
\[
\sum_{\pi\in\Pi_{\rm disc}(\bm{M}(\mathbb{A})^1)_{\underline{\delta}}}\int_{i(\mathfrak{a}_{\bm{L}_s}^{\bm{G}})^*}\! h_R^{\delta,\mu}(\nu_\pi+\lambda) {\rm tr}\big(\Delta_{\mathcal{X}_{\bm{L}_s(\underline{\beta})}}(\bm{P},\pi,\lambda)\mathcal{M}(\bm{P},s,\pi)\Pi_{K_1(\q),\tau_0}\big)\, {\rm d}\lambda.
\]
We apply the upper bound $\|A\|_1\leqslant \dim V\|A\|$  to deduce that the above expression is bounded by
\[
\sum_{\pi\in\Pi_{\rm disc}(\bm{M}(\mathbb{A})^1)_{\underline{\delta}}}\dim\mathcal{A}_\pi^2(\bm{P})^{K_1(\q), \tau_0}\int_{i(\mathfrak{a}_{\bm{L}_s}^{\bm{G}})^*}|h_R^{\delta,\mu}(\nu_\pi+\lambda)| \|\Delta_{\mathcal{X}_{\bm{L}_s(\underline{\beta})}}(\bm{P},\lambda)_{|\mathcal{A}^2_\pi(\bm{P})^{K_1(\q),\tau_0}}\| \, {\rm d}\lambda.
\]
We have omitted  $\mathcal{M}(\bm{P},s,\pi)$ from the norm $\|\cdot \|$, since it is a unitary operator commuting with $\Pi_{K_1(\q),\tau_0}$.

Following \cite[\S 6]{LapidMuller2009}, we now proceed to break up the sum-integral, so that $i{\rm Im}\,\nu_\pi+\lambda$ lies in $(1/R)$-balls centered at lattice points. Note that since $\pi\in\Pi_{\rm disc}(\bm{M}(\A_F)^1)_{\underline{\delta}}$, and $\underline{\delta}$ is represented by $(\delta,T_{0,\infty})$, the spectral parameter $\nu_\pi$ lies in $(\mathfrak{h}^{\bm{M}}_0)^*_\C$. As $\mathfrak{h}^{\bm{M}}_0$ is, by definition, orthogonal to $\mathfrak{a}_{\bm{M}}^{\bm{G}}$ it is orthogonal to $\mathfrak{a}_{\bm{L}}^{\bm{G}}$. We choose lattices $\Lambda\subset i(\mathfrak{h}^{\bm{M}}_0)^*$ and $\Gamma\subset i(\mathfrak{a}_{\bm{L}_s}^{\bm{G}})^*$ of covering radius $1/R$. We deduce that the previous display is bounded above by the sum over all $\mu_j\in\Lambda$ and $\nu_k\in\Gamma$ of
\[
\sum_{\substack{\pi\in\Pi_{\rm disc}(\bm{M}(\mathbb{A})^1)_{\underline{\delta}}\\ i{\rm Im}\, \nu_\pi\in B^{\bm{M}}(\mu_j,1/R)}}\dim\mathcal{A}_\pi^2(\bm{P})^{K_1(\q), \tau_0}\!\!\!\!\!\!\int\limits_{B(\nu_k,1/R)\cap i(\mathfrak{a}_{\bm{L}_s}^{\bm{G}})^*}\!\!\!\!|h_R^{\delta,\mu}(\nu_\pi+\lambda)| \|\Delta_{\mathcal{X}_{\bm{L}_s}(\underline{\beta})}(\bm{P},\lambda)_{|\mathcal{A}^2_\pi(\bm{P})^{K_1(\q),\tau_0}}\|\, {\rm d}\lambda.
\]
Note that Lemma \ref{aut-2-ind} implies, in particular, that only those $\pi$ which are spherical at infinity contribute non-trivially to the above sum.

We now use the Paley--Wiener estimate
\[
h_R^{\delta,\mu}(\nu_\pi+\lambda)\ll_N e^{R\|{\rm Re}(\nu_\pi)\|} \max_{w\in W(A_0)}(1+R\| i{\rm Im}\,\nu_\pi+\lambda-w\mu\|)^{-N},
\]
where $\lambda\in i(\mathfrak{a}_{\bm{L}_s}^{\bm{G}})^*$, as well as Lemma \ref{package}, to bound the contribution of the integral. We then execute the sum over $\Gamma$. For $N$ large enough, we obtain bound on \eqref{bound-me} of the form
\begin{equation}\label{bound-me5}
(1+\log\norm\q)^{2m_{\bm{L}_s}} \sum_{\mu_j\in\Lambda}\mathcal{L}_{N,R}(\mu_j,\mu) \sum_{\substack{\pi\in\Pi_{\rm disc}(\bm{M}(\mathbb{A})^1)_{\underline{\delta}}\\i{\rm Im}\, \nu_\pi\in B^{\bm{M}}(\mu_j,1/R)}}e^{R\|{\rm Re}(\nu_\pi)\|} \dim\mathcal{A}_\pi^2(\bm{P})^{K_1(\q), \tau_0},
\end{equation}
where, using $\mu_j^{\bm{L}_s}=\mu_j$,
\[
\mathcal{L}_{N,R}(\mu_j,\mu)=\max_{w\in W(A_0)}(1+R\|\mu_j-w\mu^{\bm{L}_s}\|)^{-N} (1+\log (1+\|\mu_j-w\mu^{\bm{L}_s}\|)+\log(1+\beta_{\bm{M}}(\mu_j)))^{2m_{\bm{L}_s}}.
\]
From Lemma \ref{aut-2-ind} it follows that if $\pi\simeq\pi_1\otimes\cdots\otimes\pi_m$ contributes non-trivially to \eqref{bound-me5} then $\prod_k\q_k\mid\q$. Writing $\mu_j=\sum_k \mu_{j,k}$, and applying the triangle inequality to $\|{\rm Re}(\nu_\pi)\|=\|\sum_k{\rm Re}(\nu_{\pi_k})\|$, the last sum is then majorized by
\[
(1+\log\norm\q)^n\sum_{\q_1\cdots\q_m\mid\q}\prod_{k=1}^mD_R^{\GL_{n_k}}(\q_k,\underline{\delta}_k,\mu_{j,k}),
\]
where $\delta_k$ is the trivial character on the diagonal torus of $\GL_{n_k}(F_\infty)$. It follows from the induction hypothesis that the spherical case of Theorem \ref{money-cor} holds for $\GL_m$, establishing Property (ELM) in that case. This then renders Proposition \ref{comp-mu} for $\GL_m$ unconditional when $\delta$ is the trivial character of $T_{0,\infty}$. We conclude that there is $C>0$ such that
\[
D_R^{\GL_{n_k}}(\q_k,\underline{\delta}_k,\mu_{j,k})\ll e^{CR}\norm\q_k^{n_k}\beta^{\GL_{n_k}}_0(\mu_{j,k}).
\]
Summing over all $\q_k$, using $n_k\leqslant n-1$ and the standard divisor bound, we deduce that \eqref{bound-me5} is bounded by
\[
e^{CR}\norm\q^{n-1+\epsilon}\sum_{\mu_j\in\Lambda}\; \mathcal{L}_{N,R}(\mu_j,\mu) \beta^{\bm{M}}_0(\mu_j).
\]
The above sum over $\Lambda$ has rapid decay away $\mu_j\in W\mu^{\bm{L}_s}$, and is therefore dominated by $(\log(1+\beta_{\bm{M}}(\mu^{\bm{L}_s}))^{2m_{\bm{L}_s}}\beta^{\bm{M}}_0(\mu^{\bm{L}_s})$. This completes the proof.
\end{proof}

\bibliographystyle{amsplain}
\bibliography{ccf-references}

\end{document}